\documentclass[a4paper]{article}
\usepackage{geometry}
\usepackage{geometry}

\usepackage[utf8]{inputenc}
\usepackage[utf8]{inputenc}
\usepackage{tikz-cd}
\usepackage{graphicx}
\usepackage{subcaption}
\usepackage{amsmath} 
\usepackage{amsmath}
\usepackage[english]{babel}
\usepackage{marvosym}
\usepackage{amssymb}
\usepackage{graphicx}
\usepackage[section]{placeins}
\usepackage{hyperref}
\usepackage{mathtools}
\usepackage[utf8]{inputenc}

\usepackage[backend=biber,style=alphabetic,giveninits=true,sorting=nyt,maxbibnames=99]{biblatex}
\DeclareFieldFormat
  [article,inbook,incollection,inproceedings,patent,thesis,unpublished]
  {title}{#1\isdot}
\DeclareFieldFormat[article,inbook,incollection,inproceedings,patent,thesis,unpublished]{title}{\textit{#1}}
\DeclareFieldFormat[article,inbook,incollection,inproceedings,patent,thesis,unpublished]{journaltitle}{#1}
\DeclareFieldFormat[article,inbook,incollection,inproceedings,patent,thesis,unpublished]{volume}{\textbf{#1}}
\DeclareFieldFormat{pages}{#1}
\renewbibmacro{in:}{}
\usepackage{hyperref}
\usepackage[english]{babel}
\usepackage{amsthm}
\usepackage{color, soul}
\newtheorem{theorem}{Theorem}[section]
\newtheorem{lemma}[theorem]{Lemma}
\newtheorem{proposition}[theorem]{Proposition}
\newtheorem{corollary}[theorem]{Corollary}

\newtheorem{definition}[theorem]{Definition}

\theoremstyle{remark}
\newtheorem{remark}[theorem]{Remark}
\newtheorem{convention}[theorem]{Convention}

\usepackage{mathtools}

\usepackage{float}
\usepackage{mathtools}
\usepackage{amsmath,amscd,amssymb,amsfonts,amsthm,hyperref}
\usepackage[all]{xy}

\title{From cascades to \texorpdfstring{$J$}{J}-holomorphic curves and back}
\newcommand{\mcal}{\mathcal}
\newcommand{\bb}{\mathbb}
\newcommand{\ep}{\epsilon}
\newcommand{\db}{\bar{\partial}}
\newcommand{\p}{\partial}
\newcommand{\dt}{\delta}
\newcommand{\la}{\langle} 
\newcommand{\ra}{\rangle }
\newcommand{\mf}{\mathbf}
\newcommand{\op}{\operatorname}
\newcommand{\cas}[1]{#1^{\text{\Lightning}}}

\usepackage{geometry}
\usepackage{fullpage}
\author{Yuan Yao}
\begin{document}
\maketitle
\begin{abstract}
    This paper develops the analysis needed to set up a Morse-Bott version of embedded contact homology (ECH) of a contact three-manifold in certain cases. In particular we establish a correspondence between ``cascades" of holomorphic curves in the symplectization of a Morse-Bott contact form, and holomorphic curves in the symplectization of a nondegenerate perturbation of the contact form. The cascades we consider must be transversely cut out and rigid. We accomplish this by studying the adiabatic degeneration of $J$-holomorphic curves into cascades and establishing a gluing theorem. We note our gluing theorem satisfying appropriate transversality hypotheses should work in higher dimensions as well. The details of ECH applications will appear elsewhere.
\end{abstract}

\tableofcontents

\section{Introduction}
Let $(Y^3, \lambda)$ be a contact 3-manifold. We assume the Reeb orbits of $\lambda$ are Morse-Bott and come in $S^1$ families, i.e. we have tori foliated by Reeb orbits, which we call Morse-Bott tori. Examples of this include the standard contact structure on the 3-torus, and boundaries of toric domains. See \cite{ECHT3}, \cite{embedding_convex_concave}. (Toric domains are also called Reinhardt domains in \cite{Hermann}.)

In this setup, one would like to make sense of Floer theoretic invariants constructed via counting $J$-holomorphic curves in the symplectization of our contact manifold, which we write as
\[
\left( Y^3\times \bb{R}, d(e^a \lambda), J\right ).
\]
In the above $a$ is the variable in the $\bb{R}$ direction, $d(e^a \lambda)$ is the symplectic form and $J$ is a (generic) almost complex structure compatible with $\lambda$.

However, most versions of Floer homology require the contact form to be non-degenerate. One way to get around this is as follows. We first fix a very large number $L>0$, and consider the action filtered version of our Floer theory up to action $L$. We will have embedded contact homology (ECH) in mind when we describe this process, but it also applies to other types of Floer theories assuming suitable transversality. For a Morse-Bott torus with action less than $L$, which we write as $\mcal{T}$, we perform a small perturbation of the contact form $\lambda$ written as
\[
\lambda \longrightarrow \lambda_\dt
\]
for $\dt>0$ small, in a small fixed neighborhood of $\mcal{T}$. Such perturbation requires the information of a Morse function $f:S^1 \rightarrow \bb{R}$, with two critical points. After this perturbation we also need to change the almost complex structure to $J_\dt$ to make it compatible with the new contact form $ \lambda_\dt$.

The effect of this perturbation is so that the Morse-Bott torus $\mcal{T}$ splits into two nondegenerate Reeb orbits corresponding to the critical points of $f$, one elliptic and one hyperbolic, and that no other Reeb orbits of action less than $L$ are introduced. We perform this perturbation for all Morse-Bott tori of action less than $L$. Then in this case, for at least up to action $L$, we can define our Floer theory with generators as collections of non-degenerate Reeb orbits with total action $<L$ and the differential as counts of $J_\dt$-holomorphic curves connecting between our generators (the details of which Reeb orbits/holomorphic curves to consider depend on whichever Floer theory we choose to work with.) 

However, it is often desirable to be able to compute our Floer theory purely in the Morse-Bott setting, in part because often the count of $J$-holomorphic curves is easier in the Morse-Bott setting. To this end, in order to find out what kind of objects that ought to be counted in the Morse-Bott setting, one can imagine turning the above process around. For given $\dt>0$, we know how to compute our Floer theory up to action $L$ with the contact form $\lambda_\dt$ via counts of a collection of $J_\dt$-holomorphic curves, then we take the limit of $\dt\rightarrow 0$, and see what kind of object our $J_\dt$-holomorphic curves degenerate into. It turns out in this process $J$-holomorphic curves degenerate into cascades \cite{BourPhd}, \cite{oancea},\cite{urslag}, \cite{SFT}. See Definition \ref{cascade_def} for the definition of a (height 1) cascade, and Definition \ref{def height k cascade} for the more general case. For the purposes of our paper we only need to consider height 1 cascades. See Section \ref{degenerations} and the Appendix for a fuller explanation of setup and more precise definition of degeneration of $J_\dt$-holomorphic curves into cascades.

Roughly speaking, a cascade $\cas{u} =\{u^1,...,u^n\}$ consists of a sequence of (not necessarily connected) $J$-holomorphic curves with ends on Morse-Bott tori. We think of the curves $u^i$ as living on different levels (for more precise definitions of level and height and their distinctions, see Definitions \ref{cascade_def} and \ref{def height k cascade}.)  Between adjacent levels, say $u^i$ and $u^{i+1}$, there is also the data of a number $T_i\in [0,\infty]$. The negative ends of $u^i$ and positive ends of $u^{i+1}$ are connected by gradient flow segments of length $T_i$. Said differently, recall each $S^1$ family of Reeb orbits is equipped with a Morse function $f$ on $S^1$, and if we start at a Reeb orbit reached by a positive puncture of $u^{i+1}$, follow the upwards gradient flow of $f$ on $S^1$ (this $S^1$ means the $S^1$ family of Reeb orbits) for time $T_i$, we will arrive at a Reeb orbit hit by a negative puncture of $u^{i}$. The Reeb orbits hit by the positive punctures of $u^1$ and negative punctures of $u^n$ are connected to Reeb orbits on the Morse-Bott tori corresponding to critical points of $f$ via the upwards gradient flow. The definition of cascade being height 1 is simply that no flow time $T_i$ between adjacent curves $u^i$ and $u^{i+1}$ is allowed to be infinite. A schematic picture of a height 1 cascade (of two levels) is given in Figure \ref{fig:cascade}.

\begin{figure}[!hb]
\centering
\includegraphics[width=.4\linewidth]{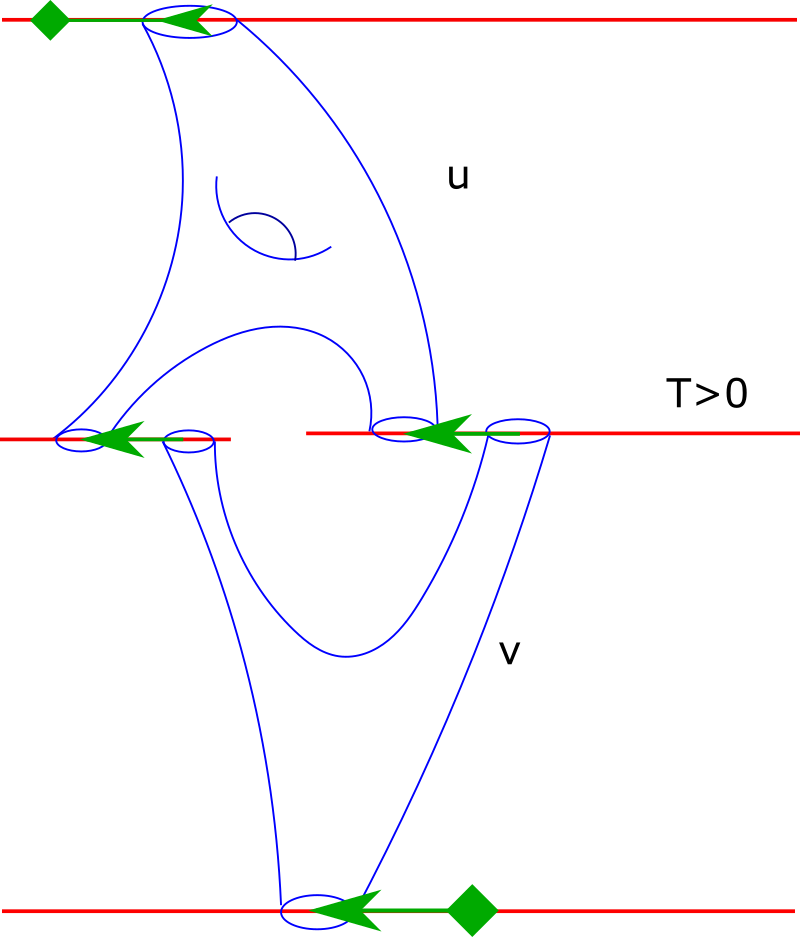}
\caption{A schematic picture of a height one 2-level  cascade: the cascade $\cas{u}$ consists of two levels, $u$ and $v$. Horizontal lines correspond to Morse-Bott tori. Moving in the horizontal direction along these horizontal lines corresponds to moving to different Reeb orbits in the same $S^1$ family. Arrows correspond to gradient flows, and diamonds correspond to critical points of Morse functions on $S^1$ families of Reeb orbits. Between the holomorphic curves $u$ and $v$, there is a single parameter $T$ that tells us how long positive ends of $v$ must follow the gradient flow to meet a negative end of $u$.}
\label{fig:cascade}
\end{figure}

We would then like a way to compute Floer homology purely in the Morse-Bott setting via enumeration of cascades. To prove that enumeration of cascades recovers the enumeration of $J_\dt$-holomorphic curves in the non-degenerate setting, we would require a correspondence theorem between the two types of objects. The correspondence theorem will of course then involve gluing cascades into $J_\dt$-holomorphic curves. We remark that we currently do not have the technology to glue together all cascades; there are issues pertaining to transversality: curves could be multiply covered, and even if they are somewhere injective and even after generic choice of $J$, there could still be non-transverse cascades because we required all negative ends of $u^i$ meet positive ends of $u^{i+1}$ after flowing for a single time length, $T_i$. In general it is convenient to think of a cascade as existing in a fiber product, and we require the fiber product to be transverse. Also we only concern ourselves with rigid cascades, and their correspondences with rigid holomorphic curves. For a more precise definition of transverse and rigid, as well as the description of this fiber product, see Definition \ref{def_rigid_cas}. Our version of the gluing theorem should work for gluing higher index (transversely cut out) cascades, but making sense of a correspondence between two high dimensional moduli spaces could be much trickier. With the above preamble we state in a slightly imprecise way our main theorem:

\begin{theorem}
Given a transverse and rigid height one $J$-holomorphic cascade $\cas{u}$, it can be glued to a rigid $J_\dt$-holomorphic curve $u_\dt$ for $\dt>0$ sufficiently small. The construction is unique in the following sense: if $\{\dt_n\}$ is a sequence of numbers that converge to zero as $n\rightarrow \infty$, and $u'_{\dt_n}$ is sequence of $J_{\dt_n}$-holomorphic curves converging to $\cas{u}$, then for large enough $n$, the curves $u_{\dt_n}'$ agree with $u_{\dt_n}$ up to translation in the symplectization direction.
\end{theorem}
See Definition \ref{def_rigid_cas} for the description of  ``transverse and rigid''. See Theorem \ref{main_theorem} for a more precise formulation of this theorem.

\begin{remark}
The purpose of Morse-Bott theory is usually that $J$-holomorphic curves are often more easily enumerated in the Morse-Bott setting due to presence of symmetry. While cascades only contain $J$-holomorphic curves in the Morse-Bott situation, counting them explicitly can be difficult in its own way. Even though rigid and transverse cascades are themselves discrete, they may be built out of curves that live in high dimensional moduli spaces. Since in principle arbitrarily high dimensions of moduli spaces can show up, one usually needs some extra simplifications for the enumeration of cascades to be tractable.
\end{remark}

\begin{remark}
Since we will have future applications to ECH in mind, we make some comments about our ``transverse and rigid'' condition versus the ECH index 1 condition:
\begin{itemize}
    \item In general restricting to cascades that have ECH index one (of course one first needs to extend the notion of ECH index one to cascades) and choosing a generic $J$  does not necessarily imply the cascades we get are transversely cut out. However there are special cases where transversality can be achieved by restricting to ECH index one cascades, and the correspondence theorem (Theorem \ref{main_theorem}) would allow us to compute ECH using an enumeration of $J$-holomorphic cascades. Work in this direction is forthcoming in \cite{Tips}.
    \item If we already had cascades that are transverse and rigid, from a gluing point of view, further restricting to the cascades that have ECH index one does not change very much: it just implies all the curves in the cascade are embedded (with the exception of unbranched covers of trivial cylinders) and distinct curves within each level do not intersect each other. We further have some partition conditions on the ends of holomorphic curves in the cascade, but again from a gluing point of view this does not make a difference.
\end{itemize}
\end{remark}

\subsection{Relations to other work}

The idea of doing Morse-Bott homology certainly isn't new. Methods of working with Morse-Bott homology predate the construction of cascades, and were described in \cite{austinbraam}, \cite{Fukaya}. The construction of cascades was discovered independently in \cite{BourPhd} and \cite{urslag}. There were then a plethora of constructions of Floer-type theories using cascades (or in many cases, constructions very similar to cascades). For Lagrangian Floer theory, in addition to \cite{urslag}, there was also \cite{biran2007quantum}. For symplectic homology, see \cite{oancea}. See also \cite{ohke}. For Morse homology, see \cite{Morse-cas}, \cite{3approach}. For special cases of contact homology, see \cite{nelhuthyper},\cite{Nelauto}. For the special case of ECH where the cascades can only have one level, see \cite{colin2021embedded}. For abstract perspectives on Morse-Bott theory, see \cite{zhengyi}, \cite{JoMiAx}. Finally, the gauge theory analogue of ECH, monopole Floer homology, has a Morse-Bott version constructed in \cite{lin}, though there they do not use a cascade model.

For cascades there are two general approaches to show the Morse-Bott homology theory constructed agrees with the original homology theory. One way is to show the differential obtained via counts of cascades squares to zero, hence one has \emph{some} homology theory. Then one shows that this homology theory is isomorphic to the original by constructing a cobordism interpolating the Morse-Bott geometry and the non-degenerate geometry. For standard Floer theory reasons this cobordism induces a cobordism map between the two homology groups. Also for standard Floer theory reasons we could show this cobordism map induces an isomorphism on homology. This is the approach taken in \cite{urslag}, \cite{biran2007quantum}.

The other approach is to directly show that non-degenerate holomorphic curves degenerate into cascades in the $\dt\rightarrow 0$ limit, and there is a correspondence between cascades and holomorphic curves. This degeneration of holomorphic curves into cascades is also sometimes called the adiabatic limit. This approach of computing Morse-Bott homology is taken in \cite{oancea} \cite{BourPhd} \cite{Morse-cas} \cite{ohke}. This is also the approach we take here. We prove the correspondence theorem under transversality assumptions (Definition \ref{def_rigid_cas}), and will take up applications to ECH in a separate paper \cite{Tips}.

The reason we take the latter approach is that in ECH, which is the application we have in mind, everything except transversality is very hard. That the differential squares to zero requires 200 pages of obstruction bundle gluing calculations \cite{obs1} \cite{obs2}, and a similar story must be repeated in the Morse-Bott case for showing the count of ECH index 1 cascades defines a chain complex. Constructing cobordism maps in ECH is even harder, and generally relies on passing to Seiberg-Witten theory. Cobordism maps on ECH defined purely using holomorphic curves techniques have only been worked out for very special cases \cite{rooney2020cobordism}, \cite{chenthesis},\cite{Gerig1},\cite{Gerig2}. Hence in light of these difficulties, it would seem the path of least resistance would be to prove a correspondence theorem and do the adiabatic limit analysis for ECH, despite this being a generally difficult approach.

We must highlight the relation of our work with \cite{oancea}, which produces a correspondence theorem in the case of symplectic homology. We borrowed heavily the techniques of that paper in the areas of analysis of linear operators over gradient flow trajectories (most notably the construction of uniformly bounded right inverses in the $\dt \rightarrow 0$ limit), as well as the degeneration of holomorphic curves into gradient trajectories near Morse-Bott tori. Both of these ideas have previously appeared in \cite{BourPhd} but were worked out in more detail in \cite{oancea}. However, our construction of gluing is markedly different from \cite{oancea}, as we were unfortunately unable to adapt their approach. Instead, our approach of both gluing and proving the gluing procedure produces a bijection between cascades and holomorphic curves mirrors the approach of \cite{obs1} \cite{obs2}, the two papers where Hutchings and Taubes show the differential in ECH squares to zero using obstruction bundle gluing. In particular, our approach can in fact be rephrased in terms of obstruction bundle gluing, see Remark \ref{obsrmk}, though in our case the obstruction bundle gluing is particularly simple and can be thought of as an application of the intermediate value theorem. For a formulation of this kind of gluing results in a simpler case in ECH where there is only 1 level in our cascades using obstruction bundle gluing, see the Appendix of \cite{colin2021embedded}, which we wrote jointly with Colin, Ghiggini and Honda.

We briefly outline the differences between our approach to gluing compared to those in \cite{oancea}, \cite{ohke}. In \cite{oancea}, \cite{ohke}, the gradient trajectories connecting different levels of the cascade are preglued to the $J$-holomorphic curves in the cascade; they consider the deformations of the entire preglued curve, and use the implicit function theorem to obtain gluing results. In our approach, in following the approach of \cite{obs1}, \cite{obs2}, the condition that a cascade can be glued to a $J_\dt$-holomorphic curve is translated into a \emph{system} of coupled nonlinear PDEs, which we loosely write as $\{\mathbf{\Theta}_i=0\}$. Gluing is established by systematically solving this system of PDEs. How this is accomplished is explained first in a simplified setting in Section \ref{semiglue}, then in the general case in Section \ref{gluing}. 

To say a bit more about this system of PDEs, we note that in this system there is a PDE for each $J$-holomorphic curve that appears in the cascade, and a PDE for each upwards gradient trajectory. In some sense this allows us to think about deformations of $J$-holomorphic curves and deformations of gradient trajectories separately from each other (of course in the end the equations are coupled, so this is only metaphorically true). The point is that in considering cascades, gradient trajectories and $J$-holomorphic curves are in some sense different kinds of objects. For small values of $\dt>0$, measured in a suitable norm, $J$-holomorphic curves in the cascade are very close to being $J_\dt$-holomorphic curves in the perturbed picture, but the gradient flow trajectories in the cascade come from very long gradient flow cylinders (their lengths go to $\infty$ as $\dt \rightarrow +\infty)$ that follow a very slow gradient flow (the gradient flow of $\dt f$). For a description of these cylinders see Section \ref{diffgeo}. The consequence of this is that deformations that appear to be small from the perspective of $J$-holomorphic curves in the cascade can be extremely large from the perspective of gradient flow cylinders in the cascade. See Figure \ref{fig:displace_figure} and the accompanying explanations. Hence one benefit of our approach in writing down a system of equations $\{\mathbf{\Theta}_i=0\}$ is that then it becomes easy to keep track of which deformations are very large, and hence easy to understand the effects of these deformations on the equations in $\{\mathbf{\Theta}_i=0\}$ and the way different equations in the system $\{\mathbf{\Theta}_i=0\}$ are coupled to each other.

\begin{figure}[!h]
\centering
\begin{subfigure}{.5\textwidth}
  \centering
  \includegraphics[width=.7\linewidth]{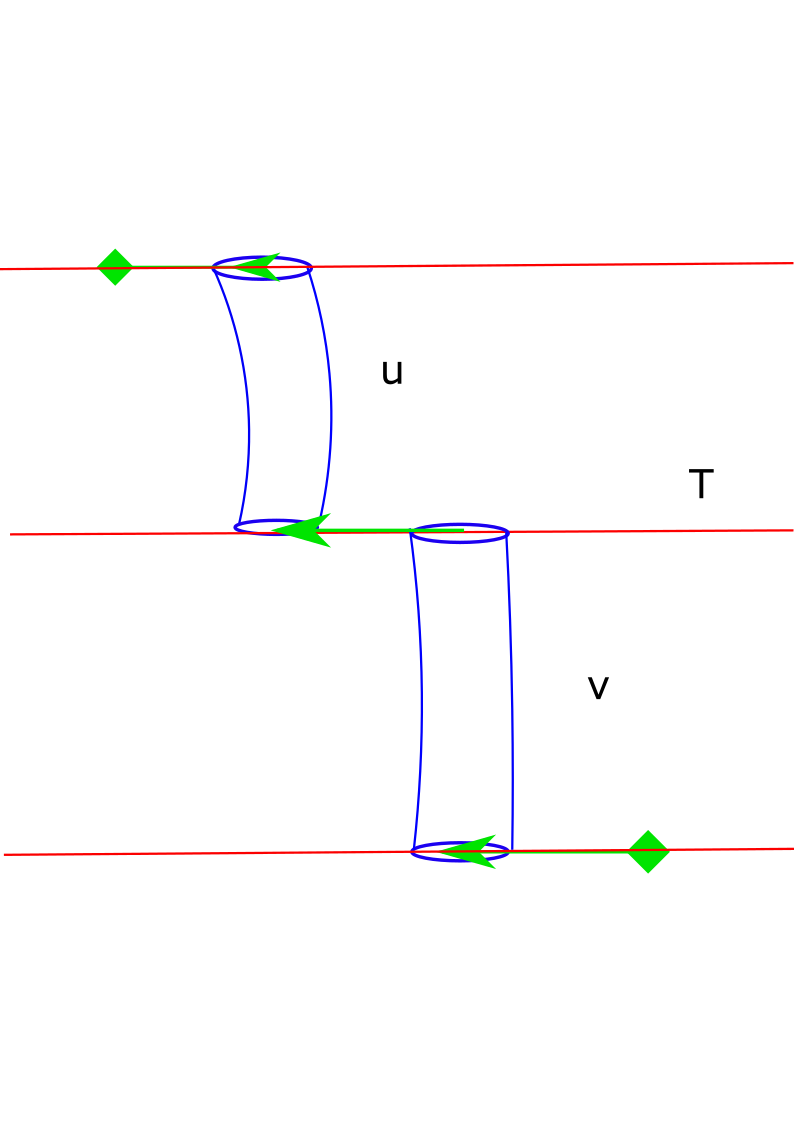}
  \label{fig:sub1}
\end{subfigure}%
\begin{subfigure}{.5\textwidth}
  \centering
  \includegraphics[width=.7\linewidth]{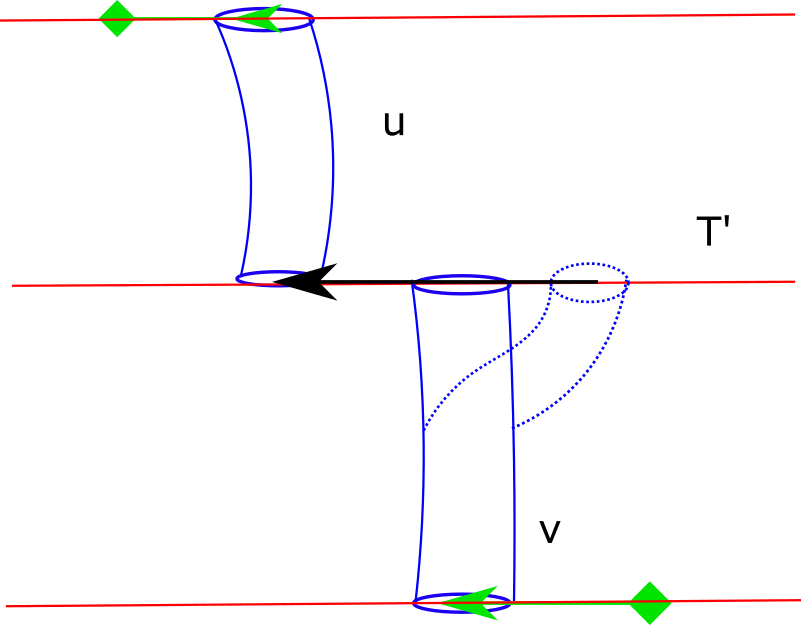}
  \label{fig:sub2}
\end{subfigure}
\caption{On the left we see a cascade of two levels consisting of the curves $\{u,v\}$. They meet along a Morse-Bott torus in the middle, and their ends are connected by a gradient flow trajectory of length $T$, shown by the green arrow. On the right, we imagine slightly displacing the end of the curve $v$ along the Morse-Bott torus, shown in dashed blue lines. From the perspective of $v$, measured with appropriate norms this is a small deformation of $v$. The flow time from this deformed $v$ to $u$ is given by $T' =T +\Delta T$, which is a slightly longer flow time, indicated by the black arrow. However the picture is deceiving, because for small values of $\dt$, the gradient flow cylinder corresponding to the black arrow is significantly longer than the gradient cylinder of the green arrow in the original picture, by an additional length of order $\Delta T /\dt$. This is because for small values of $\dt$, the gradient flow is very slow (it follows the gradient of $\dt f$), hence it needs to flow for very long to cover that extra distance. This is what we mean when we say deformations that can seem very small from the perspective of $J$-holomorphic curves in the cascade can be arbitrarily large from the perspective of gradient trajectories.}
\label{fig:displace_figure}
\end{figure}

Finally, we remark that despite only working with Morse-Bott tori, we expect our approach to work for most Floer theories based on counts of holomorphic curves that do not have multiple covers or issues with transversality (both in the non-degenerate setting and the Morse-Bott setting). We expect the generalization from Reeb orbits showing up in $S^1$ families to higher dimensional families to be straightforward, and the rest of the analysis should carry over directly. However, we do not know how our analysis or proof of correspondence theorem interact with virtual techniques that are often used to define Floer theories when classical transversality methods fail.

\subsection{Applications to ECH}
As mentioned above the main application we have in mind of this work is the computation of Embedded Contact homology in the Morse-Bott setting.
Previously several computations of ECH (or its related cousin Periodic Floer homology) have assumed results about Morse-Bott theory and cascades, for instance computations in \cite{pfh}, \cite{ECHT3}, \cite{choi2016combinatorial}.

This paper does not contain the full construction of Morse-Bott ECH, but the analysis done here will lay the groundwork for constructing a correspondence theorem for ECH index 1 cascades and ECH index 1 holomorphic currents. This is the subject of a forthcoming paper, see \cite{Tips}.

However at this juncture we make several remarks about the construction. The condition of ``transverse and rigid'' for a cascade does not hold in general, even if we restricted to cascades that have ECH index 1 and used a generic $J$. Hence in general ECH index one cascades might apriori have multiply covered components due to lack of transversality coming from the fiber products we used to define the cascades. However in simple cases where all ECH index one curves have genus zero, there is indeed enough transversality, and we expect the machinery developed here and \cite{Tips} to fill in the foundations for Morse-Bott ECH for the computations in \cite{pfh}, \cite{ECHT3}, \cite{choi2016combinatorial}, in which all ECH index one curves are shown to have genus zero.

For contact 3-manifolds, Morse-Bott degeneracy might also mean the Morse-Bott critical manifold is two dimensional, which means the manifold itself is foliated by Reeb orbits. The computation of ECH in that case was done using different techniques, see \cite{nelson2020embedded}, \cite{farris}. However, our methods (suitably extended to allow for the case where Reeb orbits can come in higher dimensional families) could potentially be applied to ECH computations in these cases as well.

\subsection{Outline}
The paper is organized as follows. After some quick description of the geometric setup, we describe in Section \ref{degenerations} how holomorphic curves in the non-degenerate case degenerate into objects we call cascades, and introduce a version of SFT type compactness, already introduced in \cite{SFT}, \cite{BourPhd}. We relegate the more technical definitions of convergence and proof of degeneration into cascades to the Appendix for the sake of exposition.

In Section \ref{transversality} we establish what we mean by generic choice of $J$, the definition of transversality, and in particular describe the set of cascades we will be able to glue into $J$-holomorphic curves.

Then we get to the most technical part of the paper, in which we prove transverse and rigid $J$-holomorphic cascades can be glued to $J_\dt$-holomorphic curves as we perturb the contact form. We first start with some preamble on differential geometry in Section \ref{diffgeo}, and describe the gradient trajectories that arise from perturbing the contact form. We then find a suitable Sobolev space for the gradient trajectories which we will use for our gluing, and prove some nice properties of the linearized Cauchy Riemann operator in this Sobolev space for later use in Section \ref{linearization}.

To initiate the gluing, first as a warm up we explain how to glue a semi-infinite gradient trajectory to a $J$-holomorphic curve in Section \ref{semiglue}. This corresponds to gluing 1-level cascades to $J_\dt$-holomorphic curves as we perturb the contact form from $\lambda$ to $\lambda_\dt$, which is also done in \cite{colin2021embedded}. We then prove an important property of the curve we constructed in this process, i.e. the solution exponentially decays along the gradient trajectory. This is done in Section \ref{expdecay}, and will be crucial for gluing together multiple level cascades.

Section \ref{gluing}, we complete the gluing construction. We first consider the simplified case of gluing together 2-level cascades, which will contain the heart of the construction and is markedly different from gluing semi-infinite trajectories. As before we first do some basic Sobolev space setup. The key idea is to first preglue, then use the solution constructed for semi-infinite trajectories to construct another pregluing on top of the original pregluing with substantially smaller pregluing error, and then use the contraction mapping principle one last time to turn the second pregluing into a genuine gluing. As illustrated in Figure \ref{fig:displace_figure}, during the $\dt \rightarrow 0$ degeneration the gradient flow cylinders correspond to very long necks, and when we try to preglue a cascade, deformations that appear small from the perspective of $J$-holomorphic curves in the cascade can become very large from the perspective of the preglued curve when we try to fit a gradient flow cylinder between adjacent levels of the cascade during the pregluing. So all of the complications in the gluing we mentioned above arise from trying to keep track of these deformations and finding a setup where all of the vectors that we see are sufficiently small, so that the contraction mapping principle can be applied.

After this, the generalization to multiple level cascades is mostly a matter of keeping track of notation.

In anticipation of proving bijectivity of gluing, we deduce some analytic estimates of how $J_\dt$-holomorphic curves behave near Morse-Bott tori as we degenerate the contact form $\lambda_\dt$. This is done in Section \ref{behav}. Much of this analysis is taken from the appendix of \cite{oancea} where they work out a very similar case in symplectic homology. This kind of analysis has also appeared in \cite{BourPhd}.

Finally we take up the bijectivity of gluing; for this step we largely follow the footsteps of \cite{obs1} \cite{obs2}. This is taken up in Section \ref{surj}. 

The appendix contains the necessary background to state the SFT compactness theorem required for our kind of degenerations, which was stated in \cite{SFT} and proved in \cite{BourPhd}. A similar result also appears in \cite{oancea}. We also provide a proof for completeness, which relies also on the analysis done in Section \ref{behav}.

\subsection{Acknowledgements}
First and foremost I would like to thank my advisor Michael Hutchings for his consistent help and support throughout this project. I would also like to thank Ko Honda, Jo Nelson, Alexandru Oancea, Katrin Wehrheim and Chris Wendl for helpful discussions and comments.

I would like to acknowledge the support of the Natural Sciences and Engineering Research Council of Canada (NSERC), PGSD3-532405-2019.
Cette recherche a été financée par le Conseil de recherches en sciences naturelles et en génie du Canada (CRSNG), PGSD3-532405-2019.

\section{Morse-Bott setup and SFT type compactness} \label{degenerations}
Let $(Y^3,\lambda)$ be a contact 3-manifold with Morse-Bott contact form $\lambda$. Throughout we assume all Reeb orbits come in $S^1$ families; hence we have tori foliated by Reeb orbits, which we call Morse-Bott tori.

\begin{convention}
Throughout this paper we fix a large number $L>0$, and only consider collections of Reeb orbits that have total action less than $L$. This is implicit in all of our constructions and will not be mentioned further. We prove the correspondence theorem between cascades and $J_\dt$-holomorphic curves up to action level $L$, and then when we need to apply this construction to Floer theories we can take $L\rightarrow \infty$.
\end{convention}

The following theorem, which is a special case of a more general result in  \cite{oh_wang_2018}, gives a characterization of the neighborhood of Morse-Bott tori. Let $\lambda_0$ denote the standard contact form on $(z,x,y) \in S^1\times S^1 \times \mathbb{R}$ of the form
\[
\lambda_0= dz-ydx.
\]
\begin{proposition} \label{prop_locform} 
$(M^3,\lambda)$ be a contact 3 manifold with Morse-Bott contact form $\lambda$. We assume the Morse-Bott Reeb orbits come in families of tori, which we write as $\mcal{T}_i$, with minimal period $T_i$. Then we can choose coordinates around each Morse-Bott torus so that a neighborhood of $\mcal{T}_i$ is described by $(z,x,y) \in S^1\times S^1 \times (-\ep,\ep)$, and the contact form $\lambda$  in this coordinate system looks  like
\[
\lambda = h(x,y,z) \lambda_0
\]
where $h(x,y,z)$ satisfies
\[
h(x,0,z)=1, \quad dh(x,0,z) =0.
\]
Here we identify $z\in S^1 \sim \bb{R}/2\pi T_i \bb{Z}$.
\end{proposition}
\begin{proof}
This is implicitly in \cite{oh_wang_2018}. We need to apply the setup of \cite{oh_wang_2018} Theorem 4.7 to \cite{oh_wang_2018}, Theorem 5.1.

In \cite{oh_wang_2018} Theorem 4.7, in their notation we have $E=0$, $Q = S^1\times S^1$, with coordinates $(z,x)$, and $\theta =dz$. The foliation $\cal{N}$ is given by $\{z\} \times S^1$. The contact form on the total space of the fiber bundle $F=T^*\mcal{N} = \mathbb{R} \times T^2$ is given by $dz+ydx$. Our proposition then follows from Theorem 5.1 in \cite{oh_wang_2018} (we need to take another transformation $y\rightarrow -y$ to get our specific choice of contact form, our sign conventions for the contact form are different from those of \cite{BourPhd}.)
\end{proof}

We assume we have chosen above neighborhoods around all Morse-Bott tori ${\mcal{T}_i}$ with action less than $L$. By the Morse-Bott assumption there are only finitely many such tori up to fixed action $L$. Next we perturb them to nondegenerate Reeb orbits by perturbing the contact form in a neighborhood of each torus. 

Let $\delta>0$, let $f:x\in \mathbb{R}/\bb{Z} \rightarrow R$ be a smooth Morse function with max at $x=1/2$ and minimum $x=0$. Let $g(y):\bb{R} \rightarrow \bb{R}$ be a bump function that is equal to $1$ on $[-\ep_{\mcal{T}_i},\ep_{\mcal{T}_i}]$ and zero outside $[-2\ep_{\mcal{T}_i},2\ep_{\mcal{T}_i}]$. Here $\ep_{\mcal{T}_i}$ is a number chosen for each $\mcal{T}_i$ small enough so that the normal form in the above theorem applies, and that all such chosen neighborhoods of Morse-Bott tori of action $<L$ are disjoint. Then in a neighborhood of the Morse Bott torus $\mcal{T}_i$, we perturb the contact form as
\[
\lambda \longrightarrow \lambda_\delta:= e^{\delta gf}\lambda.
\]
We can describe the change in Reeb dynamics as follows:
\begin{proposition}\label{prop:perturb}
For fixed action level $L>0$ there exists $\delta>0$ small enough so that the Reeb dynamics of $\lambda_\delta$ can be described as follows. In the neighborhood specified by Proposition \ref{prop_locform}, each Morse-Bott torus splits into two non-degenerate Reeb orbits corresponding to the two critical points of $f$. One of them is hyperbolic of index $0$, the other is elliptic with rotation angle $|\theta| <C\dt <<1$ and hence its Conley-Zehnder index is $\pm 1$. There are no additional Reeb orbits of action $<L$.
\end{proposition}
\begin{definition}
We say an Morse-Bott torus is \textbf{positive} if the elliptic Reeb orbit has Conley Zehnder index 1 after perturbation, otherwise we say it is \textbf{negative} Morse Bott torus. This condition is intrinsic to the Morse-Bott torus itself, and is independent of perturbations.
\end{definition}
\begin{proof}[Proof of Proposition \ref{prop:perturb}]
After we have fixed our local neighborhood near a Morse-Bott torus from Proposition \ref{prop_locform}, we get natural trivializations of the contact plane along the Morse-Bott torus given by the $x-y$ plane.  With this trivialization in mind, the linearized return map takes either of the following forms \footnote{This fact is referenced in Section 4 of \cite{colin2021embedded}, and Section 5 of \cite{beyondech}. Section 3 of the paper \cite{hwzproperties4} works out the detailed computation leading up to this result. We remark that in all of these three papers the linearized return map is lower triangular. This is because we have chosen different conventions. For instance in \cite{hwzproperties4} they chose their $y$ coordinate to denote the $S^1$ family of Reeb orbits, and their $x$ coordinate to denote the normal direction to their Morse-Bott torus. Hence their contact form is written as $dz+xdy$. Our linearized return maps agree with theirs after we switch to their coordinate system.}
\begin{itemize}
    \item  Positive Morse-Bott Torus: $\phi(t) = \begin{bmatrix}
    1 & -t \\
      0    &1 \\

\end{bmatrix}.$ 

    \item Negative Morse-Bott torus: $\phi(t) = \begin{bmatrix}
    1 & t \\
    0      &1 \\
\end{bmatrix}.$ 
\end{itemize}
They are degenerate, but they admit a Robbin-Salamon index, see Section 4 of \cite{Gutt2014}. The positive Morse-Bott torus has Robbin-Salamon index $1/2$ and the negative Morse-Bott torus has Robbin-Salamon index $-1/2$ (see \cite{Gutt2014}, Proposition 4.9). Then the claims behaviour of Reeb orbits follow from Lemmas 2.3 and 2.4 in \cite{BourPhd}.
\end{proof}

\begin{remark}
Later when we define various terms in the Fredholm index, they will depend on choices of trivializations of the contact structure along the Reeb orbits. We will always choose the trivializations specified by Proposition \ref{prop_locform}, and where the return maps take the form specified above. For notational convenience we will call this trivialization $\tau$.

We also observe that after iterating the Reeb orbits in the Morse-Bott tori, their Robbin-Salamon indices stay the same. So up to action $L$, in the nondegenerate picture, we will only see Reeb orbits of Conley-Zehnder indices in the set $\{-1,0,1\}$.
\end{remark}

Let us consider for small $\delta>0$ the symplectization
\[
(M^4, \omega_\delta) := (\bb{R} \times Y^3,de^a\lambda_\delta).
\]

We also consider the symplectization in the Morse-Bott case
\[
(M^4, \omega_0 ):= (\bb{R} \times Y^3,de^a\lambda).
\]

We fix our conventions for almost complex structures for the rest of the article as follows:

\begin{convention}
We equip $(M,\omega_0)$ with $\lambda$ compatible almost complex structure $J$ (for purposes of tranversality Definition \ref{def_rigid_cas} we may want to take $J$ to be generic). We restrict $J$ to take the following form near a neighborhood of each Morse-Bott torus (if we are using the action filtration we can only require this condition for Morse-Bott tori up to action $L$). Recall each Morse-Bott torus has neighborhood described by $(a,z,x,y) \in \bb{R} \times S^1\times S^1\times \bb{R}$, then on the surface of the Morse-Bott torus, i.e. $y=0$, we require
\[
J\p_x =\p_y.
\]

Our requirement for $J_\dt$ is that it is $\lambda_\dt$ compatible, and in a neighborhood of each Morse-Bott torus (resp. Morse-Bott tori up to action $L$), its restriction to the contact distribution agrees with the restriction of $J$. See Remark \ref{remark_generic_J} for additional comments for genericity.

\end{convention}

For fixed $L>0$ large and $\dt >0$ small enough, all collections of orbits with total action less than $L$ are non-degenerate, and hence there are corresponding $J$-holomorphic curves with energy less than $L$ with non-degenerate asymptotics. To motivate our construction, we next take $\delta \rightarrow 0$ to see what these $J$-holomorphic curves degenerate into. By a theorem that first appeared in Bourgeois' thesis \cite{BourPhd} (Chapter 4) and also stated in \cite{SFT} (Theorem 11.4), they degenerate into $J$-holomorphic cascades. (For a more careful definition see the appendix that takes into account of stability of domain and marked points, but the definition here suffices for our purposes).

\begin{definition} \label{cascade_def}
\cite{BourPhd}
Let $\Sigma$ be a punctured (nodal) Riemann surface, potentially with multiple components.
A cascade of height 1, which we will denote by $\cas{u}$, in $(\bb{R}\times Y^3,\lambda,J)$ consists of the following data :
\begin{itemize}
    \item A labeling of the connected components of $\Sigma ^*=\Sigma \setminus \{ \text{nodes} \}$ by integers in
    $\{1, . . . , l\}$, called levels, such that two components sharing a node have levels differing by at most 1. We denote by $\Sigma_i$ the union of connected components of level $i$, which might itself be a nodal Riemann surface.
    \item $T_i \in [0,\infty)$ for $ i = 1, . . . , l - 1$.
    \item  $J$-holomorphic maps $u^i: (\Sigma_i, j) \rightarrow (\bb{R}\times Y^3, J)$ with $E(u_i) < \infty$ for $ i = 1, . . . , l$, such that:
    \begin{itemize}
        \item Each node shared by $\Sigma_i$ and $\Sigma_{i+1}$, is a negative puncture for $u^i$ and is a positive puncture for $u^{i+1}$. Suppose this negative puncture of $u^i$ is asymptotic to some Reeb orbit $\gamma_i \in \mcal{T}$, where $\mcal{T}$ is a Morse-Bott torus, and this positive puncture of $u^{i+1}$ is asymptotic to some Reeb orbit $\gamma_{i+1} \in \mcal{T}$, then we have that $\phi^{T_i}_f(\gamma_{i+1}) = \gamma_{i}$. Here $\phi^{T_i}_f$ is the upwards gradient flow of $f$ for time $T_i$. It is defined by solving the ODE
        \[
        \frac{d}{ds} \phi_f(s) = f'(\phi_f(s)).
        \]
        \item $u^i$ extends continuously across nodes within $\Sigma_i$.
        \item No level consists purely of trivial cylinders. However we will allow levels that consist of branched covers of trivial cylinders.
    \end{itemize}
\end{itemize}
With $\cas{u}$ defined as above, we will informally write $\cas{u} = \{u^1,..,u^l\}$.
\end{definition}
\begin{convention}
We fix our conventions as follows. 
\begin{itemize}
    \item We say the punctures of a $J$-holomorphic curve that approach Reeb orbits as $a\rightarrow \infty$ are positive punctures, and the punctures that approach Reeb orbits as $a\rightarrow -\infty$ are negative punctures. We will fix cylindrical neighborhoods around each puncture of our $J$-holomorphic curves, so we will use ``positive/negative ends'' and ``positive/negative punctures'' interchangeably.  By our conventions, we think of $u^1$ as being a level above $u^2$ and so on.
    \item We refer to the Morse-Bott tori $\mcal{T}_j$ that appear between adjacent levels of the cascade $\{u^i,u^{i+1}\}$ as above, where negative punctures of $u^i$ are asymptotic to Reeb orbits that agree with positive punctures from $u^{i+1}$ up to a gradient flow, \textit{intermediate cascade levels}.
    \item We say that the positive asymptotics of $\cas{u}$ are the Reeb orbits we reach by applying $\phi_f^\infty$ to the Reeb orbits hit by the positive punctures of $u^1$. Similarly, the negative asymptotics of $\cas{u}$ are the Reeb orbits we reach by applying $\phi_f^{-\infty}$ to the Reeb orbits hit by the negative punctures of $u^l$. We note if a positive puncture (resp. negative puncture) of $u^1$ (resp. $u^l$) is asymptotic to a Reeb orbit corresponding to a critical point of $f$, then applying $\phi^{+\infty}_f$ (resp. $\phi_f^{-\infty}$) to this Reeb orbit does nothing.
\end{itemize}
\end{convention}

\begin{definition}[\cite{BourPhd}, Chapter 4] \label{def height k cascade}
A cascade of height $k$ consists of $k$ height 1 cascades, $\cas{u}_k =\{u^{1\text{\Lightning}},...,u^{k\text{\Lightning}}\}$ with matching asymptotics concatenated together. By matching asymptotics we mean the following. Consider adjacent height one cascades, $u^{i\text{\Lightning}}$ and $u^{i+1\text{\Lightning}}$. Suppose a positive end of the top level of $u^{i+1\text{\Lightning}}$ is asymptotic to the Reeb orbit $\gamma$ (not necessarily simply covered). Then if we apply the upwards gradient flow of $f$ for infinite time we arrive at a Reeb orbit reached by a negative end of the bottom level of $u^{i\text{\Lightning}}$. We allow the case where $\gamma$ is at a critical point of $f$, and the flow for infinite time is stationary at $\gamma$. We also allow the case where $\gamma$ is at the minimum of $f$, and the negative end of the bottom level of $u^{i\text{\Lightning}}$ is reached by following an entire (upwards) gradient trajectory connecting from the minimum of $f$ to its maximum. If all ends between adjacent height one cascades are matched up this way, then we say they have matching asymptotics.

We will use the notation $\cas{u}_k$ to denote a cascade of height $k$. We will mostly be concerned with cascades of height 1 in this article, so for those we will drop the subscript $k$ and write $\cas{u} = \{u^1,...,u^l\}$.
\end{definition}
\begin{remark}
In this paper our families of Reeb orbits are parameterized by $S^1$, and in particular there are no broken gradient flow lines on $S^1$. In general, when the critical manifold (the manifold that parameterizes the Morse-Bott family of Reeb orbits) is more complicated, the notion of matching asymptotics between height one cascades mentioned in the above definition involves going from a Reeb orbit hit by a positive puncture of the top level of $u^{i+1\text{\Lightning}}$ to a Reeb orbit hit by a negative puncture of the bottom level of $u^{i\text{\Lightning}}$ via broken Morse trajectories on the critical manifold.
\end{remark}
\begin{remark}
Once we have given the definition of cascades, we must then describe what it means for two cascades to be equivalent to each other. The precise definition of when two cascades are equivalent to one another can only be more precisely stated after we have given the more precise definition of cascades in the Appendix, where we keep track of all of the marked points and punctures of each level. Essentially we simply need to adapt the definition of when SFT buildings are equivalent to one another as stated in \cite{SFT} Section 7.2 by viewing  gradient flow trajectories in cascades as extra levels.
Here we just remark that for our gluing purposes this is not really an issue for us, all of the cascades we care about (see Definition \ref{def_rigid_cas}) will have $u^i:\Sigma_i \rightarrow \bb{R} \times Y$ be somewhere injective $J$-holomorphic curves, with the possible exception of unbranched covers of trivial cylinders, hence for us it will be obvious when two cascades are equivalent to one another.
\end{remark}

Now we state informally our version of the SFT compactness theorem, the full version with a precise definition of convergence is stated in the Appendix.

\begin{theorem}
A sequence of $J_\dt$-holomorphic curves $\{u_{\dt_n}\}$ that have fixed genus, are asymptotic to the same non-degenerate Reeb orbits, and $\dt_n \rightarrow 0$, has a subsequence that converges to a $J$-holomorphic cascade of height $k$. 
\end{theorem}

\begin{remark}
It is apparent, with the definition of convergence outlined in the Appendix, that if $u_{\dt_n}$ converges to a cascade $\cas{u_k}$ of height $k$, and all the curves in the cascade are somewhere injective (except unbranched covers of trivial cylinders), then this limit $\cas{u}_k$ is unique up to equivalence.
\end{remark}
\section{Transversality} \label{transversality}
In this section we describe the necessary transversality hypothesis we need for gluing and the correspondence theorem. 

We fix a metric $g$ that is invariant under $\bb{R}$ which we shall use for linearization purposes. We require that it is of the form
\[
g= da^2 +dx^2+dy^2+dz^2
\]
in a neighborhood of each Morse-Bott torus.

We also note the following convention that will be followed throughout this paper:
\begin{convention}
Since we will be doing a lot of gluing in the paper there is a lot of demand for various cut off functions. We fix once and for all our convention for cut off functions. We use the notation $\beta_{a;b,c;d}: \bb{R} \rightarrow \bb{R}$ to denote a function with support in $(b,c)$, all of its derivatives are also supported in this interval. $\beta_{a;b,c;d}$ is equal to $1$ on the interval $(b+a,c-d)$, and over the interval $(a,b+a)$ it satisfies a derivative bound of the form $|\beta'(s) |\leq C/(a)$, and likewise for the interval $(c-d,c)$.

If we would want cut off functions that are equal to $1$ at either $\pm \infty$, we will write $\beta_{-\infty,c;d}$ or $\beta_{a;b,\infty}$. The behaviour of the cut off function on intervals $(c-d,c)$ (resp. $(a,b+a)$) is the same as the above paragraph.
\end{convention}

Let $u: \dot{\Sigma} \rightarrow ( Y\times \bb{R}, \lambda)$ denote a holomorphic curve from a punctured Riemann surface $\dot{\Sigma}$ with $N_\pm$ positive (resp. negative) punctures labeled $p_j^\pm$, the collection of which we denote by $\Gamma_\pm$.
For each puncture $p_j^\pm$ we fix cylindrical neighborhoods around the puncture of the form $(s,t) \in [0,\pm \infty) \times S^1$. 
The punctures of $\dot{\Sigma}$ are asymptotic to Reeb orbits on Morse-Bott tori. There are moduli spaces, which we generally write as $\mcal{M}$, of $J$-holomorphic maps from $\dot{\Sigma} \rightarrow (Y\times \bb{R}, \lambda)$, which can be specified as follows. For given puncture $p_j^\pm$, we first specify which Morse-Bott torus $\mcal{T}_j^\pm$ that it lands on and with what multiplicity it covers the Reeb orbits on that Morse-Bott torus. Then we have the option of specifying whether this end is ``free'' or ``fixed'', and each choice will lead to a different moduli space. By ``free'' end we mean elements in the moduli space can have their $p_j^\pm$ puncture be asymptotic to any Reeb orbit on $\mcal{T}_j^\pm$ with given multiplicity. By ``fixed'' end we mean elements in the moduli space must have their $p_j^\pm$ end land on a specific Reeb orbit in $\mcal{T}_j^\pm$ with given multiplicity. With this designation it will enough to specify a moduli space of $J$-holomorphic curves. The virtual dimension of this moduli space with the above specifications, is given by (See Section 3 of \cite{wendlauto} or Corollary 5.4 of \cite{BourPhd})
\begin{equation}
    Ind(u) : = -\chi(u) +2c_1(u) + \sum _{p_j^+} \mu\left (\gamma^{q_{p_j^+}} \right)- \sum _{p_j^-} \mu\left(\gamma^{q_{p_j^-}}\right) + \frac{1}{2}\# \text{free ends} -\frac{1}{2} \# \text{fixed ends}
\end{equation}
where $\chi$ is the Euler characteristic, $c_1$ the first Chern class, $\mu(-)$ is the Robbin Salamon index for path of symplectic matrices with degeneracies defined in \cite{Gutt2014}. The letter $\gamma$ denotes the embedded Reeb orbit the end $p_j^\pm$ is asymptotic to, with covering multiplicity $q_{p_j^\pm}$.

To explain the notation we think of $u$ as being an element of the moduli space $\mcal{M}$, and $Ind(u)$ is the Fredholm index of $u$. Implicitly when we write $Ind(u)$ we are including the information of which punctures of $\dot{\Sigma}$ are considered free/fixed. We also note in constructing this moduli space the complex structure of $\dot{\Sigma}$ is allowed to vary.

This moduli space can be viewed as the zero set of a Fredholm map. We borrow the set up as explained in Section 3.2 of \cite{wendlauto}. To this end, consider the space of vector fields $W^{2,p,d}(u^*TM)$ with exponential weights at the cylindrical ends of the form $e^{d|s|}$. We consider the map, following \cite{wendlauto}\footnote{In \cite{wendlauto} this operator is denoted by $D\db_J$.} Section 3.2
\begin{equation}\label{equation_neighborhood}
D_J: W^{2,p,d}(u^*TM) \oplus V_\Gamma \oplus T\mcal{J} \longrightarrow W^{1,p,d}(\overline{\op{Hom}}(T\dot{\Sigma}, u^*TM))
\end{equation}
where $V_\Gamma:=\oplus V^\pm_j$ is a direct sum of vector spaces for each puncture $p_j^\pm$. For a positive puncture at a fixed end, it is 2 dimensional vector space spanned by vector fields
\[
\beta_{1;0,\infty} \partial_t, \quad
\beta_{1;0,\infty}\partial_a
\]
where $a\in \bb{R}$ is the symplectization coordinate. For negative punctures we use instead cut off functions $\beta_{-\infty,0;1}$.

For free ends we additionally include another asymptotic vector that displaces the ends along the Morse-Bott torus
\[
\beta_* \partial_x
\]
where $\beta_*$ is as above, depending on whether this end is at a positive or negative puncture. 

$T\mcal{J}$ is a finite dimensional vector space corresponding to the variation of complex structure (in \cite{wendlauto} Section 3.1 it is called a Teichmuller slice). We note we have chosen the variation of complex structure to be supported away from the fixed cylindrical neighborhoods.
\begin{remark}\label{remark_sobolev_exponents}
It will later turn out very important to us we work with $W^{2,p}$ as our domain instead of $W^{1,p}$. The reason for this is the analytical fact that product of $L^p$ functions is generally not in $L^p$ for $p>2$, but products of $W^{1,p}$ functions remain in $W^{1,p}$. In particular in Equation \ref{equation_product} we took one more $s$ derivative than usual due to translations of terms, and if we used $W^{1,p}$ spaces we would have ended up with products of $L^p$ functions.

Another possibility is working with the Morrey spaces in Section 5.5 of \cite{obs2}, where all products are allowed and the space has an $L^2$-type inner product. In fact this is the approach taken in the Appendix of \cite{colin2021embedded}, and if we did this we might be able to avoid the awkward exponential factors of $2/p$ that appear in our subsequent exponential decay estimates.
\end{remark}

If $\dot{\Sigma}$ is stable, $u$ is somewhere injective, not a trivial cylinder, and of positive index, then for generic $J$, the operator $D_J$ is surjective and its index is equal to the dimension of the moduli space $\mcal{M}$, which is given by $Ind(u)$.

If $\dot{\Sigma}$ is not stable and $u$ is not a trivial cylinder, $u$ still lives in a moduli space of dimension calculated by the index, after we quotient out by automorphisms of the domain. As an analytic matter we address this by adding some marked points to make the domain stable and make the appropriate modifications to Sobolev spaces, in the following convention:

\begin{convention}[Stabilization of Domain] \label{stabilization}
Given a cascade $\cas{u}$, each of the $u^i$ may have components that are unstable, i.e. holomorphic curves whose domain are cylinders or planes. A main source of example is trivial cylinders. Since in this paper we are gluing curves as opposed holomorphic submanifolds, we stabilize these domains following Section 5 of  \cite{pardon} (see also \cite{Cieliebak} Section 4). For each $J$-holomorphic curve whose domain is a cylinder, we first fix a surface $\Sigma$ that intersects the $J$-holomorphic curve transversely at one point. We endow the $J$-holomorphic curve with an additional marked point on its domain and require this marked point passes through $\Sigma$. For a $J$-holomorphic curve whose domain is a plane, we fix two disjoint surfaces $\Sigma_1, \Sigma_2$, each of which intersects the $J$-holomorphic curve transversely at a single point. We add two marked points $p_1,p_2$ to the domain and require the $J$-holomorphic curve maps them to $\Sigma_1$ and $\Sigma_2$ respectively.

The effect of this is that we eliminate the reparametrization symmetry of the domain. This makes (subsequent) uniqueness statements unambiguous. We note here during the gluing construction, we will be performing large scale symplectization direction translations of each of the $u^i$. We translate the surfaces $\Sigma_i$ along with $u^i$ in these large scale symplectization direction translations. We shall make no further remark on this point and henceforth assume all $u^i$ have domains that are stable.
\end{convention}

For trivial cylinders there is a tad more to be said. If both ends are free then the moduli space is transversely cut out of index 1, where the one dimension of freedom is moving the the trivial cylinder along the Morse-Bott torus. With one end fixed the other free the moduli space is still transversely cut out of index zero. However with both ends fixed the $D_J$ operator is of index $-1$, yet obviously such trivial cylinders still exist. In this discussion we will only talk about trivial cylinders with at most one fixed ends.

We now come to the definition of what we call \emph{transverse and rigid} cascade. It is these cascades that we will eventually glue.

\begin{definition}\label{def_rigid_cas}
Suppose $\cas{u} = \{u^1,..,u^n\}$ is a height 1 cascade that satisfies the following properties:
\begin{enumerate}
    \item All curves $u^i$ are somewhere injective, except trivial cylinders, which can be unbranched covers.
    \item The\footnote{We only consider $T_i>0$, the case of $T_i=0$ requires different transversality assumptions and is handled by standard gluing methods.} numbers $T_i$ satisfy $T_i \in (0,\infty)$. 
    \item Given $u^i$ and $i>1$, the $s\rightarrow \infty$ ends of $u^i$ approach distinct Reeb orbits. For $u^i$, and $i<n$, the $s \rightarrow -\infty$ ends of $u^i$ approach distinct Reeb orbits.
    \item No end of $u^i$ land on critical points of $f$, with the following exceptions:
    \begin{enumerate}
        \item If a positive end of $u^i$ lands on a Reeb orbit corresponding to a critical point of $f$ in the intermediate cascade levels, it must then be the minimum of $f$. Suppose this orbit is $\gamma$. Furthermore, for all $j<i$, $u^j$ has a trivial cylinder (potentially unbranched cover) asymptotic to $\gamma$ and no other curves asymptotic to $\gamma$.
        \item If a negative end of $u^i$ is asymptotic to a Reeb orbit corresponding to a critical point of $f$ in the intermediate cascade levels, it must be the maximum of $f$. Call this orbit $\gamma$. For all $j>i$, $u^j$ has a trivial cylinder (potentially unbranched cover) asymptotic to $\gamma$ and no other curves asymptotic to $\gamma$.
    \end{enumerate}
    We call the chain of trivial cylinders all at a critical point of $f$ \textbf{a chain of fixed trivial cylinders}.
    \item We remove all chains of fixed trivial cylinders from $\cas{u}$. For the remaining curves, if an end of $u^i$ lands on a critical point of $f$, we designate it as fixed, and if an end of $u^i$ avoids critical points of $f$, we designate it as free.
    Then each $u^i$ can be thought of as living in a moduli space of $J$-holomorphic curves. If the domain of $u^i$ is written as $\dot{\Sigma}$, this is the moduli space of $J$-holomorphic maps from $(\dot{\Sigma},j) \rightarrow (Y\times \bb{R}, J)$ where we allow the complex structure to vary, but impose the same fixed/free conditions on the punctures as $u^i$. We denote this moduli space by $\mcal{M}(u^i)$. Then $\mcal{M}(u^i)$ is transversely cut out, and its dimension is given by $Ind(u^i)$.
    \item Again we assume we have removed all chains of fixed trivial cylinders and $\cas{u}$ satisfy all of the previous conditions. Each $\mcal{M}(u^i)$ comes with two evaluation maps, $ev_i^\pm: \mcal{M}(u^i) \rightarrow (S^1)^{k_i^\pm} $ where $k_i^\pm$ refers to how many Reeb orbits are hit by free ends of $u^i$ at $\pm \infty$. Note $k_i^-=k_{i+1}^+$. The evaluation map simply outputs the location of the Reeb orbit that an end of $u^i$ is asymptotic to on the Morse-Bott torus. 
    If we let $u'^{i}\in \mcal{M}(u^i)$,
    and the map
    \begin{align}
        EV^+: \mcal{M}(u^2) \times \bb{R}\times  \mcal{M}(u^3)\times \bb{R} \times \ldots \times \mcal{M}(u^{n}) \times \bb{R} \longrightarrow (S^1)^{k_2^+} \times (S^1)^{k_3^+} \times \ldots \times (S^1)^{k_{n}^+} 
    \end{align}
    given by
    \begin{align}
        (u'^{2}, T_1',\ldots,u'^{n},T_{n-1}') \longrightarrow \left(\phi_f^{T_1'}(ev_2^+(u'^2)),  \phi_f^{T_2'}(ev_3^+(u'^3)),\ldots,\phi_f^{T_{n-1}'}(ev_{n}^+(u'^{n}))\right)
    \end{align}
    and the map
    \begin{align}
        EV^-: \mcal{M}(u^1) \times \mcal{M}(u^2) ...\times \mcal{M}(u^{n-1})  \longrightarrow (S^1)^{k_1^-} \times (S^1)^{k_2^-} \times ... \times (S^1)^{k_{n-1}^-} 
    \end{align}
    given by
    \begin{align}
        (u'^{1},...,u'^{n}) \longrightarrow \left(ev_2^-(u'^1),...,ev_{n-1}^-(u'^{n-1})\right)
    \end{align}
    are transverse at $\cas{u}$, then we say the cascade $\cas{u}$ is transversely cut out.
    \item In particular if $\cas{u}$ is transversely cut out, it lives in a moduli space that is a manifold. The manifold has dimension given by the following formula. Assuming again we have removed all chains of fixed trivial cylinders, then the dimension is given by
    \begin{align*}
    &Ind(\cas{u}):= Ind(u^1)+\ldots+Ind(u^{n})  - k_1^--\dots- k_{n-1}^- +(n-1)-n
    \end{align*}
    If $\cas{u}$ is transversely cut out and $Ind(\cas{u}) =0$, then we say $\cas{u}$ is rigid. Note here we have quotiented by the $\bb{R}$ action on each level.
    \item (Asymptotic matchings).
    \footnote{ We describe the analogue of this construction in the nondegenerate case. Suppose $u$ and $v$ are both nontrivial somewhere injective transversely cut out rigid holomorphic cylinders in $\bb{R}\times Y^3$, and the negative end of $u$ approaches an embedded (nondegenerate) Reeb orbit $\gamma$ with multiplicity $n$, and the positive end of $v$ also approaches $\gamma$ with multiplicity $n$. Then there are $n$ distinct ways to glue $u$ and $v$ together, and we use asymptotic markers to keep track of this. This is explained in Lemma 4.3 of \cite{Hutneldynam}. Our definition of matching is an analogue of this phenomenon for cascades, except we fit a gradient trajectory (or several segments of gradient trajectories connected to each other by trivial cylinders) between two non-trivial curves. Our notation for asymptotic markers is taken from Section 1 of \cite{obs1}.}
    
    Suppose $C_i$ is a nontrivial curve that appears on the $i$th level, i.e. it is a component of $u^i$, and a negative end of $C_i$ is asymptotic to a Reeb orbit $\gamma$. Suppose $\gamma$ is the $m$ times multiple cover of an embedded Reeb orbit, $\gamma'$. Consider the preimage of a point $w$ in $\gamma'$ of the covering map from $\gamma$ to $\gamma'$, which is a set of multiplicity $m$ that we write as $\{1,....,m\}$. Consider the smallest $j>i$ such that $u^j$ contains a nontrivial curve $C_j$ that has a positive end that is asymptotic to a Reeb orbit $\tilde{\gamma}$ that is connected to $\gamma$ by the upwards gradient flow segments in $\cas{u}$. (If $j>i+1$ we allow several segments of gradient flow concatenated together with trivial cylinders in the middle\footnote{As explained in the proof of Theorem \ref{SFTtheorem} in the Appendix, we should really think of collections of trivial cylinders connected by finite gradient flow segments between them as being a \emph{single} gradient flow segment that flows across different cascade levels.}.) The orbit $\tilde{\gamma}$ covers some embedded Reeb orbit $\tilde{\gamma}'$ with multiplicity $m$. Let $\tilde{w}$ denote the point in $\tilde{\gamma}$ that corresponds to $w$ under the gradient flow, i.e. in the neighborhoods we have chosen for Morse-Bott tori, $w$ and $\tilde{w}$ have the same $z$ coordinate. Then the preimage of $\tilde{y}$ under the covering map $\tilde{\gamma}\rightarrow \tilde{\gamma}'$ is a set with $m$ elements, $\{\tilde{1},..,\tilde{m}\}$. A \textbf{matching}, which is part of the data of the cascade $\cas{u}$, is a choice of a function from $\{1\}$ to  $\{\tilde{1},..,\tilde{m}\}$.
\end{enumerate}
With the above conditions all satisfied, we say $\cas{u}$ is a transverse and rigid height 1 cascade.
\end{definition}
Now we are in a position to state our correspondence theorem:
\begin{theorem} \label{main_theorem}
For all $\dt>0$ sufficiently small, a rigid transverse cascade $\cas{u}$ can be uniquely glued to a $J_\dt$-holomorphic curve $u_\dt$ with non-degenerate ends. 

The free ends in the $u^1$ level correspond to ends of $u_\dt$ that are asymptotic to Reeb orbits corresponding to the maximum of $f$. The fixed positive ends of $u^1$ correspond to positive ends of $u_\dt$ that are asymptotic to the Reeb orbits at the minimum of $f$. Similarly the free negative ends of $u^n$ correspond to negative ends of $u_\dt$ that are asymptotic to the Reeb orbits at the minimum of $f$, and the fixed negative ends of $u^n$ correspond to negative ends asymptotic to Reeb orbits corresponding to the maximum of $f$. The curve $u_\dt$ also has Fredholm index $1$. 

By uniqueness we mean that if $\{\dt_n\}$ is a sequence of numbers that converge to zero as $n\rightarrow \infty$, and $u'_{\dt_n}$ is a sequence of $J_{\dt_n}$-holomorphic curves converging to $\cas{u}$, then for large enough $n$, the curves $u_{\dt_n}'$ agree with $u_{\dt_n}$ up to translation in the symplectization direction.
\end{theorem}

\begin{remark} \label{remark_generic_J}
In Section \ref{diffgeo} we will see that in perturbing from $\lambda$ to $\lambda_\dt$, the almost complex structure $J$ will need to be perturbed to $J_\dt$ to ensure it is compatible with $\lambda_\dt$. We will specify how to perturb $J$ into $J_\dt$ near each Morse Bott torus in Section \ref{diffgeo}. We can in fact perturb $J_\dt$ to be different from $J$ away from the Morse-Bott tori as well. Our construction works as long as in $C^\infty$ norm the difference between $J$ and $J_\dt$ is bounded above by $C\dt$. We bring this up because we can choose a generic path between $J$ and $J_\dt$ as $\dt \rightarrow 0$ so that for generic $\dt>0$, the glued curve $u_\dt$ is also transversely cut out. This will be useful for Floer theory constructions in \cite{Tips}, and will be explained in more detail there.
\end{remark}
\section{Differential geometry} \label{diffgeo}
In this section we work out the differential geometry surrounding the Morse-Bott tori. We first work out the Reeb dynamics, then we show two gradient flow trajectories of $f$ correspond to $J_\dt$-holomorphic cylinders.
\subsection{Reeb dynamics}
We recall the local neighborhoood near a Morse-Bott torus: if $(Y^3,\lambda)$ is a contact 3 manifold with Morse-Bott degenerate contact form, near a Morse-Bott torus we have coordinates
$(z,x,y) \in S^1\times S^1 \times \mathbb{R}$. Let
\[
\lambda_0= dz-ydx
\]
denote the standard contact form, then by Theorem \ref{prop_locform} $\lambda$ looks like 
\[
\lambda = h(x,y,z) \lambda_0
\]
where $h(x,y,z)$ satisfies
\[
h(x,0,z)=0,\quad  dh(x,0,z) =0.
\]
Next we perturb the contact form to 
\[
\lambda \longrightarrow \lambda_\delta:= e^{\delta gf}\lambda
\]
We assume we are working in a small enough neighborhood so that $g=1$. We are interested in the Reeb dynamics on the torus $y=0$.
\begin{proposition}
On the torus $y=0$, let $R_\delta$ denote the Reeb vector field of $\lambda_\delta$. We write it in the form $R_\delta = R+X$ where $R = \partial/\partial_z$ is the Reeb vector field of $\lambda$. Then the following equaitons are satisfied
\begin{equation*}
    \iota_X \lambda = \frac{1-e^{\delta f} }{e^{\delta f}}
\end{equation*}
\[
\iota_X d\lambda = \frac{de^{\delta f}}{e^{2\delta f}}
\]
and these two equations completely characterize the the behaviour of $R_\delta$ on the $y=0$ surface.
\end{proposition}
\begin{proof}
From definition
\begin{equation*}
    \iota_{R_\delta} \lambda_\delta =1
\end{equation*}
hence we have
\begin{equation*}
    \iota_X \lambda = \frac{1-e^{\delta f} }{e^{\delta f}}.
\end{equation*}
For the second equation, 
\begin{align*}
    &\iota_{R+X} d\left(e^{\delta f}\lambda\right )\\
    &=\iota_{R+X} \left(e^{\delta f}d\lambda +\delta e^{\delta f}df \wedge \lambda \right)\\
    &=\iota_R(...) +\iota_X(...)
\end{align*}
If we look at the first term we see
\begin{align*}
    &\iota_R(e^{\delta f }d\lambda +\delta e^{\delta f} df \wedge \lambda )\\
    &=\iota_R e^{\delta f} d\lambda +\delta e^{\delta f} (\iota_R df) \wedge \lambda - \delta e^{\delta f} df \iota_R \lambda \\
    &= 0 +0 -\delta e^{\delta f} df .
\end{align*}
Next we look at the second term
\begin{align*}
    &\iota _Xd\lambda_\delta\\
    &=\iota_X(e^{\dt f} d\lambda +\delta e^{\delta f} df \wedge \lambda )\\
    & = \iota_X e^{\delta f}d\lambda + \delta e^{\delta f} (\iota_X df) \lambda - \delta e^{\delta f} df \iota_X \lambda \\
    &= e^{\delta f} \iota_X d\lambda + \delta e^{\delta \lambda} (\iota_X df) \lambda - \delta e^{\delta f} df \iota_X \lambda.
\end{align*}
Combining the above two equations we get
\begin{equation*}
    e^{\delta f} \iota_X d\lambda + \delta e^{\delta f} (\iota_X df) \lambda - \delta e^{\delta f} df \iota_X \lambda = \delta e^{\delta f} df.
\end{equation*}
Evaluate both sides with $\iota_R$ we see that
\begin{equation*}
    \iota_X df =0
\end{equation*}
so we get
\begin{align*}
    &e^{\delta f} \iota_X d\lambda = de^{\delta f}\left(1+\frac{1-e^{\delta f} }{e^{\delta f}}\right)
\end{align*}
\[
\iota_X d\lambda = de^{\delta f}/e^{2\delta f}.
\]
\end{proof}
In particular on the $y=0$ surface we can write
\[
X = \frac{1-e^{\delta f} }{e^{\delta f}}\partial_z - \frac{\delta e^{\delta f}f'(x) \partial_y} {e^{2\delta f}}
\]
\subsection{Almost complex structures and gradient flow lines}
For $(\bb{R}\times Y^3, \lambda)$ we choose a generic almost complex structure $J$ that is standard on the surface of the Morse-Bott torus, i.e. $J\partial_x = \partial_y$. After we perturb to $\lambda_\delta $, we must perturb $J$ to $J_\delta$ to make the complex structure compatible with the new contact form. However we keep the same complex structure on the contact distribution, i.e. 
\[
J_\delta \partial_x =\partial_y.
\]
We wish to understand what $J_\dt$ does to the Reeb vector field and the vector field in the symplectization direction. By definition
\begin{equation*}
    J_\delta (R+X) = -\partial_a 
\end{equation*}
\begin{equation*}
    J_\delta \partial_a = R+X.
\end{equation*}
From the above we deduce
\begin{equation*}
    J_\delta R= -\partial_a -J_\delta X= -\partial_a- J_\delta \left(\frac{1-e^{\delta f} }{e^{\delta f}}\partial_z - \frac{\delta e^{\delta f}f'(x) \partial_y} {e^{2\delta f}}\right)
\end{equation*}
\begin{equation*}
     J_\delta \partial_z = e^{\delta f} \left(-\partial_a -  \frac{\delta e^{\delta f}f'(x) \partial_x} {e^{2\delta f}}\right)= -e^{\delta f} \partial_a - \delta f'(x) \partial_x.
\end{equation*}
Next we consider $J_\delta$-holomorphic curves constructed by lifting gradient flows of $\delta f$.
Consider maps
\[
v:(s,t) \longrightarrow (a(s), z(t), x(s),y(s)) \in \bb{R} \times S^1 \times S^1 \times \bb{R}
\]
defined by
\begin{equation*}
   \partial_s a = e^{\delta f(x(s))}
\end{equation*}
\begin{equation*}
    \partial_t z(t)= R
\end{equation*}
\begin{equation*}
    \partial_s x(s) = + \delta f'(x)
\end{equation*}
\[
y=0
\]
and initial conditions
\begin{equation*}
    a(0,t) = 0, x(0) = \textup{constant}.
\end{equation*}
\begin{proposition}
The map $v$ as defined above is a $J_\delta$-holomorphic curve.
\end{proposition}
\begin{proof}
Let $\hat{v}:= (z(t),x(s),y(s))$. We apply the $J_\delta$ holomorphic curve equation to $v$
\begin{align*}
    &\partial_s \hat{v} + \partial_s a \frac{\partial}{\partial a} + J \partial_t \hat{v}\\
    &= e^{\delta f(x(s))}\frac{\partial}{\partial a} + \delta f'(x) \frac{\partial}{\partial x} -(e^{\delta f}) \partial_a -\delta f'(x) \partial_x =0.
\end{align*}
\end{proof}
We observe since there are two gradient flow lines on $S^1$, there are two $J_\dt$-holomorphic curves as above corresponding to their lifts. Further:
\begin{proposition}
The curve $v$ is transversely cut out. The same is true for unbranched covers of $v$ by cylinders.
\end{proposition}
\begin{proof}
We use Theorem 1 from Wendl's paper on automatic transversality \cite{wendlauto}. In the language of Theorem 1, $\text{Ind} (v) =1$, $\Gamma_0=1$ (only one end is asymptotic to Reeb orbits with even Conley-Zehnder index), there are no boundary components, and $c_N=0$, hence
\begin{equation*}
   \text{ Ind}(v)=1 >c_N +Z(du)=0.
\end{equation*}
The same proof works for unbranched covers of $v$ as well.
\end{proof}
For future references, we record the form of the vector field
\[
v_* \partial_s = e^{\delta f(x(s))}\partial_a + \delta f'(x) \partial_x.
\]
\section{Linearization of \texorpdfstring{$\bar{\partial}_{J_\delta}$}{Cauchy Riemann operator} over \texorpdfstring{$v$}{v}} \label{linearization}
In this section we define the linearization of the Cauchy Riemann operator $\db_{J_\delta}$ over $v$, the holomorphic cylinder constructed in the above section that corresponds to a gradient flow of $f$. We also equip it with an appropriate Sobolev space on which the linearized operator is Fredholm. This is preparation for the gluing construction.

\begin{convention}
For this point onward in the paper we will assume all gradient trajectories are simply covered for ease of notation. In practice they can be (unbranched) multiply covered. For any of the analysis we are doing this will not make any difference. 

The point to note here is that if we see any finite gradient cylinders (or chains of finite gradient cylinders connected to each over by trivial cylinders) that are multiply covered connecting between two non-trivial curves in the cascade, the number of ways to glue is counted precisely by the number of different matchings (see Definition \ref{def_rigid_cas}) we can assign to such a segment.
\end{convention}

Fix a holomorphic cylinder $v_\delta$ (we make the $\delta$ dependence explicit), consider the space of vector fields over $v_\dt$,
\[
\Gamma (v_\delta ^*TM).
\]
We take a weighted Sobolev space
\[
W^{2,p,d}(v_\delta^*TM)
\]
which is the $W^{2,p}(v_\delta^*TM)$ with exponential weight $e^{w(s)} = e^{ds}$, where $d>0$ is a small fixed number that only depends on the Morse-Bott torus. Here we can also use $e^{-ds}$.

Note as given, these are vector fields with exponential decay as $s\rightarrow \infty$ and exponential growth as $s\rightarrow -\infty$. The end with exponential growth is not suited for nonlinear analysis of the Cauchy Riemann equation, but we will find them useful as a formal device so all our linear operators have the right Fredholm index and uniformly bounded right inverse. It will be apparent from our gluing construction that vector fields with exponential growth will not cause any difficulty. This is also the approach taken in \cite{colin2021embedded}.
The main result of the section is the following:
\begin{proposition}\label{uniform_bounded_inverse_gradient_flow}
Let $D_{J_\delta}$ denote the linearization of $\db_{J_\delta}$ along $v_\delta$ using metric $g$. Then the operator
\[
D_{J_\delta}: W^{2,p,d}(v_\delta^*TM) \longrightarrow W^{1,p,d}(v_\delta^*TM)
\]
is a Fredholm operator of index 0. In particular it is an isomorphism. Further it has right (and left) inverse $Q_\delta$ whose operator norm is uniformly bounded as $\delta \rightarrow 0$.
\end{proposition}
The proof will occupy the rest of this section. The idea is for sufficiently small $\delta>0$ the $J_\dt$-holomorphic curve $v_\dt$ is nearly horizontal, and hence can be approximated by a finite collection of trivial cylinders glued together. But the linearization of $\db$ over a trivial cylinder is an isomorphism with inverse independent of $\delta$, and by standard gluing theory of operators the operator glued from linearizations of $\db$ over trivial cylinders has the properties described in the theorem.

\subsection{Linearizations over trivial cylinders}
Fix $x$, which corresponds to fixing a Reeb orbit in the Morse-Bott torus. Consider the trivial cylinder $C_x$ at $x$. The Cauchy Riemann operator $\db_J$ (with unperturbed complex structure $J$) has linearization $D_x$ of the form
\begin{equation*}
    \partial_s + J_0 \partial_t +S_x(t).
\end{equation*}
The matrix $J_0$ is the standard complex structure on $\bb{R}^4$, and $S_x(t)$ is a symmetric matrix. Considered as an operator, we have
\[
D_x:W^{2,p,d}(C_x^*TM) \longrightarrow W^{1,p,d}(C_x^*TM)
\]
with exponential weight $e^{ds}$ on both sides. 
\begin{lemma}
$D_x$ is an isomorphism.
\end{lemma}
\begin{proof}
We consider this operator defined on $W^{2,p}(C_x^*TM)$ instead of $W^{2,p,d}(C_x^*TM)$ by using the isometry
\[
e^{-ds} : W^{2,p}(C_x^*TM) \longrightarrow W^{2,p,d}(C_x^*TM).
\]
The effect of this on the operator $D_x$ is
\[
e^{ds} D_x e^{-ds} :W^{2,p}(C_x^*TM) \longrightarrow W^{1,p}(C_x^*TM)
\]
\[
e^{ds} D_x e^{-ds} = \partial_s + J_0\partial_t + S_x(t) -d.
\]
The operator $A(t):W^{2,p}(S^1) \rightarrow W^{1,p}(S^1)$ given by $A= -J_0\partial_t -S_x(t) +d$ has eigenfunctions $\{e_n\}$ with eigenvalues $\{\lambda_n\}$, and no eigenvalue $\lambda_n$ is equal to zero. This shows $D_x$ is index 0 because there is no spectral flow. An element in the kernel of $e^{ds} D_x e^{-ds}$ can be written in the form
\[
\sum c_n e^{\lambda_n s} e_n(t)
\]
but all $c_n$ must equal to zero because terms like $e^{\lambda_ns}$ have exponential growth on one end hence cannot live in $W^{2,p}(C_x^*TM)$. This implies $D_x$ is an isomorphism hence has an inverse, which we denote by $Q_x$. Note this inverse does not depend on $\delta$.
\end{proof}
Observe since $x$ varies in a $S^1$ family, there exists $C$ such that
\[
\|Q_x\| \leq C
\]
in operator norm for all $x\in S^1$.
\subsection{Uniformly bounded inverse for \texorpdfstring{$D_{J_\delta}$}{D}}
In this subsection we prove the main theorem of this section. This is inspired by analogous constructions in Proposition 4.9 in \cite{oancea} and Proposition 5.14 in \cite{BourPhd}.
\begin{proof}[Proof of Proposition \ref{uniform_bounded_inverse_gradient_flow}]
We identify $S^1$, the circle of Morse-Bott orbits, with $x\in [0,1]/\sim$, and we recall $f$ has critical points at $x=0$ and $x=1/2$. WLOG we consider the $v_\delta (s,t)$ corresponding flow from with $-\infty$ end at $x=0$, towards $x=1/2$ as $s\rightarrow +\infty$ and take $s_0=-\infty$.\\
Fix $N$ large, let $x_i = 1/2N$, $i=1,..,N$ denote Reeb orbits on the Morse-Bott torus. Let $s_i \in \bb{R}$ denote the time it takes for $v_\delta$ to flow to $x_i$, i.e. when $x\,\text{ component of } \, v_\delta(s_i, \cdot) = x_i$. We implicitly take $s_N =+\infty$. We observe $s_i$ implicitly depends on $\delta$ and 
\[
s_{i+1} -s_i \geq C/(\delta N).
\]
We let $D_{i}:=  \partial_s+J_0\partial_t +S_{x_i}(t)$ denote the linearization of the $\db_J$ operator at a trivial cylinder at $x_i$. We define the parameter
\begin{equation*}
R:= \frac{1}{5d} \log \frac{1}{\delta}
\end{equation*}
Let $\beta_o(s)$ be a cut off function equal to $1$ for $s\geq1$ and $0$ for $s\leq0$. we define the ``glued" operator
\[
\#_N D_i : = \p_s+J_0 \p_t + \sum_{i=0}^{N-1} (1-\beta_o(s-s_{i+1}))\beta_o(s-s_{i})S_{x_{i+1}}(t).
\]
So we have $\#_ND_i = D_i$ on the interval $[s_{i-1}+1,s_i]$
by construction. Viewed as operators \[
W^{2,p,d}(v_\delta^*TM) \longrightarrow W^{1,p,d}(v_\delta ^*TM)
\]
we have 
\[
\|D_{J_\delta} - \#_N D_i\| \leq C(1/N+\delta) 
\]
in operator norm with constant $C$ independent of $\dt$ or $N$. It follows from the same spectral flow argument as above that $\#_N D_i$ is Fredholm of index 0. We now proceed to construct a uniformly bounded (as $\delta \rightarrow 0$) right inverse $Q_N$ for it. Let $Q_i$ denote inverses to $D_i$, we first construct approximate inverse $Q_R$ using the following commutative diagram
\begin{equation*}
\begin{tikzcd}
W^{1,p,d}(v_\delta ^*TM) \arrow[r,"Q_R"] \arrow[d, "s_R"] & W^{2,p,d}(v_\delta^*TM)\\
\bigoplus_i[W^{1,p,d}(v_\delta^*TM)]_i \arrow[r,"\bigoplus Q_i"] &\bigoplus_i[W^{2,p,d}(v_\delta^*TM)]_i\arrow[u, "g_R"]
\end{tikzcd}
\end{equation*}
with splitting maps $s_R$ and gluing maps $g_R$ defined as follows:
if $\eta \in W^{1,p,d}(v_\delta^*TM)$, $s_R(\eta) = (\eta_i,..,\eta_N)$
where 
\[
\eta_i := \eta (1-\beta_o (s-s_{i}))\beta_o(s-s_{i-1}).
\]
We see immediately $s_R$ has uniformly bounded operator norm as $\delta \rightarrow 0$, and that its norm is also bounded above independently of $N$.
Let $\gamma_R(s)$ be a cut off function $\gamma_R(s) =1$ for $s<1$ and $\gamma_R(s) =0$ for $s>R/2$ and $\gamma'(s) \leq C/R$. If $(\xi_1, ..,\xi_N)\in \bigoplus_i W^{2,p,d}(v_\delta^*TM)$ we define
\[
g_R(\xi_1,..,\xi_N) = \sum_i \xi_i \gamma_R(s-s_i) \gamma_R(s_{i-1}-s).
\]
We also see that $g_R$ is an uniformly bounded operator as $\delta \rightarrow 0$ and its upper bound on norm is independent of $N$. We conclude $Q_R$ has uniformly bounded norm as $\delta \rightarrow 0$. We next show it is an approximate inverse to $\#_N D_i$.\\
If we start with $\eta \in W^{1,p,d}(v_\delta ^*TM)$, with $Q_R(\eta) = \sum_i\xi_i \gamma_R(s-s_i) \gamma_R(s_{i-1}-s) $. We apply $\# _N D_i$ to it and observe away from the intervals of the form $\bigcup_i [s_i-R,s_i+R]$ - which we think of the region where gluing happens,
\[
\#_ND_i Q_R \eta  = D_i Q_i \eta_i =\eta
\]
so we focus our attention to an interval of the form $[s_i-R,s_i+R]$, in which $Q_R(\eta) = \gamma_R(s-s_i)\xi_i + \gamma_R(s_i-s) \xi_{i+1}$.\\
We observe over intervals of this form $\|D_i- \#_N D_i\| \leq C/N$ in operator norm, so when we apply $\#_N D_i$ to $Q_R(\eta)$ we get
\begin{align*}
     \#_ND_i Q_R \eta
     =& \#_N D_i  (\gamma_R(s-s_i)\xi_i + \gamma_R(s_i-s) \xi_{i+1})\\
    =&\gamma_R'(s-s_i)\xi_i - \gamma_R' (s_i-s)\xi_{i+1}\\
    &+ \gamma_R(s-s_i)\#_ND_i \xi_i + \gamma_R(s_i-s) \#_ND_i \xi_{i+1}.
\end{align*}
In light of the above, in this region we have
\begin{align*}
    \#_ND_i \xi_i &= D_i\xi_i + (\#_N D_i -D_i )\xi_i \\
    &= \beta_o(s-s_i)\eta + (\#_ND_i -D_i) \xi_i 
\end{align*}
with
\[
\|(\# _ND_i -D_i) \xi_i  \| \leq C/N \|\xi_i\| \leq C/N \|\eta\|
\]
and likewise for the $\xi_{i+1}$ term in weighted Sobolev norm. We also note $\gamma_R' \leq C/R$, so we can write
\[
 \#_ND_i Q_R \eta = \gamma_R(s-s_i) \beta_o(s-s_i) \eta + \gamma_R(s_i-s)(1-\beta_o(s-s_i))\eta + \textup{error}
\]
for $s\in [s_i-R,s_i+R]$.
But by the construction of $\beta_o$ and $\gamma_R$, we have  $\gamma_R(s-s_i) \beta_o(s-s_i) \eta + \gamma_R(s_i-s)(1-\beta_o(s-s_i))\eta = \eta$ in $[s_i-R,s_i+R]$, so we have
\[
\|\#_N D_i Q_R \eta -\eta\| \leq C/N \|\eta\|
\]
in weighted Sobolev norm.
So for sufficiently large values of $N$, the operator $Q_R$ is an approximate right inverse. Then we can define a true right inverse $\#_N D_i$ by
\[
Q_N:=Q_R(\#_N D_i Q_R)^{-1}
\]
which also has uniformly bounded norm as $\delta\rightarrow 0$. This in particular implies $\#_N D_i$ is surjective.\\
Finally using 
\[
\|D_{J_\delta} - \#_N D_i\| \leq C(1/N+\delta) 
\]
in operator norm we see that $Q_N$ is an approximate right inverse to $D_{J_\delta}$ because:
\begin{align*}
    \|D_{J_\delta} Q_N \eta -\eta\|
    &=\|(D_{J_\delta}- \#_N D_i)Q_N +\#_N D_iQ_N \eta -\eta\|\\
    &=\|(D_{J_\delta}- \#_N D_i)Q_N\eta\|\\
    &\leq \|Q_N\|\cdot \|(D_{J_\delta}- \#_N D_i)\|\cdot \|\eta\|\\
    &\leq C/N \|\eta\|
\end{align*}
and hence $D_{J_\delta}$ has uniformly bounded right inverse as $\delta \rightarrow 0$. 
\end{proof}
\begin{remark}
We proved for given $D_{J_\delta}$ acting on $W^{2,p,d}(v_\delta^*TM)$ over fixed $v_\delta$ it has uniformly bounded right inverse. For our proof we assumed the exponential weight is of the form $e^{ds}$, but it should be apparent from our proof even as we translate the weight profile from $e^{ds}$ to $e^{ds-T}$ for any $T \in \bb{R}$, the same proof goes through. Said another way, for any sufficiently small $\delta$ and any $T$, the operator $D_{J_\delta}$ defined over $W^{2,p,d}(v_\delta^*TM)$ with weight $e^{\pm ds+T}$ has a uniformly bounded inverse. 
\end{remark}

\begin{remark}
In the above construction we implicitly fixed a parametrization of $v_\dt$ with respect to the $t$ variable, i.e. we picked out which point on the Reeb orbit corresponds to $t=0$. We could also have changed this, resulting in a reparametrization of $v_\dt$, of the form $t\rightarrow t+c$. For all such reparametrizations it is obvious $D_{J_\dt}$ continues to have uniformly bounded right inverse, and this upper bound is uniform across all possible reparametrizations in the $t$ variable.
\end{remark}
\section{Gluing a semi-infinite gradient trajectory to a holomorphic curve} \label{semiglue}
In this section we glue a $J$-holomorphic curve $u$ to a semi-infinite gradient trajectory $v$. This is a simpler case of gluing for multi-level cascades, and properties of this gluing developed here and in the following sections will be used extensively in gluing together multiple level cascades. The novel feature of this gluing construction, which separates it from standard types of gluing constructions, is that we will make the pregluing dependent on asymptotic vectors. The general setup will follow that of Section 5 in \cite{obs2}, and in a sense we are doing obstruction bundle gluing, see also Remark \ref{obsrmk}. This approach to gluing has appeared in the Appendix of \cite{colin2021embedded}.

The section is organized as follows: in subsection 1 we first introduce the gluing setup. In subsection 2 we do the pregluing. In subsection 3 we take care of the differential geometry/estimates needed to deform the pregluing. Further, we write down the $J$-holomorphic curve equation we need to solve, and split it into two different equations as was done in Section 5 of \cite{obs2}. And finally in subsection 4 we solve both of these equations. We do not yet say anything about surjectivity of gluing and save it for the end when we discuss surjectivity of gluing in the general case.

\subsection{Gluing setup}
Let $u: \dot{\Sigma} \rightarrow M$ be a $J$-holomorphic curve with only one positive puncture which is free, asymptotic to a Morse-Bott torus with multiplicity 1 (higher multiplicities are handled similarly). We choose local coordinates on $u$ around the puncture given by $(s,t) \in [0,\infty)\times S^1$. We also assume $\dot{\Sigma}$ is stable. Our assumptions are purely a matter of convenience since it will be apparent from our construction how to glue semi-infinite gradient trajectories with arbitrary number of positive/negative ends. We also assume (purely as a matter of notational convenience) that we have shifted our coordinates so that $\lim_{s\rightarrow \infty} u(s,t)$ converges to the Reeb orbit at $x=0$, and the critical points of $f$ are at $x=\pm1/4$ with max at $x=1/4$ and min at $x=-1/4$. We assume $u$ is rigid, i.e. the operator
\[
D_J: W^{2,p,d}(u^*TM) \oplus V_\Gamma \oplus T\mcal{J} \longrightarrow W^{1,p,d}(\overline{\op{Hom}}(T\dot{\Sigma}, u^*TM))
\]
is surjective of index 1. It has a right inverse $Q_u$. Here $V_\Gamma:=span\{\beta_{1;0,\infty} \partial_z, \beta_{1;0,\infty}\partial_a, \beta_{1;0,\infty} \partial_x\}$. This is a 3 dimensional vector space with a given basis, we denote elements of this space by triples $(r,a,p)$. The norm of elements $(r,a,p) \in V_\Gamma$ is simply $|r|+|a|+|p|.$ We will often write $|(r,a,p)|$ to mean this norm.

\begin{convention}
We will generally use the symbol $(r,a,p)$ as a shorthand for the asymptotically constant vector field
\[
r \beta_{1;0,\infty} \partial_z + a \beta_{1;0,\infty}\partial_a + p\beta_{1;0,\infty} \partial_x.
\]
This is generally the case when we use $(r,a,p)$ to deform curves, and the case later where the symbol $(r,a,p)$ appears in the equations $\Theta_\pm$. We will also sometimes to use the symbol $(r,a,p)$ to simply denote the tuple of numbers, $r,a,p$. It will be clear from context what we mean. 
\end{convention}
We observe by definition $D_J(r,a,p)$ decays exponentially (at a rate faster than $e^{-ds}$, which we denote by $e^{-Ds}$, $D>>d$) as $s\rightarrow \infty$.

\begin{convention}
We use the following convention regarding $d$ and $D$. The symbol $D$, when written as $e^{-Ds}$ will always be used to denote a rate of exponential decay that only depends on the background geometry, say the local geometry around the Morse-Bott torus. An example will be the rate of exponential convergence to a trivial cylinder of a $J$-holomorphic curve asymptotic to a Reeb orbit. The lower case $d$ will be chosen to be independent of $\dt$, $d<<D$ and as usual much smaller than the distance between the nonzero eigenvalues of operator $A(t)$ and 0. This is the exponential weight we will use in our weighted Sobolev spaces.
\end{convention}

The rest of the section is devoted to proving the following:
\begin{proposition}
For every $\delta >0$ sufficiently small, there is a $J_\delta$-holomorphic curve $u_\delta: \dot{\Sigma}\rightarrow M$ that is positively asymptotic to the Reeb orbit $x=1/4$ obtained by gluing a semi-infinite gradient trajectory along the Morse-Bott torus to $u$.
\end{proposition}

\subsection{Pregluing}
We make the pregluing dependent on the triple of asymptotic vectors $(r,a,p)$. We first describe the neighborhood of $u|_{[0,\infty)\times S^1}$. Recall
we are working in a neighborhood of the Morse-Bott torus whose local coordinates in the symplectization are given by
\begin{equation*}
    \mathbb{R} \times S^1 \times S^1 \times \mathbb{R} \ni (a,z,x,y)
\end{equation*}
where $x$ is displacement across Morse-Bott torus direction, $y$ is the vertical direction, $a$ symplectization direction, and $z$ Reeb direction.
At the surface of $y=0$, $J$ is the standard complex structure. The metric here is the flat metric, so we will simply ``add" vectors together as opposed to taking the exponential map.
The map $u$ comes in the form
\begin{equation*}
    u(s,t) = (s+ \epsilon_s, t+\epsilon_t, \eta_x,\eta_y)
\end{equation*}
where 
\begin{equation*}
    \lim_{s\longrightarrow \infty} \eta_* =0
\end{equation*}
of order $e^{-Ds}$, where $D$ is some fixed constant specific to Morse-Bott torus ($d<<D$).
We also have
\begin{equation*}
    \epsilon_* \approx O(e^{-Ds}).
\end{equation*}
Then
\begin{equation*}
    u(s,t)+(r,a,p) = (s+\epsilon_s +a\beta_{1;0,\infty} ,t+\epsilon_t+r\beta_{1;0,\infty},\eta_x+\beta_{1;0,\infty}p, \eta_y).
\end{equation*}
Recalling the important parameter $R$:
\[
R = \frac{1}{5d}\log(1/\delta)
\]
which we will take to be our gluing parameter, we cut off $u+(r,a,p)$ at $s=R$ and glue in a gradient trajectory $v_{r,a,p}(s,t)$ satisfying
\[
v_{r,a,p}(R,t) = (R+a,t+r,p,0).
\]
We observe that since $\dt <<R$, in the range of $s\in [R,5R]$, the map $v_{r,a,p}(s,t)$ remains almost a trivial cylinder, which can make precise by noting
\[
|v_{r,a,p} -(R+s,t,p,0)|_{C^k} \leq C R\delta, s\in [R,5R].
\]
We are now ready to define the pregluing. We define
\[
u_{r,a,p}(s,t):=  \begin{cases}
u(s,t) + (r,a,p), s<R-1\\
v_{r,a,p} , s\geq R-1/2\\
\text{smooth, bounded interpolation between $u+(r,a,p)$ and $v_{r,a,p}$ for}\,  s \in [R-1,R-1/2].
\end{cases}
\]
The interpolation above should be chosen so that the difference between $u_{r,a,p}$ and the trivial cylinder of the form $(s,t) \rightarrow (R+a,t+r,p,0)$ should be bounded by $e^{-DR}$ in $C^k$ norm.

We first observe the preglued curve is still defined on the same domain $\dot{\Sigma}$. It still has the same coordinate neighborhood $[0,\infty) \times S^1$ near the unique positive puncture. As a warm up to considering the deformations of this preglued curve, we next measure how non-holomorphic this preglued curve is by applying $\db_{J_\delta}$ to it. 

\begin{remark}
Here in constructing the domain for the pregluing we ``rotated" our gradient trajectory $v_{r,a,p}$ (denoted by $v$ in Section \ref{linearization}) by $r$ to match $u+(r,a,p)$. It is also possible to instead glue $u_{r,a,p}$ with $v_{0,a,p}$ by making the identification $t \sim t+r$ at $s=R$. In this case we get back the same surface, however when we later glue over finite cylinders this will make a difference, as it corresponds to the same topological surface but a new complex structure on the preglued domain.
\end{remark}

\begin{convention}
We adopt the convention that for terms that are supposed to be small, e.g. uniformly bounded by $C\epsilon$ (in say $C^k$ norms or any norm we care about), we just write the upper bound $C \epsilon$ instead of the specific term in its entirety.
\end{convention}

\begin{proposition}
After we apply the $\db_{J_\delta}$ operator to the preglued curve $u_{r,a,p}$ over the interval $(s,t) \in [0,\infty) \times S^1$ we get terms of the form
\begin{align*}
    & [D_J(r,a,p) + C(s,t) |(r,a,p)|^{2}e^{-Ds} ] \beta_{[0,R+1;1]}\\   
    &+ C[\delta(1+(r,a,p))] \beta_{[0,R+1;1]} \\
    &+ C[e^{-DR}(1+|(r,a,p)|)  + C\delta(1+|(r,a,p)| )] \beta_{[1;R-2,R+2;1]}.
\end{align*}
By $C(s,t)$ or oftentimes $C$, we mean a function of $(s,t)$ and occasionally also including the variables $(r,a,p)$, whose derivatives are uniformly bounded. When we write $|r,a,p|$ we mean the absolute value of the numbers $|r|, |a|, |p|$.

Note the term $ D_J(r,a,p)$, which is the only term in this expression that is not ``small".
Figuratively we can write this as
\begin{equation*}
    D_J(r,a,p) +\mcal{F}(r,a,p) + \mcal{E}(r,a,p)
\end{equation*}
where 
\begin{equation*}
   \mcal{F}(r,a,p)= C(s,t) |(r,a,p)|^{2}e^{-Ds}
\end{equation*}
and 
\begin{equation*}
\mcal{E} (r,a,p) =   C\left[\delta(1+(r,a,p))\right] \beta_{[0,R+1;1]} 
    + C\left[e^{-DR}(1+(r,a,p))  + C\delta(1+(r,a,p) )\right] \beta_{[1;R-2,R+2;1]}
\end{equation*}
where we think of $\mcal{F}(r,a,p)$ as a quadratic order term and $\mcal{E}(r,a,p)$ as an error term.
\end{proposition}
\begin{remark}\label{int_err}
We first note that $u$ is holomorphic with respect to $J$, but in the above theorem we applied the $\db$ operator with respect to $J_\delta$, which is responsible for the appearance of several error terms. Further since $u$ is not holomorphic with respect to $J_\delta$, there is another error term that appears in the interior of $u$, i.e. $\dot{\Sigma} \setminus [0,\infty)\times S^1$ of size $C\delta$. Note no such error term appears in the interior region of $v_{r,a,p}$ This term is not very important because by our metric it is (uniformly) small, we will include it when we solve for the equation more globally. 
\end{remark}

\begin{proof}
We first consider downwards of the pregluing region, in the region $s\in [0,R-1]$, the pregluing is simply
consider $u+(r,a,p)$, then after applying $\db_{J}$ we get
\begin{equation*}
    \partial_s [u+(r,a,p)]+J(u+(r,a,p)) \partial_t (u+(r,a,p)) = \beta'_{1;0,\infty}(r\p_z+a\p_a+p\p_x) + \p_su +J(u+(r,a,p))\p _t u.
\end{equation*}
To this end, observe $\p_su +J(u)\p_t u =0$ so we get an expansion of the form
$D_J(r,a,p)+ \sum C|(r,a,p)|^{n\geq 2}\partial_{r,a,p}^n J(u) \partial_t u $. This is a $C^0$ bound, we will need a better bound since eventually the size of the vector will be measured with respect to weighted Sobolev norms. Observe
$\partial_t u$ is of the form 
\begin{equation*}
    (\p_t \ep_s,1 + \p_t \epsilon_t, \partial_t \eta_x,\partial_t \eta_y).
\end{equation*}
All $\eta_*$ terms decay like $e^{-Ds}$, except $(0,1,0,0)$. But we observe by compatibility of $J$, the term $(\partial^n_{r,a,p}J(u(\infty, t))) (0,1,0,0) =0$. Hence overall the second term $C|(r,a,p)|^{n\geq 2}\partial_{r,a,p}^n J(u) \partial_t u$ is of the form 
\[
C|r,a,p|^2 e^{-Ds}.
\]
Next let's include the effect of $J_\delta$, now we have
\begin{equation*}
    (\partial_{J_\delta} -\partial_J )(u+(r,a,p))= (J_\delta(u+(r,a,p))-J (u+(r,a,p)) \partial_t (u+(r,a,p)).
\end{equation*}
This term has size $\delta C$ and it only exists for length $s\in [0,R]$ and disappears after the pregluing region.
We clarify its dependence on various variables: it is of the form
\begin{equation*}
    C\delta(1+|r,a,p| ) \beta_{[0,R+1;1]}
\end{equation*}
and this is everything in the region $s\in[0,R-1]$.
We observe by definition $u_{r,a,p}$ is $J_\delta$-holomorphic in the region $s>R$ so we only need to look at the pregluing region to find rest of the pregluing error. It follows from the uniform boundedness of our interpolation construction in the pregluing that this error is of the form $C[e^{-DR}(1+|r,a,p|)  + C\delta(1+|r,a,p| )] \beta_{[1;R-2,R+2;1]}$, whence we complete our proof.
\end{proof}
\begin{remark}
The reason we are painstakingly computing all of these terms carefully (and in our subsequent computations) is because later we will be differentiating this entire expression with respect to $(r,a,p)$ so we must take note how our expressions depend on these asymptotic vectors.
\end{remark}
\subsection{Deforming the pregluing}
Now that we have constructed $u_{r,a,p}$, we deform it to try to make it $J_\delta$-holomorphic. We recall a neighborhood of $u$ is given by:
$W^{2,p,d}(u^*TM) \oplus V_\Gamma \oplus T\mcal{J}$. We recall for $[0,\infty)\times S^1$ there is an exponential weight $e^{ds}$. We already explained how to construct the pregluing with asymptotic vector fields $(r,a,p)$. We fix $\psi \in W^{2,p,d}(u^*TM)$, $\dt j\in T\mcal{J}$. Recall deformations of complex structure of the domain $\dt j$ is away from the cylindrical neighborhood so does not affect our gluing construction for the most part, so unless it is explicitly needed for rest of this section we will drop it from our notation. Now for $v_{r,a,p}$ fix $\phi \in W^{2,p,w}(v_{r,a,p}^*TM)$. Note this choice of Sobolev space is dependent on the asymptotic vectors $(r,a,p)$. We equip the space $W^{2,d,w}(v_{r,a,p}^*TM)$ with weighted Sobolev norm $e^{w(s)} = e^{ds}$.

We fix cut off functions 
\[
\beta_u:=\beta_{[-\infty, 2R;R/2]}
\]
and 
\[
\beta_v:=\beta_{[R/2;R, +\infty]}.
\]
We deform the pregluing $u_{r,a,p}$ via
\begin{equation}
    (u_{r,a,p},j_0)\longrightarrow (u_{r,a,p}+\beta_u\psi + \beta_v \phi, j_0 +\dt j).
\end{equation}
The next proposition describes what happens to the deformed curve when we apply $\db_{J_\dt}$ to it.
\begin{proposition}\label{proposition_equations}
The deformed curve $(u_{r,a,p}+\beta_u\psi + \beta_v \phi, j_0 +\dt j)$ is $J_\delta$-holomorphic if and only if the equation
\[
\beta_u \Theta_u + \beta_v \Theta_v=0
\]
is satisfied. $\Theta_u$ and $\Theta_v$ are equations depending on $\psi_u,\psi_v, \dt j$, and they take the following form
\[
\Theta_u = D_{J} \psi + \beta_v'\phi+\mcal{F}_u(\psi,\phi) + \mcal{E}_u(\psi,\phi)
\]
and 
\[
\Theta_v= \beta_u' \psi + D_{J_\dt} \phi + \mcal{F}_v(\phi,\psi).
\]
The forms of functionals $\mcal{F}_*,\mcal{E}_*$ are given in the course of the proof.
\end{proposition}

\begin{remark}
We will write the equation $\Theta_u$ and $\Theta_v$ in two different forms, one form will make it easy to apply elliptic regularity, the other makes it easy to use the contraction mapping principle. It will be later crucial for us to use elliptic regularity, as when we do finite trajectory gluing we will lose one derivative by lengthening/shortening the domain of the neck, and we will use elliptic regularity to gain one derivative to make up for this. The key ingredient is to arrange things so that $\Theta_v$ does not contain derivatives of $\psi$, and $\Theta_u$ does not contain derivatives of $\phi$. We shall see that this requires some careful differential geometry to achieve.
\end{remark}
\begin{proof}
\textbf{Step 0}
We first prepare to write our equation in a way that makes apparent the elliptic regularity in the equation, then we will linearize everything to make linear operators appear. We first consider
\[
\db_{J_\dt} (u_{r,a,p} + \beta_u \psi + \beta_v \phi)
\]
in the region $s>R$. We recall over in this region $u_{r,a,p} = v_{r,a,p}$. Let's use $u_*$ to denote $u_{r,a,p}$ for short. Then we are looking at the equation
\[
\partial_s (u_* + \beta_u \psi + \beta_v \phi) + J_\dt (u_* + \beta_u \psi + \beta_v \phi) \p_t (u_* + \beta_u \psi + \beta_v \phi) =0.
\]
We rewrite this in the following fashion
\begin{align*}
&\partial_su_* +\beta_u'\psi + \beta_v'\phi + \beta_u\p_s \psi + \beta_v \p_s \phi + J_\dt (u_* +\beta_u \psi + \beta_v \phi) \p_t (u_* + \beta_u \psi + \beta_v \phi)\\
=& \beta_v (\partial_s \phi + J_\dt(u_*+\beta_u\psi + \beta_v \phi) \p_t \phi + \beta_u'\psi)+  \beta_v' \phi + \partial_s u_* + J_\dt (u_*+\beta_u\psi +\beta_v\phi) \partial_t u_*  \\
&+\beta_u (\p_s \psi + J_\dt (u_*+\beta_u\psi +\beta_v\phi) \p_t \psi)\\
=&0.
\end{align*}
Recalling that $\p_s u_* + J_\dt (u_*) \p_t u_*=0$ in this region, we can write $\partial_s u + J_\dt (u+\beta_u\psi +\beta_v\phi) \partial_t u$ as
\[
\p_s u_* + J_\dt(u_*) \p_t u_* + \partial_{\beta_u\psi} J_\dt (u_*) \p_tu_* + \p_{\beta_v \phi} J_\dt(u_*) \p_t u_* + G(\beta_v\psi, \beta_u \phi) =0
\]
where we can further write 
\[
G(\beta_v\psi, \beta_u \phi) = \beta_v\phi g_v(\beta_u\psi,\beta_v \phi) + \beta_u\psi g_u(\beta_u\psi,\beta_v \phi)
\]
where $g_u(x,y)$ and $g_v(x,y)$ are smooth functions so that pointwise
\begin{equation}\label{equation:property of g}
|g_*(x,y)|\leq C (|x| + |y|)
\end{equation}
and $g_*$ have uniformly bounded derivatives.
If we introduce the modified cutoff functions,
\[
\beta_{ug}:= \beta_{[1/2;R-1/2,2R;R/2]}
\]
\[
\beta_{vg}:= \beta_{[R;R/2,2R+2;1/2]}.
\]
Then we define $\Theta_v$ to be
\begin{equation}\label{theta_v_elliptic}
    \Theta_v : = \p_s \phi + J_\dt(v_{r,a,p} + \beta_{ug}\psi + \beta_{[1;R-2,\infty]}\beta_v\phi) \p_t \phi+ \partial_\phi J_\dt(v_{r,a,p}) \p_t v_{r,a,p} + \beta_{[1;R-2,\infty]}\phi g_v(\beta_{ug}\psi,\beta_v\phi) + \beta_u'\psi.
\end{equation}
We make a few remarks about the important features of our definition of $\Theta_v$. We first remark by our cut off function $\beta_{[1;R-2,\infty]}$, the equation becomes linear for $s<R-1$, as all the quadratic terms have disappeared. This is desirable as we will be solving $\Theta_v$ with $W^{2,p}(v^*TM)$ with exponential weight $e^{ds}$. Usually having vector fields that grow exponentially is undesirable for doing analysis, but in our case where the vector field grows exponentially the equation is linear, and hence poses no problem for the solution for our equation. The deformed preglued curve also doesn't see the segments of $\phi$ that grows exponentially by our choice of cut off functions.

We also remark that $\Theta_v$ appears in a form that allows us to apply elliptic regularity as stated in Theorem \ref{Mcduff_elliptic}, which we will need much later on.

The definition of $\Theta_u$ is slightly more involved. From now on we think of $(s,t)$ as coordinates in the cylindrical ends of $u$. Let $\overline{u}$ denote the interpolation from $v_{r,a,p}$ to $u+(r,a,p)$ that starts at $s=2R+1$ and finishes the interpolation process at $s=2R+2$. The difference between $v_{r,a,p}$ and $u+(r,a,p)$ in this interval is uniformly bounded in $C^k$ norm by $C(e^{-2DR} +R\dt)$ over $s\in [2R+1,2R+2]$.  Note also where $\beta_u$ is nonzero and $s>R$, $\overline{u}$ agrees with $u_*$. Let us also consider 
\[
\p_s \overline{u} + J_\dt (\overline{u} +\beta_u\psi +\beta_{vg} \phi) \p_t \overline{u}
\]
which we expand as
\[
\p_s \overline{u} + J_\dt(\overline{u}) \partial_t \overline{u} + \partial_{\beta_u \psi} J_\dt (\overline{u})\p_t\overline{u} + \p_{\beta_{vg}\phi} J_\dt (\overline{u})\p_t\overline{u} + \overline{G}(\beta_u\psi, \beta_{vg}\phi) 
\]
where the definition of $\overline{G}$ is analogous to that of $G$.
We recognize that the first term $\p_s \overline{u} + J_\dt(\overline{u})\p_t\overline{u}$ is supported for $s\in [2R+1,\infty]$ whenever $s>R+1$ and is of size (in $C^k$ norm) $C(e^{-2DR} +R\dt)$. The term $\overline{G}(\beta_u\psi, \beta_{vg}\phi)$ admits a similar expansion as $G$ above that gives
\[
\overline{G}(\beta_{u}\psi, \beta_{vg}\phi) = \beta_{vg}\phi \overline{g}_v(\beta_u\psi,\beta_{vg} \phi) + \beta_u\psi \overline{g}_u(\beta_u\psi,\beta_{vg} \phi)
\]
with $\overline{g}_*$ satisfying the same norm bound as before.
Then for $s>R$ we define $\Theta_u$ to be
\begin{equation*}
    \Theta_u : = \beta_v' \phi + \p_s \psi + J_\dt(\overline{u} + \beta_{vg}\phi + \beta_u \psi) \p_t \psi + \p_\psi J_\dt (\overline{u}) + \psi \overline{g}_u(\beta_u\psi,\beta_{vg}\phi).
\end{equation*}
Note that we choose $\overline{g}_u$ to agree with $g_u$ for $s<2R$.
Then we observe by this construction over $s\in [R,\infty)$, the equation:
\[
\beta_u \Theta_u + \beta_v \Theta_v=0
\]
implies directly that the deformation of the pregluing $u_*$ under $\beta_u \psi + \beta_v \phi$ is $J_\dt$ holomorphic.

The definition of $\Theta_u$ extends also naturally to $s\in [0,R]$ as
\begin{equation*}
    \partial_s (u_* +\psi) + J_\dt(u_*+\psi)\p_t(u+\psi) =0
\end{equation*}
as in this region the effect of $\phi$ vanishes. The extension of $\Theta_u$ to the interior of $u$ is standard, albeit one also needs to take into account of deformation of complex structure $\dt j$ in the interior of $u$.

As promised the derivatives of $\psi$ does not appear in $\Theta_v$ and vice versa. As written it is manifest that solutions of $\Theta_u$ and $\Theta_v$ satisfy elliptic regularity. We next rewrite them into a form that makes the linearizations of operators appear, and hence more amendable to fixed point techniques.

\textbf{Step 2} We now establish an alternative form of $\Theta_v$, namely we take Equation \ref{theta_v_elliptic} and expand the nonlinear terms. We get
\begin{equation*}
    \Theta_v = D_{J_\dt} \phi + \beta_u' \psi + \beta_{[1;R-2,\infty]}\phi g_{v1}(\beta_{ug}\psi,\beta_v\phi) + \p_t\phi g_{v2}(\beta_{ug}\psi,\beta_{[1;R-2,\infty]}\beta_v\phi)
\end{equation*}
where $g_{v*}$ have the same properties as $g_*$. Even though they are different functions, we will sometimes just write $g_v (\phi +\p_t\phi)$ for convenience. We then can take
\begin{equation*}
    \mcal{F}_v:=\beta_{[1;R-2,\infty]}\phi g_{v1}(\beta_{ug}\psi,\beta_v\phi) + \p_t\phi g_{v2}(\beta_{ug}\psi,\beta_{[1;R-2,\infty]}\beta_v\phi)
\end{equation*}
which we think of a quadratic term. There is no error term.

\textbf{Step 3} The analogous expression for $\Theta_u$ is more complicated, in part because we have to deal with asymptotic vectors and have to pull back everything to $W^{2,p,d}(u^*TM)\oplus V_\Gamma$. We first focus on $s>R$ part of $\Theta_u$ from which we can write this as
\[
\beta_v' \phi + [(\p_s + J_\dt (\overline{u})\p_t)\psi + \p_\psi J_\dt(\overline{u})] + \psi \overline{g}_u(\beta_u\psi,\beta_{vg}\phi) + (J_\dt(\overline{u} +\beta_{vg}\phi + \beta_u\psi)- J_\dt(\overline{u}))\p_t \psi=0. 
\]
We loosely think of $[(\p_s + J_\dt (\overline{u})\p_t)\psi + \p_\psi J_\dt(\overline{u})]$ as the linearization, and the rest of the terms as quadratic perturbation. The quadratic terms
\[
\psi \overline{g}_u(\beta_u\psi,\beta_{vg}\phi) + (J_\dt(\overline{u} +\beta_{vg}\phi + \beta_u\psi)- J_\dt(\overline{u}))\p_t \psi
\]
generally take the form:
$\psi \cdot g(\psi,\phi) + \p_t\psi g(\psi,\phi)$
where $g$ is the function having the property of Equation \ref{equation:property of g} and uniformly bounded derivatives.

Next we consider what happens for $s\leq R$, where $\Theta_u$ takes the form 
\[
\p_s(\overline{u} + \psi) + J_\dt (\overline{u} +\psi) \p_t(\overline{u} +\psi)=0
\]
which we can rewrite as
\[
\p_s\psi + J_\dt(\overline{u}) \p_t \psi + \p_\psi J_\dt(\overline{u}) \p_t \overline{u} + g(\psi) \p_t \psi  + \p_s \overline{u} + J_\dt(\overline{u}) \p_t \overline{u}.
\]
We think of $\p_s\psi + J_\dt(\overline{u}) \p_t \psi + \p_\psi J_\dt(\overline{u}) \p_t \overline{u}$ as the linear term, $g(\psi) \p_t \psi$ as the quadratic correction ($g$ is just some function satisfying property of Equation \ref{equation:property of g}), and $\p_s \overline{u} + J_\dt(\overline{u}) \p_t \overline{u}$ the pregluing error, which was already estimated in the previous proposition.

We next wish to understand how the linear terms in the various segments of $\Theta_u$ compare with the linearization of $\db_J$ along $u$, which we turn to in the next step.

\textbf{Step 4} We first focus on $s<R$. We are trying to compare the linearization term in $\Theta_u$ to $D_J$, which can be written as 
\[
\p_s\psi +J(u)\p_t\psi + \p_\psi J(u) \p_t u.
\]
We compare their difference. We first consider the linear term in $\Theta_u$ with $J$ instead of $J_\dt$, and in taking their difference we see terms of the form
\[
(J(u) -J(\overline{u}) )\p_t \psi + \p_{\psi} J(u) \p_t u -\p_\psi J(\overline{u}) \p_t \overline{u}.
\]
In the first term above the difference is of the form $C(s,t) (r,a,p) \p_t\psi + C(s,t)\beta_{[1;R-2.R+2;1]}e^{-DR} \p_t\psi$ where $e^{-DR} \p_t\psi$ is coming from pregluing error. 
In the second term above we can write it as:
\[
C(s,t) (r,a,p) \psi + C \psi \beta_{[1;R-2.R+2;1]} e^{-DR}
\]
the second term coming from the difference between $\p_t u$ and $\p_t \overline{u}$.

Then we must take into accout the difference between $J_\dt$ and $J$, this introduces terms of the form
\[
C \dt \psi + C \dt \p_t \psi.
\]
This concludes our computations for the $s<R$ region. For $s>R$, we repeat a similar procedure, we recall the linear term in $\Theta_u$ in this region takes the form
\[
\p_s\psi  +J_\dt(\overline{u})\p_t\psi  + \p_\psi J_\dt (\overline{u})\p_t \overline{u}.
\]
As before to understand this difference we first replace instances of $J_\dt$ with $J$, and get
\begin{align*}
&D_J\psi - (\p_s\psi  +J(\overline{u})\p_t + \p_\psi J(\overline{u})\p_t\overline{u})\\
=& C(s,t) \{((r,a,p)+ C(\dt +e^{-Ds})) \beta_{[1;R-1, 2R+2;1]} + \beta_{[1;2R-2,2R+2;1]} (e^{-DR} +\dt)\} \p_t \psi \\
&+C(s,t) \{((r,a,p)+ C(\dt +e^{-Ds})) \beta_{[1;R-1, 2R+2;1]} + \beta_{[1;2R-2,2R+2;1]} (e^{-DR} +\dt)\} \psi
\end{align*}
where the terms of the form $C(\dt +e^{-Ds}) \beta_{[1;R-1, 2R+2;1]}$ and $\beta_{[1;2R-2,2R+2;1]} (e^{-DR} +\dt)$ came from the difference between $v_{r,a,p}$ and $\overline{u}$. Finally the effect of putting $J_\dt$ is to add a term of size:
\[
C\dt \psi + C \dt \p_t\psi.
\]

Hence collecting all of the above computations, we can write \begin{equation*}
    \Theta_u = \beta_v'\phi+ D_J \psi + \mcal{F}_u + \mcal{E}_u
\end{equation*}
where we think of $\mcal{F}_u$ as the quadratic term and $\mcal{E}_u$ as the error term. They take the following forms:
\begin{equation*}
    \mcal{F}_u = \begin{cases}
    g(\psi)\p_t\psi + C(s,t)(r,a,p)\psi + C(s,t)(r,a,p) \p_t\psi+C(s,t) (r,a,p)^{2}e^{-Ds}, s<R\\
    \psi \overline{g}_u(\beta_u\psi,\beta_{vg}\phi) + (J_\dt(\overline{u} +\beta_{vg}\phi + \beta_u\psi)- J_\dt(\overline{u}))\p_t \psi + C(s,t)(r,a,p) \psi + C(s,t) (r,a,p) \p_t\psi, s\geq R
    \end{cases}
\end{equation*}
and for $s<R$
\begin{align*}
    \mcal{E}_u =& 
    C \dt \psi + C \dt \p_t \psi +\beta_{[1;R-2.R+2;1]}e^{-DR} \p_t\psi +\beta_{[1;R-2.R+2;1]}e^{-DR} \p_t\psi\\
            & +C\delta(1+|r,a,p|) \beta_{[0,R+1;1]} \\
            &+ C[e^{-DR}(1+|r,a,p|)  + C\delta(1+|r,a,p| )] \beta_{[1;R-2,R+2;1]}.
\end{align*}
For $s\geq R$ we have
\begin{align*}
    \mcal{E}_u =&  C\{(\dt +e^{-Ds}) \beta_{[1;R-1, 2R+2;1]} + (\beta_{[1;2R-2,2R+2;1]} (e^{-DR} +\dt))\}\p_t \psi \\
    &+ C\{(\dt +e^{-Ds}) \beta_{[1;R-1, 2R+2;1]} + (\beta_{[1;2R-2,2R+2;1]} (e^{-DR} +\dt))\} \psi \\
    &+C\dt \psi + C \dt \p_t\psi.
\end{align*}
\end{proof}

We also need to version of elliptic regularity given in Proposition B.4.9 in \cite{mcdfuff}, which we reproduce here.
\begin{theorem}\label{Mcduff_elliptic}
Let $\Omega' \subset \Omega$ be open domains in $\mathbb{C}$ so that $\overline{\Omega}' \subset \Omega$. Let $l$ be a positive integer and $p>2$. Assume $J\in W^{l,p}(\Omega,\mathbb{R}^{2n})$ satisfies $J^2 =-1$, and $\|J\|_{W^{l,p}} \leq c_0$, then:
\begin{enumerate}
    \item If $u\in L^p_{loc}(\Omega, \bb{R}^{2n})$, $\eta \in W^{l,p}_{loc}(\Omega,\bb{R}^{2n})$, and $u$ weakly solves 
    \begin{equation}
        \p_s u +J\p_tu = \eta.
    \end{equation}
    Then $u\in W^{l+1,p}_{loc} (\Omega,\bb{R}^{2n})$, and satisfies this equation almost everywhere.
    \item \begin{equation}
        \|u\|_{W^{l+1,p}(\Omega',\bb{R}^{2n})} \leq c( \|\p_su +J\p_tu\|_{W^{l,p}(\Omega,\bb{R}^{2n})} +\|u\|_{W^{l,p}(\Omega,\bb{R}^{2n})})
    \end{equation}
    where $c$ only depends on $c_0, \Omega, $ and $\Omega'$.
\end{enumerate}
\end{theorem}

\begin{remark}\label{quadratic_term}
In what follows, ignoring for now our choice of cut off functions, we will think of $\mcal{F}_v$ in the following form:
\begin{equation*}
\mcal{F}_v (\phi,\psi)= g(\phi,\psi) \phi + h(\phi,\psi)\p_t(\phi)
\end{equation*}
where measured in $C^0$ norm we have, \[
|g(x,y)| \leq C (|x|+|y|)
\]

\[
|h(x,y)| \leq C(|x|+|y|)
\]
We also have $g,h$ are both smooth functions whose derivatives are uniformly bounded, which in particular implies that the $W^{k,p}$ norm of $g(\phi,\psi)$ and $h(\phi,\psi)$ are bounded above by the $W^{k,p}$ norm of $\phi$ and $\psi$.

In comparison with Section 5 of \cite{obs2}, our conditions on $\mcal{F}_v$ are slightly stronger than the condition in there called \emph{quadratic of type 2} because only the derivative of $\phi$ is allowed to appear.
We will think of $\mcal{F}_u$ in the following form (note this is slightly different from above conventions):
\begin{equation*}
\mcal{F}_u = g(\beta_v \phi, \psi,r,a,p) + h(\beta_v \phi, \psi,r,a,p) \p_t (\psi)
\end{equation*}
Ignoring the precise details of cut off functions, we have (all norms below are the $C^0$ norm)
\[
\|g\| \leq C(\|\phi\| \cdot \|\psi\| +\|\psi\|^2 + |(r,a,p)|^2 e^{-Ds} + |(r,a,p)|(\|\phi\|+\|\psi\|))
\]
\[
\|h\| \leq C (|(r,a,p)| + \|\phi\| + \|\psi\|)
\]

Analogous expressions for pointwise bounds for higher derivatives of $\mcal{F}_u$ also hold, essentially because $\mcal{F}_u$ comes from expanding a smooth function. For most of our purposes the bounds above will suffice.
\end{remark}

\begin{remark}
The terms $\mcal{F}_*$ and $\mcal{E}_*$ are viewed as error terms, so what is important is their relative sizes and not the constants appearing in front of them. In what follows we will not be too careful to distinguish $+\mcal{F}_*$ and $-\mcal{F}_*$ and similarly for $\mcal{E}_*$.
\end{remark}

\subsection{Solution of \texorpdfstring{$\Theta_u =0, \Theta_v=0$}{Theta=0}}
In this subsection we will finally solve the system of equations $\Theta_u=0$, $\Theta_v=0$. We will adopt the following strategy:
\begin{itemize}
    \item Given fixed $(r,a,p),\psi$, construct our lift of gradient trajectory, $v_{r,a,p}$, which we preglue to $u+(r,a,p)$.
    \item For this fixed choice, we solve $\Theta_v(\phi)=0$ over $v_{r,a,p}^*TM$ using the contraction mapping principle to obtain an unique solution, $\phi(r,a,p,\psi)$.
    \item Then we try to solve $\Theta_u=0$ over $u^*TM$. We do this via another contraction mapping principle with input variables $(\psi,r,a,p, \dt j)$. The function $\phi$ enters the equation, but as a dependent on these variables. As such, we need to understand how $\phi$ varies when we change  $\psi,r,a,p$. This is made non-trivial by the fact when we change variables $r,a,p$, the deformation is not local, we are twisting/moving an entire semi-infinite cylinder. We will need to understand under these changes, how the $\phi$ terms that enter $\Theta_u$ change. Hence to keep track of these changes, we will make certain identifications of bundles $v_{r,a,p}^*TM$ and $v_{r',a',p'}^*TM$ so we can compare the solutions of different equations over the same space. Then from that we get from the perspective of the equation $\Theta_u$ over $u^*TM$, $\phi_{r,a,p}$ depends nicely on the variables $(r,a,p,\psi,\dt j)$.
    \item We apply the contraction mapping principle over $u^*TM$ to solve $\Theta_u$.
\end{itemize}
\begin{proposition}
Let $\ep>0$ be fixed and sufficiently small (small relative to the constants $C$ that describe the local geometry of Morse-Bott torus but fixed with respect to $\dt>0$). Let the tuple $(\psi,r,a,p,\dt j)$ be fixed and in an $\ep$ ball around zero. Then we can view $\Theta_v=0$ as an equation with input $\phi \in W^{2,p,d}(v_{r,a,p}^*TM)$. This equation has an unique solution $\phi \in W^{2,p,d}(v_{r,a,p}^*TM)$ whose norm is bounded by
\begin{equation*}
    \|\phi\| \leq \ep/R.
\end{equation*}
Furthermore, this solution $\phi$ is actually in $W^{3,p,d}(v_{r,a,p}^*TM)$, its $W^{3,p,d}$ norm is likewise bounded by $C\ep/R$.
\end{proposition}
\begin{proof}
The equation we need to solve is
\begin{equation*}
    D_{v} \phi +  \mcal{F}_v(\phi, \psi) + \beta_u' \psi  =0
\end{equation*}
where $D_v$ is the linearization of $\db_{J_\dt}$ along $v_{r,a,p}$, which we previously denoted by $D_{J_\dt}$. We dropped the subscript $(r,a,p)$ to make the notation manageable.

Let $Q$ denote the inverse of $D_{v}$. Now consider the map $I:  W^{2,p,d}(v_{r,a,p}^*TM) \rightarrow  W^{2,p,d}(v_{r,a,p}^*TM)$ defined by
\begin{equation*}
    \phi \longrightarrow Q(- \beta_v' \psi- \mcal{F}_v(\phi, \psi))
\end{equation*}
We note a solution to $\Theta_v=0$ is equivalent to $I$ having a fixed point. We demonstrate a fixed point exists via the contraction mapping principle.
Since $\psi \in W^{2,p,d}(u^*TM)$ has norm $\leq \ep$, the function $\beta_u\psi$ viewed as a element in $W^{2,p,d}(v_{r,a,p}^*TM)$ also has norm bounded above by $\epsilon$, hence $\|\beta_u'\psi\| \leq \ep/R$. Also we note for $\|\phi\| \leq \ep$, $\|Q\circ \mcal{F}_v\|\leq C\ep^2$, both of these being measured in $W^{2,p,d}(v_{r,a,p}^*TM)$ norm.

If we let $B_\ep(0)$ denote the $\ep$ ball in $W^{2,p,d}(v_{r,a,p}^*TM)$ then by the above, we see $I$ sends $B_\ep(0)$ to itself. We also see it satisfies the contraction mapping property. If $\phi,\phi' \in B_\ep(0)$
then
\begin{align*}
    \|I(\phi)-I(\phi')\| &\leq \|\mcal{F}_v(\phi,\psi)- \mcal{F}_v(\phi',\psi)\|\\
    &\leq C max \{\|\phi\|_{W^{2,p,d}}, \|\phi'\|_{W^{2,p,d}},\|\psi\|_{W^{2,p,d}}\}\|\phi-\phi'\|\\
    &\leq C\epsilon \|\phi-\phi'\|
\end{align*}
Hence for small enough $\ep$ the conditions for contraction principle is satisfied, the map $I$ has a unique fixed point. Since $D_{v_{r,a,p}}$ is invertible, this is equivalent to $\Theta_v$ having a unique solution.

We can estimate the size of this fixed point. Consider the equation
\begin{equation*}
    \phi = Q( -\beta_v' \psi- \mcal{F}_v(\phi, \psi))
\end{equation*}
If we measure the size of both sides in $W^{2,p,d}(v_{r,a,p}^*TM)$ we get
\[
\|\phi\|\leq C\epsilon/R + C\epsilon \|\phi\|
\]
hence we get
\[
\|\phi\| \leq C\epsilon/R.
\]
The fact we can improve the regularity and bound the $W^{3,p,d}$ norm of $\phi$ follows directly from Theorem \ref{Mcduff_elliptic}.
\end{proof}

We next need to solve $\Theta_u$. As we mentioned in the introduction to this subsection, we think of this equation taking place over $W^{2,p,d}(u^*TM)\oplus V_\Gamma \oplus T\mcal{J}$, with input variables $(r,a,p,\psi,\dt j)$ in an $\ep$ ball. We think of the $\phi(r,a,p,\psi,\dt j)$ term that appears as a dependent variable. From above we know that for each tuple $(r,a,p,\psi,\dt j)$ there exists a unique solution $\phi(r,a,p,\psi,\dt j)$ of small norm. To apply the contraction mapping principle we need to see the \emph{derivative} of $\phi(r,a,p,\psi,\dt j)$ with respect to the tuple $(r,a,p,\psi,\dt j)$ behaves nicely as well. This is made nontrivial by the fact when we vary $(r,a,p)$ we are pregluing different gradient trajectories, hence the solution of $\Theta_{v_{r,a,p}}=0$ takes place in different spaces. We take the approach of identifying all the solutions into one space, and that as $(r,a,p)$ vary the equation over the \emph{same} space changes, and by understanding this change, we understand how the terms in $\Theta_u$ change.

To this end we let the pair $(D_v, W)$ denote the vector space
\[
\left\{W^{2,p,d}(v_{0,0,0}^*TM),e^{ds}\right\}
\]
with operator $D_v$ given by $D_{v_{0,0,0}}$-the linearization of $\db_{J_\dt}$ over $v_{0,0,0}$.
We first consider varying  $r,a$, and keeping $p=0$. Let $\phi_{r,a,0} \in W^{2,p,d}(v_{r,a,0}^*TM)$, then there is an obvious parallel transport map $PT: W^{2,p,d}(v_{r,a,0}^*TM) \rightarrow W$ that sends
\[
\phi_{r,a,0} \longrightarrow \phi_{r,a,0} -r - a \in W
\]
which is an isometry. Here we are using additive notation for parallel transport maps because the metric is flat. We denote its image by $\hat{\phi}_{r,a,0}$. Under this identification $\hat{\phi}_{r,a,0}$ satisfies a different equation, which we denote by $\hat{\Theta}_v$. This equation is of the form
\[
\hat{D}_{r,a,0} \hat{\phi}_{r,a,0} + \hat{\mcal{F}_v}(\hat{\phi}_{r,a,0}, \psi)=0
\]
If we write $D_v = \p_s+ J_\dt(s,t) \p_t + S(s,t)$, then $\hat{D}_{r,a,0}$ is given by
\[
\p_s +J_\dt (s+a,t+r) \p_t +S(s+a,t+r)
\]
and the term $\hat{\mcal{F}_v}$ has some mild dependence on $r,a$ depending on the local geometry of the Morse-Bott torus. The point here is that the term $\phi(r,a,p,\psi,\dt j)$ that enters directly  $\Theta_u$ can be identified with $s\leq 5R$ portion of $\hat{\phi}_{r,a,0}$ solving $\hat{\Theta}_v=0$. Hence to understand how $\phi(r,a,p,\psi,\dt j)$ feeds back into $\Theta_u$ as we vary $(r,a)$ we only need to understand how the parametrized solutions $\hat{\phi}_{r,a,0}$ solving the $(r,a)$ parametrized family of PDEs $\hat{\Theta}_v$ changes.

We would like to extend the previous discussion to include variations of $p$. To do that we need the next lemma which concerns the differential geometry of the situation.

\begin{lemma}
In our chosen coordinate system, let $v_p(s,t)= (a_p(s),t,x_p(s),0)$ be a lift of gradient trajectory satisfying 
\[
a_p(0)=0, x_p(0)=p
\]
and let $v_{p'}=(a_{p'}(s),t,x_{p'}(s),0)$ be another lift satisfying
\[
a_{p'}(0)=0, x_{p'}(0)={p'}
\]
with $|p-p'|\leq \ep$
then there exists a $C$ independent of $\delta$ such that:
\[
|v_p(s,t)-v_{p'}(s,t)| \leq C|p-p'|
\]
for all $s,t$.
\end{lemma}
\begin{proof}
This is a fundamentally a statement about gradient flows. 
We recall $x_p$ satisfies the equation
$x_p(s)_s=\delta f'(x_p)$.
If we reparametrize $x_p(s)$ to be $x_p(\frac{s}{\dt})$ then it is simply a gradient flow of $f$, then we have for all $\delta$ and all $s$
\[
\left |x_p\left (\frac{s}{\dt}\right ) -x_{p'}\left (\frac{s}{\dt}\right ) \right| \leq C |p-p'|
\]
where $C$ is independent of $\delta$, and the claim follows.
To verify the claim about $a_p$ and $a_{p'}$, we need to be a bit more careful. Assume we have re-chosen coordinates $[U,V]$ around $p$ so that on $x\in[U,V]$, we have $f(x) =Mx$ and on $x\in [0,U+\epsilon]$ $f(x)=M'x^2/2$.

This fixed choice of coordinate is independent of $\delta$ so our quantitative conclusions drawn from this coordinate system continues to hold in our original coordinate system up to a change of constant. Then we analyze the behaviour of $a_p(s)$ and $a_{p'}(s) $ as $s\rightarrow-\infty$,  with the positive end being similar. For $s$ so that $x_p(s)\in [U,V]$, $x_p(s) = \dt Ms+p$ and $x_{p'}(s) = \dt Ms+p'$, since 
$a_p' = e^{\delta f(x_p)}$ this is equivalent to
\[
a_p(s)' = e^{\delta (\dt Ms+p)} 
\]
and 
\[
a_{p'}(s)' = e^{\delta (\dt Ms+p')}.
\]
If we take $T$ large enough so that $x_p(-T), x_{p'}(-T) \in [U,U+\epsilon]$, we have the upper bound:
\[
T< C( \max(p,p')-U)/\delta.
\]
Hence  we have a uniform upper bound in the integral of the form:
\begin{align*}
|a_p(-T)-a_{p'}(-T)| &\leq \left|\int_0^{A/\dt} e^{\dt^2 Ms +\dt p} - e^{\dt^2Ms +\dt p'}ds\right| \\
& \leq \frac{1}{\dt^2M}(e^{\dt^2 M A /\dt}-1)(e^{\dt p} - e^{\dt p'})\\
&\leq C|p-p'|
\end{align*}
where $A$ is just some constant.
Next for $s<-T$, the curves $x_p(s)$ and $x_{p'}(s)$ enter the region where $f(x)=M'x^2/2$, and they satisfy the differential equation
\[
x_p(s)' = \delta M'x
\]
so they satisfy
\[
x_p(s) =x_p(-T)e^{\delta M's}
\]
and likewise for $x_{p'}(s)$, hence the difference between $a_p(s)$ and $a_{p'}(s)$ satisfies 
\[
\p_s (a_p-a_p') = e^{\dt x_p(-T)e^{\delta M's}}-e^{\dt x_{p'}(-T)e^{\dt M's}}.
\]
Since we are taking $s\rightarrow -\infty$ the right hand side can be bounded by
\[
C\dt(x_p(-T)-x_{p'}(-T))e^{\delta M's} \leq \dt C|p-p'|e^{\delta M's}.
\]
Hence for $s<-T$
\[
|a_p-a_{p'}|  \leq C|p-p'| \dt \int_{-\infty}^{-T} e^{\dt M's}ds \leq C |p-p'|
\]
and hence our conclusion follows.
\end{proof}
With the above calculations we have can extend the parallel transport map
\[
PT: W^{2,p,d}(v_{0,0,p}^*TM) \longrightarrow W
\]
defined by
\[
PT(\psi) = \psi-(v_{0,0,p}-v_{0,0,0})
\]
which is also an isometry.
The $\Theta_v=0$ equation also pulls back to $W$, the linearized operator $\hat{D}_{0,0,p}$ given by
\[
\p_s\hat{\phi}_{0,0,p} + J_\dt (p) \p_t \hat{\phi}_{0,0,p} + S(p,s,t) \hat{\phi}_{0,0,p}.
\]
This is what $D_{v_{0,0,p}}$ would look like in the coordinates we chose for $v_{0,0,p}$. The full equation $\hat{\Theta}_v=0$ looks like
\[
\hat{D}_{0,0,p} \hat{\phi}_{0,0,p} + \hat{\mcal{F}}_v(p,\hat{\phi}_{0,0,p},\psi)=0.
\]

The previous lemma ensures the coefficient matrices $J_\dt$, $S(p,s,t)$, as well as $\hat{\mcal{F}}_v$ are uniformly well behaved (say in $C^k$ norm) as we vary $p$ as $\dt \rightarrow 0$. And as before the components of $\phi(r,a,p,\psi,\dt j)$ that enters $\Theta_u$ can be identified with the $s>5R$ component of $\hat{\phi}_{0,0,p}$.
Combining the previous discussion, we have:

\begin{proposition}
The derivative of the operator $\hat{D}_{v_{r,a,p}}:W\rightarrow W^{1,p,d}(v_0^*TM)$ with respect to $(r,a,p)$ is well defined, and satisfies for a fixed constant $C$
\[
\|\p_a D_{v_{r,a,p}}\|=0
\]
\[
\|\p_r D_{v_{r,a,p}}\| \leq C
\]
\[
\|\p_pD_{v_{r,a,p}}\| \leq C.
\]
Further we have the bound:
\[
\|\p_* \hat{ \mcal{F}_v}\| \leq C.
\]
In the above, $\p_*D_{v_{r,a,p}}$ is viewed as an operator $W\rightarrow W^{1,p,d}(v_0^*TM)$, and $\hat{\mcal{F}}_v$ is viewed as a map from $W\rightarrow W^{1,p,d}(v_0^*TM)$ and its derivative is a map over the same space.
\end{proposition}
The next step is to understand how $\hat{\phi}_{r,a,p}$ varies with respect to the variables $(r,a,p,\psi)$.

\begin{proposition}
Fix $(r,a,p,\psi)$, let $\hat{\phi}(r,a,p,\psi,\dt j)$ (which for the purpose of this proof we abbreviate $\hat{\phi}$) denote the solution of $\hat{\Theta}_v=0$ viewed as an element of $W$. We can take the derivative of $\p_* \hat{\phi}(r,a,p,\psi)$ where $*=r,a,p,\psi$. They satisfy the equations
\[
\|\p_* \hat{\phi}\| \leq C \epsilon,\quad  * = r,a,p,\psi,\dt j
\]
where the norm is measured in $W^{2,p,d}(v_0^*TM)$ for $*=r,a,p$ and in $Hom(W^{2,p,w}(u^*TM), W^{2,p,w}(v_0^*TM))$ for $* = \psi$, and $Hom(T\mcal{J},W^{2,p,w}(v_0^*TM))$ for $*=\dt j$. See Remark \ref{Frechet} for the interpretation of terms $\p_* \hat{\phi}$.
\end{proposition}
\begin{proof}
The fixed point equation looks like
\begin{equation*}
    \hat{\phi}= \hat{Q}_{r,a,p}( -\beta_u' \psi  - \hat{\mcal{F}_v}(\hat{\phi}, \psi,r,a,p)).
\end{equation*}
This is really a family of equations over $\psi,r,a,p$.

We first differentiate w.r.t. $\psi$. Note unlike differentiating w.r.t. $r,a,p$, the term $\frac{d\phi}{d\psi}$ is a Frechet derivative which should be viewed as a linear operator 
\[
\frac{d\hat{\phi}}{d\psi}: W^{2,p,w}(u^*TM) \longrightarrow W^{2,p,w}(v_0^*TM)
\]
and when we measure its norm it is the operator norm. When we write $\beta_u'$ below we mean the operator defined by multiplication by $\beta_u'$ etc. Differentiating both sides of the fixed point equation we have
\begin{equation*}
\left\|\frac{d\hat{\phi}}{d\psi}\right \| \leq C \left \|Q\circ \left(1/R+ \psi \frac{ d\hat{\phi}}{d \psi} +(\p_t \hat{\phi})\cdot \frac{d\hat{\phi}}{d\psi} + \hat{\phi} \p_t \frac{d \hat{\phi}}{d\psi}\right)\right\|
\end{equation*}
where the norms of both sides are operator norms. To make sense of $Q \hat{\phi} \p_t \frac{d \hat{\phi}}{d\psi}$ see Remark \ref{Frechet}. We recall the $C^1$ norm of is $\psi$ bounded above by $C\epsilon$ via the Sobolev embedding theorem, so the above equation can be rearranged to be
\[
(1-C\epsilon)\|d\hat{\phi}/d\psi \| \leq C(1/R)\leq \epsilon
\]
and this proves the claim for $*=\psi$.
We next consider the case $*=p$, the cases for $*=r,a$ are analogous. We consider the equation
\begin{equation*}
\hat{D}_{r,a,p} \hat{\phi} + \hat{\mcal{F}}_v(\psi,r,a,p, \hat{\phi}) +\beta_u'\psi =0
\end{equation*}
and differentiate both sides w.r.t. $p$:
\begin{equation}\label{pderivative}
    \frac{d\hat{D}_{r,a,p}}{dp} \hat{\phi} + \hat{D}_{r,a,p}\frac{d\hat{\phi}}{dp} = -\frac{d}{dp}\mcal{ F}_v(r,a,p,\psi,\hat{\phi
    })
\end{equation}
rearranging to get
\begin{equation*}
    \frac{d\hat{\phi}}{dp} =\hat{Q}_{r,a,p}\left[-\frac{d\hat{D}_{r,a,p}}{dp} \hat{\phi} - \frac{d}{dp}\mcal{F} (r,a,p,\psi,\hat{\phi})\right].
\end{equation*}
The only new thing we need to estimate is $\frac{d}{dp}\hat{\mcal{F}} (r,a,p,\psi,\hat{\phi})$
which we calculate as
\begin{align*}
&\left\|\hat{Q}_{r,a,p} \circ \frac{d}{dp} \hat{\mcal{F}}_v(r,a, p,\psi, \hat{\phi})\right\| \\
&=\left \| \hat{Q}_{r,a,p} \circ \left\{\p_{\hat{\phi}}g (\hat{\phi},\psi) \frac{d\hat{\phi}}{dp} +\p_{\hat{\phi}}h(\hat{\phi},\psi)\p_t\hat{\phi} \frac{d\hat{\phi}}{dp} + h(\hat{\phi},\psi )\p_t \frac{d\hat{\phi}}{dp} + \p_p g(\hat{\phi},\psi) + \p_p h(\hat{\phi},\psi) \p_t (\hat{\phi} )\right\} \right\|\\
& \leq C(\|\psi\| \cdot \|\hat{\phi}\| + \|\hat{\phi}\|^2 ) +\epsilon \left\|\frac{d\hat{\phi}} {dp}\right\|
\end{align*}
where $\p_p h$ and $\p_pg$ refer to how the functions $h$ and $g$ themselves depend on $p$. The norms above are all in $W^{2,p,d}$ norm in the domain.
Combining the above inequalities we conclude
\[
\|d\hat{\phi}/dp\| \leq C\epsilon.
\]
The same proof works for $*=r,a,\dt j$.

\end{proof}

\begin{remark}\label{Frechet}
As we mentioned in the course of the proof, the attentive reader might feel uneasy about the appearance of the term $Q\circ \p_t \frac{d\hat{\phi}}{d\psi}$. The proper way to take this Frechet derivative is explained in Proposition 5.6 in \cite{obs2}. The idea is to take the fixed point equation
\[
 \hat{\phi}= \hat{Q}_{r,a,p}(- \beta_u' \psi  - \hat{\mcal{F}}(r,a,p,\hat{\phi}, \psi)).
\]
Let $\mcal{D}_\psi$ and $\mcal{D}_{\hat{\phi}}$ denote the derivative of the right hand side of the above equation with respect to $\psi,\hat{\phi}$ respectively, then the derivative $\frac{d\hat{\phi}}{d\psi}$ is defined to be $(1-\mcal{D}_{\hat{\phi}})^{-1} \mcal{D}_\psi$, and sends $W^{2,p,w}(u^*TM) \rightarrow W^{2,p,w}(v_0^*TM)$. In light of this the composition $Q\circ \p_t \frac{d\hat{\phi}}{d\psi}$ is not at all problematic, and our estimates of norms continue to hold.

We take our approach to things because in the case of differentiation with respect to the parameters $*=r,a,p$ (say for $p$ for definiteness), the resulting derivative is a linear map from $\mathbb{R} \rightarrow W^{2,p,w}(v_0^*TM)$, and hence has a right to be viewed as an honest function in $W^{2,p,w}(v_0^*TM)$. It further satisfies an elliptic PDE as in Equation \ref{pderivative}, which gives us estimates on its norms and exponential decay properties, which will be essential for our later purposes.
\end{remark}

\begin{remark}
We can actually show $\frac{d\hat{\phi}}{dp}$ (and likewise for $*=r,a$) belongs in $W^{3,p,d}(v_0^*TM)$ with the help of elliptic regularity. If we recall the form of $\Theta_v=0$ written in Equation \ref{theta_v_elliptic}, we differentiate with respect to $p$ to see $\frac{d\hat{\phi}}{dp}$ weakly satisfies an equation of the form
\[
\p_s \frac{d\hat{\phi}}{dp} + J_\dt(v_{r,a,p} +\beta_{ug}\psi+ \beta_{[1;R-2,\infty]}\beta_v \hat{\phi}) \p_t \frac{d\hat{\phi}}{dp} + \mcal{F}(\hat{\phi},\p_t\hat{\phi}, \frac{d\hat{\phi}}{dp},\psi)=0
\]
where $\mcal{F}$ is a smooth function. Using elliptic regularity we see that $\frac{d\hat{\phi}}{dp}$ is in $W^{3,p,d}(v_0^*TM)$ with norm bounded above by $C\ep$.
\end{remark}
Finally we turn our attention to solving $\Theta_u(\psi,\phi,r,a,p)=0$. 
\begin{proposition}
The equation $\Theta_u=0$ has a solution.
\end{proposition}
\begin{proof}
Recall we view $\Theta_u(r,a,p,\psi, \phi(r,a,p,\psi))$ as an equation with independent variables $(r,a,p,\psi,\dt j) \in W^{2,p,d}(u^*TM) \oplus V_\Gamma \oplus T\mcal{J}$, and we have the surjective linearized operator
\[
D_J: W^{2,p,d}(u^*TM) \oplus V_\Gamma \oplus T\mcal{J} \longrightarrow W^{1,p,d}(\overline{\op{Hom}}(T\dot{\Sigma}, u^*TM)).
\]
The equation $\Theta_u=0$ over the entire domain $\dot{\Sigma}$ can be written to be of the form
\begin{equation}
    D_J (r,a,p,\psi,\phi) + \mcal{F}_u +\mcal{E}_u + \mcal{F}_{int}(\psi,\dt j)=0
\end{equation}
where $ \mcal{F}_u$ is supported in the cylindrical neighborhood $[0,\infty)\times S^1$ and the term $\mcal{F}_{int}$ is quadratic in $\psi,\dt j$ and supported in $\dot{\Sigma}\setminus [0,\infty)\times S^1$ and $\mcal{E}_u$ is described by Proposition \ref{proposition_equations} and supported in the cylindrical neck. We let $Q_u$ denote a right inverse to $D_J$. Then to find a solution to $\Theta_u$ it suffices to find a fixed point of the map
\[
I:(r,a,p,\psi, \dt j)\longrightarrow Q_u(-\mcal{F}_u-\mcal{E}_u-\mcal{F}_{int}).
\]
Let $B_\ep$ denote the $\epsilon$ ball in $W^{2,p,d}(u^*TM) \oplus V_\Gamma \oplus T\mcal{J}$. It follows from the fact $\mcal{F}_{int}$ is quadratic, our previous derived expressions for $\mcal{F}_u$, $\mcal{E}_u$, and our size estimate $\|\phi\|\leq C\epsilon$ that for $\dt <<\epsilon$ small enough, $I$ maps $B_\ep$ to itself. It further follows from the fact that $\|\p_*\phi \|\leq C\ep$ and the explicit expressions of $\mcal{F}_u,\mcal{F}_{int}, \mcal{E}_u$ that $I$ is a contraction mapping, and hence a solution to $\Theta_u=0$ exists.
\end{proof}
\begin{remark}
The contraction mapping principle actually says the fixed point of $I$ is unique. However this does not mean the solution to $\Theta_u$ is unique. If $D_J$ is not injective (which it never is if the curve is nontrivial due to translations in the symplectization direction, and if the curve is a free trivial cylinder there is also a global translation along the Morse-Bott torus), we could have chosen a different right inverse $Q'$ which leads to a (presumably) different solution of $\Theta_u$. We leave discussions of uniqueness of gluing to after when we glued together general cascades.
\end{remark}
We note that even though the previous construction was only for one end, the construction works for arbitrary number of free ends.
\begin{corollary}
Given a transversely cut out $J$-holomorphic curve $u$ with free-end Morse-Bott asymptotics, the ends can be glued with semi-infinite gradient trajectories into $J_\delta$-holomorphic curves.
\end{corollary}
\begin{remark}
In the above we only glued gradient flows to free ends. We could have also glued in trivial cylinders to fixed ends. The only difference is instead of $V_\Gamma$ being spanned by $r,a,p$ it is only spanned by $r,a$. The rest of the argument follows exactly the same way.
\end{remark}

\section{Exponential decay for solution of \texorpdfstring{$\Theta_v$}{Theta}} \label{expdecay}
Consider, in notation of previous section, $\Theta_v=0$ for $s>3R$, then it is an equation of the form
\begin{equation*}
    D_v \phi + \mcal{F}(\phi)=0
\end{equation*}
where $D_v$ denotes the linearization of $\p_{J_\dt}$ operator, and $\mcal{F}$ we loosely think of as an quadratic expression in $\phi$ and $\p_t\phi$, see Remark \ref{quadratic_term}. 
In this section we study the properties of this solution, in particular, it exhibits exponential decay as $s\rightarrow \infty$ for $\dt$ sufficiently small (exponential decay beyond what is imposed by the exponential weight $e^{ds}$). This property will be crucial for our gluing construction for multiple level cascades. 
The idea why $\phi$ undergoes exponential decay is the following: for $\delta>0$ sufficiently small, the gradient flow cylinder $v_{r,a,p}$ flows so slowly that locally the geometry resembles that of a trivial cylinder, and the usual proof that $J$-holomorphic curve decays exponentially along asymptotic ends can be applied.\\
The section is organized as follows: We first remove the exponential weights from our Sobolev spaces and work instead in $W^{2,p}(v^*TM)$. Next we follow the strategy of \cite{Hofer2} Section 2(this strategy as far as we know also dates back at least to \cite{Floer}, see Section 4, and is used frequently in various kinds of Floer homologies to prove exponential convergence), using second derivative estimates to derive the exponential decay, and finally we show the various derivatives of $\phi$ also decay exponentially.

\subsection{Exponential decay for solutions of \texorpdfstring{$\Theta_v$}{Theta}}
We begin by studying the exponential decay of $\phi$, then move on to study the exponential decay of its derivatives. First some setup that will be used for both cases.
\subsubsection{Change of coordinates and setup}
We study $\Theta_v=0$ for $s>3R$, which takes the form
\begin{equation*}
    D_v \phi + \mcal{F}_v(\phi)=0.
\end{equation*}
WLOG we assume $(r,a,p)=(0,0,0)$ and write $v$ instead of $v_{r,a,p}$. It will be clear our analysis holds for any value of $(r,a,p)$ and later we will identify sections of $v_{r,a,p}^*TM$ with sections of $v_0^*TM$ via parallel transport.

Recall $\mcal{F}_v(\phi)$ takes the form
\begin{equation*}
\mcal{F}_v (p,\phi)= g(p,\phi) \phi + h(p,\phi)\p_t\phi.
\end{equation*}
Here we have made explicit the dependence of this term on the $p$, which controls the background geometry. It also implicitly depends on $(r,a)$, which we suppress from our notation. 
Here the functions satisfy (uniformly in $p$) $\|g(\phi)\|_{C^0}\leq \|\phi\|_{C^0}$, $\|\nabla g(\phi)\|_{C^0} \leq \|\nabla \phi\|_{C^0}$.
For $h(\phi)$ we have $\|h(\phi)\|_{C^0} \leq \|\phi\|_{C^0}$, $\|\nabla h(\phi)\|_{C^0} \leq \|\nabla \phi\|_{C^0}$. These bounds will be important to us in the subsequent estimates.

Next we change variables to $W^{2,p}(v^*TM)$ i.e. by conjugation we remove the exponential weights on our space. We use the following diagram:
\begin{equation*}
\begin{tikzcd}
W^{2,p}(v^*TM) \arrow[r,"\Theta_v'"] \arrow[d, "e^{-d(s)}"] & W^{2,p}(v^*TM) \\
W^{2,p,d}(v^*TM) \arrow[r,"\Theta_v"] & W^{2,p,d}(v^*TM) \arrow[u, "e^{d(s)}"].
\end{tikzcd}
\end{equation*}
We use this to define the operator $\Theta_v'$. In terms of actual equations it looks like this: if $\zeta$ is the corresponding element of $\phi$ without exponential decay (i.e. $\zeta = e^{ds}\phi$), then $\Theta_v'$ is the same as
\begin{equation*}
    D_{J_\dt}' \zeta + e^{d(s)} \mcal{F}_v(e^{-d(s)}\zeta)=0
\end{equation*}
where $D_{J_\dt}' = e^{ds}D_{J_\dt}e^{-ds}$.
We decompose $D'_{J_\dt}$ as follows
\begin{equation*}
    D'_{J_\dt} = d/ds - A(s)- \delta A
\end{equation*}
where by $A(s) $ we denote self adjoint operator associated with linearizing $\db_J$ along Morse- Bott orbit plus $d$ due to the exponential conjugation
\begin{equation*}
    A = -J_0 \partial_t - S+d.
\end{equation*}
Consequently 
the eigenvalues of $A(s)$ are bounded away from zero, say by a factor of $\lambda >0$.
\begin{remark}
We will often change the value of $\lambda$ from one line to another, as long as it is bounded away from zero. The choice of $\lambda$ above depends somewhat on the choice of $d$, because the operator $-J_0 \frac{d}{dt} -S $ has zero as an eigenvalue. With more careful estimates we can make the decay rate only depend on local geometry, but this won't be necessary for us for purpose of gluing.

In the Section \ref{surj} we also use a $\lambda$ to describe exponential decay behaviour of $J_\dt$-holomorphic curve near a Morse-Bott torus, there the $\lambda$ is genuinely independent of $d$ and only dependent on the local geometry, as will be apparent from our proofs.
\end{remark}

$\delta A$ is the perturbed correction to $A$ due to the fact we are using $J_\dt$ instead of $J$. It has the form 
\begin{equation*}
    \delta A = \delta M d/dt + \delta N
\end{equation*}
where we use $M,N$ to denote matrices whose entries are uniformly bounded in $C^k$ (by abuse of notation we will later use them to denote other matrices
where each of the coefficient terms is uniformly bounded).

\subsubsection{Exponential decay estimates}
Let us define
\begin{equation*}
    g(s) := \int_{S^1} \langle \zeta(s,t), \zeta(s,t) \rangle dt.
\end{equation*}
We shall show:
\begin{proposition} \label{prop_exp_decay}
\begin{equation}
    g''(s) \geq \lambda^2 g.
\end{equation}
\end{proposition}
This proposition combined with the following proposition, will imply exponential decay:
\begin{proposition}[Lemma 8.9.4 in \cite{AudinDamian}]
If $g''(s) \geq \lambda ^2 g(s)$ for $s>s_0$, then either:
\begin{itemize}
    \item $g(s)\leq g(s_0) e^{-\lambda (s-s_0)}$,
    \item $g(s)\rightarrow \infty$ as $s\rightarrow \infty$.
\end{itemize}
\end{proposition}
\begin{proof} [Proof of Proposition \ref{prop_exp_decay}]
\begin{align*}
    g''(s) &= 2(\langle \zeta_s, \zeta_s \rangle + \langle \zeta_{ss} ,\zeta \rangle)
\end{align*}
where when we write $\la \cdot,\cdot \ra $ we implicitly take the $S^1$ integral over $t$. The proof is long and we separate it into steps.\\
\textbf{Step 1}
Let us first determine $\langle \zeta_s, \zeta_s \rangle$.
This is given by 
\begin{align*}
    \langle \zeta_s, \zeta_s  \rangle =& \langle (A+\delta A) \zeta + e^{d(s)} \mcal{F}_v(e^{-d(s)}\zeta), (A+\delta A) \zeta + e^{d(s)} \mcal{F}_v(e^{-d(s)}\zeta) \rangle\\
    =&\langle A\zeta, A\zeta \rangle \\
    &+ \langle A \zeta, \delta A \zeta \rangle\\
    &+ \langle A \zeta, e^{d(s)} \mcal{F}_v(e^{-d(s)}\zeta) \rangle\\
    &+ \langle \delta A \zeta, \delta A \zeta \rangle\\
    &+ \langle \delta A \zeta, e^{d(s)} \mcal{F}_v(e^{-w(s)}\zeta)\rangle\\
    & +\langle e^{d(s)} \mcal{F}(e^{-d(s)}\zeta),e^{d(s)} \mcal{F}(e^{-d(s)}\zeta)\rangle.
\end{align*}
We recall we are not tracking the signs in front of the term $\mcal{F}_v$ since it will eventually be upperbounded.
We look at the \textbf{six} terms, which we label by $\textbf{T1-T6}$, in the above expression one by one, we use \textbf{bold} to remind the reader which term we are referring to since the computation gets very long. Also when we upper bound terms from \textbf{T1-T6} we are implicitly taking the absolute value of terms. We shall keep this convention for all proofs involving exponential decay.

\textbf{T1} gives
\begin{equation*}
    \langle A\zeta, A\zeta \rangle \geq \lambda^2 \langle \zeta, \zeta \rangle.
\end{equation*}
This is because if we expand $\zeta =\sum a_ne_n$, with the collection $\{e_n(s)\}$ of eigenbasis for $A(s)$, we have $Aa_ne_n = \sum \lambda_na_ne_n$. We see this is greater than $\lambda^2 \la \zeta,\zeta \ra $.

\textbf{T2} is given by
\begin{align*}
    \langle A \zeta, \delta A \zeta \rangle &= \langle A \zeta, \delta (MA+N) \zeta \rangle\\
    &= \delta \langle A\zeta, MA \zeta \rangle + \delta \langle A\zeta, N \zeta \rangle.
\end{align*}
The first term above is bounded in absolute value by 
\begin{equation*}
    |\delta \langle A\zeta, MA \zeta \rangle| \leq \delta (\|A\zeta\|^2+\|MA\zeta\|^2)\leq  C\delta \langle A\zeta, A\zeta \rangle.
\end{equation*}
The second term satisfies
\begin{align*}
   | \delta \langle A\zeta, N\zeta \rangle |\leq \delta (\langle A \zeta, A\zeta \rangle +\langle N\zeta, N\zeta \rangle).
\end{align*}
Hence for the \textbf{T2} term we have the overall bound by
\[
 \langle A \zeta, \delta A \zeta \rangle \leq C\dt (\la A\zeta,A\zeta \ra +\la \zeta,\zeta \ra).
\]
\textbf{T3} satisfies
\begin{align*}
    &|\langle A\zeta, e^{d(s)} \mcal{F}_v(e^{-d(s)}\zeta) \rangle |\\
    =& \left|\langle \sqrt{\epsilon}A\zeta, \frac{1}{\sqrt{\epsilon}}e^{d(s)} \mcal{F}_v(e^{-d(s)}\zeta) \rangle \right|\\
    \leq& \epsilon \langle A\zeta, A\zeta \rangle + \frac{1}{\epsilon}\langle e^{d(s)} \mcal{F}_v(e^{-d(s)}\zeta),e^{d(s)} \mcal{F}(e^{-d(s)}\zeta) \rangle\\
    \leq& \epsilon \langle A\zeta, A\zeta \rangle +\epsilon \langle \zeta,\zeta \rangle.
\end{align*}
In the last line we used the fact $\p_t \zeta$ and $\zeta$ have $C^0$ norm uniformly bounded above by $C\ep$.
\textbf{T4} satisfies
\begin{align*}
    \langle \delta A \zeta, \delta A \zeta \rangle =&\delta ^2 \langle MA+N \zeta, MA+N \zeta \rangle \\
     =&\delta^2 \langle MA \zeta, MA\zeta \rangle +\langle N\zeta, N\zeta\rangle +\langle MA \zeta, N\zeta\rangle\\
    \leq& 2 \delta^2 \langle MA\zeta, MA\zeta \rangle + \langle N\zeta, N\zeta \rangle\\
    \leq&  C \delta ^2 (\langle A\zeta, A\zeta \rangle + \langle \zeta, \zeta \rangle).
\end{align*}
\textbf{T5} satisfies 
\begin{align*}
   & \langle \delta A \zeta, e^{d(s)} \mcal{F}_v(e^{-d(s)}\zeta)\rangle\\
     =& \delta \left \langle MA +N \zeta, e^{d(s)} \mcal{F}_v(e^{-d(s)}\zeta)\rangle \right \rangle \\
    \leq& \delta \left[ C \langle A\zeta,e^{d(s)} \mcal{F}(e^{-d(s)}\zeta)\rangle  + \langle N \zeta , e^{d(s)} \mcal{F}_v(e^{-d(s)}\zeta)\rangle \right]\\
     \leq& \delta \left [C \epsilon \langle A\zeta, A\zeta \rangle + (1+\epsilon)\langle \zeta, \zeta \rangle\right ]
\end{align*}
as the second last line follows from previous computation.
\textbf{T6} satisfies
\[
\langle e^{d(s)} \mcal{F}(e^{-d(s)}\zeta),e^{d(s)} \mcal{F}(e^{-d(s)}\zeta)\rangle \leq C\epsilon \la \zeta,\zeta \ra
\]
simply because $\mcal{F}_v$ is quadratic. Putting all these terms together we conclude
\[
\la \zeta_s,\zeta_s \ra \geq (\lambda^2-C\ep)\la \zeta,\zeta \ra
\]
and this concludes the first step.\\
\textbf{Step 2}
We next compute
\begin{align*}
    &\langle \zeta_{ss}, \zeta \rangle\\
    =&\langle \frac{d}{ds}[(A+\delta A) \zeta + e^{d(s)} \mcal{F}_v(e^{-d(s)}\zeta)], \zeta \rangle\\
    =& \langle d/ds(A+\delta A) \zeta,\zeta \rangle\\
    &+ \langle (A+\delta A) \zeta_s,\zeta \rangle\\
    & +\langle \frac{d}{ds} e^{d(s)}\mcal{F}_v(e^{-d(s)})\zeta, \zeta \rangle.
\end{align*}
We will need to dissect these terms one by one. We label them \textbf{T1-T3}.
For \textbf{T1} recall 
\begin{equation*}
    A=-J_0 d/dt -S(s,t)+d,
\end{equation*}
hence its $s$ derivative is a uniformly bounded matrix $\frac{dS}{ds}$ of norm $\leq C \delta$. Here we are using the fact $J$ is the standard almost complex structure along the surface of the Morse-Bott torus.
We also recall
\begin{equation*}
    \delta A = \delta (M \partial_t +N).
\end{equation*}
When we take its $s$ derivative we get
\begin{equation*}
    \frac{d}{ds} \delta A = \delta \frac{dM}{ds} \partial_t + \delta \frac{dN}{ds}.
\end{equation*}
Again we have
\begin{equation*}
    \partial_s (\delta A)= \delta^2 (M A + N)
\end{equation*}
where $M,N$ denotes matrices with bounded entries.
So the \textbf{T1} term is given by:
\begin{align*}
    &|\langle \partial_s( A +\delta A) \zeta, \zeta \rangle |\\
    & \leq C\delta \langle \zeta, \zeta \rangle + \langle \delta ^2 A \zeta, \zeta \rangle\\
    & \leq C \delta \langle \zeta, \zeta \rangle + \delta^2 \langle A \zeta, A\zeta \rangle.
\end{align*}
The \textbf{T2} term looks like
\begin{align*}
    &\langle (A+ \delta A) \zeta_s , \zeta \rangle \\
    =& \langle \zeta_s , (A + \delta A^T) \zeta \rangle \\
    =& \langle (A+\delta A)\zeta +e^{d(s)} \mcal{F}_v(e^{-d(s)}\zeta), (A + \delta A ^T ) \zeta \rangle\\
    =& \langle A \zeta, A\zeta \rangle\\
    & + \langle \delta A \zeta, A\zeta \rangle + \langle A\zeta, \delta A^T \zeta \rangle \\
    & + \langle e^{d(s)} \mcal{F}_v(e^{-d(s)}\zeta), A+ \delta A ^T \zeta \rangle.
\end{align*}
We can estimate the above using the same analysis as before ($\delta A^T$ behaves really similarly to $\delta A$ since we only care about the $\delta$ factor in front, technically we will need to take a $t$ derivative of $M$ but in our case this is still upper bounded by $\dt$ multiplied by a uniformly bounded matrix). This shows all these terms combine to make the \textbf{T2} term satisfy
\begin{equation*}
    \geq \lambda^2/2 g(s).
\end{equation*}

Finally we look at the \textbf{T3} term
\begin{align*}
    &\left \langle \frac{d}{ds} e^{d(s)} \mcal{F}_v(e^{-d(s)}\zeta), \zeta \right \rangle\\
    =& \left \la \frac{d}{ds} [ (g(e^{-ds}\zeta) \zeta + h(e^{-ds}\zeta) \zeta_t)], \zeta  \right \ra\\
    =& \left \la \frac{d}{ds}  (g(e^{-ds}\zeta) \zeta, \zeta\ra + \la \frac{d}{ds}h(e^{-ds}\zeta) \zeta_t),\zeta \right \ra \\
    \leq& \epsilon \langle \zeta, \zeta \rangle + \epsilon \langle \zeta_s, \zeta \rangle
\end{align*}
where we used the elliptic estimate $\|\zeta_{st}\|_{C^0} \leq C\ep$ and $\|\zeta_t\|_{C^0} \leq C\ep$ (technically the version of elliptic regularity in \cite{AudinDamian} or \cite{mcdfuff} only applies to $\phi = e^{-ds}\zeta$, but seeing everything we used above is local, that this implies corresponding bounds on $\zeta$ is immediate).
As before we need to estimate 
\begin{align*}
    &\epsilon \langle \zeta_s, \zeta \rangle\\
    &= \epsilon \langle A +\delta A \zeta + e^{d(s)}\mcal{F}_v(e^{-d(s)})\zeta , \zeta \rangle .
\end{align*}
The third term in the equation above is easily bounded above by
\begin{equation*}
    \epsilon \langle \zeta, \zeta \rangle.
\end{equation*}
The first term is bounded by
\begin{equation*}
   \epsilon( \langle A \zeta, A \zeta \rangle + \langle \zeta, \zeta \rangle).
\end{equation*}
The second term is similarly bounded by 
\begin{equation*}
    \epsilon \delta C (\langle A\zeta, A\zeta\rangle +\langle \zeta, \zeta \rangle)
\end{equation*}
thus the entire \textbf{T3} term satisfies
\begin{equation*}
    \leq C\epsilon( \langle A\zeta, A\zeta \rangle + \langle \zeta, \zeta \rangle)
\end{equation*}
then putting all of these terms together globally we have
\begin{equation*}
    g'' \geq \lambda^2 g
\end{equation*}
for small enough $\epsilon > 0$ and this concludes the proof.
\end{proof}
It still requires some work to go from this to exponential decay in the Sobolev spaces we want. The easiest way to do this is to realize our solution $\zeta$ has $C^0$ norm bounded above by $C \ep/R$. Hence for small enough $\ep$, its $C^0$ norm always below 1. This means $L^2$ norm bounds over intervals of form $[k,k+1]\times S^1$ gives rise to $L^p$ norm bounds over this interval. Using a version of elliptic regularity found in Theorem 12.1.5 in \cite{AudinDamian} (there's a typo in this version) or appendix B of \cite{mcdfuff}, reproduced in Theorem \ref{Mcduff_elliptic}, we conclude the exponential decay bounds can be improved to $W^{k,p}$, which we can then turn to pointwise bounds. We summarize this in the following theorem:
\begin{proposition} \label{prop:higher_order_decay}
For $s>3R$, $j=0,1,..k$,
\[
\|\nabla ^j\zeta\|(s,t) \leq C\|\zeta\|^{\frac{2}{p}}_{W^{2,p}(S^1\times [3R,\infty))}e^{-\lambda (s-3R)}.
\]
Note the $\lambda$ here is not the same as $\lambda$ from before.
\end{proposition}
More details of the elliptic bootstrapping argument is written up in Corollary \ref{cor:expdecay_p}.
\subsection{Exponential decay w.r.t. \texorpdfstring{$p$}{p}}

In this subsection we show the derivative of $\zeta$ with respect to $p$ also decays exponentially. To explain the notation, we recall for each $p$ we can use the parallel transport map to transport $\phi(r,a,p,\psi)$ to $W$. We remove the exponential weights to view them as vector fields:
\[
\zeta(p) \in W^{2,p}(v_0^*TM)
\]
(we suppress the dependence on $r,a,\psi$), and for $s>3R$ they satisfy equations
\[
D'(p) \zeta +e^{d(s)} \mcal{F}_v(p,e^{-d(s)}\zeta(p))=0
\]
where $D'(p)$ is of the form
\begin{equation*}
    D'(p) = d/ds -( A(p) + \delta A(p)).
\end{equation*}
As before $A(p)$ and $\dt A(p)$ take the form
\begin{equation*}
    A(p) =-J_0d/dt - S(p) +d
\end{equation*}
\begin{equation*}
    \delta A(p) = \delta MA-N.
\end{equation*}
The nonlinear term takes the form
\[
\mcal{F}_v (\phi)=  g(p,\phi)\phi + h(p,\phi)\p_t\phi
\]
where $g$ and $h$ and their $p$ derivatives (uniformly with respect to $p$) satisfy the assumptions listed in Remark \ref{quadratic_term} as well as Proposition \ref{prop_locform}.

We know from the above subsection that for each fixed $p$, the vector field $\zeta(p)$ is exponentially suppressed as $s\rightarrow \infty$. In this subsection we show the derivative of this family of vector fields
\[
\frac{d}{dp} \zeta(p)
\]
is exponentially suppressed as $s\rightarrow \infty$, as this will be crucial for our applications in gluing together multiple level cascades.
In this subsection we use $\zeta(p)$ to make explicit the dependence on $p$, and use subscripts $\zeta_p$ to denote the partial derivative with respect to $p$.
For this subsection we define
\[
p'=p/\ep
\]
for $\ep>0$ small enough. This $\epsilon$ is comparable to the $\ep$ balls we have chosen (we can take them to be the same), and depends only on the local geometry near the Morse-Bott torus, and is in particular independent of $\dt$. We write everything in terms of $p'$ instead of $p$. We next differentiate the defining equation for $\zeta(p)$ w.r.t to $p'$:
\begin{equation*}
    \frac{dD'}{dp'}\zeta(p) + D'(p) \frac{d\zeta(p)}{dp'} = \frac{d e^{d(s)} \mcal{F}_v(p,e^{-d(s)}\zeta(p))}{dp'} .
\end{equation*}
By elliptic regularity we can assume $\zeta$ in this region is infinitely differentiable in $s,t$, and its $p'$ derivative is also infinitely differentiable in $s,t$. Further the $s,t$ derivatives of $\zeta_{p'}$ are bounded in $W^{1,p}$ norm by $W^{2,p}$ norms of $\zeta_{p'}$ and $\zeta$.
Now we observe that 
\begin{equation*}
   \frac{ dD'(p)}{dp }\zeta(p) =  (MA-N) \zeta(p)
\end{equation*}
because when we are differentiating $D'(p)$ w.r.t. $p$ we are really looking at how the coefficient matrices $J_\dt$, $S(p)$ behave w.r.t. $p$, and this is determined by the local geometry and hence their variation is uniformly bounded. Hence by our definition of $p'$ we have
\begin{equation*}
   \frac{ dD(p')}{dp'} = \ep (MA-N)=: \ep B \phi.
\end{equation*}
Recalling the form of $\mcal{F}_v$:
\[
\mcal{F}_v (p,\phi)= g(p,\phi) + h(p,\phi)\p_t\phi.
\]
Here we have a $p$ dependence on both $g$ and $h$ since we are shifting the local geometry when we change $p$. Thus 
\[
e^{ds}\mcal{F}_v (p,e^{-ds}\zeta)= g(p,e^{-ds}\zeta)\zeta + h(p,e^{-ds}\zeta)\p_t\zeta
\]
Hence the $p'$ derivative of $e^{ds}\mcal{F}_v (p,e^{-ds}\zeta)$ looks like
\begin{align*}
    & \frac{d}{dp'}e^{ds}\mcal{F}_v (p,e^{-ds}\zeta) \\
    =&\ep g_1(p,e^{-ds}\zeta)\zeta + \ep h_1(p,e^{-ds}\zeta)\p_t\zeta\\
    & + g(p,e^{-ds}\zeta) \zeta_{p'}+ h(p,e^{-ds}\zeta)\p_t\zeta_{p'}\\
    &+ g_2(p,e^{-ds}\zeta)e^{-ds}\zeta_{p'} \zeta + h_2(p,e^{-ds}\zeta) \p_t \zeta \zeta_{p'}
\end{align*}
where $g_1$ and $h_1$ denote the derivative with respect to its first variable, namely $p$. The $\ep$ appears because we are differentiating with $p'$ instead of $p$. The functions $g_1$ and $h_1$ have the same properties as $g$ and $h$, i.e.
\[
g_1(x,y) \leq |x| + |y|
\]
and $g_1$ has uniformly bounded derivatives with respect to each of its variables; similarly for $h_1$.

$g_2$ and $h_2$ are the derivatives of $g$ and $h$ on their second variable. They are just bounded functions whose derivatives are also bounded. 

Hence we can write 
\[
 \frac{d}{dp'}e^{ds}\mcal{F}_v (p,e^{-ds}\zeta) = F + G(\zeta,\zeta_t) \zeta_{p'} + h(p,e^{-ds}\zeta)\p_t\zeta_{p'}
\]
where 
\[
F = \ep g_1(p,e^{-ds}\zeta)\zeta + \ep h_1(p,e^{-ds}\zeta)\p_t\zeta
\]
which essentially behaves like $e^{ds}\mcal{F}_v(p,e^{-ds\zeta})$, and 
\[
G(\zeta,\zeta_t) = g(p,e^{-ds}\zeta) + g_2(p,e^{-ds}\zeta)\zeta + h_2(p,e^{-ds}\zeta)\zeta_t.
\]
Hence $G$ is obviously bounded pointwise by $C(|\zeta| +|\zeta_t|)$, and the derivatives of $G$ w.r.t $s,t$ are also bounded by the corresponding derivatives of $\zeta$ and $\zeta_t$.

So the equation satisfied by $\zeta_{p'}$ is
\begin{equation*}
    \frac{d}{ds} \zeta_{p'} = A \zeta_{p'} +\delta A \zeta_{p'} + \ep B \zeta + F +G(\zeta,\zeta_t) \zeta_{p'} + h(p,e^{-ds}\zeta)\p_t\zeta_{p'}
\end{equation*}
The idea is to let 
\[
g(s) :=\langle \zeta, \zeta \rangle + \langle \zeta_{p'} ,\zeta_{p'} \rangle
\]
and repeat the proof of the previous subsection to show:
\begin{proposition}
$g''(s)\geq \lambda^2g(s)$.
\end{proposition}

\begin{proof}
The term involving $\langle \zeta, \zeta \rangle$ behaves exactly the same way. So let's examine 
\[
 \langle \zeta_{p'}, \zeta_{p'} \rangle''= 2(\langle \zeta_{p's}, \zeta_{p's} \rangle+ \langle \zeta_{p'}, \zeta_{p'ss} \rangle).
\]
\textbf{Step 1}:
The first term looks like
\begin{equation*}
    \langle \zeta_{p's}, \zeta_{p's} \rangle.
\end{equation*}
This is equal to
\[
\la A\zeta_{p'}, A\zeta_{p'} \ra + [...].
\]
We have as before $\la A\zeta_{p'}, A\zeta_{p'} \ra \geq \lambda^2 \la \zeta_{p'},\zeta_{p'} \ra$, we think of the terms in $[...]$ as error terms. We will introduce them one by one and show they are bounded. The first few are of the form (the list continues)
\begin{align*}
    \langle A\zeta_{p'}, \ep B\zeta \rangle, \quad  \langle A\zeta_{p'},  F\ra, \quad \langle A \zeta_{p'}, G(\zeta,\zeta_t) \zeta_{p'} + h(p,e^{-ds}\zeta)\p_t\zeta_{p'} \rangle.
\end{align*}
We shall resume our convention of using \textbf{bold} face letters \textbf{T1-T3} to refer to the above terms. 
\textbf{T1} can be bounded
\begin{align*}
    &\leq \langle \sqrt{\ep} A\zeta_{p'}, \sqrt{\ep} B\zeta \rangle \\
    &\leq \ep \langle A \zeta_{p'}, A \zeta_{p'} \rangle + \langle  \sqrt{\ep} B \zeta , \sqrt{\ep} B \zeta \rangle\\
    &\leq  \ep \langle A \zeta_{p'}, A \zeta_{p'} \rangle + \ep (\langle A \zeta, A \zeta \rangle + \langle \zeta, \zeta \rangle).
\end{align*}
For \textbf{T2}
\begin{align*}
    &\leq \langle A\zeta_{p'} ,\epsilon \zeta\rangle+ \la A\zeta_{p'}, \ep h_1(p,e^{-ds}\zeta)\p_t\zeta\ra \\
    & \leq \epsilon( \langle A \zeta_{p'}, A \zeta_{p'} \rangle + \langle \zeta, \zeta \rangle ) + \ep (\la A\zeta_{p'}, A\zeta_{p'}\ra + \la h_1,h_1\ra\\
    &\leq C \epsilon( \langle A \zeta_{p'}, A \zeta_{p'} \rangle + \langle \zeta, \zeta \rangle ).
\end{align*}
\textbf{T3} is bounded by
\begin{align*}
    &\leq \langle A\zeta_{p'}, \epsilon  \zeta_{p'} \ra +\la A\zeta_{p'}, h(p,e^{-ds}\zeta)\zeta_{p't} \rangle \\
    & \leq \epsilon (\langle A \zeta_{p'}, A\zeta_{p'} \rangle  + \langle \zeta_{p'}, \zeta_{p'} \rangle  +\la \zeta,\zeta\ra).
\end{align*}
There are actually several ways to bound this term. The easiest way as above is to observe $\zeta_p$ has $W^{2,p}$ norm $\leq C\ep$, hence $\zeta_{pt}$ has $C^0$ norm bounded by $C\ep$, and hence $\zeta_{p't}$ has $C^0$ norm bounded by $C\ep $, hence the second term is bounded above by $\ep (\la A\zeta_{p'},A\zeta_{p'}\ra + \la h,h\ra) $ which implies the overall bound by the form of $h$.

More terms that also appear in $\langle \zeta_{p's}, \zeta_{p's} \rangle$ are given below:
\begin{align*}
    &\langle \delta A \zeta_{p'}, \ep B\zeta\rangle, \quad  \langle \delta A \zeta_{p'}, F +G(\zeta,\zeta_t) \zeta_{p'} + h(p,e^{-ds}\zeta)\p_t\zeta_{p'}\rangle,\\
    &\langle \ep B\zeta, \ep B \zeta\rangle,\quad \langle \ep B \zeta, F +G(\zeta,\zeta_t) \zeta_{p'} + h(p,e^{-ds}\zeta)\p_t\zeta_{p'}\rangle,\\
    &\langle F +G(\zeta,\zeta_t) \zeta_{p'} + h(p,e^{-ds}\zeta)\p_t\zeta_{p'}, F +G(\zeta,\zeta_t) \zeta_{p'} + h(p,e^{-ds}\zeta)\p_t\zeta_{p'} \rangle.
\end{align*}
The common feature with all of the above terms is that both inputs into the inner product are small, hence we can bound all of the terms above by
\begin{align*}
&\la \dt A\zeta_{p'}, \dt A \zeta_{p'} \ra, \quad  \la \ep B \zeta,\ep B \zeta\ra, \quad \la F,F\ra, \quad  \la G(\zeta,\zeta_t) \zeta_{p'},G(\zeta,\zeta_t) \zeta_{p'}\ra, \quad\\
&\la h(p,e^{-ds}\zeta)\p_t\zeta_{p'},h(p,e^{-ds}\zeta)\p_t\zeta_{p'}\ra .
\end{align*}

Using techniques already established when we considered exponential decay in the previous subsection, we can bound each of these above terms by (respectively)
\begin{align*}
    &C\dt (\langle A\zeta_{p'},A\zeta_{p'}\rangle + \la \zeta_{p'}, \zeta_{p'} \ra), \quad 
     C\ep (\langle A\zeta, A\zeta \rangle + \langle \zeta,\zeta \rangle), \quad  C\ep (\langle A\zeta, A\zeta \rangle + \langle \zeta,\zeta \rangle),\\
     & 
     \ep \la \zeta_{p'},\zeta_{p'}\ra, \quad  C\ep \la \zeta,\zeta \ra .
\end{align*}
This concludes the first step, in which we bounded all terms appearing in $\la \zeta_{p's},\zeta_{p's}\ra$.

\textbf{Step 2} We next compute
\begin{align*}
    &\langle \zeta_{p'ss}, \zeta_{p'}\rangle\\
    &= \langle  \partial_s((A +\delta A) \zeta_{p'} + \ep B \zeta+ F +G(\zeta,\zeta_t) \zeta_{p'} + h(p,e^{-ds}\zeta)\p_t\zeta_{p'}), \zeta_{p'}\rangle\\
    & = \left \langle (A'+\delta A')\zeta_{p'} + \ep B' \zeta +\frac{d}{ds}F + \frac{d}{ds}(G(\zeta,\zeta_t)\zeta_{p'}) + \frac{d}{ds}(h(p,e^{-ds}\zeta)\p_t\zeta_{p'}) ,\zeta_{p'} \right \rangle \\
    & + \langle (A+\delta A) \zeta_{p's} +\ep B \zeta_s ,\zeta_p \rangle.
\end{align*}
We label the above two terms by \textbf{T1} and \textbf{T2} respectively.
We first examine \textbf{T1}.
In order to make the sizes of various terms more apparent, we shall replace $\frac{d}{ds}F$ with 
\[
C\zeta^2 +C\zeta \p_t \zeta + +C \zeta \zeta_s + C \zeta_s\zeta_t + C \zeta \zeta_{ts}
\]
where the $C$ as it appears in each of the above terms may be different, but they are all uniformly bounded smooth functions of $(s,t)$ with uniformly bounded derivatives. 

Similarly we shall replace $\frac{d}{ds} G(\zeta,\zeta_t)\zeta_{p'}$ with
\[
(C\zeta + C\zeta_s + C\zeta^2 +C\zeta \zeta_s + C\zeta \zeta_t + C\zeta_s\zeta_t + C\zeta \zeta_{ts})\zeta_{p'} + (C\zeta+ C\zeta^2 + C\zeta \zeta_t) \zeta_{p's}
\]
with the same convention on $C$ as before. Finally we shall replace $\frac{d}{ds}h(p,e^{-ds}\zeta)\zeta_{tp'}$ with
\[
(C\zeta + \zeta_s)\zeta_{tp'} + C\zeta \zeta_{p'ts}.
\]
We examine various components of the \textbf{T1} term, starting with
\[
\la (A' +\dt A') \zeta_{p'} + \ep B' \zeta,\zeta_{p'}\ra.
\]
The operator $A'+ \delta A'$ for our purposes looks like $\ep (A+N)$, since the derivatives of the coefficient matrices with respect to $p'$ are bounded by $\ep$. Similarly $\ep B'$ behaves like $\ep B$
so we have, using these estimates
\[
 \langle (A'+\delta A')\zeta_{p'} + \ep B' \zeta ,\zeta_{p'} \ra \leq C \ep ( \la A\zeta_{p'}, A\zeta_{p'}\ra + \la \zeta_{p'}, \zeta_{p'} \ra) + \ep (\la \zeta_{p'}, \zeta_{p'} \ra + \la \zeta,\zeta\ra + \la A\zeta,A\zeta\ra).
\]
We next estimate
\begin{align*}
& \left \la \frac{d}{ds}F, \zeta_{p'} \right \ra\\
\leq & \la C\zeta^2 +C\zeta \p_t \zeta + +C \zeta \zeta_s + C \zeta_s\zeta_t + C \zeta \zeta_{ts}, \zeta_{p'}\ra\\
\leq& C\ep (\la \zeta,\zeta\ra + \la \zeta_{p'},\zeta_{p'}\ra) + C \ep \la \zeta_t,\zeta_{p'}\ra
\end{align*}
where we used the fact $C^0$ norm of $\zeta, \p_t\zeta,\p_s\zeta, \zeta_{st}$ are all uniformly bounded by $C\ep$ using elliptic regularity. The term $ C \ep \la \zeta_t,\zeta_{p'}\ra$ is bounded by
\begin{align*}
    & C \ep \la \zeta_t,\zeta_{p'}\ra \\
    \leq & C\ep (\la \zeta_t,\zeta_t\ra + \la \zeta_{p'},\zeta_{p'}\ra\\
    \leq& C \ep (\la A\zeta,A\zeta\ra + \la \zeta,\zeta\ra + \la \zeta_{p'},\zeta_{p'}\ra).
\end{align*}
which concludes the estimates for $\la \frac{d}{ds}F, \zeta_{p'}\ra$.

We next examine $\left \la \frac{d}{ds} G(\zeta,\zeta_t) \zeta_{p'}, \zeta_{p'}\right \ra$, which we can bound by
\begin{align*}
    &\leq \ep \la\zeta_{p'},\zeta_{p'} \ra + \ep \la \zeta_{p's}, \zeta_{p'} \ra.
\end{align*}
The second term in the above inequality is in turn bounded by
\begin{align*}
    &\leq \epsilon \langle \zeta_{p's} ,\zeta_{p'} \rangle\\
    &\leq \epsilon \langle A \zeta_{p'} +\delta A \zeta_{p'} + \ep B \zeta + F + G(\zeta,\zeta_t)\zeta_{p'} + h(p,e^{-ds}\zeta)\p_t\zeta_{p'}, \zeta_{p'}\rangle\\
    & \leq \epsilon (\langle A\zeta_{p'}, A\zeta_{p'} \rangle + \langle \zeta_{p'}, \zeta_{p'} \rangle + \langle \zeta, \zeta \rangle + \langle A \zeta, A\zeta \rangle )
\end{align*}
using techniques of the previous step. This concludes all bounds for $\la \frac{d}{ds} G(\zeta,\zeta_t) \zeta_{p'}, \zeta_{p'}\ra$.

We next turn to $\left \la \frac{d}{ds}h(p,e^{-ds}\zeta) \zeta_{tp'}, \zeta_{p'}\right \ra$, which we bound by
\begin{align*}
    &\la (C\zeta + C\zeta_s)\zeta_{tp'} + C\zeta \zeta_{p'ts}, \zeta_{p'} \ra \\
    \leq & \ep \la \zeta_{tp'}, \zeta_{p'}\ra + \ep \la \zeta_{p'ts}, \zeta_{p'}\ra\\
    \leq & \ep (\la A\zeta_{p'}, A\zeta_{p'}\ra + \la \zeta_{p'}, \zeta_{p'}\ra + \ep \la \zeta_{p'ts}, \zeta_{p'}\ra.
\end{align*}
To bound $\ep \la \zeta_{p'ts}, \zeta_{p'}\ra$, we use
\begin{align*}
    &\ep \la \zeta_{tp's},\zeta_{p'} \ra\\ &\leq \ep \la \zeta_{p's}, \zeta_{p't} \ra \\
    &\leq \ep \la  A \zeta_{p'} +\delta A \zeta_{p'} + \ep B \zeta + F +G(\zeta,\zeta_t) \zeta_{p'} + h(p,e^{-ds}\zeta)\p_t\zeta_{p'}), (MA+N)\zeta_{p'} \ra\\
    &\leq \ep [\la A \zeta_{p'}, A\zeta_{p'}\ra + \la \zeta_{p'},\zeta_{p'} \ra + \la \zeta,\zeta\ra + \la A\zeta,A\zeta\ra ].
\end{align*}
This concludes the bounds for $\la \frac{d}{ds}h(p,e^{-ds}\zeta) \zeta_{tp'}, \zeta_{p'}\ra$, and consequently all of \textbf{T1}.

We now turn to \textbf{T2}. We first examine $\la \ep B \zeta_s ,\zeta_{p'}\ra$. It can be rewritten as
\begin{equation*}
    \langle \zeta_s, \ep B^T \zeta_{p'} \rangle =\langle (A +\delta A )\zeta  +e^{d(s)} \mcal{F}(e^{-d(s)}\zeta), \ep B^T \zeta_{p'} \rangle.
\end{equation*}
We recall that $\ep B = \ep MA +N$. Now in taking the adjoint we had to differentiate the coefficient matrices w.r.t the variable $t$, but in our case $\ep B^T$ would still take the same form. Hence these terms can be handled by entirely similar techniques as before, giving
\begin{equation*}
    \leq \ep (\langle A\zeta_{p'}, A\zeta_{p'} \rangle + \langle A \zeta, A \zeta \rangle + \langle \zeta, \zeta \rangle + \langle \zeta_{p'}, \zeta_{p'} \rangle).
\end{equation*}
We consider the remaining term $\la A+\dt A \zeta_{p's},\zeta_{p'}\ra$. We can rewrite it as
\[
\langle \zeta_{p's} , (A+\delta A^T) \zeta_{p'}\rangle .
\]
Noting that $\dt A^T$ essentially takes the same form as $\dt A$, the above term will resemble the terms we computed in step 1. Hence it is equal to 
\[
\la A\zeta_{p'}, A\zeta_{p'}\ra
\]
plus an error term which is uniformly bounded by
\[
\ep( \la A\zeta_{p'},A\zeta_{p'} \ra + \la \zeta_{p'},\zeta_{p'}\ra + \la A\zeta,A\zeta\ra + \la \zeta,\zeta\ra).
\]
This gives bounds on all of the terms appearing in $g''(s)$, from which we conclude that 
\[
g''(s) \geq \lambda^2 g(s)
\]
for $\ep>0$ sufficiently small.
\end{proof}

We now switch to trying to understand $\la \zeta_p,\zeta_p\ra$, we can get this simply by rearranging terms in $g(s)$ and realizing derivatives w.r.t $p$ versus $p'$ differ by a factor of $\ep$.
\begin{corollary}
\begin{equation*}
\la \zeta_p,\zeta_p\ra_{L^2(S^1)}(s) \leq C\frac{\la \zeta,\zeta\ra_{L^2(S^1)}(s_0)+ \ep^2 \la \zeta_p,\zeta_p\ra_{L^2(S^1)}(s_0)}{\ep^2} e^{-\lambda (s-s_0)}.
\end{equation*}
for $s>s_0$ (in our case we can take $s_0=3R$, we are just stating the corollary more generallly to indicate the decay starts at $s_0$.)
\end{corollary}
It might seem unpleasant we are dividing by $\ep^2$, but in practice by elliptic regularity (and the $\zeta$ term we will be working with) we will have $\la \zeta ,\zeta \ra \sim C\ep^2/R^2$, so the decay really is of the form $C(\ep^2+\frac{1}{R^2})e^{-\lambda(s-s_0)}$. Also in the cases that interest us the decay will be so large factors of size $1/\ep^2$ will become irrelevant.

Using the same argument as before $\zeta_p$ has $W^{2,p}$ norm of size $C \ep$ so our previous strategy of bounding $L^p$ norm with $L^2$ norm continues to work, so we obtain the bound:
\begin{corollary} \label{cor:expdecay_p}
For $s>s_0>3R$, we have
\[
|\zeta_p (s,t)| \leq C \left [\frac{(\|\zeta\|^2_{W^{2,p}} + \|\zeta_p\|^2_{W^{2,p}})}{\ep^2}\right ]^{\frac{1}{p}} e^{-\lambda (s-s_0)}.
\] 
Here $\lambda$ is different from the $\lambda$ we chose previously. We will abbreviate this by writing $|\zeta_p(s,t)| \leq C e^{-\lambda(s-s_0)}$ as some more careful estimates can show the coefficient in front to be of order $O(1)$, similarly using elliptic regularity we can bound 
\[
|\zeta_{p*}(s,t)| \leq C e^{-\lambda(s-s_0)}, \quad * =s,t \, \text{and higher derivatives.}
\]
\end{corollary}

\begin{proof}
For completeness we explain how elliptic regularity is used. First using $W^{2,p}\hookrightarrow C^0$, we have
\[
\la \zeta_p,\zeta_p\ra_{L^2(S^1)}(s) \leq C\frac{\|\zeta\|^2_{W^{2,p}} + \|\zeta_p\|^2_{W^{2,p}}}{\ep^2} e^{-\lambda (s-s_0)}.
\]
Using the fact $C^0$ norm of $\zeta_p$ is $<1$, we have
\[
\|\zeta_p\|_{L^p([k-1,k+2]\times S^1)}^p \leq C\int_{k-1}^{k+2}\la  \zeta_p,\zeta_p\ra_{L^2(S^1)}(s) ds \leq \frac{\|\zeta\|^2_{W^{2,p}} + \|\zeta_p\|^2_{W^{2,p}}}{\ep^2} e^{-\lambda k}.
\]
Given this $L^p$ norm bound, we can use elliptic regularity and the fact $\zeta_p$ satisfies an equation of the form
\begin{equation*}
    D' \zeta_p=  B \zeta + F +G(\zeta,\zeta_t) \zeta_{p} + h(p,e^{-ds}\zeta)\p_t\zeta_{p}.
\end{equation*}
Here we are differentiating with respect to $p$ instead of $p'$ so we are rescaling some of the above terms so that they have norm $O(1)$ instead of $O(\ep)$.

From elliptic bootstrapping we have 
\[
\|\zeta_p\|_{W^{1,p}([k,k+1]\times S^1)} \leq C\|\zeta_p\|_{L^p([k-1,k+2]\times S^1)}  + \|\zeta\|_{W^{1,p}([k-1,k+2]\times S^1)} \leq Ce^{-\lambda k}
\]
where we used the exponential decay estimate of $\zeta$. Note we have slightly shrunk the domain to $[k,k+1]\times S^1$ to use elliptic regularity. We can iterate this argument to show
\[
\|\zeta_p\|_{W^{l,p}([k,k+1]\times S^1)} \leq C_le^{-\lambda k}
\]
and use Sobolev embedding theorems to obtain pointwise bounds as in the proposition.
\end{proof}
We also note the could have used the exact same techniques when applied to the $r$ asymptotic vector. There we need to identify $r \in S^1 = [0,1]/\sim$, and $r':= r/\ep \in S^1= [0,1/\ep]/\sim$. The result is very similar: we can obtain exponential decay bounds on $\zeta_r$, given as:
\begin{corollary}
For $s>s_0>3R$, we have
\[
|\zeta_r (s,t)| \leq C \left[\frac{(\|\zeta\|^2_{W^{2,p}} + \|\zeta_r\|^2_{W^{2,p}})}{\ep^2}\right]^{\frac{1}{p}} e^{-\lambda (s-s_0)}
\] 
\[
|\zeta_{r*}(s,t)| \leq C e^{-\lambda(s-s_0)}, \quad * =s,t \, \text{and higher derivatives.}
\]
\end{corollary}

We do not need such a result for the $a$-asymptotic vector since the geometry is invariant in the $a$ direction.

\section{Gluing multiple-level cascades} \label{gluing}
We have assembled all the technical ingredients we need to do gluing, which we take up in this section. We note gluing together cascades with finite gradient trajectories is substantially harder than semi-infinite gradient trajectories. We start with a simplified setup of gluing together 2-level cascades, which captures most of the technical difficulty. The generalization to $n$ level cascades is then a problem of linear algebra.

Our simplified setup is this:
let $u_\pm:\Sigma_\pm \rightarrow (M,J)$ be two rigid (nontrivial) $J$-holomorphic curves. $u_+$ has one negative end, asymptotic to Reeb orbit $\gamma_+$; and $u_-$ has one positive end, asymptotic to Reeb orbit $\gamma_-$. Both of these ends are on the same Morse-Bott torus, and in fact they are connected by a gradient trajectory of length $T$. We will perturb $J$ to $J_\dt$ near this Morse-Bott torus (and nowhere else), and glue $u_\pm$ along with the finite gradient trajectory into a $J_\dt$-holomorphic curve. This construction ignores the other ends of $u_\pm$, which we assume to remain on other Morse-Bott tori, and we only have two levels. The reason for this is that the process of gluing together two $J$-holomorphic curves along a finite gradient trajectory is very technical, and we would like to carry out the heart of the technical construction with as little extra baggage as possible. 

This section is organized as follows: we first introduce the general setup and the process of pregluing. We, as before, show gluing can be realized by solving a system of three equations. We then proceed to discuss the linear theory required to describe the linearization of $\db_{J_\dt}$ over the finite gradient trajectory. After that we show the feedback terms coming from $\Theta_v$ and going into $\Theta_\pm$ (defined in the subsection below) depend nicely on the input - this is the most technical step and will take some careful estimates. Finally after this we will be able to solve the three equations as we did in the previous section. Finally, we explain the generalization to $n$-level transverse index one cascades.

\subsection{Setup and pregluing}
Recall near the Morse-Bott torus we have coordinates $(z,x,y)\in S^1  \times S^1 \times \bb{R}$. For definiteness we assume $\gamma_+$ is the Reeb orbit with $x$ coordinate $x_+$ and $\gamma_-$ is at $x_-$. To simplify notation we assume we have rescaled the $x$ coordinate so that $f(x) = x+C$ over the interval on $S^1$ connecting $x_-$ and $x_+$.

We recall for each $u_\pm$ we choose a cylindrical neighborhood around each of its punctures $(s,t) \in S^1 \times (0,\pm \infty)$. We also recall near our punctures $u_\pm$ has the coordinate form
\[
u_\pm = (a_\pm,z_\pm,x_\pm,y_\pm).
\]
We assume $u_+(s=-\infty,t)\rightarrow (-\infty, t,x_+,0)$ and $u_-(s=\infty,t)\rightarrow (\infty, t,x_-,0)$. For each $u_\pm$ we describe a neighborhood of this map as
\[
W^{2,p,d}(u_\pm^*TM)\oplus T\mcal{J}_\pm \oplus  V'_\pm \oplus V_\pm
\]
where $W^{2,p,d}(u_\pm^*TM)$ is the weighted vector space of vector fields with weight $e^{\pm ds}$ at positive/negative punctures. We use $T\mcal{J}_\pm$ to denote a Teichmuller slice. We use $V'_\pm$ to denote asymptotic vectors at other ends of $u_\pm$, and $V_\pm$ is the end that we are considering, being a 3 dimensional space consisting of vectors $(r,a,p)_\pm$.\\
We recall the important gluing constant 
\[
R:= \frac{1}{5d}\log(1/\dt)
\]
which we think of our gluing parameter.\\
Let $v_\dt$ be a gradient trajectory suitably translated so that over the interval $s\in[0,T/\dt]$, the map $v_\dt$ corresponds to the gradient flow that connects $\gamma_\pm$, in particular this means the $x$ component of $v_\dt$ satisfies
\[
x \,\text{ component of }\, v_\dt(R)=x_-
\]
\[
x\, \text{component of} \, v_\dt(T/\dt-R) =x_+.
\]
We next construct our preguling, similar to the semi-infinte case our pregluing will depend on our asymptotic vectors $(r,a,p)_\pm$.

Given fixed $(r,a,p)_\pm$, let $v_{r,a,p}=(a_v(s),t_v(t),x_v(s),0)$ (we suppress the $\pm$ that should appear in the subscript to ease the notation) denote the suitably translated gradient trajectory, so that when restricted to $s\in [0,T_p/\dt]$ satisfies
\[
v_{r,a,p}(T_p/\dt-R,t) = (a_+(-R,0)+a_+,t+r_+,x_++p_+,0)
\]
and 
\[
v_{r,a,p}(R,t) = (a_v(R),t+r_+,x_-+p_-,0).
\]
We observe that due to the form of $f$ in this region, we have $T_p= T +(p_+-p_-)$. 
We preglue this gradient trajectory to the deformed curve $u_+ + (r,a,p)_+$ at $s=T_p/\dt-R$ of $v_{r,a,p}$. This value of $s$ over $v_{r,a,p}$ is identified with $s_+=-R$ over $u_+$. At the other end we consider $u_-$ translated in $a$ direction so that $a_-(R)= a_v(R)-a_-$. Then we would like to preglue $v_{r,a,p}$ at $s=R$ to $u_-+(r,a,p)_-$ at $s_-=R$, except there is an issue that since $r_+$ is in general different from $r_-$, the curve $v_{r,a,p}(-,t)$ has $z$ component $t+r_+$, while $u_-(s,t)+(r,a,p)_-$ has $t$ component roughly equal to $t+r_-$. To remedy this we need to preglue with a different domain $\Sigma_{r,a,p}$ so that at $s=T_p/\dt-R$ we do our usual pregluing (as in the semi-infinite gradient trajectory case), but at $s_-=s=R$ we glue with a twist: recall $(s_-,t_-) \in \bb{R} \times S^1$ is a cylindrical neighborhood on $u_-$ and $(s,t)$ is the usual coordinate on $v_{r,a,p}$, then we construct the domain $\Sigma_{r,a,p}$ by identifying $t_-+r_- \sim t+r_+$ at $s=s'=R$. Then we can construct a preglued map
\[
u_{r,a,p}:\Sigma_{r,a,p} \longrightarrow (M,J)
\]
that depends on the asymptotic vectors $(r,a,p)_\pm$.
\begin{remark}
We first observe that here the domain depends non-trivially on the asymptotic vectors $(r,a,p)_\pm$, in fact in the case where the domains for $u_\pm$ are stable, changing the pregluing in $r_\pm$, i.e. ``twisting", or changing the length of the cylindrical neck by changing $p_\pm\, \text{or} \, a_\pm$ correspond to changing the complex structure of the domain curve.\\
We also observe here that if we change $a_\pm$ by size $\ep$, then the length of cylindrical length changes by size $\ep$. Similarly if we change $r_\pm$ by $\ep$m in some appropriate sense the complex structure changes within an $\ep$ neighborhood. However when we change $p_\pm$ by $\ep$, the length of the neck changes by $\ep/\dt$. This is in some sense the main source of difficulty in studying this degeneration. Since $\dt << \ep$ they operate on different scales, and care must be taken to ensure all the vectors we encounter have the right sizes.
\end{remark}
\begin{remark}
Here because we only have one end the pregluing is rather simple, when there are multitple ends and/or when we talk about degeneration into cascades more care must be taken to into pregluing, which we defer to subsection \ref{subsection_gluing_multiple_level}.
\end{remark}

\subsection{Linear theory over \texorpdfstring{$v_{r,a,p}$}{v}}
In this subsection we take a detour to study the linearization of $\db_{J_\dt}$ over $v_{r,a,p}$. In particular we find a suitable Sobolev space with suitable exponential weights so that for given $(r,a,p)_\pm$, the said linearization denoted by $D_{J_\dt}$ is surjective with uniformly bounded right inverse as $\dt \rightarrow 0$.\\
After fixing $(r,a,p)_\pm$, we consider 
\[
D_{J_\dt}: W^{2,p,w_p}(v_{r,a,p}^*TM)\longrightarrow W^{1,p,w_p}(v_{r,a,p}^*TM).
\]
Here $w_p$ is a piecewise linear function that is zero at $s=0$ and $s=T_p/\dt$, has a peak at $s=T_p/2\dt$, and has slope $\pm d$. Explicitly it is given by
\[
w_p= -|d(s-T_p/2\dt)|+dT_p/2\dt.
\]
It looks like an inverted $V$. The space $W^{2,p,w_p}(v_{r,a,p}^*TM)$ is a weighted Sobolev space with exponential weight $e^{w_p(s)}$. As is with the case for semi-infinite ends these vector fields have exponential growth as $s\rightarrow \pm \infty$, but we do not care about them because those regions do not make an appearance in our construction.
\begin{remark}
Observe with our choice of $w_p(s)$, which we sometimes denote by $w(s)$ for brevity, over the preglued curve $u_{r,a,p}$, the pregluing takes place at $s=R$ and $s=T/\dt-R$, and at these two values of $s$ where the pregluing takes place, the exponential weight profile of $v_{r,a,p}$ agrees with the exponential weight profile over $u_\pm$.
\end{remark}
\begin{theorem}
$D_{J_\dt}$ as defined above is surjective of index 3. It has a uniformly bounded right inverse as $\dt \rightarrow 0$.
\end{theorem}
\begin{proof}
We can view $D_{J_\dt}$ as the gluing of two operators $D_1$ and $D_2$. The operators $D_i$ are both defined over $W^{2,p,w_i}(v_{r,a,p}^*TM)$, except they use different exponential weights. We let $w_1(s)=d(T_p/\dt-s)$,  and $w_2(s)=ds$
We glue $D_1$ and $D_2$ together at $s=T_p/2\dt$ to recover $D_{J_\dt}$. By results in Section \ref{linearization}, $D_i$ are both surjective with uniformly bounded right inverse $Q_i$, hence as before we can construct approximate right inverse of $D_{J_\dt}$ via $Q_1 \# Q_2$, hence $D_{J_\dt}$ is surjective with uniformly bounded right inverse as $\dt \rightarrow 0$.

The index computation is done by conjugating to $W^{2,p}(v_{r,a,p}^*TM)$ via multiplication by $e^{w_p(s)}$. There we observe by shape of $w_p(s)$ there are 3 eigenvalues that cross $0$ as $s$ goes from $-\infty$ to $\infty$, hence by spectral flow this operator has index 3.
\end{proof}
We now proceed to describe the kernel of $D_{J_\dt}$ and a codimensional $3$ subspace $H_0$ of its domain so that $D_{J_\dt}|_{H_0}$ is an isomorphism with uniformly bounded inverse as $\dt \rightarrow 0$. This will be crucial for us when we try to solve equations over $W^{2,p,w}(v_{r,a,p}^*TM)$.

Consider the vector fields
\[
\p_z,\p_a \in W^{2,p,w_p}(v_{r,a,p}^*TM).
\]
They are asymptotically constant, but they live in $W^{2,p,w_p}(v_{r,a,p}^*TM)$ because as $|s| \rightarrow \infty$ the Sobolev norm is exponentially suppressed (written as is they still have very large norm, of order $e^{\frac{dT_p}{2\dt}}$.) Also observe they live in the kernel of $D_{J_\dt}$.
Recall from the differential geometry section
\[
v_* \partial_s = e^{\delta f(x(s))}\partial_a + \delta f'(x) \partial_x.
\]
This vector field also lives in the kernel of $D_{J_\dt}$, and is linearly independent of $\{\p_z,\p_a \}$, we modify it to have more palatable form. Consider
\[
\frac{v_* \partial_s -\p_a}{\dt}=  [e^{\delta f(x(s))}-1]/\dt\partial_a+f'(x) \partial_x.
\]
This still lives in the kernel of $D_{J_\dt}$, and we see from Taylor expansion that the coefficient in front of $\p_a$ is bounded above as $\dt \rightarrow 0$. We defined the vector field $\p_v$ to be $a\frac{v_* \partial_s -\p_a}{\dt} +b\p_s$ where $a,b$ are constants (both of order 1, bounded above and away from 0) chosen so that $\partial_v(s=T_p/2\dt,t)=\p_x$.
Thus the kernel of $D_{J_\dt}$ is spanned by $\{\p_z,\p_a,\p_v\}$. We construct a complement of this space. Consider the linear functionals $L_*, *=z,a,v: W^{2,p,w}(v_{r,a,p}^*TM)\rightarrow \bb{R}$ defined by
\[
L_*: \phi \in W^{2,p,w}(v_{r,a,p}^*TM)\longrightarrow  \int_0^1 \la \phi(s,t), \p_*\ra dt \in \bb{R}.
\]
We define the complement subspace of $ker D_{J_\dt}$, which we write as $H_0$, via
\[
H_0 := \{ \phi \in W^{2,p,w}(v_{r,a,p}^*TM)| L_*(\phi)=0,*=z,a,v\}.
\]
We next show:
\begin{proposition}
The projection map
\begin{equation*}
    \Pi: W^{2,p,w}(v_{r,a,p}^*TM) \longrightarrow H_0
\end{equation*}
has uniformly bounded norm as $\dt \rightarrow 0 $. The map $\Pi$ also commutes with $D_{J_\dt}$.
\end{proposition}
\begin{proof}
We first observe $\Pi$ is defined by
\[
\Pi(\phi) = \phi -\sum_* L(\phi) \p_*.
\]
We now estimate the norm of this operator. By the Sobolev embedding theorem
\begin{equation*}
    W^{2,p,w}(v_{r,a,p}^*TM) \hookrightarrow C^0(v_{r,a,p}^*TM).
\end{equation*}
In view of the fact we have exponential weights, we have the upper bound
\[
L_*(\phi) \leq C e^{-dT_p/2\dt}\|\phi\|_{W^{2,p,w}}.
\]
Hence to estimate the norm of $\Pi$ it suffices to calculate
\begin{align*}
    &\frac{\|L_*(\phi) \partial_*\|}{\|\phi\|}\\
    \leq& C e^{-dT_p/2\dt} \|\p_*\|\\
    \leq& C e^{-dT_p/2\dt}\left[ \int_0^{T_p/2\dt}e^{ds}ds + \int_{-\infty}^0 e^{ds}ds\right]\\
     \leq & C e^{-dT_p/2\dt}\frac{(e^{dT_p/2\dt})}{d}
    \leq C.
\end{align*}
We only integrated from $(-\infty, T/2\dt)$ because the integral over $(T/2\dt,\infty)$ takes the same form. And hence we see readily the operator norm of $\Pi$ is uniformly bounded above independently of $\dt$.\\
The fact that $\Pi$ commutes with $D_{J_\dt}$ follows from the fact $\Pi$ subtracts off elements that are in the kernel of $D_{J_\dt}$.

Hence we conclude $\Pi \circ Q$ is a uniformly bounded inverse to $D_{J_\dt}$ restricted to $H_0$.
\end{proof}
\subsection{Deforming the pregluing}
Recall that given a pair of asymptotic vectors over $u_\pm$, which we denote by $(r,a,p)_\pm$, we constructed a preglued map $u_{r,a,p}: \Sigma_{r,a,p} \rightarrow M$. Next given vector fields with exponential decay, $\psi_\pm \in W^{2,p,d}(u_\pm^*TM)$, and $\phi \in W^{2,p,w}(v_{r,a,p}^*TM)$, we use them to deform $u_{r,a,p}$. Technically the space of deformations of $u_\pm$ also includes $T\mcal{J}_\pm \oplus V_\pm'$, but we suppress them from our notation because these deformations happen away from the region where the pregluing takes place. For $s\in [R,T_p/\dt-R]$ considered over $v_{r,a,p}$, we  define the cut off functions
\begin{equation*}
    \beta_- =\beta_{[-\infty, 2R;R/2]}
\end{equation*}
\[
\beta_+:=\beta_{[R/2;T_p/\dt-2R,\infty]}
\]
\[
\beta_v:=\beta_{[R/2;R,T_p/\dt-R;R/2]}.
\]
We would like to deform $u_{r,a,p}$ by $\beta_+\psi_++\beta_-\psi_-+\beta_v\phi$, however there is one subtlety that when we constructed $\Sigma_{r,a,p}$ there was a twist at $s=R$ when we identified $t_-+r_- \sim t+r_+$ when we glued $v_{r,a,p}$ with $u_-$. Since $\beta_-$ cuts off $\psi_-$ within the interior of $v_{r,a,p}$, the only effect of this is that when we view the equation over $v_{r,a,p}$ instead of seeing $\psi_-(s,t)$, the term we see is $\psi(s,t+(r_+-r_-))$. Aside from this point, as before we can add the vector field $\beta_+\psi_++\beta_- \psi_- +\beta_v\phi$ to $u_{r,a,p}$, and apply the $\db_{J_\dt}$ operator. Using the same ``splitting up the equations" trick as we did for semi-infinite trajectories we get:
\begin{proposition}
The deformed curve $u_{r,a,p} + \beta_+\psi_++\beta_- \psi_- +\beta_v\phi$, where $\psi_\pm \in W^{2,p,d}(u_\pm^*TM)\oplus T\mcal{J}\oplus V'_\pm$ implicitly includes the variations of complex structure away from the gluing region, is $J_\dt$-holomorphic iff the following 3 equations are satisfied
\[
\Theta_v(\phi,\psi_\pm)=0
\]
\[
\Theta_\pm(\phi,\psi_\pm)=0
\]
where $\Theta_v$ is of the form
\[
D_{J_\dt} \phi + \beta_\pm' \psi_\pm + \mcal{F}_v(\phi,\psi_\pm).
\]
Here $\mcal{F}_v$ is of the same form as semi-infinite case (except at the end near $s=R$ we see effects of $\psi_-$ and near $s=T_p/\dt-R$ we see the effect of $\psi_+$). The equations $\Theta_\pm$ take the form
\[
\Theta_+ = D_J\psi_+ + D_J(r,a,p)_+ + \mcal{F}_+ + \mcal{E}_+ +\beta_{v}'\phi
\]
\[
\Theta_- = D_J\psi_- + D_J(r,a,p)_- + \mcal{F}_- + \mcal{E}_- +\beta_{v}'\phi
\]
where the scripted expressions $\mcal{F}_\pm, \mcal{E}_\pm$ taking the same form as they did in the semi-infinite case. Implicit in the above notation is also the variation of the domain complex structure $u_\pm$, which we denote by $\dt j_\pm$ when we need to make them explicit.
\end{proposition}

\subsection{Solving the equations \texorpdfstring{$\Theta_\pm, \Theta_v$}{Theta}}
\subsubsection{Preamble}
In this very lengthy subsection we show the system $\Theta_\pm=0, \Theta_v=0$ has a solution with nice properties. Since this is a long process we give a preamble: 
\begin{itemize}
    \item We first show as before given fixed tuple of input data $(\psi_\pm, (r,a,p)_\pm)$ there exists a unique solution $\phi(r,a,p,\psi_\pm)\in H_0$ to $\Theta_v$.
    \item Then we verify that when we vary the input, $\psi_\pm$ the solution $\phi$ behaves nicely (in the sense that its input into the equations into $\Theta_\pm$ varies differentiably, as was the case for the gluing of semi-infinite gradient trajectories.)
    \item Then we verify as we change $p_\pm$ the solution is well behaved. This is the crux of the matter, because when we vary $p_\pm$ what is actually happening is that we are drastically changing the pregluing by dramatically lengthening/shortening the length of the neck. We do this via the following process:
    \begin{itemize}
        \item We make sense of what it means for $\phi$ to be well behaved when we vary $p_\pm$.
        \item We translate $\Theta_v$ into the vector space $W^{2,p}(v^*TM)$ by removing exponential weights.
        \item We write the solution $\phi$ as a sum of two terms: an approximate solution $\gamma_+\zeta_+ + \gamma_- \zeta_-$ to $\Theta_v$ that behaves nicely when we vary $p_\pm$ and a correction to this approximate solution $\dt \zeta$, we show $\dt \zeta$ is extremely small. Here $\gamma_\pm$ are cut off functions (and definitely not Reeb orbits).
        \item We consider the behaviour of $\dt \zeta$ as we vary $p_\pm$. We consider two ways $p_\pm$ can vary called ``lengthening/stretching" and ``translation". We show $\dt \zeta$ varies nicely with $p_\pm$, hence the entire solution $\phi$ varies nicely with $p_\pm$.
    \end{itemize}
    \item We finally show as a much easier step $\phi$ varies nicely with $(r,a)$.
    \item Using all of the above steps, we solve $\Theta_\pm$ with the contraction mapping principle.
    
\end{itemize}
\subsubsection{Solution to \texorpdfstring{$\Theta_v$}{Theta}}
\begin{proposition}
For $\ep>0$ sufficiently small, for all $\dt>0$ sufficiently small, for fixed tuple $(\psi_\pm,(r,a,p)_\pm)$ with norm less than $\ep>0$, there exists a unique solution $\phi(\psi_\pm,r,a,p_\pm)\in H_0$ to $\Theta_v=0$ of size $C\ep/R$. Moreover the regularity of $\phi$ can be improved to $W^{3,p,w}(v_{r,a,p}^*TM)$ with its norm similarly bounded above by $C\ep/R$.
\end{proposition}
\begin{proof}
Let $Q$ denote the uniformly bounded right inverse to $D_{J_\dt}$. Consider $\Pi \circ Q: W^{1,p,w}(v_{r,a,p}^*TM) \rightarrow H_0$. We observe this operator has uniformly bounded norm as $\dt \rightarrow 0$. Further we claim this is an inverse to $D_{J_\dt}|_{H_0}$. To see this first oberseve $D_{J_\dt}|_{H_0}$ is an isomorphism, as it has the same image as $D_{J_\dt}|_{W^{2,p,w}(v_{r,a,p}^*TM)}$ and has index 0. Hence it suffices to show $\Pi \circ Q$ is a right inverse for $\phi \in H_0$. This follows from
\[
D_{J_\dt} \Pi \circ Q  \phi = \Pi(\phi) =\phi.
\]
Hence we consider the map $I: H_0 \rightarrow H_0$ defined by
\[
I(\phi) = \Pi\circ Q (- \beta_\pm' \psi_\pm -\mcal{F}_v(\phi,\psi_\pm)).
\]
(For ease of notation we will write $\psi_\pm$ when both $\psi_+$ and $\psi_-$ appear in similar ways).
It is apparent that a solution $\phi\in H_0$ to $\Theta_v$ is equivalent to a fixed point of $I(\phi)$. We show that a fixed point in an epsilon ball $B_\ep \in H_0$ exists and is unique via the Banach contraction mapping principle. Since $\psi_\pm$ has norm $\leq \ep$, we have $I(\phi)\leq C(\ep/R +C\ep^2)$ hence it sends $B_\ep$ to itself. That $I$ satisfies the contraction property follows from the fact $\mcal{F}_v$ is quadratic in $\phi,\psi_\pm, \p_t \phi, \p_t\psi_\pm$, as well as the fact $\|\psi_\pm\|\leq \ep$. Hence it follows from contraction mapping principle there exists unique $\phi(\psi_\pm,r,a,p_\pm)$ solving $\Theta_v$ in $B_\ep$. We can use the equation itself to estimate the size of $\phi$ as before and get the size estimate of $C\ep/R$. The improvement to $W^{3,p,w}$ and its norm bound follows from elliptic regularity.
\end{proof}
\subsubsection{How \texorpdfstring{$\phi(\psi_\pm,r,a,p)$}{phi} varies with \texorpdfstring{$\psi_\pm$}{psi}}
For fixed $(r,a,p)_\pm$ we consider the variation of $\phi(\psi_\pm,(r,a,p)_\pm)$ as above with respect to $\psi_\pm$. As we recall from the above the expression $\frac{d\phi}{d\psi_\pm}$ is a linear operator $W^{2,p,w}(u_\pm^*TM)\rightarrow W^{2,p,w}(v_{r,a,p}^*TM)$ and has its sized measured via the operator norm. When we write below $\beta_\pm'$ we really mean the multiplication map operating between Sobolev spaces. As in the case of semi-infinite gradient trajectories, we have:
\begin{proposition}
$\left\|\frac{d\phi}{d\psi_\pm}\right \|_{W^{2,p,w}(u_\pm^*TM) \rightarrow W^{2,p,w}(v_{r,a,p}^*TM)} \leq C\ep$.
\end{proposition}
\begin{proof}
Consider the fixed point equation 
\[
\phi = \Pi\circ Q (-\beta_\pm' \psi_\pm - \mcal{F}_v(\phi,\psi_\pm)).
\]
We differentiate both sides w.r.t $\psi_\pm$ to get (see Remark \ref{Frechet} for this kind of operation)
\[
\frac{d\phi}{d\psi_\pm} = -\Pi \circ Q (\beta_\pm' +\frac{d}{d\psi_\pm} \mcal{F}_v).
\]
Since we know $\mcal{F}_v$ is a polynomial expression of $\psi_\pm,\phi, \p_t\phi,\psi_\pm$, we can bound (norm wise) 
\[
\left \|\Pi\circ Q \frac{d}{d\psi_\pm} \mcal{F}_v\right \| \leq C (\|\phi\| + \|\psi_\pm\|) +C \ep \left \|\frac{d\phi}{d\psi_\pm}\right \|
\]
where in the above equation, the norm for $\|\Pi\circ Q \frac{d}{d\psi_\pm} \mcal{F}_v\|$ and $\|\frac{d\phi}{d\psi_\pm}\|$ are operator norms, and $\|\phi\|$ and $\|\psi_\pm\|$ are $W^{2,p,w}$ norms.\\
Since $\psi_\pm$ have $C^1$ bounds of size $\leq C\ep$, we can move the term $d\phi/d\psi_\pm$ to the left and get
\[
(1-C\ep)\left\| \frac{d\phi}{d\psi_\pm} \right\| \leq C (1/R)
\]
which implies our conclusion.
\end{proof}
\subsubsection{Variation of \texorpdfstring{$\phi$}{phi} w.r.t. \texorpdfstring{$p_\pm$}{p}}
In this subsubsection we study the variation of $\phi$ w.r.t. $p_\pm$. When we change $p_\pm$, we are considerably changing the pregluing. So we need to make sense of what kind of result that we want. We recall from previous section we already found a solution to $\Theta_v$ in $H_0$ for every choice of $(\psi_\pm,(r,a,p)_\pm)$, so our next order of business is to solve $\Theta_\pm$, and in order to do that we need to show as we vary $p_\pm$, the part of $\phi$ that enters into equations $\Theta_\pm$ varies nicely w.r.t. $p_\pm$. We recall $\Theta_\pm$ is an equation defined over $u_\pm^*TM$. What is happening is as we vary $p_\pm$, the maps $u_\pm$ are translated further/closer to each other, but since our equations are invariant in the symplectization direction, we can identify all those translates of $u_\pm$ and consider one set of equations $\Theta_+,\Theta_-$ as we vary $p_\pm$. Thus we need to understand how $\phi$ behaves near the pregluing region. We make this a definition.
\begin{definition}
Let $s \in [-3R,3R]$, recall if we let $s_\pm$ denote coordinates near the cylindrical neighborhoods of punctures of $u_\pm$, then we have identified $s \sim s_-$ and $s  \sim -s_+ +T_p/\dt$. Then for $s\in [-3R,3R]$ (resp. $[-3R+T_p/\dt, 3R+T_p/\dt ]$), the vector field $\phi(s,t)$ can be viewed as a vector field in $W^{2,p,d}(u_- ^*TM)$ (resp. $W^{2,p,d}(u_+ ^*TM)$ ), as we noted in the pregluing section.
We say $\phi(\psi_\pm,r,a,p)$ is well behaved w.r.t. $p_\pm$ if over $s \in [-3R,3R]$,
\[
\left \|\frac{d}{dp_\pm} \phi(s+T_p/\dt,t) \right \| \leq C\ep
\]
and
\[
\left \|\frac{d}{dp_\pm} \phi(s,t) \right \| \leq C\ep
\]
where $\frac{d\phi}{dp_\pm}$ is viewed as a vector field over $W^{2,p,d}(u_\pm^*TM)$, and the norm is the weighted Sobolev norm in $W^{2,p,d}(u_\pm^*TM)$.
\end{definition}
\begin{remark}
Actually because no derivatives of $\phi$ appears in $\Theta_\pm$, only the $W^{1,p}$ norm is enough for our purposes.
\end{remark}
The main theorem of this subsubsection is then:
\begin{proposition}
$\phi$ is well behaved w.r.t. $p_\pm$.
\end{proposition}
To do this we need to very carefully analyze the solutions to $\Theta_v$. It turns out it is not so convenient to analyze this equation with exponential weights, because the weights themselves depend on $p_\pm$. So we first remove the exponential weights via conjugation. We use the following convention:
\[
\zeta:= e^{w(s)}\phi,\]
\[
\psi'_\pm:=e^{w(s)}\psi_\pm.
\]
The exponential weights are removed and $\Theta_v$ is rewritten using the following diagram:
\begin{equation}
\begin{tikzcd}
W^{2,p}(v_{r,a,p}^*TM) \arrow[r,"\Theta_v'"] \arrow[d, "e^{-w(s)}"] & W^{2,p}(v^*TM) \\
W^{2,p,d}(v_{r,a,p}^*TM) \arrow[r,"\Theta_v"] & W^{2,p,d}(v_{r,a,p}^*TM) \arrow[u, "e^{w(s)}"].
\end{tikzcd}
\end{equation}
Then the equation $\Theta_v$ can be rewritten as
\[
\Theta_v' := D_{J_\dt}' \zeta + \beta_\pm' \psi_\pm' + e^{w(s)}\mcal{F}_v(e^{-w(s)}\zeta,e^{-w(s)} \psi_\pm')=0
\]
where $\zeta \in H_0' \subset W^{2,p}(v_{r,a,p}^*TM)$. We use $H_0'$ to denote the subspace in $W^{2,p}(v_{r,a,p}^*TM)$ corresponding to $H_0$. To better understand $\zeta$, let us focus our attention near $s\in[0,T_p/2\dt]$. For this range of $s$, the equation $\Theta_v$ is exactly the same equation as we had solved for semi-infinite gradient trajectories since we do not see the effects of $\psi_+$. Then by previous result we have a (uniquely constructed) solution $\phi_- \in W^{2,p,d}(v_{r,a,p}^*TM)$ for $s\in [0, T_p/2\dt]$ subject to exponential weight $e^{ds}$ (which for our range of $s$ agrees with $e^{w(s)}$). Defining \[
\zeta_- := e^{w(s)}\phi_-
\]
we see $\zeta_-$ is a solution to $\Theta_v'$ for $s\in [0,T_p/2\dt]$. 

There is a slight subtlety in that near $u_-$ there is a twist in the $t$ coordinate as we constructed the pregluing domian, $\Sigma_{r,a,p}$. By the construction in the semi-infinite gradient trajectory case, $\phi_-$ should depend on input variables $(s_-,t_-)$, which we write as $\phi_-(s_-,t_-)$, but when we view it as a vector field over $v_{r,a,p}^*TM$, using coordinates $(s,t)$ it should be written as $\phi_-(s,t+(r_+-r_-))$. This won't make a difference for us as we consider variations in the $(p_-,p_+)$ direction, and for the most part we will suppress the $t$ coordinate for brevity of notation. We will take up variations in the $(r_+,r_-)$ variables after considerations of $p_\pm$.

We similarly construct $\zeta_+$. The point is: 
\begin{proposition}
$\zeta_\pm$ is well behaved w.r.t. $p$. i.e. the part of $\zeta_\pm$ that enters into $\Theta_\pm$ has derivative w.r.t. $p_\pm$ bounded above by $C\ep$.
\end{proposition}
\begin{proof}
This follows from our results on $\phi_\pm$ when we proved this property for semi-infinite trajectories. 
\end{proof}
The next step is to actually construct $\zeta$ from approximate solutions $\zeta_\pm$. Consider the cut off functions $\gamma_\pm$ defined by
\[
\gamma_+ := \beta_{[\infty,T_p/2\dt-1;1]}
\]
\[
\gamma_+:= \beta_{[-\infty; T_p/2\dt-1;1]}.
\]
Then we consider the approximate solution
\[
\gamma_+ \zeta_+ +\gamma_-\zeta_-.
\]
We also observe by construction that $\gamma_+ \zeta_+ +\gamma_-\zeta_- \in H_0'$.
We plug this into $\Theta'_v$, we observe by definition this produces zero for all $s$ except $s\in[T_p/2\dt-2,T_p/2\dt+2]$. In this interval the $\Theta_v'$ takes the form:
\begin{equation}
    D'_{J_\dt}(\gamma_+ \zeta_+ +\gamma_-\zeta_-) +e^{w(s)}\mcal{F}_v(e^{-w(s)}\gamma_\pm \zeta_\pm)
\end{equation}
which equals
\[
E:= \sum_\pm (\gamma_\pm'\zeta_\pm +\gamma_\pm D'_{J_\dt}\zeta_\pm) + e^{w(s)}\mcal{F}_v(e^{-w(s)}\gamma_\pm \zeta_\pm).
\]
Observe $D'_{J_\dt}\zeta_\pm =-e^{w(s)}\mcal{F}_v(e^{-w(s)}\zeta_\pm)$
so the error term takes the form 
\begin{align*}
E =&  \gamma_+'\zeta_++ \gamma_-'\zeta_-\\
&+ [e^{w(s)}\mcal{F}_v(e^{-w(s)} \gamma_+\zeta_+)-\gamma_+e^{w(s)}\mcal{F}_v(e^{-w(s)} \zeta_+)] +[e^{w(s)}\mcal{F}_v(e^{-w(s)} \gamma_-\zeta_-)-\gamma_-e^{w(s)}\mcal{F}_v(e^{-w(s)} \zeta_-)].
\end{align*}
We can estimate the size of this term (say in $C^1$ norm), by elliptic regularity it is easily bounded by the $W^{2,p}$ norm of $\zeta_\pm$ restricted to $ s\in[T_p/2\dt-2,T_p/2\dt+2]$. (We actually see $t$ derivatives of $\zeta_\pm$ in $\mcal{F}_v$ but this is fine, we can bound them by elliptic regularity). We know the norm of $\zeta_\pm$ undergoes exponential decay as $s$ moves into this center region, so the size of the error term is bounded above by
\[
C \op{max}\{\|\zeta_+\|,\|\zeta_-\|\}^{2/p} e^{-\lambda (T/2\dt-3R)}
\]
where in the above equation $\|\zeta_\pm\|$ denotes the full norm of $\zeta_\pm$ over $W^{2,p}(v_{r,a,p}^*TM)$, or equivalently the norm of $\phi_\pm \in W^{2,p,d}(v_{r,a,p}^*TM)$.\\
From the above we conclude the error term to the approximate solution $\gamma_+ \zeta_+ +\gamma_-\zeta_-$ is very small-exponentially suppressed in fact. We now perturb it by adding a small term $\dt \zeta \in H_0'$ to make it into a solution to $\Theta_v'$. We state this in the form of a proposition:
\begin{proposition}
We can choose $\dt \zeta \in H_0'$ so that $\zeta = \gamma_+ \zeta_+ +\gamma_-\zeta_- +\dt \zeta$. Further, the norm of $\dt \zeta$, as measured in $W^{2,p}(v_{r,a,p}^*TM)$ is bounded above by
\begin{equation*}
    C\ep^{2/p} e^{-\lambda( T_p/2\dt-3R)}.
\end{equation*}
The vector field $\dt \zeta$ also lives in $W^{3,p}(v_{r,a,p}^*TM)$, and its $W^{3,p}(v_{r,a,p}^*TM)$ norm is similarly bounded above by
\[
C\ep^{2/p} e^{-\lambda( T_p/2\dt-3R)}.
\]
\end{proposition}
\begin{remark}
We remark in the term $T_p/2\dt-3R$, the term $3R$ appears because we can only start the exponential decay after the effects of $\psi_\pm'$ in $\Theta_v'$ disappear. (Technically we could have used $2R$ but this will not make a difference).
\end{remark}
\begin{proof}
We plug $\zeta := \gamma_+ \zeta_+ +\gamma_-\zeta_- +\dt \zeta$ into $\Theta_v'$ and solve for $\dt \zeta$ using the contraction mapping principle. We are now looking at an equation of the form:
\[
D_{J_\dt}' (\gamma_+\zeta_+ + \gamma_-\zeta_- + \dt \zeta) + \beta_\pm' \psi_\pm' +e^{w(s)}\mcal{F}_v(e^{-w(s)}\zeta, e^{-w(s)}\psi_\pm')=0.
\]
We examine the term $e^{w(s)}\mcal{F}_v(e^{-w(s)}\zeta, e^{-w(s)}\psi_\pm')$, 
recall $\mcal{F}_v$ generally takes the form:
\[
\mcal{F}_v:=\beta_{[1;R-2,\infty]}\phi g_{v1}(\beta_{ug}\psi,\beta_v\phi) + \p_t\phi g_{v2}(\beta_{ug}\psi,\beta_{[1;R-2,\infty]}\beta_v\phi).
\]
Hence our expression can really be expanded as
\begin{align*}
e^{w(s)}\mcal{F}_v(e^{-w(s)}\zeta, e^{-w(s)}\psi_\pm') =&
e^{w(s)}\mcal{F}_v(e^{-w(s)}\gamma_+ \zeta_+ +\gamma_-\zeta_-, e^{-w(s)}\psi_\pm') \\
&+ G_1(e^{-w(s)}\zeta_\pm,e^{-w(s)}\p_t \zeta_\pm, e^{-w(s)}\psi_\pm',e^{-w(s)}\dt \zeta)\dt \zeta+G_2(e^{-w(s)}\zeta_\pm,e^{-w(s)}\psi_\pm')\p_t \dt \zeta.
\end{align*}
The functions $G_*$ (the functions themselves, ignoring its inputs such as $\zeta_\pm$) have uniformly bounded smooth derivatives and are bounded in the following way:
\[
G_*(x_1,\ldots,x_n) \leq |x_1| +\ldots+|x_n|
\]
for $x_*$ small. Recalling our choice of cut off functions we always have $w(s)>1$, so this assumption is always satisfied.
Recalling the elliptic regularity results on $\zeta_\pm$ above we can actually bound the $W^{2,p}(v_{r,a,p}^*TM)$ norm of $G_1$ and $G_2$ by $\ep$. 
Then our equation for $\Theta_v'$ simplifies to
\[
D_{J_\dt}' \dt \zeta +G_1\dt \zeta + G_2 \p_t \delta \zeta =E
\]
where $E$ was defined as the error term above. We now apply the contraction mapping principle to this equation, let $\Pi' \circ Q$ denote the right inverse to $D_{J_\dt}'|_{H_0'}$ (where $\Pi'$ corresponds to projection to $H_0'$ as we have removed exponential weights). Consider the linear functional $I(\dt \zeta)$:
\[
\dt \zeta \longrightarrow \Pi'\circ Q(-G_1 \dt \zeta - G_2 \p_t \dt \zeta +E).
\]
Let $B_\ep \subset H_0'$ denote a ball of size $\ep$, then it follows from the form of $G_*$ as well as the size estimate of $E$ that $I$ maps $B_\ep$ to itself. It follows similarly from above that $I$ is a contraction mapping, hence it follows from the contraction mapping principle that such $\dt \zeta$ is unique. It follows from uniqueness of $\zeta \in H_0'$ in previous theorem that this $\zeta$ from this contraction mapping is the $\zeta$ we constructed earlier.\\
The norm estimate of $\dt \zeta$ follows directly from the norm estimate of $E$. The improvement from $W^{2,p}$ to $W^{3,p}$ is as follows: we first realize $\zeta = \gamma_+\zeta_+ + \gamma_-\zeta_- +\dt \zeta$ lives in $W^{3,p}$, the same is true for $\gamma_\pm\zeta_\pm$, hence $\dt \zeta$ also lives in $W^{3,p}$. To get the actual norm estimates, we recall the fixed point equation
\[
\dt \zeta = \Pi'\circ Q(-G_1\dt \zeta -G_2 \p_t \dt \zeta -E).
\]
We first realize $-G_1\dt \zeta -G_2 \p_t \dt \zeta +E$ actually lives in $W^{2,p}(v_{r,a,p}^*TM)$ by previous elliptic regularity results. We then realize $\Pi' \circ Q$ restricts to a bounded operator from $W^{2,p}(v_{r,a,p}^*TM) \rightarrow W^{3,p}(v_{r,a,p}^*TM)$ with image in $H_0' \subset W^{3,p}(v_{r,a,p}^*TM)$ by applying elliptic regularity to $D_{J_\dt}$. Finally we observe the $W^{2,p}$ norm of $E$ is similarly bounded above by $
C max\{\|\zeta_+\|,\|\zeta_-\|\}^{2/p} e^{-\lambda (T/2\dt-3R)}$ owing to the fact in the region where $E$ is supported, $\zeta_\pm$ is smooth. Then to get the $W^{3,p}$ norm of $\dt \zeta$ we just measure the $W^{3,p}$ norm of both side of the fixed point equation and conclude.
\end{proof}
We now investigate how $\dt \zeta$ varies w.r.t.  $p_\pm$, because we already understand $\zeta_\pm$ is well behaved w.r.t. $p_\pm$. Instead of varying $p_\pm$ individually, we find it is more convenient to change basis and distinguish two kinds of variations.  We introduce the new variable $p$.
\begin{itemize}
    \item We call the transformation of this type: $(p_-,p_+) \rightarrow (p_--p,p_++p)$ a stretch. 
    \item We can transformation of the type $(p_-,p_+)\rightarrow (p_-+p,p_++p)$ a translation.
\end{itemize}
We shall vary $\dt \zeta$ w.r.t. $p$ with these kind of transformations. In both cases we shall show $\dt \zeta$ is well behaved w.r.t. differentiating via $p$.
\subsubsection{Stretch}
Observe in our region of interest we assumed $f'(x)=1$, and that $x'(s)=\dt f'(x)$. The effect of stretch will be thought of as keeping the same gradient trajectory $v_{r,a,p}$ prescribed by $(p_+,p_-)$ but lengthen the interval $s\in [0,T_p/\dt]$ to $[-p/\dt, T_p+p/\dt]$ over $v_{r,a,p}^*TM$ with the peak of exponential weight profile $w_p(s)$ still at $s=T_p/2\dt$. We translate $u_+$ and $u_-$ in opposite directions along symplectization coordinate. We then think of equation $\Theta_v$ as taking place over the same gradient cylinder, but various terms like  $\psi_\pm'$ being translated as we stretch along $p$. (There is some abuse in notation here, $T_p$ refers to the gradient flow length for original pair $(p_+,p_-)$, and $p$ is how much we stretched).

The $a$ distance between $a(-p/\dt)$ and $a(T_p+p/\dt)$ also changes but not in a linear fashion since $a'(s) = e^{\dt f(x)}$ but this is fine since none of our operators depend on $a$.

We make the following important observation about $\phi_\pm$. In our section dealing with semi-infinite trajectories when we moved the asymptotic vector $p$ ($p$ here as in an element among the tuple $(r,a,p)$) we preglued to a different gradient trajectory. To be specific, let's focus on $\phi_-$. In the case of semi-infinite trajectories, after changing gradient trajectories, no matter the value of $p_-$ the pregluing always happened at $s=R$. We denote the resulting function of $(s,t)$ by  $\tilde{\phi}_-(p)$ so that in this system preluing always happened at $s=R$. Now in the stretch picture we are taking a different perspective, that when we deform by $p$ we are pregluing to a different segment of the same gradient trajectory $v_{r,a,p}$, so $\tilde{\phi}_-(p)$ and $\phi_-$ are related via translation, to be precise
\[
\phi_-(s+p/\dt)=\tilde{\phi}_-(p)(s)
\]

Here we only consider variations in the $p_\pm$ directions and have suppressed the $t$ variable - there should be some identification of $t+(r_+-r_-)$ and $t_-$. Variations in $r_\pm$ will be considered in a subsequent section. The feedback into $\Theta_-$ is given precisely by $\tilde{\phi}(p)(s)$ for $s\in [-3R,3R]$.
And we understand how  $\tilde{\phi}_-(p)$ depends on $p$, and by our previous sections its feedback into equation $\Theta_-$ is \emph{well behaved} w.r.t. $p$. A similar relation also holds to $\phi_+$, and $\psi_\pm$.

Here we see the advantage of working in $W^{2,p}(v_{r,a,p}^*TM) $ instead of $W^{2,p,w}(v_{r,a,p}^*TM) $ since our norms are independent of $p$. Observe similarly our definition of $H_0'$ is independent of $p$. The only dependence in $p$ comes from terms of the form $e^{w_p(s)}$ (we include, where relevant, the subscript $p$ into our exponential weight profiles), which we will be able to describe explicitly. The formulate the following proposition:
\begin{proposition}
In the case of a stretch,
\[
\left \|\frac{d}{dp} \dt\zeta \right \|\leq \frac{C }{\dt} e^{-\lambda (T_p/2\dt-3R)}
\]
where the norm of $\|d/dp \dt\zeta\|$ is measured w.r.t. $W^{2,p}(v_{r,a,p}^*TM)$. We are taking the derivative at $p=0$, but it is obvious a similar formula holds for all small values of $p$ uniformly.
\end{proposition}
\begin{proof}
We already know for every $p$ there is a $\dt \zeta$ (we suppress the dependence on $p$) satisfying
\[
D_{J_\dt}' \dt \zeta +G_1\dt \zeta + G_2 \p_t\dt \zeta =E
\]
which we may rewrite as 
\[
\dt \zeta = \Pi'\circ Q(-G_1\dt \zeta - G_2 \p_t\dt \zeta +E).
\]

We next proceed to differentiate both sides w.r.t. $p$.
We see that the result is an expression of the form
\begin{align*}
    & \frac{d}{dp} (\delta \zeta)\\
    =& (\frac{d}{dp} \Pi' \circ Q)( -G_1\dt \zeta - G_2 \p_t\dt \zeta+ E)\\
    &+ \Pi' \circ Q \cdot(- \frac{dG_1}{dp} \delta \zeta - \frac{dG_2}{dp}\p_t \delta \zeta )\\
    &+ \Pi' \circ Q  (-G_1 \cdot \frac{d}{dp} \delta \zeta  - G_2 \frac{d}{dp} \p_t \dt \zeta)\\
    &+ \Pi' \circ Q \frac{dE}{dp}.
\end{align*}
See Remark \ref{Frechet} for this kind of differentiation.

\textbf{Step 1}. We first differentiate $\Pi \circ Q$ w.r.t. $p$. Recall over $W^{2,p,w}(v_{r,a,p}^*TM)$ $\Pi$ takes the form:
\[
\Pi(\phi) = \phi -\sum_* L_*(\phi) \p_*
\]
after we remove the exponential weights the corresponding operator $\Pi'$ takes the form:
\[
\Pi' \zeta = \zeta - \sum_* L_*(e^{-w(s)}\zeta) e^{w(s)}\p_*.
\]
For stretch $\p_*$ is independent of $p$, so the only dependence we see is on $w(s)$. We realize $L((e^{-w(s)}\zeta)) = L(\zeta) e^{-w(T_p/2\dt)}$, but we realize that $w(s)-w(T_p/2\dt)$ is independent of $p$, so we conclude $\Pi'$ is independent of $p$.

We next consider $\frac{d}{dp}\Pi' \circ Q$. We observe this is a map from $W^{1,p}(v_{r,a,p}^*TM)\rightarrow H_0$. It is the inverse of $D_{J_\dt}|_{H_0'}$, so we can instead differentiate the relation
\begin{equation*}
    (\Pi' \circ Q ) \circ (D'_{J_\dt} \circ \iota) =I.
\end{equation*}
where $\iota: H_0' \rightarrow W^{2,p}(v_{r,a,p}^*TM)$, to get
\begin{equation*}
     \frac{d(\Pi' \circ Q )}{dp} (D'_{J_\dt} \circ \iota) +  (\Pi '\circ Q )\frac{ (D'_{J_\dt} \circ \iota)}{dp}=0
\end{equation*}
\begin{equation*}
     \frac{d (\Pi '\circ Q )}{dp}  =-  (\Pi' \circ Q )\frac{d (D_{J_\dt}' \circ \iota)}{dp} (\Pi' \circ Q ).
\end{equation*}
We already know $\Pi' \circ Q$ is uniformly bounded w.r.t. $\dt \rightarrow 0$, now recall that
\[
D_{J_\dt}' = D_{J_\dt} +w'(s).
\]
Of course we know that $w(s)$ has a ``kink" where the absolute value bends (see definition equation for $w$) but we can smooth it. Noting that $D_{J_\dt}$ is independent of $p$ and $w'(s)$ is independent of $p$, we conclude that
\[
\left\| \frac{d (\Pi' \circ Q )}{dp}\right \| \leq C
\]
where we use the operator norm. (In this case it's in fact zero).

\textbf{Step 2}
We next examine the term $\frac{dG_*}{dp}$. There are two kinds of dependencies, one on how the function $G_*$ depends on $p$, which we denote by $\p G/\p p$, and second how its arguments $\zeta_\pm'$and $\psi_\pm$ depend on $p$. We first recall $G_*$ comes about from the expansion
\begin{align*}
e^{w(s)}\mcal{F}_v(e^{-w(s)}\zeta, e^{-w(s)}\psi_\pm') =&
e^{w(s)}\mcal{F}_v(e^{-w(s)}(\gamma_+ \zeta_+ +\gamma_-\zeta_-), e^{-w(s)}\psi_\pm') \\ &+G_1(e^{-w(s)}\zeta_\pm,e^{-w(s)}\p_t \zeta_\pm, e^{-w(s)}\psi_\pm',e^{-w(s)}\dt \zeta)\dt \zeta  \\
&+G_2(e^{-w(s)}\zeta_\pm,e^{-w(s)}\psi_\pm')\p_t \dt \zeta.
\end{align*}
The function $\mcal{F}_v$ only depends on the geometry, so the only dependence of the function $G$ on $p$ is given by $dw/ds$:
\[
\frac{\p G_*}{\p p} \leq C\left|\frac{dw}{dp}\right| \leq C/\dt.
\]
Next we try to understand the dependence of $dG_*/dp$ through its dependence on terms like $d\zeta_\pm/dp$ and $d\psi'/dp$. From previous remark by the semi-infinite trajectory case we understand $d\tilde{\zeta}_\pm/dp \leq C\ep$, and $\zeta_\pm = \tilde{\zeta}_\pm(s \mp p/\dt)$. Similarly we have $\psi'_\pm=\tilde{\psi'}(s \mp p/\dt)$. Noting $\tilde{\psi}_\pm$ doesn't depend on $p$, we have 
\[
\frac{d}{dp}\psi'_\pm = -\frac{1}{\dt} \frac{d}{ds}\tilde{\psi}'_\pm
\]
thus estimates:
\[
\left\|\frac{d\psi_\pm'}{dp}\right\| \leq C\ep/\dt.
\]
Note taking the $p$ derivative of $\psi'$ has cost us a derivative, hence the above norm can only be measured in $W^{1,p}$. Thankfully this is enough for our purposes because $Q$ brings back another derivative.
\[
\left \|\frac{d \zeta_\pm}{dp}\right \| \leq \left\|\frac{d}{dp} \tilde{\zeta}_\pm \right\| + \frac{1}{\dt} \left\|\frac{d}{ds} \tilde{\zeta}_\pm\right \| \leq C\ep(1 +1/\dt)
\]
Now in this computation $\frac{d}{dp} \tilde{\zeta}_\pm$ lives naturally in $W^{2,p}$, by elliptic regularity $\tilde{\zeta}_\pm$ lives in $W^{3,p}$, so its $s$ derivative lives in $W^{2,p}$. Hence the above inequality can at most hold in $W^{2,p}$, which suffices for our purposes.

Next we need to consider the $W^{1,p}$ norm of $\frac{d}{dp} \p_t \zeta_\pm$, which we re-write as
\[
\frac{d}{dp} \p_t \tilde{\zeta}_\pm (s\mp p/\dt) = \p_t \p_p \tilde{\zeta}_\pm (s \mp p/\dt) - \frac{1}{\dt}\p_s\p_t  \tilde{\zeta}_\pm(s\mp p/\dt).
\]
We first make a remark about commutativity of derivatives, e.g. we have commuted $\p_p \p_t \tilde{\zeta}_\pm = \p_t\p_p \tilde{\zeta}_\pm$. We know $\p_p \tilde{\zeta}_\pm$ is in $W^{2,p}$, hence we can commute the derivatives using the following version of Clairut's theorem:
\begin{proposition}
If $f: \bb{R}^2 \rightarrow \bb{R}$ is so that $\p_1 f,\p_2 f, \p_{2,1}f$ exists everywhere (here $\p_1f$ denotes the partial derivative of $f$ w.r.t the first variable), $\p_{2,1}f$ is continuous, then $\p_{1,2}f$ exists and is equal to $\p_{2,1}f$.
\end{proposition}
Hence we can commute the derivative, measure the $W^{1,p}$ norm of $\p_t \p_p \tilde{\zeta}_\pm (s \mp p/\dt)$, and bound it by the $W^{2,p}$ norm of $\frac{d}{dp}\tilde{\zeta}_\pm$, which is bounded above by $C\ep$.
The $W^{1,p}$ norm of $\p_s\p_t  \tilde{\zeta_\pm}(s\mp p/\dt)$ is bounded by the $W^{3,p}$ norm of $\tilde{\zeta}_\pm$, which is also bounded by $C\ep$ by elliptic regularity.

But observe the expression involving $\frac{dG_*}{dp}$ is multiplied by $\dt \zeta $ or $\p_t \dt \zeta$, so overall we have the estimate:
\begin{equation}\label{equation_product}
\left\|Q \circ \left\{ \frac{dG_1}{dp}\cdot  \dt\zeta +\frac{dG_2}{dp}\p_t \dt \zeta \right\}\right\|_{W^{2,p}} \leq C\ep/\dt e^{-\lambda (T_p/2\dt-3R)} +\ep \|d\dt \zeta/dp\|_{W^{2,p}}.
\end{equation}
The last term coming from the dependence of $G_1$ on $\dt \zeta$. 
We remark in the above the term $\frac{dG_1 }{dp} \cdot \p_t\dt \zeta$ we have a product of $W^{1,p}$ functions, which remains in $W^{1,p}$. This is where we justify our use of $W^{2,p}$ instead of $W^{1,p}$. See Remark \ref{remark_sobolev_exponents}.\\
\textbf{Step 3}
The next term is $\Pi' \circ Q  (-G_1 \cdot \frac{d}{dp} \delta \zeta  - G_2 \frac{d}{dp} \p_t \dt \zeta)$. Note we have $C^1$ bound on $G_*$, which is bounded by $C\ep$, so after we apply $\Pi' \circ Q$ the norm of this term is overall bounded by 
$C\ep \|\frac{d\dt\zeta }{dp}\|_{W^{2,p}}$, and we move this term to the left hand of the equation.\\
\textbf{Step 4} We finally estimate how the error term $E$ depends on $p$, and here we shall use the exponential decay estimates proved in Section \ref{expdecay}. Recall $E$ takes the form
\[
E = \gamma_\pm' \zeta_\pm + [e^{w(s)}\mcal{F}_v(e^{-w(s)} \gamma_+\zeta_+)-\gamma_+e^{w(s)}\mcal{F}_v(e^{-w(s)} \zeta_+)] +[e^{w(s)}\mcal{F}_v(e^{-w(s)} \gamma_-\zeta_-)-\gamma_-e^{w(s)}\mcal{F}_v(e^{-w(s)} \zeta_-)].
\]
The important feature of this expression is that it has support in $ s\in[T_p/2\dt-2,T_p/2\dt+2]$, the term $E$ and its derivative over $p$ can be upper bounded by the terms:
\[
|E(s,t)| \leq C |\zeta_\pm(s,t)|(1 + |\p_t \zeta_\pm (s,t)|)
\]
\[
\left|\frac{dE}{dp}(s,t)\right| \leq C\left|\frac{d \zeta_\pm}{dp}(s,t)\right| + C\left|\zeta_\pm\right|\left(\left|\frac{d}{dp} e^{-w(s)} \zeta_\pm\right| + \left|\frac{d}{dp}e^{-w} \p_t\zeta_\pm\right|\right)
\]
where for both equations size refers to $C^1$ norm. Since this is supported over $ s\in[T_p/2\dt-2,T_p/2\dt+2]$, bounds on the uniform norm imply bounds on Sobolev norms. Furthermore we know by elliptic regularity $\zeta_\pm$ and its $p$ derivative are smooth over this region so it make sense to talk about $C^1$ norms. We first note
\[
\frac{d}{dp}e^{-w} =\frac{C}{\dt} e^{-w}.
\]
So the only terms we need to worry about are 
\[
\frac{d}{dp} \zeta_\pm, \quad \frac{d}{dp} \p_t \zeta_\pm
\]
over the interval $ s\in[T_p/2\dt-2,T_p/2\dt+2]$. We recall by our convention $\zeta_\pm(s)= \tilde{\zeta}_\pm(s\mp p/\dt)$
so we have
\[
\frac{d\zeta_\pm}{dp} = \frac{d}{dp}\tilde{\zeta}_\pm(s\mp p/\dt) + \frac{1}{\dt}\frac{d}{ds}(\tilde{\zeta}_\pm(s\mp p/\dt)).
\]
By the constraint that $s\in[T_p/2\dt-2,T_p/2\dt+2]$, the terms on the right hand side have already decayed substantially, hence they are bounded by
\[
\frac{C}{\dt}e^{-\lambda (T_p/2\dt-3R)}
\]
which quickly decays to zero as $\dt \rightarrow 0$.
Finally we compute the derivative $\frac{d}{dp} \p_t \zeta_\pm$ for $ s\in[T_p/2\dt-1,T_p/2\dt+1]$. We can also break this down into
\[
\frac{d\zeta_{\pm,t}}{dp} = \frac{d}{dp}\tilde{\zeta}_{\pm,t}(s\mp p/\dt) + \frac{1}{\dt}\frac{d}{ds}(\tilde{\zeta}_{\pm,t}(s\mp p/\dt)).
\]
The exponential decay estimates in Corollary \ref{cor:expdecay_p}, as well as exponential decay in Proposition \ref{prop:higher_order_decay}, say in the interval $s\in[T_p/2\dt-1,T_p/2\dt+1]$ the above is also bounded by
\[
\frac{C}{\dt}e^{-\lambda (T_p/2\dt-3R)}.
\]
\textbf{Step 5} Combining all of the above estimates we see that 
\[
\left\|\frac{d}{dp} \dt\zeta\right \|\leq \frac{C}{\dt} e^{-\lambda (T_p/2\dt-3R)}
\]
as claimed.
\end{proof}
Now we use the above to show $\zeta$ is well behaved in the sense we originally described. 
\begin{proposition}
$\zeta$, and hence $\phi$ is well behaved with respect to $p$ when $p$ controls a stretch.
\end{proposition}
\begin{proof}
Note what we feed into $\Theta_\pm$ are vector fields with exponential weights, so we put $\zeta$ back into $\Theta_\pm$ we need to turn it back to $\phi$ via
\[
\phi = e^{-w(s)}\zeta
\]
but note that from above we have
\[
\phi = \gamma_+ \phi_+ + \gamma_-\phi_- +e^{-w(s)} \dt \zeta.
\]
And we know terms like $\gamma_\pm \phi_\pm$ behave nicely with respect to $p$. So it suffices to understand how $e^{-w(s)} \dt \zeta$ feeds back into $\Theta_\pm$. For simplicity we focus on $\Theta_-$. For fixed $p$, and for $s\in [-p/\dt,-p/\dt +3R]$, if we define $\dt \phi := e^{-w(s)} \dt \zeta$, then from the perspective of $\Theta_-$, the vector field we see is $\dt \phi (s'-p/\dt)$ for $s'\in [0,3R]$ equipped weighted norm $e^{ds'}$. We observe over the region $s' \in [0,3R]$, the weight function coming from $e^{w(s')} = e^{ds'}$, so when we calculate how the $p$ variation feeds back into $\Theta_-$ we are really looking at
\[
\left\|\frac{d}{dp}\dt \zeta(s'-p/\dt)e^{-ds'}\right\|
\]
for $s' \in [0,3R]$ with respect to the norm $W^{2,p,d}(u_-^*TM)$, which is equivalent to the expression:
\[
\left\|\frac{d}{dp} \dt \zeta(s'-p/\dt) \right\|
\]
with the unweighted $W^{2,p}$ norm over the interval $s' \in [0,3R]$. We observe
\[
\left\|\frac{d}{dp} \dt \zeta(s'-p/\dt)\right \|_{W^{2,p}} \leq \frac{C}{\dt} \left\| \frac{d}{ds} \dt \zeta (s' -p/\dt)\right\|_{W^{2,p}} + \left\|\frac{d}{dp}\dt\zeta(s'-p/dt)\right\|_{W^{2,p}}
\]
Here we have used elliptic regularity on $\dt \zeta$ to control its $W^{3,p}$ norm by its $W^{2,p}$ norm. The by the preceding proposition both of the above expressions are bounded above by $\frac{C}{\dt} e^{-\lambda (T_p/2\dt-3R)}$, hence the proof.
\end{proof}
\subsubsection{Translation}
The case of translation is much easier than the case of stretch, as it bears many similarities with the case of semi-infinite trajectory. We don't even need to remove exponential weights. The only salient difference is we now have to work in a subspace $H_0$.

Let us first recall/set up some notation. Fix tuples $(r,a,p)_\pm$, and they determine a pregluing between $u_+$ and $u_-$. We use  $v_{p_\pm}$ to denote the intermediate trajectory that connects between $u_+ + (r,a,p)_+$ and $u_- + (r,a,p)_-$ in the pregluing. As before we define $w_{p\pm}(s)$ as our exponential weight profile, and we have the codimension 3 subspace $H_0$. We fix $(s,t)$ coordinates over $v_{p_\pm}$, with gluing happening at $s=R$ and $s=T_p/\dt -R$. Let $p\in \bb{R}$ be a small number denoting the size of the translation, let $p^*_\pm = p_\pm +p$, and let $v_{p^*}$ denote the gradient trajectory between the pregluing determined by $p^*_\pm$. We equip vector fields over $v_{p^*}$ with Sobolev norms as previous described and it also has a subspace $H_0^*$. On $v_{p^*}$ we choose coordinates $(s^*,t^*)$ and because we assumed the function $f(x)$ is locally linear (after maybe a change of coordinates) we have that pregluing happens at $s^*=R$ and $s^*=T_p/\dt -R$. Observe there is a parallel transport map using the flat metric
\[
PT:W^{2,p,w}(v_{p^*}^*TM)\longrightarrow W^{2,p,w}(v_{p}^*TM)
\]
such that if $\phi^*(s^*,t^*)$ is a vector based at $v_{p^*}(s^*,t^*)$, it is transported to $ \phi(s=s^*,t=t^*)$ over $v_{p}(s,t)$. Note the parallel transport map send $H_0^*$ to $H_0$. And the solution $\phi^*_{p^*_\pm}$ to $\Theta_v$ over $v_{p^*}$ can be identified with $\phi(p)\in H_0$ to an equation of the form
\[
D_{J_\dt}(p) \phi_p + \mcal{F}_v(p,\psi_\pm, \phi_p) + \beta_\pm' \psi=0
\]
and the feedback term from $\phi_{p_\pm^*}^*$ into $\Theta_\pm (p_\pm^*)$ can be identified with the feedback of $\phi_p$ which corresponds to regions $s \in [-3R,3R]$ for $\Theta_-$ and $s \in [T_p/\dt-3R,T_p/\dt+3R]$ for $\Theta_+$. Then it suffices to calculate the norm of $d\phi_p/dp$ in $H_0$.
\begin{proposition}
$\left\|\frac{d\phi_p}{dp}\right\|_{W^{2,p,w}(v_{p}^*TM)}\leq C\ep$.
\end{proposition}
\begin{proof}
Observe that $\|d/dpD_{J_\dt}\|\leq C$ when measured in the operator norm because the coefficient matrices in this operator only depend on the background geometry. The same is true for $\|\frac{\p \mcal{F}_v}{\p p} (p,-,-)\|_{C^1} \leq C$. We recall $D_{J_\dt}(p)$ is an isomorphism from
\[
H_0 \longrightarrow W^{1,p,w}(v_p^*TM)
\]
hence has an inverse whose operator norm is uniformly bounded over $p$ and as $\dt\rightarrow 0$. The same is true for the derivative in $p$ of this inverse. To see this, we recall $D_{J_\dt}(p): W^{2,p,w}(v_p^*TM)\rightarrow W^{1,p,w}(v_p^*TM)$ has a right inverse  uniformly bounded in $p$ and $\dt \rightarrow 0$, which we denote by $Q$. We also recall the inverse for $D_{J_\dt}(p)$ is obtained by $\Pi \circ Q$. Hence it suffices to show $\Pi$ has uniformly bounded norm as $p$ changes in a translation. \\
Recall
\[
\Pi (\phi_p)= \phi_p-\sum_* L_*(\phi_p) \p_*
\]
as an operator we see that the terms involving $*=z,s$ are independent of $p$, the vector field $v :=a\frac{v_* \partial_s -\p_s}{\dt} +b\p_s $ depends on $p$ but we see in $C^1$ norm that $|\frac{dv}{dp}| \leq C$, so we see $\Pi$ has uniformly bounded norm as $p$ varies, which in turn implies $\Pi \circ Q$ has uniformly bounded norm. We now investigate $\frac{d}{dp} \Pi \circ Q$, which we can understand by differentiating the expression
\[
\Pi \circ Q \circ D_{J_\dt}(p) |_{H_0} = id|_{H_0}
\]
w.r.t. $p$, which yields
\[
\frac{d}{dp} \Pi \circ Q  = -\Pi \circ Q  (\frac{d}{dp}  D_{J_\dt}(p)) \circ \Pi \circ Q 
\]
which implies as an operator $\frac{d}{dp} \Pi \circ Q $ has uniformly bounded norm.
Next we recast the equation:
\[
D_{J_\dt}(p) \phi_p + \mcal{F}_v(p,\psi_\pm, \phi_p) + \beta_\pm' \psi=0
\]
as a fixed point equation
\[
\phi_p = \Pi \circ Q (-\mcal{F}_v(p,\psi_\pm, \phi_p) - \beta_\pm' \psi)
\]
using the exact same procedure as we did for for semi infinite gradient trajectories, we differentiate this equation in $p$ to show $\|d\phi_p/dp\| \leq C\ep$. Observe after parallel transport there was no translation of $\psi_\pm$ involved.
\end{proof}
Since in this case we worked directly with weighted norms we can directly conclude:
\begin{corollary}
With respect to translations, the vector field $\phi_p$ is well behaved.
\end{corollary}
With this and the previous subsection, we conclude that $\phi$ is well behaved with respect to variations of $p_\pm$. In the next part we examine how $\phi$ varies when we change $r_\pm,a_\pm$.

\subsubsection{Variations in \texorpdfstring{$r_\pm$}{r}, \texorpdfstring{$a_\pm$}{a}}
In this subsection we show that when we vary the parameters $a_\pm$ and $r_\pm$ the solution $\phi$ is well behaved. 
\begin{proposition}
The solution $\phi$ to $\Theta_v$ is well behaved w.r.t. $a_\pm$.
\end{proposition}
\begin{proof}
Observe that changing $a_\pm$ can also have the effect of lengthening and shortening the gradient trajectory we need to glue between $u_+$ and $u_-$, though the process is substantially less dramatic than when we changed $p_\pm$. For instance when we change $a_\pm$ by size $\ep$, the connecting gradient trajectory may lengthen/shrink by size $C\ep$, instead of $C\ep/\dt$. In particular we can redo all of the previous subsection. We separate the change to stretch and translation. We first observe in case of translation the equation $\Theta_v$ actually stays invariant, because all of our background geometry is invariant in the $a$ direction. In the case of stretch, we remove the exponential weights, and repeat the above proof. The difference is that no factor of $C/\dt$ ever appears, so we don't even need the exponential decay estimates. The rest follows as above.
\end{proof}

\begin{proposition}
$\phi(r,a,p)$ is well behaved as we vary $r_\pm$.
\end{proposition}
\begin{proof}
Recall for $r_+$ in the pregluing construction we are rotating the entire gradient trajectory $v_{r,a,p}$ along with it, so we can again use parallel transport in $r_+$ to turn it into a family of equations over the same space, which we denote by $W$ as before, and the resulting $(r_+,r_-)$ family of PDEs over $W$ by $\hat{\Theta}_v$. We use $\hat{\phi}(r_+,r_-)$ to denote the solution to $\hat{\Theta}_v$. By assumption, the almost complex structure $J$, when restricted to the surface of the Morse-Bott torus, is $r$ invariant, however, the local geometry is not necessarily invariant. Therefore, the linearized operator as well as nonlinear term picks up a $r_+$ dependence, so the equation solved by $\phi(r_+,r_-)$  has a linear operator $D_{J_\dt}(r_+)$ and a nonlinear term $\hat{\mcal{F}}_v(r_+,-)$ with $r_+$ dependence.
Also observe $H_0$ is invariant under changing $r_\pm$, so we denote it by the same letter when viewed as subspace in $W$.

We now recall what happens to the pregluing near the $u_-$ end, the domain Riemann surface $\Sigma_{r,a,p}$ is constructed at $s=R$ with the identification $t+r_+ \sim t_-+r_-$. So we see this effect in the equation $\hat{\Theta}_v$ via the dependence of $\psi_-$ on $r_\pm$, in particular the $\psi_-$ term in $\hat{\Theta}_v$ should be instead $\psi_-(s,t+r_+-r_-)$. Hence after parallel transport we see $\hat{\phi}(r_+,r_-)$ is the unique solution to the equation in $H_0$:
\[
D_{J_\dt} (r_+) \hat{\phi} +\beta_+' \psi_+(s,t) + \beta_-' \psi(s,t+r_+-r_-) + \hat{\mcal{F}}_v(r_+,\psi_\pm)=0.
\]
Again, following the same procedure as we did for semi-infinite gradient trajectories we recast this as a fixed point equation
\[
\hat{\phi} = \Pi \circ Q (-\beta_+' \psi_+(s,t) - \beta_-' \psi(s,t+r_+-r_-) - \hat{\mcal{F}}_v(r_+,\psi_\pm))
\]
and differentiate both sides with respect to $r_\pm$, observing that $\frac{\p \hat{\mcal{F}}_v(r_+,\psi_\pm)}{\p r_+} =  g(\phi,\psi) + h(\phi,\psi)\p_t(\phi) $ as in Remark \ref{quadratic_term}. However, it is important to note that taking an $r_\pm$ derivative of the above equation will produce a $t$ derivative of $\psi_-$, which will produce a function in $W^{1,p}$ (we neglect any mention of weights for now). But since we are not taking any further derivatives of $\psi_\pm$, this is fine as $Q$ will send this to $W^{2,p}$, then the same argument as before shows that 
\[
\left\|\frac{d\hat{\phi}}{dr_\pm}\right\|_{W^{2,p,w}(v_{r,a,p}^*TM) }\leq C\ep
\]
as desired.
\end{proof}
\subsubsection{Solution of \texorpdfstring{$\Theta_\pm$}{Theta}}
In this subsection we use the results from previous section to finally solve $\Theta_\pm$ and hence conclude gluing exists. Recall deformations of $u_\pm$ are given by the elements  $(\psi_\pm, (r,a,p)_\pm, \p'_\pm, \dt j_\pm )\in W^{2,p,d}(u_\pm^*TM)\oplus V_\pm \oplus V_\pm'\oplus T \mcal{J}_\pm$, and the linearized Cauchy Riemann operator 
\[
D\db_{J\pm}:W^{2,p,d}(u_\pm^*TM)\oplus V_\pm \oplus V_\pm'\oplus T \mcal{J}_\pm \longrightarrow  W^{1,p,d}(\overline{\op{Hom}}(T\dot{\Sigma}, u_\pm^*TM))
\]
is surjective with right inverse $Q_\pm$. Then $\Theta_\pm$ are equations of the form
\[
D\db_{J_\pm}((\psi_\pm, (r,a,p)_\pm, \p'_\pm, \dt j_\pm )) + \mcal{F}_\pm ((\psi_\pm, (r,a,p)_\pm, \p'_\pm, \dt j_\pm ,\phi)+\mcal{E}_\pm + \beta_v' \phi=0
\]
where $\mcal{F}_\pm$ is a quadratic expression in each of its variables, implicit in $\mcal{F}_\pm$ are quadratic terms depending on $(\dt j_\pm, \psi_\pm)$ responsible for variation of domain complex structure away from the punctures. And implicit in term $\mcal{E}_\pm$ are error terms uniformly bounded by $C \dt$ in the interior of $u_\pm$ responsible for the fact that $u_\pm$ are $J$-holomorphic, instead of $J_\dt$-holomorphic.
\begin{theorem}
The system of equations $\Theta_\pm =0$ has a solution, and hence 2 level cascades with one intermediate end can be glued. Furthermore, for specific choices of $Q_\pm$, which are right inverse to $D\bar{\p}_{J_\pm}$, there is a unique solution in the image of $(Q_+,Q_-)$.
\end{theorem}
\begin{proof}
We consider the system of $\Theta_\pm$ as a map from 
\begin{align*}
(\Theta_+,\Theta_-):&
(W^{2,p,d}(u_+^*TM)\oplus V_+ \oplus V_+'\oplus T \mcal{J}_+) \oplus 
(W^{2,p,d}(u_-^*TM)\oplus V_- \oplus V_-'\oplus T \mcal{J}_-) \longrightarrow\\
&W^{1,p,d}(\overline{\op{Hom}}(T\dot{\Sigma}, u_+^*TM))\oplus W^{1,p,d}(\overline{\op{Hom}}(T\dot{\Sigma}, u_-^*TM)).
\end{align*}
We solve this via a fixed point theorem by finding a fixed point to the map
\[
[(\psi_+, (r,a,p)_+, \p'_+, \dt j_+ ),(\psi_-, (r,a,p)_-, \p'_-, \dt j_- )] \longrightarrow [Q_+(-\mcal{F}_+-\mcal{E}_+-\beta_v'\phi), Q_-(- \mcal{F}_--\mcal{E}_--\beta_v'\phi)].
\]
We show it maps the $\epsilon$ ball to itself. This follows from the size estimates we had of $\phi$ relative to $\psi_\pm$, as well as the fact $\mcal{F}_\pm$ is quadratic, and the size of the terms that appear in $\mcal{E}_\pm$ are very small. We next argue this map has the contraction property as we vary $\psi_\pm, (r,a,p)_\pm, \p'_\pm, \dt j_\pm$; this follows directly from the previous subsection in which we showed $\phi$ is well behaved with respect to these input variables, plus the fact $\mcal{F}_\pm$ is quadratic (see remark \ref{quadratic_term}).The sizes of terms that appear in $\mcal{E}_\pm$ are also uniformly small, as we derived in the pregluing section. Hence the contraction mapping principle shows there is a unique solution in the image of $(Q_+,Q_-)$.
\end{proof}

\begin{remark}\label{obsrmk}
\textbf{Relation to obstruction bundle gluing}. We remark we could have proved a gluing exists via obstruction bundle gluing methods in \cite{obs1},\cite{obs2}. This is more similar to how the gluing of 1 level cascades was constructed in \cite{colin2021embedded}. We explain this in the simplified setting as above, and the general case of multiple level cascade can be done analogously.  Recall $D\db_{J,\pm}$ is index 1 (i.e. $u_\pm$ is rigid), we let $U_\pm$ be a 1 dimensional vector space in $W^{1,p,d}(\overline{\op{Hom}}(T\dot{\Sigma}, u_-^*TM))$ spanned by image of asymptotically constant vector field $\p_x$ under $D\db_{J,\pm}$, and let $U'_\pm$ denote a fixed complement given by the image of $(\psi_\pm,r_\pm,a_\pm,\dt j_\pm)$ under $D\p_{J,\pm}$. The fact $U_\pm'$ is closed follows from the fact our operators are Fredholm, and it form a complement for index reasons (as long as neither $u_\pm$ is a trivial cylinder).

Then we form the (trivial) obstruction bundle with base $(p_+,p_-)\in [-\ep,\ep]^2$ and fiber $U_+\oplus U_-$. Then instead of solving $\Theta_\pm$ on the nose we introduce projections $\Pi_{U_\pm'}$ that project to $U_\pm'$. Then for fixed input data $\{(p_+,p_-), (\psi_\pm,r_\pm,a_\pm,\dt j_\pm)\}$ in an epsilon ball, we solve the equation $\Theta_v$ for $\phi$
\[
D_{J_\dt} \phi +\beta_+' \psi_+(s,t) + \beta_-' \psi(s,t+r_+-r_-) + \mcal{F}_v(r_+,\psi_\pm)=0.
\]
Its (unique) solution $\phi$, which depends on all input data $\{(p_+,p_-), (\psi_\pm,r_\pm,a_\pm,\dt j_\pm)\}$, will have norm uniformly bounded by $C\ep/R$ (the $C$ is uniform as we vary $(p_+,p_-)$). Then the solution to the system of equations $\Theta_\pm=0$ is equivalent to the solution of the following system of equations
\[
\theta_\pm :=D\db_{J,\pm} (\psi_\pm,(r,a)_\pm,\p_\pm',\dt j_\pm) +\Pi_{U_\pm'}[(+\mcal{F}_\pm+\mcal{E}_\pm+\beta_v'\phi]
\]
\[
D\db_{J,+} p_+  + (1-\Pi_{U_+'})[(+\mcal{F}_++\mcal{E}_++\beta_v'\phi)] =0
\]
\[
D\db_{J,-} p_- +(1-\Pi_{U_-'})[(+\mcal{F}_-+\mcal{E}_-+\beta_v'\phi)] =0.
\]

We observe for fixed $(p_+,p_-)$ the equations $\theta_\pm$ can always be solved via contraction mapping principle, essentially because the nonlinear term under the projection $\Pi_{U_\pm'}$ always lands in the image of $D\db_{J,\pm}$ by construction, and we have estimates $\|\phi\| \leq C\ep/R$. The other two equations in the language of \cite{obs2} define an obstruction section to the obstruction bundle, as
\[
\mathfrak{s}:=\{p_+  + (1-\Pi_{U_+})[(\mcal{F}_++\mcal{E}_++\beta_v'\phi], p_- +(1-\Pi_{U_-})[(\mcal{F}_-+\mcal{E}_-+\beta_v'\phi]\} \in \Gamma(U_+\oplus U_- \longrightarrow [-\ep,\ep]^2)
\]
and the vanishing of $\mathfrak{s}$ corresponds to gluing. In the above expression we think of $p_\pm$ as real numbers (because we have projected to the one dimensional spaces $U_\pm$).  But we observe by the size estimates of $\phi, \psi_\pm$, the size of the nonlinear term  $(1-\Pi_{U_\pm})[(\mcal{F}_\pm+\mcal{E}_+\pm\beta_v'\phi)]$ under $\Pi_{U_\pm}$ is uniformly bounded above by $C\ep^2/R$. However the linear term $p_\pm$ varies freely from $-\ep$ to $\ep$. The nonlinear term is clearly continuous with respect to variations in $p_\pm$. Hence from topological considerations the obstruction section must have at least one zero, hence we have at least one gluing.

The difficulty with the above approach, is of course there is at least one gluing, but it is unclear how many there are in total. One could improve the above conclusion by trying to argue that $\mathfrak{s}$ is not only $C^0$ close to the $(p_+,p_-)$ but also $C^1$ close, and this would imply the zero is unique. In fact what we proved about ``$\phi$ being well behaved w.r.t. $p_\pm$" is tantamount to showing $\mathfrak{s}$ is $C^1$ close to $(p_+,p_-)$. This required we do very careful exponential decay estimates as well as another contraction mapping principle. That we previously proved gluing via contraction mapping and here phrased it here as obstruction bundle gluing is purely a matter of repackaging.
\end{remark}
\begin{remark}
Another possible approach to obstruction bundle gluing might be to show for generic choice of $J_\dt$ we can arrange to have the zeros of the obstruction section be transverse to the zero section. This will show there is only one gluing up to sign. This is more in line with the strategy taken in \cite{obs2}. However, it's unclear whether we can choose generic enough $J_\dt$ since here we have a family of $J_\dt$ degenerating as $\dt\rightarrow 0$ as opposed to some fixed generic $J$.
\end{remark}

\begin{remark}
We shall later prove surjectivity of gluing. The appendix of \cite{colin2021embedded} used a different strategy for surjectivity, hence did not need to prove the solution obtained via obstruction bundle gluing is unique. Conceivably the methods there could also be applied here, but the construction would be difficult for two  reasons: one they used stable Hamiltonian structures as opposed to contact structures, therefore their equation is nicer than ours. Two it seems their methods would be difficult to carry out in multiple level cascades where the dimensions of moduli spaces that appear could be very high. Instead in what follows we use an approach in Section 7 of \cite{obs2}.
\end{remark}

\subsection{Gluing multiple level cascades}\label{subsection_gluing_multiple_level}
In this subsection we generalize gluing to multiple cascade levels. Given what we have proved above, this is mostly a matter of linear algebra. However there are still subtle details we need to take care of, we first take care of the simple case where we are still gluing together a 2-level cascades, except now with multiple ends meeting in the middle. This contains all the important features required for the gluing. Then we will simply generalize this situation to $n$ level cascades.

\subsubsection{2-level cascade meeting at multiple ends}
We consider a 2-level cascade built out of two $J$-holomorphic curves $u_+$ and $u_-$ meeting along $n$ free ends along an intermediate Morse-Bott torus. It does not matter how many intermediate Morse-Bott tori are there, so for simplicity we assume there is only one. We assume all ends of $u_-$ and $u_+$ landing on this Morse-Bott torus avoid critical points of $f$, and we have chosen coordinates so that the Morse function looks like $f(x)=x$. We assume this cascade is rigid, and $ev_-(u_+)$ and $ev_+(u_-)$ are separated by gradient flow of $f$ for time $T$. We also assume the $x$ coordinates of the positive asymptotic Reeb orbits of $u_-$ are labelled by $x_1,...,x_n$. 

In this example, for simplicity of exposition, we only focus on gluing finite gradient cylinders, and ignore gluing for semi-infinite trajectories. Hence we assume no positive end of $u_+$ nor negative end of $u_-$ lands on the Morse-Bott torus that appear in the intermediate cascade level, and we only perturb the contact form to be nondegenerate in a neighborhood of this torus.

The fact the cascade is rigid and transverse implies the following operator is surjective
\begin{align*}
D_+\oplus D_-:& W^{2,p,d}(u_+^*TM)\oplus T\mcal{J}_+\oplus V_+' \oplus V_+'' \oplus W^{2,p,d}(u_-^*TM)\oplus T\mcal{J}_-\oplus V_-' \oplus V_-'' \oplus (\Delta_t) \longrightarrow \\
&W^{1,p,d}(\overline{\op{Hom}}(T\dot{\Sigma}, u_+^*TM))\oplus W^{1,p,d}(\overline{\op{Hom}}(T\dot{\Sigma}, u_-^*TM)).
\end{align*}
$V_\pm'$ denotes asymptotic vectors associated to ends away from glued ends. $V_\pm''$ denotes asymptotic vector fields at glued ends except they only include $(r,a)_\pm$ components. $\Delta_t$ is a $n+1$ dimensional vector space that consists of asymptotic vectors that satisfy relations $p_i^+-p_i^- = t$, where $t$ is a positive real number that varies freely.  $D_\pm$ is our shorthand for the linearization of the Cauchy Riemann operator, which implicitly also depends on the complex structure of the domain.

We next do a much more careful pregluing. The main difficulty is for fixed $\dt>0$, suppose we start at $x_i$, and connect to a lift of a gradient trajectory that flows for distance $T$ in the $x$ direction, if we used $(s,t)$ coordinates on this gradient flow cylinder, the $s$ coordinate has range $s\in [0,T/\dt]$ which is independent of $i$. But the $a$ distance (i.e. distance in the symplectization direction) traveled by this gradient trajectory for the same $s$ from $0$ to $T/\dt$ varies depending on $i$, fundamentally this is because the $a$ coordinate satisfies the ODE
\[
a'(s) = e^{\delta f(x(s))}
\]
which depends on the value of $f$. Hence in the pregluing, instead of using the vector field $\Delta_t$ where $p_i^+-p_i^-=t$, there would be some nonlinear relations between the asymptotic vectors $\p_s$ and $\p_x$. \\
Fix cylindrical coordinates $(s_i^+,t_i^+)$ around each of the punctures of $u_+$ that hits the intermediate cascade level (i.e. the Morse Bott torus) and likewise $(s_i^-,t_i^-)$ for punctures of $u_-$. Near each of the punctures the maps $u_\pm$ takes the form
\[
(a_{i\pm}(s_i^\pm,t_i^\pm),z_{i\pm}(s_i^\pm,t_i^\pm),x_{i\pm}(s_i^\pm,t_i^\pm),y_{i\pm}(s_i^\pm,t_i^\pm)).
\]
We use $\pi_*$ for $*=a,z,x,y$ to denote the relevant component of a map, i.e. $(\pi_a u_+)_i$ denotes the $a$ component of $u_+$ at its $i$th end. We use the following notation to denote the various evaluation maps
\begin{align*}
ev^-_i(a)(R) := \int_{S^1}\pi_a(u^+)_i(-R,t)dt,& \quad  ev^+_i(a)(R) :=\int_{S^1}\pi_a(u^-)_i(R,t)dt
\end{align*}

\begin{align*}
ev^-_i(x)(R) := \int_{S^1}\pi_x(u^+)_i(-R,t) &,\quad ev^+_i(x)(R) :=\int_{S^1}\pi_x(u^+)_i(-R,t)(R,t)dt.
\end{align*}
We observe the deformation with respect to the asymptotically constant vector field $\p_r$ is constructed the same way as before, so we focus our attention on the vector spaces $V_+(x)\oplus V_-(x) \oplus V_+(a)\oplus V_-(a)$ consisting of the tuples $(p_{i}^+,p_{i}^-,a_{i}^+,a_{i}^-)$. \\
Let $T'>0$. Consider the submanifold $\hat{\Delta}$ in $V_+(x)\oplus V_-(x) \oplus V_+(a)\oplus V_-(a)$ defined as follows
\[
p^+_1-p^-_1=T'
\]
\[
p^+_i-p^-_i=T' + f_i(a_i^\pm, a_{1}^\pm,p_i^-,R)
\]
\[
|a_{i}^{\pm}|, |p_{i}^{\pm}| < \ep
\]
where $f_i$ is defined as follows: let $v_{1p}$ denote the gradient trajectory connecting the $i=1$ ends between $u_+$ and $u_-$. We endow it with the following specification: its $a$ coordinate at $s=R$ starts at $ev^+_1(a)(R)+a_{i}^+$, and its $x$ coordinate at $s=R$ starts at $ev_1^+(x)(R)+p^-_1$. It follows the gradient flow for $s$ length $T'$. We then translate $u_+$ in the $a$ direction so that $ev_1^-(a)(R) +a_{1-}= \pi_a (v_{1p}(T'/\dt-R,t))$. Further we have $ev_1^-(x)(R) + p_1^+ = \pi_x(v_{1p}(T'/\dt-R,t))$.

Then for $a_i^\pm, i\geq 2$, we define $f_i$ to be the amount of displacement in the $x$ direction required so that a gradient flow of $s$-length $(T'+f_i(a_i^\pm, a_{1}^\pm,p_i^\pm,R))/\dt$ flows from 
$ev_i^-(a)(R) + a_i^-$ to $ev_i^+(a)(R) + a_i^+$ at the $i$ th end between $u_+ + p_i^+$ and $u_- + p_i^-$. By $s$-length we mean for a finite segment of gradient cylinder, after having chosen coordinates $(s,t)$ on the gradient cylinder, the amount by which $s$ needs to change to go from one end of the gradient cylinder to the other end. We see immediately that 
\[
f_i \leq C\dt
\]
and 
\[
\frac{\p f_i}{\p a_{j}^{\pm}} \leq C\dt \,\, \text{where} \,\, j =1,i
\]
\[
\frac{\p f_i}{\p p_i^-} \leq C \dt.
\]
From this it follows immediately that $\hat{\Delta}$ is a submanifold. And that for small enough $\dt$ the operator
\begin{align*}
D_+\oplus D_-: &W^{2,p,d}(u_+^*TM)\oplus T\mcal{J}_+\oplus V_+' \oplus V_+(r)'' \oplus W^{2,p,d}(u_-^*TM)\oplus T\mcal{J}_-\oplus V_-' \oplus V_-''(r) \oplus (\hat{\Delta}) \longrightarrow \\
&W^{1,p,d}(\overline{\op{Hom}}(T\dot{\Sigma}, u_+^*TM))\oplus W^{1,p,d}(\overline{\op{Hom}}(T\dot{\Sigma}, u_-^*TM))
\end{align*}
is surjective with uniformly bounded right inverse. By $V''_\pm(r)$ we mean the subspace of $V''_\pm$ that only includes the $r$ components of the asymptotic vectors.
Then it follows immediately that any element in $\hat{\Delta}$ gives rise to a pregluing, since the $a$ and $x$ components of $u_\pm$ and the intermediate gradient trajectories match.

\begin{remark}
Our operator 
\begin{align*}
D_+\oplus D_-: & W^{2,p,d}(u_+^*TM)\oplus T\mcal{J}_+\oplus V_+' \oplus V_+(r)'' \oplus W^{1,p,d}(u_-^*TM)\oplus T\mcal{J}_-\oplus V_-' \oplus V_-''(r) \oplus (\hat{\Delta}) \longrightarrow \\
&W^{1,p,d}(\overline{Hom}(T\dot{\Sigma}, u_+^*TM))\oplus W^{1,p,d}(\overline{Hom}(T\dot{\Sigma}, u_-^*TM))
\end{align*}
has a two dimensional ``kernel". The kernel is in quotations because $\hat{\Delta}$ is a submanifold instead of a vector subspace, but as we have seen it is exceedingly close to a linear subspace, so we gloss over this point. The two dimensional ``kernel'' consists of two kinds of elements, they both come from the fact $u_\pm$ are $J$-holomorphic curves in symplectizations and hence there is a translation symmetry. The first kind of kernel element comes from translating $u_+$ and $u_-$ by the same amount in the sympletization direction. This is an genuine kernel element of $D_1\oplus D_2$. The other kernel element is translation $u_+$ and $u_-$ in opposite directions, so that they become closer/farther away from each other. This is no longer in the kernel because of the nonlinearities of $\hat{\Delta}$, but as we see the corrections are small. For the purposes of this section choosing a right inverse for $D_1\oplus D_2$ doesn't matter, since we only need to show a gluing exists. Later when we need to prove surjectivity of gluing we will choose specific right inverses for $D_1\oplus D_2$, which amounts to saying we consider vector fields where there are approximately no $\bb{R}$ translations over the curves $u_+$ and $u_-$.
\end{remark}
We can now state the gluing construction.
\begin{theorem}
2-level cascades of the form above can be glued. The gluing is unique up to choosing a right inverse for $D_+\oplus D_1$ when we restrict the allowed asymptotic vectors corresponding to ends that meet on the intermediate cascade level on the Morse-Bott torus to $\hat{\Delta}\oplus  V_+(r)''\oplus V_-''(r)$ as above.
\end{theorem}
\begin{proof}
Given a tuple of elements $(a_i^\pm, p_i^\pm)\in \hat{\Delta}$, as well as $r_i^\pm \in V_+(r)''\oplus V_-(r)''$ as twist parameters, we can define a preglued curve $u_*$ by pregluing $n$ gradient trajectories $v_i$ between $u_-$ and a translated $u_+$. Then, just as how we proved gluing for two curves with a single end meeting at intermediate cascade level, we deform the pregluing with appropriate vector fields, i.e. starting with vector fields $\psi_\pm \in W^{2,p,d}(u_\pm^*TM)$ and $\phi_i\in W^{2,p,w_i}(v_i^*TM)$. We also implicitly deform the domain complex structures of $u_\pm$ using $\dt j_\pm$; we also deform using asymptotic vectors at other ends in $u_\pm$, they live in $V'_\pm$ and we denote them by $\p_\pm'$; since they are not super relevant to our construction we suppress them from our notation. We construct the perturbation
\[
\beta_+ \psi_+ + \beta_-\psi_- + \sum \beta_{v_i} \phi_i.
\]
And as before, the deformation is holormophic iff the system of equations can be solved:
\begin{align*}
& \Theta_+ (\psi_+,(r,a,p)_{\pm i},\p_+',\dt j_+)=0\\
&\Theta_- (\psi_-,(r,a,p)_{\pm i},\p_-',\dt j_-)=0\\
&\Theta_i(\psi_\pm, \phi_i)=0.
\end{align*}
Then we follow the same strategy of proof as before, given the tuples $(\psi_\pm,(r,a,p)_{i\pm},\p_\pm', \dt j_\pm)$ of input along $u_\pm$ we can define subspaces $H_{0i} \subset W^{2,p,w_i}(v_i^*TM)$ such that there exists unique solution to $ \mathbf{\Theta}_i=0$, $\phi_i\in H_{0i}$. It follows immediately from previous theorems that $\phi_i$ has norm bounded above by $C\ep/R$ and is nicely behaved with respect to variations of all input data $(\psi_\pm,(r,a,p)_{i\pm},\p_\pm', \dt j_\pm)$. We then view the system $\Theta_\pm=0$ as looking for a zero of a map
\begin{align*}
&W^{2,p,d}(u_+^*TM)\oplus T\mcal{J}_+\oplus V_+' \oplus V_+(r)'' \oplus W^{2,p,d}(u_-^*TM)\oplus T\mcal{J}_-\oplus V_-' \oplus V_-''(r) \oplus (\hat{\Delta}) \longrightarrow \\
&W^{1,p,d}(\overline{\op{Hom}}(T\dot{\Sigma}, u_+^*TM))\oplus W^{1,p,d}(\overline{\op{Hom}}(T\dot{\Sigma}, u_-^*TM)).
\end{align*}

It follows from our previous calculations of how $\Theta_\pm$ looks like in these coordinates,
as well as the fact the operator $D_+\oplus D_-$ restricted to $ W^{2,p,d}(u_+^*TM)\oplus T\mcal{J}_+\oplus V_+' \oplus V_+(a)'' \oplus W^{2,p,d}(u_-^*TM)\oplus T\mcal{J}_-\oplus V_-' \oplus V_-''(a) \oplus (\hat{\Delta})$ being surjective that $\Theta_\pm$ can be solved simultaneously for $(\psi_\pm,(r,a,p)_{i\pm}, \p_\pm', \dt j_\pm)$ via the contraction mapping principle. Such a solution is unique provided we fix a right inverse for $D_+\oplus D_-$.
\end{proof}
We now turn to gluing $n$-level cascades. It will follow the same strategy as above as long as we introduce some new notations so we will be brief. The main purpose of the ensuing proof is to introduce some useful notations.
\begin{theorem}
$n$-level tranverse and rigid cascades can be glued, and the solutions are unique up to choosing a right inverse, as specified in the proof.
\end{theorem}
\begin{proof}
Let $\cas{u}= \{u^i\}_{i=1,..,n}$ be an $n$ level cascade that is transverse and rigid. For each $u^i$ we let $W_i$ denote the vector space $W^{2,p,d}(u_i^*TM)\oplus T\mcal{J}_+\oplus V_+' \oplus V_+(r)'' $ and $L_i$ the vector space $W^{1,p,d}(\overline{\op{Hom}}(T\dot{\Sigma}, u^{i*}TM))$ and let $\hat{\Delta}_{i,i+1}$ denote the submanifold consisting of asymptotic vectors in $a,x$ directions corresponding to free ends that meet each other between $u^i$ and $u^{i+1}$, analogous to $\hat{\Delta}$ for the 2 level case, so that pregluing makes sense. Then the fact that the cascade exists, is transversely cut out, and of Fredholm index 0 implies the operator
\[
\oplus D_i: W_1\oplus \hat{\Delta}_{1,2}\oplus\ldots\oplus W_n \longrightarrow L_1\oplus ..\oplus L_n
\]
is surjective with uniformly bounded right inverse. Hence for each element in $\hat{\Delta}_{i,i+1}$ we preglue together $u^i$ and $u^{i+1}$ by inserting a collection of gradient trajectories in the middle. In case $u^i$ and $u^{i+1}$ have components consisting of trivial cylinders that begin and end on critical points, we recall such chains of trivial cylinders will eventually meet a non-trivial $J$ holomorphic curve with fixed end at the critical point. We replace such chains of trivial cylinders with a single fixed trivial cylinder as in the case of gluing fixed trivial cylinders in the case of semi-infinite trajectories. We add marked points to unstable components in cascade levels to make them stable, see Convention \ref{stabilization}. For the positive ends of $u^1$ and negative ends of $u^n$, if it a free end we glue in a semi-infinite gradient trajectory, and if it is a fixed end we glue in a trivial cylinder. This constructs for us a preglued curve $u_*$. Then we deform this preglued curve using vector fields $\psi_i$ over $u^i$ and $\phi_i$ over the gradient flow lines we preglued. We require that $\phi_i$ lives in the vector space $H_{0i}$, which is defined analogously to $H_0$ in the case of 2 level cascades, if $\phi_i$ corresponds to a finite gradient flow trajectory, and no such requirement is imposed if $\phi_i$ is over a semi-infinite gradient trajectory. As before the entire preglued curve can be deformed to be holomorphic iff the system of equations
\[
\mathbf{\Theta_i}(\psi_i, (r,a,p)_{\pm i},\phi_j,\p_i',\dt j_i)=0, \quad
\mathbf{\Theta}_{v_i}(\phi_i, \psi_j)=0
\]
can be solved. We use bold to denote a system of equations. $\mathbf{\Theta_i}(\psi_i, (r,a,p)_{\pm i},\phi_j,\p_i',\dt j_i)$ corresponds to equations over $u^i$, and $\mathbf{\Theta}_{v_i}$ corresponds to equations over gradient flow trajectories, which implicitly includes semi-infinite trajectories. As before for fixed epsilon ball in $\oplus W_i \oplus \hat{\Delta_{i,i+1}}$, the equations $\Theta_{v_i}$ have unique solutions in $H_{0i}$ that are well behaved w.r.t. input. Then the equations $\Theta_i(\psi_i, (r,a,p)_{\pm i},\phi_j,\p'_i,\dt j_i)=0$ have unique solutions follow from the fact $\oplus D_i$ is surjective with uniformly bounded inverse and the contraction mapping principle. The solution is unique to a choice of right inverse for the operator $\oplus D_i$.
\end{proof}
\begin{remark}
We note here by elliptic regularity all of our solutions are smooth, with their higher $W^{k,p}$ norms bounded by their $W^{1,p}$ norm.
\end{remark}
\begin{remark}
We note by the additivity of the relative first Chern class and the Euler characteristic, the resulting glued curve has Fredholm index one.
\end{remark}
\section{Behaviour of holomorphic curve near Morse-Bott tori} \label{behav}
In this section we prove a series of results concerning how $J$-holomorphic curves behave near Morse-Bott tori. This is part of the analysis that is needed to prove the degeneration result from $J$-holmoprhic curves to cascades in Bourgeois' thesis \cite{BourPhd}. We redo this part of the analysis, not only to prove the degeneration result in our case in the Appendix, but we will also need them to later show that the gluing we construct is surjective. The analysis here is very similar to the analysis performed in the Appendix of Bourgeois and Oancea's paper \cite{oancea}, the only major difference is we are working in symplectizations where they work near a Hamiltonian orbit.
We start with a series of analytical lemmas.
\subsection{Semi-infinte ends}

Recall the neighborhood of Morse-Bott torus we have coordinates
$(z,x,y) \in S^1\times S^1 \times \mathbb{R}$, with $J$ chosen so that at the surface of the Morse-Bott torus $J\p_x =\p_y$. The linearized Cauchy Riemann operator along trivial cylinders that land on this Morse-Bott torus takes the form $\p_s + A$, where 
\[
-A: = J_0 (d/dt)+ S_0(x,z)
\]
$S(x,z)$ is a symmetric matrix that only depends on $x$ and $z$. The kernel of $A(s)$ is spanned by $\p_a,\p_z,\p_x$.
Let $P$ denote the $L^2$ projection to its kernel, and let $Q$ denote the projection $\text{ker} A ^\perp$. 
\begin{theorem} \label{inf conv}
Let $u_\delta(s,t)= (a(s,t),z(s,t),x(s,t),y(s,t))$ be a $J_\delta$-holomorphic map that converges to a simply covered Reeb orbit corresponding to a critical point of $f$ as $s\rightarrow \infty$. We also assume for $s>0$, the map $u_\dt$ stays away from all other Reeb orbits corresponding to other critical points of $f$, uniformly as $\dt \rightarrow 0$. Assume for $s>0$ we have
\[
|y|,|z-t|, |\partial^{\leq k }_* x|, |\partial^{\leq k}_* y| \leq \epsilon
\]
where $*=s,t$, and $\ep>0$ is sufficiently small (but independent of $\dt$). We also assume all other derivatives are uniformly bounded above by $C$. There is some $r>0$ independent of $\dt$ and only depending on the local geometry of Morse-Bott tori so that
\begin{equation*}
    |y|,|z-t+c| \leq C\|Q(Y)(0,t)\|^{2/p}_{L^2(S^1)}e^{-rs}
\end{equation*}
\[
|x - x_p(s)| \leq C\|Q(Y)(0,t)\|^{2/p}_{L^2(S^1)} e^{-rs}
\]
\[
\left|a(s,t) - c - \int_{s_0}^s e^{\delta f(x_p(s'))}ds'\right| \leq C\|Q(Y)(0,t)\|^{2/p}_{L^2(S^1)}e^{-rs}
\]
where inheriting previous notation, we use $x_p(s)$ to denote a gradient trajectory of $\dt f(x)$, the definition of $Y$ is given in the proof. Further, inequalities of the above form continue to hold after we differentiate both sides with respect to $(s,t)$, in other words the inequalities hold in the $C^k$ norm.
\end{theorem}
\begin{proof}
In the course of this proof we first perform some important calculations which we will later reuse for decay estimates over finite gradient trajectories. 

\textbf{Step 0}
In our coordinates system the equation looks like (we will drop the $\delta$ subscript from $u$)
\[
\partial_s u + J_\delta(u) \partial_t u =0.
\]
Following the Appendix of \cite{oancea}, let's change variables
\[ 
Y:= (w,v,x,y)
\]
where $w:=a(s,t)-s$ and $v:=z-t$.
Then the equation changes to
\[
\partial_s (Y)+ J_\delta(u) \partial_t Y + (\p_s+J_\delta(u)\partial_t) =0.
\]
We simplify this as
\begin{align*}
    &\partial_s Y  + J(u) \partial_t (Y) + \delta J(u) (\partial_t Y)+ (\partial_s+ J_\delta(u) \partial_t)\\
     =& \partial_s Y + J_0 \partial_t (Y) + S_1(x,y,z)(\partial_t Y) + (\partial_s+ J_\delta(u) \partial_t) +\delta J \partial_t Y\\
    =& \partial_s Y + J_0 \partial_t (Y) + S_1(x,y,z)(\partial_t Y) + (\partial_s+ J(u) \partial_t) +\delta J(u) \partial_t +\delta J \partial_t Y\\
    =&\partial_s Y + J_0 \partial_t (Y) + S_1(x,y,z)(\partial_t Y + \partial_t) +\delta J(u) \partial_t +\delta J \partial_t Y\\
    =&\partial_s Y + J_0 \partial_t (Y)  +S_0(x,z)y \partial_t\\
    &+ S_1(x,y,z)(\partial_t Y) + S_2(x,y,z)\partial_t+ \delta J(u) \partial_t +\delta J \partial_t Y
\end{align*}
We clarify $J(u)$ is the Morse-Bott complex structure evaluated at $u$, and $J_0$ is the standard complex structure, which coincides with the Morse-Bott almost complex structure on the surface of the Morse-Bott torus. $\dt J := J_\dt -J$ is the difference between the Morse-Bott complex structure and the perturbed almost complex structure, and as a matrix is has norm bounded above by (its derivatives are also bounded above by) the expression $C\dt$. The term $S_0$ is a 4 by 4 matrix coming from the linearization of $\db_{J}$ on the surface of the Morse-Bott torus, and hence it only depends on $x,z$.

In the above expansion, we have the estimates
\[
S_1(x,y,z) \leq C(x,y,z) |y|
\]
and 
\[
S_2(x,y,z) \leq C(x,y,z) y^2
\]
We implicitly assume we have taken absolute values of both sides. Similar expressions hold for their derivatives (lowering orders in $y$ as we differentiate).\\
Next consider the term
\begin{align*}
&\delta J(a,z,x,y) \partial_t u \\
=& \delta J(a,z,x,y) \partial_t Y+ \partial_t\\
=& \delta J(u,v,x,0) \partial_t u + \delta T(v,x,y) y \partial_t u\\
=& \delta J(u,v,x,0) (\partial_t Y+\partial_t) + \delta T(v,x,y) y (\partial_t Y + \partial_t)
\end{align*}
where $\dt T$ is some matrix whose $C^k$ norm is bounded above by $C\dt $. 
We further observe $\delta J(x,0,v)$ doesn't depend on $v$ since at the surface of Morse-Bott torus it is rotationally symmetric.

We further examine 
\begin{align*}
    & \delta J(u,v,x,0) (\partial_t Y+\partial_t) \\
     =&\delta J(u,0,x,0) \partial_t (a,t+v,x,y)\\
     = &
\begin{pmatrix}
(1-e^{\delta f}) \partial_t (t+v)\\
-\frac{e^{\delta f}-1}{e^{\delta f}} \partial_t a \\
-\delta f'(x) \partial_t (t+v)\\
-\delta f'(x)\partial_t a
\end{pmatrix}
\end{align*}
where we used the fact $J$ restricted to the surface of the Morse-Bott torus is invariant in the $(x,y)$ direction.\\
Now we recall our assumptions about the form of $z(s,t)$. In particular we assume $z(t) = t + v$ with $|v| \leq \epsilon$ (this can always be achieved via a reparametrization of the neighborhood around the puncture), so if we plug that in to the above expression it is equal to:
\[
\begin{pmatrix}
(1-e^{\delta f})\partial_t v \\
-\frac{e^{\delta f}-1}{e^{\delta f}} \partial_t a \\
\delta f'(x)  \partial_t v\\
-\delta f'(x) \partial_t a
\end{pmatrix}
+ 
\begin{pmatrix}
1-e^{\delta f (x(s,t))}\\
0\\
-\partial_x \delta f(x(s,t))\\
0
\end{pmatrix}.
\]
Having performed these computations we return to the overall equation of the form
\begin{align*}
    &\partial_s Y + J_0 \partial_t (Y)  +S_0(x,z)y\\
    &+ S_1(x,y,z)(\partial_t Y) + S_2(x,y,z)\partial_t+ \delta J(x) \partial_t +\delta T(x,y,z)y (\partial_tY +\partial_t) +\delta J (u) \partial_t Y.
\end{align*}
For later elliptic regularity purposes it will be useful to write the above in the following form:
\begin{align*}
    \p_s Y + J_\dt(u) \p_t Y +\dt Ty(\p_tY +\p_t) + S_1 \p_t Y + \dt J( \p_t)=0 .
\end{align*}
\textbf{Step 1}
As before consider the operator 
\[
-A(s): = J_0 (d/dt)+ S_0(x,z).
\]
Note this operator as it appears in the above equation depends on the $x(s,t),z(s,t)$ coordinates of $u$, but we observe it remains true there exists a $\lambda$ so that for all functions $h(t)\in W^{1,2}(S^1)$, 
\[
\la A h, Ah \ra_{L^2(S^1)} \geq \lambda ^2 \la h,h\ra_{L^2(S^1)}.
\]
As a matter of bookkeeping we observe our vector field $Y$ is smooth, hence $Y$ has a well defined restriction to $\{s\} \times S^1$ for any value of $s$ .\\
We define 
\[ g(s) = \langle Q Y, Q Y \rangle_{L^2(S^1)}
\]
as before for our decay estimates we compute
\begin{align*}
    g''(s) = 2\langle \partial_s Q Y, \partial_s Q Y \rangle + 2\langle QY, \partial^2_s QY \rangle.
\end{align*}
We observe both $Q$ and $P$ commute with $\partial_*, *=s,t$.\\ 
\textbf{Step 2} Examining the first term above
\begin{align*}
    &\langle \partial_s Q Y, \partial_s Q Y \rangle\\
    =& \|Q (AY +S_1(x,y,z)(\partial_t Y) + S_2(x,y,z)\partial_t+ \delta J(x) \partial_t\\
    &+\delta T(x,y,z)y (\partial_tY +\partial_t) +\delta J (u) \partial_t Y)\|^2.
\end{align*}
Let's dissect these terms one by one.
First since $Q$ commutes with $A$ we have
\[
\|AQY\|^2 \geq \lambda\|QY\|^2
\]
for some $\lambda$ independent of $s$. \\
We next consider
\[
Q\delta J (x(s,t),z(s,t)) \partial_t 
\]
which warrants special treatment. For fixed $s$ we denote by $\bar{x}$ the average value of $x(s,t)$ over $t$.\\
Then we can write terms like 
\[ 
f(x(s,t)) = f(x(s,t) -\bar{x} +\bar{x}) = f(\bar{x}) + G_x(x-\bar{x}) \leq f(\bar{x}) + G_x(QY)
\]
and for $|Y|_{C^0} \leq C\ep$ we have $G_x (x) \leq C|x|$. 
Therefore we observe $Qf(\bar{x})=0$ and hence we have the estimate
\[
Qf(x(s,t)) \leq CQY.
\]
The same also applies to other functions built out of $f(x)$, hence we have
\[
\|Q\delta J(x(s,t))\partial_t\|\leq C\delta \|QY\|.
\]
Here we also note that the equation satisfied by $QY$ is of the form
\begin{equation} \label{Q-elliptic_reg}
\p_s QY + J_\dt(u) \p_t QY +\dt T(x,y,z) QY \cdot (\p_t QY +\p_t) + S_1 \p_t Y + \dt C(x,y,z) QY=0 
\end{equation}
where $C(x,y,z)$ is just a function of $x,y,z$ whose derivatives are uniformly bounded.

Aside from the two terms we calculated above, applying $Q$ to $Y$ does not have a major impact on other terms. To consider the rest of the terms appearing in $  \langle \partial_s Q Y, \partial_s Q Y \rangle$ let's estimate their norms (since later we can just use the triangle inequality to either estimate their cross term with themselves or with terms involving $\|AQY\|^2$).\\
The norms of the terms below after we apply $Q$
\[
S_1(x,y,z)(\partial_t Y),\quad S_2(x,y,z)\partial_t, \quad \delta J(x) \partial_t, \quad
\delta T(x,y,z)y (\partial_tY +\partial_t), \quad \delta J (u) \partial_t Y)
\]
are respectively bounded by the norms
\begin{align*}
&\epsilon^2(\|QAY\|^2+\|QY\|^2),\quad, \epsilon^2\|QY\|^2,\delta^2(\|QY\|^2),\quad 
\delta^2 \epsilon^2(\|AQY\|^2+\|QY\|^2)+\delta^2\|QY\|^2,\\
&\delta^2(\|AQY\|^2+\|QY\|^2).
\end{align*}
The key observation is $\partial_*^*y \leq \epsilon$ as part of our assumption, as well as the fact all occurrences of $y$ are upper bounded by $QY$. Another key observation is $\partial_t= \partial_t Q$, so every time we see $\partial_t Y$ we replace it by $\partial_t Q$ hence the appearance of the many $Q$ in the above expression.

\textbf{Step 3}
We look at the next term
\begin{align*}
    &\langle QY, \partial^2_s QY \rangle\\
    =&\la QY, Q \partial_s (AY+S_1(x,y,z)(\partial_t Y) + S_2(x,y,z)\partial_t+ \delta J(x) \partial_t\\
    &+\delta T(x,y,z)y (\partial_tY +\partial_t) +\delta J (u) \partial_t Y)\ra.
\end{align*}
Note for terms we think of being small, we are not careful about their signs. We introduce some more convenient notation. We write the $J$-holomorphic curve equation as 
\[
\partial_sY -AY +E(Y)=0
\]
Then 
we have
\begin{align*}
    &\la QY,Q\partial_s (AY+E(Y))\ra\\
    =& \la QY, \epsilon QY\ra + \la QY, Q[A(AY+E)+ \partial_sE]\ra\\
     =& \la QY, \epsilon QY\ra + \la QAY, QAY\ra + \la QAY,QE\ra  +\la QY, Q\partial_sE\ra.
\end{align*}
To obtain the first term in the above expression we used the fact that
\[
\partial_s A
\]
is a 4 by 4 matrix whose only nonzero entry is the diagonal entry corresponding to $y$,
so
\[
Q(\partial_sAY) = \epsilon y.
\]
The only term we don't know how to control is the last one $\la QY, Q\partial_sE\ra$, the previous ones follow from computation in previous steps. Let's recall the terms in $E$:
\[
S_1(x,y,z)(\partial_t Y), \quad S_2(x,y,z)\partial_t, \quad \delta J(x) \partial_t, \quad
\delta T(x,y,z)y (\partial_tY +\partial_t), \delta J (u) \partial_t Y.
\]
We need to compute the $L^2(S^1)$ norm of these terms after we take their $s$ derivative. We first only consider the $s$ derivatives on $S_1,\delta T, \delta J (u)$, by assumption that $\partial^{k}_*y\leq \epsilon$, when we take the $s$ derivatives of $S_1,\delta T, \delta J (u)$, they are still operators of the same form. For example $\p_s S_1$ is of the form $C_1(x,y,z) y + C_2(x,y,z)y_s$, and the norm of each term can be bounded above by $\ep$. The same can be said about $\p_s\dt T$, $\p_s \dt J(u)$, so by abuse of notation we use the same symbols. Then techniques from previous steps immediately show the norm of these terms are upper bounded by
\[
\{\partial_s(S_1) \partial_tY, \quad \partial_s(\delta Ty) \partial_tY, \quad \partial_s (\delta J) \partial_t Y\} \leq C(\ep+\dt)\|Q\p_tY\|^2 \leq C(\ep+\dt)(\|AQY\|^2 + \|QY\|^2).
\]
We next consider the $s$ derivative of $Q \delta J(z,x) \partial_t$. We first observe $Q$ commutes with $\p_s$ so we are evaluating $Q\p_s \delta J(z,x) \partial_t$. Recalling the previous form of this vector field, the components are essentially built out of $f(x(s,t))$, so we need to take its $s$ derivative and projection via $Q$. Using the previous trick of introducing $\bar{x}$
\begin{align*}
&\partial_s f(x(s,t))\\
 =&\partial_s f (x(s,t)-\bar{x}(s) +\bar{x}(s))\\
= &f_x(x(s,t)-\bar{x}(s) +\bar{x}(s))(x_s(s,t) - \bar{x}_s + \bar{x}_s)\\
= &f_x(x)(x_s(s,t) - \bar{x}_s) + \bar{x}_sf_x(x(s,t)-\bar{x}(s) +\bar{x}(s))\\
= &f _x(x)(Qx_s) + \bar{x}_s [f_x (\bar{x}_s) + G_x (Qx)].
\end{align*}
Observe $Q(\bar{x}_s f_x(\bar{x}_s))=0$ because this term doesn't depend on $t$. Hence pointwise
\[
Q\partial_s f(x_s(t)) \leq C |QY| + |QY_s|.
\]
Hence:
\[
\|\partial_s Q \delta J \partial_t \|_{L^2(S^1)} \leq C \dt (\|QY\|_{L^2} + \|QY_s\|_{L^2})
\]
and we have seen above how to bound the norm of $\|QY_s\|_{L^2}$.
Next:
\[
\partial_s S_2 \partial_t =C yy_s \partial_t.
\]
We assumed $y_s \leq \epsilon$ this term can be upper bounded by
\[
C\epsilon QY .
\]
Finally we turn our attention to terms of the form
\[
\epsilon \partial_s \partial_t Y
\]
which appear once the $s$ derivative hits $QY$. Here $\ep$ denotes a matrix whose $C^{k-1}$ norm is uniformly upper bounded by the real number $\ep$. We can insert a factor of $Q$ after the $t$ derivative and get
\begin{align*}
    &\la QY, \epsilon \partial_s\partial_t QY \ra \\
    =&\la \epsilon^T Q\partial_t Y,\partial_s QY \ra+ \la \epsilon^T_t QY, \partial_s QY \ra\\
    \leq& \epsilon(\|Q\partial_tY\|^2+ \|\partial_s QY\|^2 +\|QY\|^2).
\end{align*}
The terms in the last line are already well understood by previous computations. In particular $\|\partial_s QY\|^2$ was worked out in the previous step and $\|Q\partial_tY\|^2$ was worked out in this step. Hence putting all of these terms together we have 
\[
g''(s) \geq (4\lambda^2 -C\ep) g(s)
\]
hence from previous lemma we have $g(s) \leq g(0)e^{-\lambda s}$, hence the $L^2$ norm of $QY$ undergoes exponential decay. That this extends to pointwise $C^k$ norm follows from elliptic regularity, using equation \ref{Q-elliptic_reg}.

\textbf{Step 4}
In this step we look at what equation $PY$ satisfies. Let's recall the original equation
\[
\partial_s (Y)+\partial_s + J_\delta(u) \partial_t Y + J_\delta(u)\partial_t =0.
\]
We split $Y=QY+PY$ and plug into the above equation to get the pointwise bound:
\[
|\partial_sPY +\delta J \partial_t | \leq C\|Q(Y)(0,t)\|^{2/p}_{L^2(S^1)} e^{-\lambda s}.
\]
where we used the previous bound on the norm of $QY$. We remind ourselves $\lambda$ might change from previous parts because of various changes in norm.
We can replace $f(x)$ with $f(Px)$ because $\partial_t x$ is bounded by $\p_t QY$, which decays exponentially, so we can take the error term to the right hand side to get
\[
|\partial_sPY +\delta J(PY) \partial_t | \leq C\|Q(Y)(0,t)\|^{2/p}_{L^2(S^1)} e^{-\lambda s}.
\]
Observe that the function $PY$ only depends on $s$, and the above are pointwise inequalities.  Differentiability of $PY$ comes from bootstrapping and observing the differentiability of $QY$ in the $s$ variable. The decay estimates of the higher order $s$ derivatives follow as well.
We let $PY_*$, where $*=a,z,x,y$ denote the various components of $PY$. The equations in these coordinates are
\[ 
PY_y=0
\]
\[ 
|\partial_s PY_z| \leq C\|Q(Y)(0,t)\|^{2/p}_{L^2(S^1)} e^{-\lambda s}
\]
\[
|\partial_s PY_x - \delta f'(PY_x)| \leq C\|Q(Y)(0,t)\|^{2/p}_{L^2(S^1)}e^{-\lambda s}
\]
\[
|\p_s PY_a - e^{\dt f(PY_x)}|\leq C\|Q(Y)(0,t)\|^{2/p}_{L^2(S^1)}e^{-\lambda s}.
\]

We now solve the above inequalities. For brevity we denote by $G(s)_*$ the expression on the right hand side for $*=z,x,a$, and the only property we will need about $G(s)$ is that it is asymptotically of the form $e^{-\lambda s}$.
The inequality 
\[
|\partial_s PY_z| \leq C\|Q(Y)(0,t)\|^{2/p}_{L^2(S^1)}e^{-\lambda s}
\]
integrates to 
\[
|PY_z-c| = \left|\int_0^s G_z(s') ds' \right| \leq C\|Q(Y)(0,t)\|^{2/p}_{L^2(S^1)}e^{-\lambda s}.
\]
Next $|\p_s PY_x - \dt f'(PY_x)|\leq C\|Q(Y)(0,t)\|^{2/p}_{L^2(S^1)}e^{-\lambda s}$, we pick a coordinate neighborhood so that $f(x) = \mp \frac{1}{2}x^2+C$. We can do this because we know $u$ eventually limits to a critical point of $f$ as $s\rightarrow \infty$ and stays away from all other critical points of $f$, the choice of $\mp$ corresponds to whether we are in a neighborhood of maximum or minimum of $f$. Then this is an equation of the form
\[
\partial_s PY_x \pm \dt PY_x =G(s)_x.
\]
Then we have
\[
(PY_x e^{\pm \delta s})_s = G_x(s) e^{\pm \delta s}.
\]
We can write
\[
PY_x=c(s) e^{\mp \delta  s}
\]
where $c(s)$ satisfies the equation
\[
\frac{d}{ds}c(s) = G(s)_x e^{\pm \delta s}.
\]
Since we known $G(s)_x$ decays quickly when $s \rightarrow \infty$, the function $c$ must have a limit as $s\rightarrow \infty$, call this limit $c_\infty$. Then we have
\[
c(s) = c_\infty + \int_s^\infty G_x(t)_x e^{\pm \delta t}dt
\]
hence 
\[
PY_x(s) = c_\infty e^{\mp \delta  s} + e^{\mp \delta s} \int^\infty_s G_x (t) e^{\pm\delta t} dt.
\]
We recognize $c_\infty e^{\mp\delta s}$ as the gradient flow $x_p(s)$ we identified earlier and $e^{\mp\delta s} \int^\infty_s G(s)_x e^{\pm\delta t}$ is considered the error term, and by the form of $G_x$ the error term has the decay we needed.

We note in the case $f=-\frac{1}{2}x^2+C$ the gradient flow converges to zero, and this corresponds to ``free'' ends converging to the maximum of $f$ on positive punctures. In the case where $f=+\frac{1}{2}x^2 +C$, the gradient flow segment $c_\infty e^{\dt s}$, if we have $c_\infty \neq 0$, will actually flow away from the critical point $x=0$, so it will eventually leave the neighborhood where the expression $f(x) = \frac{1}{2}x^2+C$ is valid, and instead flow to the other critical point/maximum of $f$, for which we can use the above analysis directly. The exception is if $c_\infty=0$, and this end will converge to the $x=0$, or the minimum of $f$. This corresponds to the case of a ``fixed'' end converging to the minimum of $f$. Implicit in the above discussion is the assumption that $u_\dt$ stays away from all except one critical point of $f$ uniformly as $\dt \rightarrow 0$. This, in the language of our equations, means $c_\infty(\dt)$ (this constant implicitly depends on $\dt$), is either bounded away from zero for all $\dt$ small enough, or is identically zero for $\dt$ small enough. These correspond respectively to the above two cases. The case where $c_\infty(\dt) \rightarrow 0$ and $c_\infty(\dt)\neq 0$ as $\dt \rightarrow 0$ corresponds to the $J_\dt$-holomorphic curves $u_\dt$ breaking into a cascade of height $>1$, and is outside the scope of our discussion.

Finally we consider the equation
\[
\p_s PY_a - e^{\dt f(Px)} =G_a(s).
\]
Now by the above estimate on $P(x)$, there is a gradient trajectory $v$ whose $x$ component, $\pi_x v$ is approximated by $PY_x$, in the sense that
\[
|PY_x-\pi_xv|\leq C\|Q(Y)(0,t)\|^{2/p}_{L^2(S^1)}e^{-\lambda s}.
\]
Then for small enough $\ep$, we have the estimate
\[
|e^{\dt f(PY_x)} - e^{\dt f(\pi_xv)}| \leq C\dt\|Q(Y)(0,t)\|^{2/p}_{L^2(S^1)}e^{-\lambda s}
\]
hence we can write 
\[
\p_s P_a - e^{\dt f(\pi_x v)} =G_a(s)
\]
where we absorbed the error term 
$C\dt\|Q(Y)(0,t)\|^{2/p}_{L^2(S^1)}e^{-\lambda s}$ into $G_a(s)$ since they are of the same form. Then we integrate both sides to get:
\[
P_a(s) - \int_0^se^{\dt f(\pi_xv)} = \int_0^s G_a(s')ds' .
\]
Using the same trick as before we write $\int_0^s G_a(s')ds' = c_\infty - \int_s^\infty G_a(s')ds' $, recognizing $c_\infty+\int_0^se^{\dt f(\pi_xv)} $ is the $a$ component of a lift of a gradient trajectory, we arrive at the desired bound.
\end{proof}
\begin{remark}
In the above proof and what follows we assume that $u_\dt$ stays uniformly away from all but one critical point of $f$. The estimates for $Q(Y)$ is largely unaffected by this assumption, the main reason we use this is so that we could have nice exponential decay estimates for $PY$ (this is where we used local form of $f$). In general (and in manifolds where the critical Morse-Bott manifolds are higher dimensional) we could have $u_\dt$ degenerate into a broken trajectory of $f$ along the critical set, and the estimate there is more involved. Fortunately our transverse rigid constraint means our assumption about $u_\dt$ being away from except at most one critical point of $f$ will be sufficient for our purposes.
\end{remark}
\subsection{Finite gradient segments}
We now extend these exponential decay estimates to finite gradient trajectories.
\begin{theorem}\label{finite conv}
Consider an interval $I=[s_0,s_1]$ and a $J_\dt$-holomorphic curve $u$ so that when restricted to $s\in I$ the map $u$ is close to the Morse-Bott torus, i.e. in a neighborhood of the Morse-Bott torus $u$ has coordinates $(a,z,x,y)$ and the functions $a,z,x,y$ satisfy
\[
|y|,|\p_*^{\leq k}(z-t)|, |\partial^{\leq k }_* x|, |\partial^{\leq k}_* y| \leq \epsilon
\]
for some $\ep > 0$ depending only on the local geometry and independent of $\dt$, then
\[
\|QY\|_{C^{k-1}} \leq \max (\|QY(s_0,t)\|^{2/p}_{L^2(S^1)}, \|QY(s_1,t)\|^{2/p}_{L^2(S^1)}) \frac{\cosh(\lambda(s-(s_0+s_1)/2))}{\cosh(\lambda(s_1-s_0)/2)} 
\]
for some $\lambda>0$ only depending on the local geometry.\\
If $u$ is uniformly bounded away from all critical points of $f$ except maybe one, there is a lift of a gradient trajectory, which we denote by $v$, so that 
\[
\|PY-v\|_{C^{k-1}} \leq \max (\|QY(s_0,t)\|^{2/p}_{L^2(S^1)}, \|QY(s_1,t)\|^{2/p}_{L^2(S^1)}) \frac{\cosh(\lambda(s-(s_0+s_1)/2))}{\cosh(\lambda(s_1-s_0)/2)} .
\]
\end{theorem}
\begin{proof}
The proof will follow the general thread of the semi-infinite case. We recall our convention $\lambda$ may change from line to line, but not in a fashion that depends on $\dt$. Recall we defined the function
\[
g(s):=\la QY,QY\ra_{L^2(S^1)}
\]
then we have the inequality
\[
g''\geq \lambda ^2 g.
\]
We define the auxiliary function
\[
k(s):= \max(\|QY(s_0)\|_{L^2(S^1)}^2, \|QY(s_1)\|_{L^2(S^1)}^2)\frac{\cosh(\lambda(s-(s_0+s_1)/2))}{\cosh(\lambda(s_1-s_0))} 
\]
then we have the inequality
\[
(g-k)'' \geq \lambda^2 (g-k).
\]
Then $g-k$ cannot have positive maximum, and by construction $g-k \leq 0$ at $s=s_0,s_1$. Hence $g\leq k$ globally for $s\in I$.

With elliptic regularity as before, we obtain the pointwise bound
\[
|Q(Y)(s,t)|\leq k(s)^{1/p}
\]
which by elliptic regularity can be improved to bound the derivatives of $QY$.
Using the inequalities 
\[
c_1 \cosh(x/p) \leq {\cosh(x)}^{1/p} \leq c_2 \cosh(x/p).
\]
We then obtain inequalities:
\[
\|QY\|_{C^{k-1}} \leq \max (\|QY(s_0,t)\|^{2/p}_{L^2(S^1)}, \|QY(s_1,t)\|^{2/p}_{L^2(S^1)}) \frac{\cosh(\lambda(s-(s_0+s_1)/2))}{\cosh(\lambda(s_1-s_0)/2)} 
\]
where we have of course changed the definition of $\lambda$. We also have
\[
|\partial_s P Y- J_\delta \partial_t| \leq k_1
\]
where for brevity we have defined
\[
k_1 = max (\|QY(s_0,t)\|^{2/p}_{L^2{S^1}}, \|QY(s_1,t)\|^{2/p}_{L^2{S^1}}) \frac{cosh(\lambda(s-(s_0+s_1)/2))}{cosh(\lambda(s_1-s_0)/2)}.
\]
We try to integrate this inequality as before: $
|\partial_s PY - J_\delta \partial_t| \leq k_1 $. 
There are various components to this equation, which we examine one by one. For the easiest case we have:
\[
|\p_s PY_z| \leq k_1.
\]
Integrating both sides we get
\begin{align*}
|PY_z(s) - PY_z((s_0+s_1)/2)| &\leq \int_{(s_0+s_1)/2}^s k_1\\
&\leq C \frac{\max (\|QY(s_0,t)\|^{2/p}_{L^2{S^1}}, \|QY(s_1,t)\|^{2/p}_{L^2{S^1}}))}{\cosh(\lambda(s_1-s_0)/2)} \int_{(s_0+s_1)/2}^s \cosh ( \lambda (s'-(s_1+s_0)/2)ds'\\
& \leq C \frac{\max (\|QY(s_0,t)\|^{2/p}_{L^2{S^1}}, \|QY(s_1,t)\|^{2/p}_{L^2{S^1}}))}{\cosh(\lambda(s_1-s_0)/2)}|\sinh (\lambda (s-(s_0+s_1)/2))|\\
&\leq C  \max (\|QY(s_0,t)\|^{2/p}_{L^2{S^1}}, \|QY(s_1,t)\|^{2/p}_{L^2{S^1}}) \frac{\cosh(\lambda(s-(s_0+s_1)/2))}{\cosh(\lambda(s_1-s_0)/2} .
\end{align*}
Identifying $PY_z((s_0+s_1)/2$ as a constant, we obtain the required estimate. We next examine:
\[
|\partial_s PY_x - (\partial_x \dt f)(PY_x)| \leq k_1.
\]
For segments of gradient flow uniformly away from all critical points of $f$, then we can choose our coordinates so that locally $f(x) =x+c$. Then the above equation takes the form:
\[
|\partial_s PY_x - (\partial_x \dt f)(PY_x)| \leq k_1.
\]
Using the exact same techniques as above, we conclude
\[
| PY_x(s) -PY_x((s_0+s_1)/2) - (s-(s_0+s_1)/2)| \leq C  \max (\|QY(s_0,t)\|^{2/p}_{L^2{S^1}}, \|QY(s_1,t)\|^{2/p}_{L^2{S^1}}) \frac{\cosh(\lambda(s-(s_0+s_1)/2))}{\cosh(\lambda(s_1-s_0)/2} .
\]
Identifying $PY_x((s_0+s_1/2)) - (s-(s_0+s_1))/2$ as the $x$ component of a lift of the gradient flow, the conclude the required estimate.\\
If $u$ is uniformly bounded away from all critical points of $f$ except one, then we can only choose coordinates so that $f(x) =\frac{1}{2}x^2$ (the case for $f(x) =-\frac{1}{2}x^2$ is similar), then the above equation takes the form
\[
|\partial_s PY_x - \dt PY_x| \leq k_1.
\]
Recyling notation from the previous proof we get
\[
\partial_s PY_x - \dt PY_x =G_x(s)
\]
where $G_x(s) \leq k_1(s)$
Using integration factors as before we obtain
\[
(PY_x e^{-\delta s})' = G_x(s) e^{-\delta s}.
\]
Integrating both sides, from $(s_0+s_1)/2$ to $s$
\[
PY_x = c(s) e^{\delta (s-(s_0+s_1)/2)}
\]
where $c(s)' = G_x(s) e^{-\delta s}$. Then
\[
c(s)= c_0 + \int_{s_0+s_1/2} ^s G_x(s') e^{-\delta s'} ds'
\]
\[
|PY_x - c_0 e^{\delta M(s-(s_0+s_1)/2)}|\leq e^{\delta (s-(s_0+s_1)/2)}\int_{(s_1+s_0)/2}^s G_x(s') e^{-\delta s'} ds'.
\]
Here we need be a bit careful about this integral, by our assumptions on $G_x(t)$ it is upper bounded by:
\[
G_x(t) \leq   C  \max (\|QY(s_0,t)\|^{2/p}_{L^2{S^1}}, \|QY(s_1,t)\|^{2/p}_{L^2{S^1}}) \frac{\cosh(\lambda(s-(s_0+s_1)/2))}{\cosh(\lambda(s_1-s_0)/2} .
\]
WLOG we assume $s>0$ and $ (s_0+s_1)/2=0$, then we have the inequalities:
\[
C' \cosh (\lambda s)\leq e^{\lambda s} \leq C \cosh (\lambda s).
\]
Then the integral

\begin{align*}
&e^{\delta (s)}\int_{0}^s G_x(s') e^{-\delta s'} ds' \\
&\leq e^{\dt s} C\frac{max (\|QY(s_0,t)\|^{2/p}_{L^2(S^1)}, \|QY(s_1,t)\|^{2/p}_{L^2(S^1)})}{cosh(\lambda(s_1-s_0)/2)}\frac{e^{(\lambda -\dt )s}-1}{\lambda-\dt} \\
&\leq C max (\|QY(s_0,t)\|^{2/p}_{L^2(S^1)}, \|QY(s_1,t)\|^{2/p}_{L^2(S^1)}) \frac{cosh(\lambda s)}{cosh(\lambda (s_1-s_0)/2)}
\end{align*}
which is exactly our estimate. The same works for $s<0$.

Finally we consider $|\p_s PY_a - e^{\dt f(PY_x)}| \leq k_1(s)$. As before we replace $f(PY_x)$ with $f(\pi_x v)$, which introduces an error of the same form as $k_1$ due to our above estimate, so we simply absorb it into $k_1$ on the right hand side, we then integrate both sides to get
\[
\left|PY_a +c -\int_{s_0+s_1/2}^s e^{\dt f(\pi_x v)}\right| \leq C \frac{\max (\|QY(s_0,t)\|^{2/p}_{L^2(S^1)}, \|QY(s_1,t)\|^{2/p}_{L^2(S^1)})}{\cosh(\lambda(s_1-s_0)/2)}\cosh(\lambda (s-(s_0+s_1)/2))
\]
from this we conclude the proof.
\end{proof}
\section{Surjectivity of gluing} \label{surj}
In previous sections we proved every transverse rigid cascade glues to a $J_\dt$-holomorphic curve of Fredholm index 1. In this section we show this gluing is unique, i.e. if a $J_\dt$-holomorphic curve is sufficiently close to the cascade, then it must have come from our gluing construction. The main strategy is to consider a degeneration $u_\dt \rightarrow \cas{u}= \{ u^i\}$ of a $J_\dt$-holomorphic curve $u_\dt$ into a cascade $\{ u^i \}$. Using the compactness results stated in Section 11.2 of \cite{SFT} (See also Chapter 4 of \cite{BourPhd}) and proved in our appendix, we know the convergence is $C^{\infty}_{loc}$, using our local estimates we show $u_\dt$ corresponds to a solution of our tuple of equations
\[
\mathbf{\Theta}_u=0,\quad  \mathbf{\Theta}_v=0.
\]
Here we use $\mathbf{\Theta}_u=0$ to denote the system of equations over the $J$-holomorphic curves in the cascade, and $\mathbf{\Theta}_v=0$ denotes the system of equations over each gradient trajectory (finite or semi-infinte) that appear in the cascade.
Furthermore, we show we can arrange that vector fields among the equations in  $\mathbf{\Theta}_v=0$ that correspond to finite gradient trajectories all live in $H_0$, the codimension 3 subspace we fixed for each finite gradient trajectory (we abuse notation slightly, there is a different $H_0$ for each different finite gradient trajectory). We showed in the gluing section such vector fields in $H_0$ are unique. We also make a choice of right inverse for $\oplus D_i$ for the system $\mathbf{\Theta}_u=0$, and show we can arrange so that the vector field producing $u_\dt$ lands in the image of said right inverse for $\oplus D_i$. Therefore from the uniqueness of our gluing construction there is a 1-1 correspondence between $J_\dt$-holomorphic curves and cascades.

The outline of this section is as follows. We will first focus on the simplest possible case: a two level cascade $\cas{u}= \{u^1,u^2\}$ meeting along a single Reeb orbit in the intermediate cascade level. Even in this simplified setting there are several stages to our construction: we first use the previous decay estimates to show that $u_\dt$ is in an $\ep$ neighborhood of a preglued curve constructed from the cascade $\cas{u}$. Then we adjust the pregluing using the asymptotic vectors so that the vector field over the finite gradient trajectory $v$ lives in $H_0$, and the part of the vector field living over $u^i$ lives in the preimage of our specified right inverse, while maintaining the fact the vector field still lives in the $\ep$ ball. Finally we extend the vector fields over all of $u^{i*}TM$ and $v^*TM$ so that they become solutions to $\mathbf{\Theta_u}=0$ and $\mathbf{\Theta_v}=0$, using tools from Section 7 of \cite{obs2}. We also develop some properties of linear operators for this purpose.

Then after the 2-level cascades case has been thoroughly analyzed and proper tools developed, we introduce some more elaborate notation to set up the more general $n$-level cascade case.
\subsection{Notation and setup, for 2-level cascades}
We note here that we are not proving the SFT compactness statement, we are simply using it. For ease of exposition, we first describe the case with $u_\dt$ degenerating into a 2 level cascade consisting of $u^1$ and $u^2$ and such that they only have 1 intermediate end meeting in the cascade level. We let $\gamma_1:= ev^-(u^1)$, and $\gamma_2 := ev^+(u^2)$ denote the Reeb orbits on the Morse-Bott torus. We fix domains $\Sigma_1$ and $\Sigma_2$ for $u^1$ and $u^2$. We fix cylindrical neighborhoods near punctures of $\Sigma_i$, and let $(s, t)_i$ denote coordinates near the puncture that meet along the intermediate Morse-Bott torus. We let also let $(s,t)'_i$ denote the cylindrical coordinates on $u^i$ that are on punctures away from the Morse-Bott torus that appear in the intermediate cascade level. Recall a neighborhood of the maps $u^i$ is given by
\[
W^{2,p,d}(u^{i*}TM)\oplus V_i \oplus V_i'\oplus T \mcal{J}_i.
\]
We let $\Sigma_\dt$ denote the domain for $u_\dt$. Then by the analog of SFT compactness, for each $\dt$ we can break down the domain $\Sigma_\dt$ into 3 regions, 
\[
\Sigma_\dt = \Sigma_{\dt+}\cup N_\dt \cup \Sigma_{\dt-}
\]
where we think of $\Sigma_{\dt\pm}$ as regions that converge to $\Sigma_i$, and $N_\dt$ the thin region biholomorphic to a very long cylinder that converges to the finite (yet very long) gradient trajectory connecting $u^1$ and $u^2$. To be more precise, we can translate $u_\dt$ globally so that over $\Sigma_{\dt+}$ the map $u_\dt$ converges in $C^\infty_{loc}$ to $u^1$, and there exists a sequence of $a$ translations, we denote by $a_\dt$, so that after we translate $u^2$ by $a_\dt$, which we denote by $u^2+a_\dt$, the map $u_\dt$ when restricted to $\Sigma_{\dt-}$ converges in $C^\infty_{loc}$ to $u_2+a_\dt$. Technically the convergences to $u^1$ and $u^2$ are over compact subsets of $\Sigma_{\dt \pm}$, near the other punctures $(s,t)_i'$ there are additional convergences to semi-infinite gradient trajectory. Here we only concern ourselves with convergences near $N_\dt$, and worry about semi-infinite gradient trajectories in a later section.

\subsection{Finding appropriate vector fields}
We first consider the degeneration in the intermediate cascade level. We will later consider degeneration to the configuration of a semi-infinite gradient trajectory.\\
\subsubsection{Finding a global vector field}
Let $0<\ep' <<\ep$, the specific size of $\ep'$ will be specified in the course of the construction. We fix a large real number $K>0$, then we consider the region $|s_i|\leq K, |s_i'|\leq K$ as subsets of $\Sigma_i$. We denote this compact subset of the domain by $\Sigma_{iK}$. We take $K$ large enough so that for $|s_i|\geq K$ the maps $u^i$ are in a small enough neighborhood of $\gamma_i$, that up to $k$ derivatives, we can think of $u^i$ as exponentially decaying to trivial cylinders, with exponential decay bounded by $e^{-Ds_i}$.

This choice of $K$ also determines a decomposition of the domain of $u_\dt$, to wit
\[
\Sigma_\dt = \Sigma_{+\dt K} \cup N_{\dt K} \cup \Sigma_{-\dt K} .
\]
Then the convergence statement in $C^\infty_{loc}$ implies there are vector fields $\zeta_{i\dt} \in u^{i*}TM|_{\Sigma_{iK}}$ of $C^1$ norm $<\ep'$ and variation of complex structure  $\dt j_i \in T\mcal{J}_i$ of size $\leq \ep'$ so that
\[
u_\dt|_{\Sigma_{\dt+K}} = exp_{u^1,\dt j_1}(\zeta_{1\dt})
\]
and 
\[
u_\dt|_{\Sigma_{\dt-K}} = \exp_{u^2,\dt j_2}(\zeta_{2\dt}).
\]
We shall for now suppress the variation of complex structure $(u^i,\dt j_i)$ and simply write $u^i$. When later we want to include it in the notation we shall write $(u^i,\dt j_i)$. We also recall that our metric is flat around Morse-Bott tori, so for small enough $\zeta_{i\dt}$, we have $\exp_{u^i}(\zeta_{i\dt}) = u^i+\zeta_{i\dt}$ near Morse-Bott tori.

We here simply note the $W^{2,p,d}$ norm of $\zeta_{i\dt}$ is then bounded above by $C\ep' e^{dK}$. For fixed $K$, as $\dt \rightarrow 0$, by $C^\infty_{loc}$ convergence we can take $\ep'(\dt) \rightarrow 0$ to make this expression as small as we please. We also observe for fixed $K$ and small enough $\ep'$ the deformations $(\zeta_{i\dt},\dt j_i)$ are within an $\ep$ ball of $W^{2,p,d}(u^{i*}TM)\oplus V_i \oplus V_i'\oplus T \mcal{J}_i$. We next consider the behaviour of $u_\dt$ when restricted to the neck region $N_\dt$. We first informally write $N_{\dt K}$ as the cylinder $[0,N_{\dt K}]\times S^1$. We start with the following lemma:
\begin{lemma}\label{glob}
By our assumption as $K\rightarrow \infty$ (which would take $\dt \rightarrow 0$ with it in order to satisfy our previous assumptions) we have $u_\dt|_{N_{\dt K}}$ converges in $C^\infty_{loc}$ to trivial cylinders. This is also true uniformly, i.e. for given $\ep''>0$, there is a $K$ large enough so that for every small enough values of $\dt>0$, $u_\dt |_{[k,k+1]\times S^1}$ is within $\ep''$ (in the $C^k$ norm) of a trivial cylinder of the form $\gamma \times \bb{R}$ for all values of $k$ so that $[k,k+1]\times S^1 \subset N_{\dt K} \times S^1$. 
\end{lemma}
\begin{proof}

\textbf{Step 1} We claim for $K$ large enough $|d u_\dt| <C$ for all of $N_{\dt K}$. Suppose not, then we can find a sequence $(s_\dt,t_\dt)$ where $|du_\dt(s_\dt,t_\dt)| \rightarrow \infty$, by Gromov compactness a holomorphic plane bubbles off. But a holomorphic plane must have energy bounded below, by the Morse-Bott assumption. However as $K \rightarrow \infty$ the energy of $u_\dt |_{N_{\dt K}} $ goes to zero, which in particular is less than the minimum energy required to have a holomorphic plane, this is a contradiction.

\textbf{Step 2} We argue by contradiction, Suppose for all $K>0$ there exists an interval $[a_K,a_K+1] \times S^1$ so that the distance of $u_\dt |_{[a_K,a_K+1] \times S^1}$ and any trivial cylinder is $\geq \ep''$. However we observe as $K\rightarrow \infty$ the energy of $u_\dt |_{[a_K,a_K+1] \times S^1}$ goes to zero uniformly in $K$, then by Azerla-Ascoli this converges to a holomorphic curve of zero area, which must be a segment of a trivial cylinder. Hence we have a contradiction.
\end{proof}
Then the previous convergence estimate implies the following:
\begin{proposition}
We take $\ep''>0$ small enough so that previous convergence estimates near Morse-Bott tori apply. Then there is a large enough $K$, and small enough $\ep'$ (which depends on $K$), so that if we choose small enough $\dt>0$ (which depends on the choice of $\ep'$ and $K$ but can always be achieved), there is a gradient trajectory $v_K$ defined over the cylinder $(s_v,t_v)\in [0,N_{\dt K}]\times S^1$ so that there is a vector field $\zeta_K$ over $v_K$ so that
\[
u_\dt |_{N_{\dt K}} = \exp_{v_K}(\zeta_K)
\]
and the norm of $\zeta_K$ satisfies
\[
\|\zeta_K\|_{C^{k-1}} \leq C \max (\|\zeta_K(0,-)\|^{2/p}_{L^2(S^1)},\|\zeta_K(N_{\dt K},-)\|^{2/p}_{L^2(S^1)}) \frac{\cosh(\lambda(s-N_{\dt K}/2))}{\cosh(\lambda N_{\dt K}/2)} 
\]
and in particular, if we choose $\dt>0$ small enough, by $C^\infty_{loc}$ convergence
\[
\|\zeta_K(0)\|^{2/p}, \|\zeta_K(N_{\dt K})\|^{2/p}\leq \ep'.
\]
\end{proposition}
We estimate the norm of $\zeta_K$ for later use. With some foresight we realize we need to use a weighted norm $e^{w(s_v)}$ for $s_v\in [0,N_{\dt K}]$, where
\[
w(s) = d(N_{\dt K}/2 + K -|s-N_{\dt K}/2|).
\]
Then we measure the $W^{1,p}$ norm of $\zeta_K$ with respect to $e^{w(s)}$, but by the previous proposition the norm of $\zeta_K$ undergoes exponential decay as it enters the interior of $N_{\dt K}$. Hence we have
\[
\int_{S^1}\int_0^{N_{\dt K}} \|\zeta_K\|_{C^k} dsdt \leq C \ep' e^{dK}.
\]
We now come to the main result of this subsection. We combine $\zeta_{i\dt}$ and $\zeta_K$ into a vector field over some preglued curve built from $\Sigma_{\dt \pm K}$, the curve $v_K$ and some asymptotic vector fields. We first recall that
\[
u^1(-K,t) + \zeta_{1\dt}(-K,t) = u_\dt(N_{\dt K},t)  = v_K + \zeta_K(N_{\dt K},t).
\]
Now we have the $C^k$ norm of $\zeta_{1\dt}(-K,t)$ and $\zeta_K(0,t)$ are both bounded above by $\ep'$, then we can deform $u^1|_{\Sigma_{1K}}$ by asymptotic vectors $r_1,a_1,p_1$ all of which are of size $\leq \ep'$ so that
\[
\|u^1(-K,t_1) +(r_1,a_1,p_1) -v_K(N_{\dt K},t_v)\|_{C^{k}} \leq \ep'.
\]
There are naturally several possible choices possible for $(r_1,a_1,p_1)$. In anticipation of our later constructions, we make the following important specification.\\
Recall for $(s_1,t_1) \in (-\infty,0]\times S^1$, for $s_1<<0$ the map $u^1$ converges to a parametrized trivial cylinder
\[
\tilde{\gamma}_1(s_1,t_1) : \bb{R}\times S^1 \longrightarrow M
\]
whose image is of course the trivial cylinder $\gamma_1\times  \bb{R}$. The key property is that $u^1|_{(-\infty,0]\times S^1}$ decays exponentially to $\tilde{\gamma}_1$:
\[
\|u^1-\tilde{\gamma_1}\|_{C^k} \leq C e^{-Ds_1}.
\]
We also recall properties of $v_K$, which is the finite gradient trajectory $u_\dt$ converges to. For small enough $\dt>0$, the gradient flow is extremely slow, so for $s_v \in [N_{\dt K}-2R,N_{\dt K}]$, there is another trivial parametrized cylinder
\[
\hat{\gamma}_1(s_v,t_v): \bb{R} \times S^1 \longrightarrow M
\]
so that for $s_v \in[N_{\dt K}-2R,N_{\dt K}] $
\[
\|\hat{\gamma_1} -v_K\|_{C^k} \leq C R \dt
\]
which goes to zero as $\dt \rightarrow 0$. By the comparison result above there are vectors $(r_1,a_1,p_1)\leq \ep'$ so that:
\[
\tilde{\gamma_1} + (r_1,a_1,p_1)=\hat{\gamma_1}.
\]
Then we choose this particular choice of $(r_1,a_1,p_1)$. There is some free choice of $(r_1,a_1,p_1)$ up to size $R\dt$, which for our purpose is extremely small. We will always make a choice so that the $s_1=R$ end of $u^1+(r_1,a_1,p_1)$ and $s_v = N_{\dt K}-R+K$ of $v_K$ can be preglued together, in the sense we preglued them together in Section \ref{gluing}. (Also see below).

Similarly we recall that
\[
u^2(K,t) + \zeta_{2\dt}(K,t) = u_\dt(0,t)  = v_K + \zeta_K(0,t_v).
\]
By the same reasoning there is a parametrized trivial cylinder $\tilde{\gamma}_2:\bb{R} \times S^1 \rightarrow M$ that $u^2$ decays exponentially to:
\[
\|u_2-\tilde{\gamma}_2\|_{C^k} \leq C e^{-Ds_2}.
\]
And we can find parametrized trivial cylinder $\hat{\gamma}_2$ so that for $s_v \in [0,3R]$ we have
\[
\|\hat{\gamma}_2 -v_K\|_{C^k} \leq C R \dt.
\]
Hence by comparison we choose asymptotic vectors $(r_2,a_2,p_2)$ of size bounded above by $\ep'$ over $u^2$ so that
\[
\|u^2(K,t_2) +(r_2,a_2,p_2)- v_K(0,t)\|_{C^{k}(S^1)}\leq \ep'.
\]
The trivial cylinders satisfy the relation
\[
\tilde{\gamma}_2 + (r_2,a_2,p_2)=\hat{\gamma_2}.
\]
Observe since $v_K$ as a parametrized cylinder does not rotate in the $z$ direction, here we have $r_1=r_2$. We note this here because in our gluing construction earlier where we identified $t_v \sim t_- + r_+-r_-$. We shall see where this is used in a later section.

Then we construct the preglued domain by gluing together
\[
\Sigma_{\dt, K, (r,a,p)_i} :=(u^1,\dt j_i)+(r_1,a_1,p_1)|_{\Sigma_{1R}} \cup [R-K,N_{\dt K}+K-R]\times S^1 \cup (u^2,\dt j_2)+(r_2,a_2,p_2)|_{\Sigma_{2R}}
\]
by $\Sigma_{1R}$ we mean the domain of $u^1$ with $s_1<-R$ removed. In other words $\Sigma_{1R}:= \Sigma_K \cup (s_1,t_1)\in [-R,-K]\times S^1$. (We ignore the ends of $u^1$ glued to semi-infinite trajectories for now). A similar expression holds for $\Sigma_{2R}$.
By $(u_i,\dt j_i)+(r_i,a_i,p_i)$ we mean $\Sigma_{iR}$ with complex structure deformed by $\dt j_i$ and the cylindrical neck twisted/stretched/translated by asymptotic vector fields $(r_i,a_i,p_i)$. We specify the gluing as follows. We glue together
\[
[u^1+(r_1,a_1,p_1)](s_1=-R,t_1) \sim v_K(s_v=N_{\dt K}-R+K, t_v).
\]
Using the same pregluing interpolation as we did in our pregluing construction.
At $u^2$ end we are making the identification 
\[
[u^2+(r_2,a_2,p_2)(s_2=R,t_2) \sim v_K(s_v=R-K, t_v)
\]
and this determines our preglued domain, $\Sigma_{\dt, K, (r,a,p)_i}$. In constructing this preglued domain, we have identified:
\[
-s_1-R \sim s_v-N_{\dt K}-R+K
\]
\[
s_2-R \sim s_v-R-K
\]
\[
t_1\sim t_v\sim t_2.
\]
Since $r_1=r_2$, here the $t_v$ is identified with $t_2$ without any twist. It carries a natural preglued map into $M$ by defining it to be $(u^i,\dt j_i)+(r_i,a_i,p_i)$ on $\Sigma_{iR}$ and $v_K$ on $[R-K,N_{\dt K}-R+K]\times S^1 $, and interpolated in the pregluing region the same way we preglued in Section \ref{gluing}. We call the preglued map $u_{\dt, K, (r,a,p)_i}$.
Then we can form the interpolation of the vector fields $\zeta_{i\dt}$ and $\zeta_K$ into a vector field we call $\zeta_{\dt, K,(r,a,p)_i}$ so that \:
\[
u_\dt = u_{\dt, K, (r,a,p)_i} + \zeta_{\dt, K, (r,a,p)_i}
\]
We should at this stage measure the size of $\zeta_{\dt, K, (r,a,p)_i}$. We need to measure it with exponential weights. The weight in question takes the form $e^{d|s|}$ over $\Sigma_{iR}$ and of the form $e^{w(s)}$ over $N_{\dt K}$.
\begin{proposition} \label{norm_estimate}
The norm of $\zeta_{\dt, K, (r,a,p)_i}$ measured over $u_{\dt, K, (r,a,p)_i}$ with weights as specified above is bounded above by
\[
C\ep' (C+e^{dK})+CRe^{dR}\dt + Ce^{-D'K}.\]
For small enough $\dt$ we can make this bound be as small as we please. For convenience we use another letter $\tilde{\ep}$, informally thought of as $\ep' << \tilde{\ep} << \ep$, and say given $\tilde{\ep}>0$, we can take $\dt >0$ small enough so that norm of $\zeta_{\dt, K, (r,a,p)_i}$ is bounded above by $\tilde{\ep}$. With some foresight, we will need to make it a bit smaller than $\tilde{\ep}$, we can take $\dt$ small enough so that the vector field is bounded above by $\tilde{\ep}^2$. 
\end{proposition}
\begin{proof}
The norm of $\zeta_{\dt, K, (r,a,p)_i}$ measured over $\Sigma_{iK}$ is upper bounded by $C\ep' e^{dK}$ as we discussed earlier. \\
Next consider the segment of $\zeta_{\dt, K, (r,a,p)_i}$ over $v_K$ for $s_v\in[R-K,N_{\dt K}-R+K]$. Recall we glued at the end points of this interval, hence by previous estimates the norm of $\zeta_{\dt, K, (r,a,p)_i}$ it is bounded above by $C\ep' e^{dK}$. Next we address the remaining region. WLOG we focus on $s_1\in[K,R]$ for $u^2$. In this region, the distance between $v_K$ and $u_\dt$ is bounded above (even when integrated against weights) by $C\ep' e^{dK}$. The distance between $v_K$ and the trivial cylinder $\hat{\gamma}_2$ is bounded above by $Re^{dR}\dt$ after integrating with the exponential weights. The distance between $u^2+(r_2,a_2,p_2)$ and $\hat{\gamma}_2$ in pointwise $C^k$ norm is bounded above by
\[
C e^{-Ds_1}
\]
so when we integrate this pointwise difference over $s_1\in [K,R-K]$ with weight $e^{ds_1}$, we have the upper bound
\[
C e^{-(D-d)K}
\]
and hence our overall bound on $\zeta_{\dt, K, (r,a,p)_i}$ is as claimed in the proposition.

To explain how we make the vector field smaller than $\tilde{\ep}$, we first choose $K$ fixed large enough so that $e^{-dK} << \tilde{\ep}$, then by choosing $\ep'$ small enough we can make $C\ep' (C+e^{dK})$ much less than $\tilde{\ep}$, and we recall as $\dt \rightarrow 0$, $\ep' \rightarrow 0$. Finally $Re^{dR}\dt \rightarrow 0$ as $\dt\rightarrow 0$ by the definition of $R$. 
\end{proof}

\subsubsection{Separating global vector field into components}

After we have obtained the preglued map $u_{\dt, K, (r,a,p)_i}$ and vector field $\zeta_{\dt, K, (r,a,p)_i}$, there are a few more steps to complete our construction. They are:
\begin{enumerate}
    \item Truncate the vector field $\zeta_{\dt, K, (r,a,p)_i}$ into 
    \[
    \zeta_{\dt, K, (r,a,p)_i} = \zeta_{1,\dt, K, (r,a,p)_1} +\zeta_{\dt,K,(r,a,p)_i,v} + \zeta_{2,\dt, K, (r,a,p)_2}
    \]
    where
    \[
    \zeta_{i,\dt, K, (r,a,p)_i} \in u^{i*}(TM)
    \]
    \[
    \zeta_{\dt,K,(r,a,p)_i,v} \in v_K^*(TM).
    \]
    \item Adjust the asymptotic vectors $(r,a,p)_i$ in the pregluing so that the vector fields $ \zeta_{i,\dt, K, (r,a,p)_i}, \zeta_{\dt,K,(r,a,p)_i,v}$ live in images of $Q_i$ and $H_0$ respectively, where the definition of these conditions will be specified below.
    \item Show $ \zeta_{i,\dt, K, (r,a,p)_i}$ and $\zeta_{\dt,K,(r,a,p)_i,v}$ can be extended to (unique) solutions of the equations $\mathbf{\Theta_i}=0$ and $\mathbf{\Theta_v}=0$, subject to our choice of right inverse in the previous step.
\end{enumerate}
In this subsection we address the first two bullet points, and the third bullet point will conclude the surjectivity of gluing, which we will take up after a technical detour.

To address the first bullet point we introduce the cut off functions over $v_K$. We define
\[
\beta_1(s_v):= \beta_{[R/2;N_{\dt K}-K-2R,\infty]}
\]
\[
\beta_2(s_v):=\beta_{[-\infty,2R-K;R/2]}
\]
\[
\beta_v:=\beta_{[R/2;R-K,N_{\dt K}+K-R;R/2]}.
\]

The obvious inference is that if we imagine we constructed $\Sigma_{\dt, K, (r,a,p)_i}$ from a prelguing construction by deforming $u^i$ with $(r_i,a_i,p_i)$ and gluing to it a finite gradient segment, the cut off functions listed above should correspond to the cut off functions we used for our gluing construction. In fact this is exactly the case, $\beta_i$ ought to be identified with $\beta_\pm$. The only difference is a change in notation where our coordinates are shifted by a factor of $K$.

Then to address the first bullet point, we take some care to specify what we mean in our definition of $\zeta_{*,\dt, K, (r,a,p)_1}, * = 1,2,v$ in anticipation of our upcoming proof of surjectivity of gluing. In particular we must define $\zeta_{*,\dt, K, (r,a,p)_i}$ so that they satisfy the following properties:
\begin{itemize}
    \item \[
    \zeta_{\dt, K, (r,a,p)_i} = \beta_1 \zeta_{1,\dt, K, (r,a,p)_1} +\beta_2\zeta_{\dt,K,(r,a,p)_i,v} + \beta_v\zeta_{v,\dt, K, (r,a,p)_i}.
    \]
    \item Their norms satisfy
    \[
    \|\zeta_{i,\dt, K, (r,a,p)_i}\| \leq 
    C\ep' (C+e^{dK})+CRe^{dR}\dt + Ce^{-D'K}\] as measured in $W^{2,p,d}(u^i+(r_i,a_i,p_i) ^*TM) \oplus T\mcal{J}_i$ with weighted norm (we ignore the other ends of $u^i$ for now).
    As well as the fact 
    \[
    \|\zeta_{v,\dt, K, (r,a,p)_i}\| \leq C\ep'(C+e^{dK})+CRe^{dR}\dt + Ce^{-D'K}
    \]
    as measured in $W^{k,p,w}(v_K^*TM)$.
    \item The vector fields $\zeta_{*,\dt, K, (r,a,p)_i}$ have support as follows. Using coordinates $(s_v,t_v)\in [0, N_{\dt K}] \times S^1$ over  $\zeta_{2,\dt, K, (r,a,p)_2} = 0 $ for $s_v >3R-K$. The vector field $\zeta_{1,\dt, K, (r,a,p)_1} = 0 $ for $s_v< N_{\dt K} -K-3R$, and $\zeta_{v,\dt, K, (r,a,p)_i} = 0 $ for $s_v< R-K$ and $N_{\dt K} -K-R<s_v$. (These supports are not too significant as we will find some other way to extend them later).
\end{itemize}
We observe such extensions are always possible.
We note the previous theorem on the norm of the global vector field $\zeta_{\dt, K, (r,a,p)_i}$ implies analogous statements on the individual vector fields $ \zeta_{i,\dt, K, (r,a,p)_i}, \zeta_{\dt,K,(r,a,p)_i,v}$. The bullet point about support follows from choice of cut off functions $\beta_*$.

To address the second bullet point, we specify what we mean by the ``right spaces". Assume $u^i$ are nontrivial with stable domains, recall for $u^1$ and $u^2$ (or rather a suitable translate of $u^2$ in the symplectization direction) with domain $\Sigma_i$  the space of deformations is given by
\[
W^{2,p,d}(u^{i*}TM)\oplus V_i \oplus V_i'\oplus T \mcal{J}_i.
\]
The operators $D_i$ with domain $W^{2,p,d}(u^{i*}TM)\oplus V_i \oplus V_i'\oplus T \mcal{J}_i$ are defined as the linearization of $\db_J$ along with variations of complex structure of $\Sigma_i$. By assumption $D_i$ have index $1$ with kernel $\p_a$, i.e. global translation in the $a$ direction. Recall to define the equations $\mathbf{\Theta}_i$ we needed to fix a right inverse $Q_i$ to $D_i$. We choose $Q_i$ as follows, consider the codimension 1 subspace $W_i \subset W^{2,p,d}(u_i^*TM)$ defined by 
\begin{equation}\label{integral_over_a}
\zeta \in W' \, \textup{iff} \, \int_{\hat{\Sigma}_i} \la \zeta, \p_a\ra =0
\end{equation}
where $\hat{\Sigma}_i$ is the compact subset of $\Sigma_i$ with all cylindrical neighborhood around punctures removed. Then $D_i$ restricted to $W'_i\oplus V_i \oplus V_i'\oplus T \mcal{J}_i$ is an isomorphism with inverse $Q_i$, and we take this $Q_i$ to be the right inverse used in the contraction mapping principle we use to solve $\mathbf{\Theta_i}=0$. In the case where the domain is not stable, we note the following convention. 

\begin{convention}
For definiteness say the domain of $u^1$ is either a plane or a cylinder, then the act of placing marked points in the domain presents us a subspace $W'\subset W^{2,p,d}(u_i^*TM)$ so that the restriction of $D_1$ to $W'\oplus V_1 \oplus V_1'\oplus T \mcal{J}_1$ is an isomorphism: we simply take $W'$ to be vector fields that preserve the condition that marked points remain on the auxiliary surfaces we chose in Convention \ref{stabilization}. If $u^1$ has several connected components, some of which are stable, and some of which are unstable prior to adding marked points, then we impose the integral condition \ref{integral_over_a} for vector fields over the domains that are stable without adding marked points, and the marked point condition in Convention \ref{stabilization} for domains are stable only after adding marked points. This picks out the subspace $W'$.
\end{convention}

We now explain our definition of $H_0$. Recall $v_K$ is a segment of a gradient trajectory that has coordinates $(s_v,t_v)$, with the segment of interest being $[0,N_{\dt K}]\times S^1$, with exponential weight $e^{w(s)}$ with peak at $s_v= N_{\dt K}/2$. We recall the functionals, analogous to previous section: $L_*, *=a,s,v: W^{2,p,w}(v_K^*TM)\rightarrow \bb{R}$ defined by
\[
L_*: \zeta \in W^{2,p,w}(v_K^*TM)\longrightarrow  \int_0^1 \la \zeta(s_v=N_{\dt K}/2,t), \p_*\ra dt \in \bb{R}
\]
and we define
\[
H_0:= \{\zeta \in W^{2,p,w}(v_K^*TM) | L_*(\zeta)=0, *=a,x,z\}.
\]
We now deform the pair $(r,a,p)_i$ to ensure our vector fields lie in the correct subspace. We first observe we can ensure $\zeta_{1,\dt, K, (r,a,p)_1}$ lives in the image of $Q_1$ by using the global $a$ translation of $u_\dt$, i.e. when we first started talking about the degeneration of $u_\dt$ into $u^1$ and $u^2$, we translate $u_\dt$ by $a$ so that $u_\dt$ always converges to $u^1$ in $C^\infty_{loc}$ via a vector field that lives in image of $Q_1$. This degree of freedom is afforded to us by the fact that the problem is $\bb{R}$ invariant. Hence we next focus on vector fields $\zeta_{2,\dt, K, (r,a,p)_2}$ and $\zeta_{v,\dt, K, (r,a,p)_i}$. We imagine $u_1|_{\Sigma_{1,K}}$ is fixed in place. To make $\zeta_{2,\dt, K, (r,a,p)_2}$ live in the image of $Q_2$, we change $p_2 \rightarrow p_2 +\dt p_2$. This changes the pregluing domain:
$\Sigma_{\dt, K, (r,a,p)_1,(r,a, p+\dt p)_2}$ as well as the associated map into $M$, which is given by
\[
u_{\dt, K, (r,a,p)_1,(r,a,p+\dt p)_2}
\]
The effect of changing $p_2$ changes the length of the finite gradient trajectory that is glued between $u^1$ and $u^2$. Naturally changing the preglued map also deforms the global vector field $\zeta_{\dt, K, (r,a,p)_1,(r,a,p+\dt p)_2}$ and its cutoffs. (Adjusting our cut off functions accordingly). We observe to make $\zeta_{2,\dt, K, (r,a,p+\dt p)_2}$ live in the image of $Q_2$ we need to lengthen/shorten the glued cylinder by $a$-length $\ep'$, this corresponds to a $\dt p_2$ of size $\ep'\dt$. After this adjustment, if we take $\dt$ sufficiently small, the global vector field still has size bounded above by $\tilde{\ep}$ (or in the case we will need, $\tilde{\ep}^2$), and hence the same is true of its cut offs.

Finally we turn our attention to $\zeta_{v,\dt, K, (r,a,p)_i}$. To do this we need the following lemma:
\begin{lemma}
\[
L_*(\zeta_{v,\dt, K, (r,a,p)_i}) \leq C\ep'^{2/p} e^{-\lambda (N_{\dt K}-CR)/2}.
\]
for $*=a,z,x$.
\end{lemma}
\begin{proof}
Follows directly from exponential decay estimates
\end{proof}
Observe this upper bound is extremely small in the following sense. If we consider the vector field 
$C\ep'^{2/p} e^{-\lambda (N_{\dt K}-CR)/2} \p_*$ where $*=a,z,x$ and measured the size of this vector field over $v_k$ with domain $s_v \in [0,N_{\dt K}]$, with the exponential weight $e^{w(s)}$, we would still get an extraordinarily small number, of size $C\ep'^{2/p} e^{-(\lambda -d)N_{\dt K}/2} e^{C\lambda R}$, which goes to zero as $\dt \rightarrow 0$. This means we can apply a constant translation of form $C\ep'^{2/p} e^{-\lambda (N_{\dt K}-CR)/2} \p_*$ over $s_v\in [-K,N_{\dt K}+K]$ to the vector field $\zeta_{v,\dt, K, (r,a,p)_i}$ to try to make it land in $H_0$ while still keeping its overall norm $< \tilde{\ep}$. In practice this is done by further adjusting the pair of asymptotic vectors $(r_1,a_1,p_1), (r_2,a_2,p_2+\dt p_2)$ used in the pregluing, which we take up in the next lemma.
\begin{lemma}
By adjusting the pair of asymptotic vectors $(r_1,a_1,p_1), (r_2,a_2,p_2+\dt p_2)$ we can make $\zeta_{v,\dt, K, (r,a,p)_i} \in H_0$ while keeping its norm $\leq \tilde{\ep}^2$. We can also maintain the vector fields $\zeta_{i,\dt, K, (r,a,p)_i}$ are still within the image of $Q_i$.
\end{lemma}
\begin{proof}
We examine $L_*$ for $*=r,a,z$ one by one, as the different cases are relatively independent of each other.\\

Let us consider $L_x$. The idea is to change both $p_1$ and $p_2$ in the same direction, which we denote by $p_i +\Delta p$, and preglue to a different gradient trajectory $v_K'$, but the trouble is as we change $p_i$ to $p_i +\Delta p$, the new gradient trajectory $v_K'$ connecting $p_1+\Delta p$ to $p_2+\Delta p$ travels a different amount of $a$ distance as $s_v'$ (the variable for $v_K'$) ranges from $s_v =0$ to $s_v'=N_{\dt K}$, hence there must be a corresponding deformation in the pair of asymptotic vectors $a_1$ and $a_2$ to make the curves still match up and glue. Further we must also choose the deformation of $a_1$ and $a_2$ so that $\zeta_{i,\dt,K,(r,a,p)_i}\in Im Q_i$. Said differently we deform $a_1$ and $a_2$ so that there is no induced global translation of $u^1$ or $u^2$ that enters the pregluing. This is always possible. The exact expressions for these quantities are not so important, the important information is their sizes. The size of $\Delta_p$ is $L_x(\zeta_{v,\dt, K, (r,a,p)_i}) \leq C\ep'^{2/p} e^{-\lambda (N_{\dt K}-CR)/2} $, to make the new $\zeta'_{v,\dt, K, (r,a,p)_i}$ evaluate to $0$ under $L_x$, the corresponding change to $a_1,a_2$ is also of size $C C\ep'^{2/p} e^{-\lambda (N_{\dt K}-CR)/2}$, which we absorb into our notation $(r,a,p)_i$. It is also apparent after this deformation all vector fields are still small.
Next we adjust both $a_1$ and $a_2$ by $L_a(\zeta_{v,\dt, K, (r,a,p)_i})$ to make $L_a(\zeta_{v,\dt, K, (r,a,p)_i})=0$. We adjust $a_1$ and $a_2$ by the same amount in the same direction so as to maintain $\zeta_{i,\dt,K,(r,a,p)_i}\in Im Q_i$. It is clear this will land $\zeta_{v,\dt, K, (r,a,p)_i}$ in $H_0$ and keep the norm of $\zeta_{v,\dt, K, (r,a,p)_i}$ small.

Finally we consider $\p_z$. We shift $r_1$ by size $-L_r(\zeta_{v,\dt, K, (r,a,p)_i})$, and twist the segment of gradient trajectory $v_K$ along with it. But we do not change $r_2$, hence there is an new identification $t_v +L_r(\zeta_{v,\dt, K, (r,a,p)_i}) \sim t_2$ near the $u^2$ end. The result is a new $\zeta_{v,\dt, K, (r,a,p)_i}$, denoted by the same symbol by abuse of notation, so that $L_r(\zeta_{v,\dt, K, (r,a,p)_i})=0$. We also observe by the previous discussion the norm of $\zeta_{v,\dt, K, (r,a,p)_i}$ changed at most by $C\ep'^{2/p} e^{-(\lambda -d)N_{\dt K}/2} e^{C\lambda R}$.

It is also clear that this process will keep $\zeta_{i,\dt, K, (r,a,p)_i}$ in the image of $Q_i$ because the regions in which we are performing these deformations are disjoint.
\end{proof}

To summarize:
\begin{proposition}
We can choose suitable asymptotic vectors $(r,a,p)_1$, $(r,a,p)_2$, from which to construct a preglued domain $\Sigma_{\dt, K, (r,a,p)_i}$ that decomposes as
\[
\Sigma_{\dt, K, (r,a,p)_i} :=(u^1,\dt j_i)+(r_1,a_1,p_1)|_{\Sigma_{1R}} \cup [R-K,N_{\dt K}+K-R]\times S^1 \cup (u^2,\dt j_2)+(r_2,a_2,p_2)|_{\Sigma_{2R}}
\]
where $\dt j_i$ represents variation of complex structure on $\Sigma_{iR}$, and we let $v_K$ denote the segment of gradient trajectory whose domain is $[R-K,N_{\dt K}+K-R]\times S^1$. There is a preglueing  map $u_{\dt, K, (r,a,p)_i} : \Sigma_{\dt, K, (r,a,p)_i} \rightarrow M $ that agrees with our prescription for constructing pregluing maps in Section \ref{gluing}, so that there exists a vector field $\zeta_{\dt, K, (r,a,p)_i}$ so that 
\[
u_\dt = u_{\dt, K, (r,a,p)_i} + \zeta_{\dt, K, (r,a,p)_i}.
\]
If we split $\zeta_{\dt, K, (r,a,p)_i}$ into components that live over $u^1$, $u^2, v_K$ using cut off functions $\beta_*$ as we did in our gluing sections. The resulting vector fields $\zeta_{1,\dt, K, (r,a,p)_1} ,\zeta_{\dt,K,(r,a,p)_i,v}, \zeta_{2,\dt, K, (r,a,p)_1}$ satisfy
\[
\zeta_{i,\dt, K, (r,a,p)_1} \in \op{Im} Q_i
\]
\[
\zeta_{\dt,K,(r,a,p)_i,v} \in H_0
\]
and they all have norm $<\tilde{\ep}$ when measured with exponential weights. For $\zeta_{i,\dt, K, (r,a,p)_1}$ this means $W^{2,p,d}(u_i^*TM)$ and $\zeta_{\dt,K,(r,a,p)_i,v} \in W^{2,p,w(s)}(v_K^*TM)$. We remark here we are bounding the size of our vector fields by $\tilde{\ep}$, but it practice we can make them as small as we please, by $\tilde{\ep}^2$, for instance.
\end{proposition}
Now we are in the position to extend $\zeta_{v,\dt, K, (r,a,p)_i}$ and $\zeta_{i,\dt, K, (r,a,p)_i}$ to solutions of $\Theta_i, \Theta_v$, but before that we need to take a detour on linear operators.

\subsection{A detour on linear operators}
In this detour of a subsection we prove several key facts about linear operators to be used later. Naturally, very similar lemmas appear in Section 3 of \cite{obs2} since we are using their strategy for surjectivity of gluing.

We shall first consider the case for semi-infinite trajectories, then we will do the case for finite gradient trajectories.

We shall first work out the case for $p=2$, then deduce the necessary results $p>2$ from Morrey's embedding theorem. For this section we shall work with Sobolev regularity $k>3$, this will not make a difference to us since elliptic regularity will afford us all the regularity we need.

Let 
\begin{equation*}
    v:[0,\infty) \times S^1 \longrightarrow M
\end{equation*}
be a semi-infinite gradient trajectory, equipped with linearized operator
\begin{equation}
D_\delta = \partial_s -(A(s,t) + \delta A)
\end{equation}
where $A(s,t) = -(J_0 \frac{d}{dt} +S)$ corresponds to the linearized operator of the Morse-Bott contact form, and $\dt A$ is a operator of the form $\dt ( M\frac{d}{dt} +N)$ is the correction due to having used the $J_\dt$ almost complex structure.

We equip it with the weighted Sobolev space
$W^{k,2,w(s)}$ where 
\begin{equation*}
    w(s) = d(s+R).
\end{equation*}
We conjugate this over to $W^{k,2}$ at which point it becomes
\begin{equation}
    D_\delta' =\partial_s -(A + \delta A)-d.
\end{equation}
Let's first consider the restriction of $A+d$ to $s=0$, which we shall denote by $A_0$. By the spectral theorem there exists an orthonormal basis of $L^2(S^1)$ given by eigenfunctions of $A_0$, which we write as $\{e_n\}_{n\in \bb{Z}}$ with eigenvalue $\lambda_n$. By assumption $0$ is not an eigenvalue of $A_0$, and by convention we say $\lambda_n >0$ for $n>0$ and vice versa. \begin{theorem}
There is a continuous trace operator $T: W^{k,2}([0,\infty)\times S^1) \rightarrow W^{k-1/2,2}(S^1)$ given as follows, if $f\in W^{k,2}([0,\infty)\times S^1)$:
\[
(Tf)(t) = f(0,t)
\]
The norm in $W^{k-1/2,2}(S^1)$is given as follows, every $f(t)\in W^{k-1/2,2}(S^1)$ has a Fourier expansion
\[
f(t) = \sum_n a_n e_n(t)
\]
then the norm is equivalent to following expression:
\[
\|f(t)\|^2 : = \sum_n |a_n|^2 \lambda_n^{2k-1}.
\]
\end{theorem} 
\begin{proof}
This is a standard theorem in analysis, for a description of this see for instance proof of Lemma 3.7 in \cite{obs2}.
\end{proof}
Then we come to the first main theorem of this detour.
\begin{theorem}
Let $W_-^{k-1/2,2}(S^1)$ denote the subspace of $W^{k-1/2,2}(S^1)$ such that $a_n=0$ for all $n>0$, let $\Pi_-: W^{k-1/2,2}(S^1) \rightarrow W_-^{k-1/2,2}(S^1)$ denote the projection. Then the map $(\Pi_-, \p_s -A_0): W^{k,2}([0,\infty)\times S^1) \rightarrow W_-^{k-1/2,2}(S^1) \times W^{k-1,2}([0,\infty)\times S^1)$ taking
\[
f(s,t) \longrightarrow (\Pi_- f(0,t), (\p_s - A_0) f(s,t))
\]
is an isomorphism.
\end{theorem}
\begin{proof}
We can solve this equation explicitly. Given a pair $(g,h) \in W^{k-1/2,2}(S^1) \times W^{k-1,2}([0,\infty)\times S^1)$, we can write
\[
g=\sum_{n<0} c_n e_n(t)
\]
\[
h = \sum_n b_n (s) e_n(t)
\]
where 
\[
\|g\|^2 = \sum_{n<0} |c_n|^2 |\lambda_n|^{2k-1}
\]
\[
\|h\|^2 = \sum _n \int_0^\infty (|b_n(s)|^2 |\lambda_n|^{2k-2}+ |b_n'(s)|^2|\lambda_n|^{2k-4}  +..+|b_n^{(k-1)}(s)|^2)  ds.
\]
The usual Sobolev norms are equivalent to the expressions we've written above. Comparing term by term we see that $a_n$ satisfies the following ODE:
\[
a_{ns}-\lambda_n a_n=b_n(s)
\]
with boundary condition $a_n(0)= c_n$ for all $n<0$.
They have solutions
\[
a_n = e^{\lambda_n s}\int_0^s b_n(s') e^{-\lambda_n (s')} ds' + c_n e^{\lambda_n s}
\]
where the $c_n$ term only appears for $n<0$. We need to verify several things:
\begin{enumerate}
    \item The terms $e_n(t)e^{\lambda_n s}\int_0^s b_n(s') e^{-\lambda_n (s')}$ and $c_n e_n(t) e^{\lambda_n s}$ are in $W^{k,2}([0,\infty)\times S^1)$.
    \item $\int_0^\infty (|a_n|^2|\lambda_n|^{2k} +|a_n'(s)|^2\lambda_n|^{2k-2} +..+|a_n^{(k)}(s)|^2)|\leq C(|c_n|^2|\lambda_n|^{2k-1}) + C \int_0^\infty (|b_n(s)|^2 |\lambda_n|^{2k-2}+ |b_n'(s)|^2|\lambda_n|^{2k-4}  +..+|b_n^{(k-1)}(s)|^2)$.
\end{enumerate}
The first item says our constructed solution $f$ lives in our Sobolev space, the second item says its norm is upper bounded by our input.

First consider $c_n e_n(t) e^{\lambda_n s}$, its norm in $ W^{k,2}(S^1\times [0,\infty))$ is given by
\[
 \int_0^\infty|c_n|^2 |\lambda_n|^{2k} e^{2\lambda_n s}ds \leq C |c_n|^2|\lambda_n|^{2k-1}
\]
and that this is finite after we sum over $n$ follows from our assumptions on $g$. Similarly consider $d_n :=e^{\lambda_n s}\int_0^s b_n(s') e^{-\lambda_n (s')}ds'$, its norm as measured in $W^{k,2}([0,\infty) \times S^1)$ is given by
\[
\int_0^\infty (|d_n(s)|^2 |\lambda_n|^{2k}+ |d_n'(s)|^2|\lambda_n|^{2k-4}  +..+|d_n^{(k)}(s)|^2).
\]
We have
\[
|d_n^{(l)}| \leq C(l)e^{\lambda_ns} \left\{|\lambda_n|^l \int_0^s |b_n(s')| e^{-\lambda_n (s')}ds' + |\lambda_n|^{l-1}|b_n(s)| e^{-\lambda_n s} + |\lambda_n|^{l-2}|b_n^{(1)}(s)| e^{-\lambda_n s} +\ldots + |b_n^{(l-1)}| e^{-\lambda_n s} \right \} 
\]
and we need to take derivatives up to $l=0,...,k$.
We remind ourselves we need to place upper bounds on terms of the form $|\lambda_n|^{2k-2l}|d_n^{(l)}|^2$,
hence from above it suffices to bound terms of the form
\[
\int_0^\infty |\lambda_n|^{2(k-j-1)} |b_n^{(j)}|^2ds, j=0,...,k-1
\]
\[
\int_0^\infty e^{2\lambda_ns} |\lambda_n|^{2k} \left(\int_0^s b_n(s') e^{-\lambda_n (s')}ds'\right )^2 ds.
\]
The first term is bounded by the norm of $g$. The second term really is the $L^2$ norm of $a_n$ multiplied by $\lambda_n^{2k}$. We use a technique (probably much more well known) we found in \cite{donaldson}, Chapter 3. We first observe by Sobolev embedding our functions are at least $C^1$, so we can use the fundamental theorem of calculus. We then consider the defining equation for $a_n$
\[
\frac{d}{ds}a_n - \lambda_n a_n = b_n
\]
from which we get 
\[
\left(\frac{d}{ds}a_n\right)^2 + (\lambda_n a_n)^2 = b_n^2 + \lambda_n \frac{d}{ds} (a_n)^2.
\]
Integrate both sides from $[0,\infty)$ to get
\[
\int_0^\infty \left(\frac{d}{ds}a_n\right)^2 + (\lambda_n a_n)^2 ds = \int_0^\infty b_n^2 ds - \lambda_n|c_n|^2
\]
where we used continuity to apply fundamental theorem of calculus. We also used the fact for any fixed $b_n$, we have $\lim_{s\rightarrow \infty} e^{2\lambda_n s}\int_0^s b_n^2(s') e^{-2\lambda_n (s')} ds' \rightarrow 0$.
Hence we get
\[
\int_0^\infty |a_n|^2 ds \leq \frac{1}{|\lambda_n|^2} \int_0^\infty |b_n|^2 ds + |\lambda_n|^{-1} |c_n|^2.
\]
From which we deduce the second term is also bounded by norm of $g$ and $h$. Combining the above computations we see that our solution $f$ is indeed in $W^{k,2}([0,\infty)\times S^1)$, and the inequality 
\[
\|f\| \leq C (\|g\| + \|h\|)
\]
holds, from which we conclude the theorem.
\end{proof}
\begin{corollary}
Let $D_{\dt'0} $ denote the operator $D_\dt$ restricted at $s=0$, i.e. $D_{\dt'0} = \p_s -A(0,t) -\dt A (0,t) -d$, then for small enough $\dt >0$, the map $(\Pi_-, D_{\dt'0}): W^{k,2}([0,\infty)\times S^1) \rightarrow W_-^{k-1/2,2}(S^1) \times W^{k-1,2}([0,\infty)\times S^1)$ is an isomorphism with inverse $Q_0$ whose operator norm is uniformly bounded as $\dt \rightarrow 0$.
\end{corollary}
Using the above results we come to the theorem we will really need later on:
\begin{theorem}
For small enough $\dt>0$, the operator $(\Pi_-, D_\dt'): W^{k,2}([0,\infty)\times S^1) \rightarrow W_-^{k-1/2,2}(S^1) \times W^{k-1,2}([0,\infty)\times S^1)$ is an isomorphism whose inverse $Q$ has operator norm uniformly bounded with respect to $\dt \rightarrow 0$.
\end{theorem}
\begin{proof}
The proof is reminiscent of our original proof that $D_\dt$ (which we earlier denoted by $D_{J_\dt}$) has uniformly bounded inverse over the entire gradient trajectory, i.e. we approximate it by a sequence of operators over trivial cylinders.\\
Let $N$ be a large integer, choose $x_i$ for $i=0,1,..N$ so that $x_0$ is the $x$ coordinate on the Morse-Bott torus of $v(0,t)$, we have $|x_i-x_{i-1}| \leq 1/N$, and $x_N$ is distance $<1/N$ away from the critical point on the Morse-Bott torus corresponding to $v(\infty,t)$. We let $D_i':W^{k,2}(\bb{R}\times S^1) \rightarrow W^{k-1,2}(\bb{R}\times S^1) $ denote the linearization of $\db_J$ at the trivial cylinder located at $x_i$ on the Morse-Bott torus, and conjugated by exponential weights to remove exponential weight. In formulas we have
\[
D_i' = \p_s +J\p_t + S(x_i,t) -d.
\]
Uniformly in $N$ and $\dt>0$ and independently of $i$, the $D_i'$ are isomorphisms with uniformly bounded inverses $Q_i'$. Then similar to previous section we construct the glued operator $\# D_i'$ which satisfies
\begin{equation*}
    \|D'_\dt -\#D_i\| \leq C(1/N +\dt).
\end{equation*}
As before we  construct an approximate inverse to $\#D_i'$, which we call $Q_R'$
via the following diagram:
\begin{equation*}
\begin{tikzcd}
 W^{k-1,2}([0,\infty) \times S^1) \oplus W^{k-1/2,2}_-(S^1)] \oplus [W^{k-1,2}((-\infty,\infty)\times S^1)]_1\oplus\ldots \arrow[r,"Q_R'"] \arrow[d, "s_R"] & W^{k,2}([0,\infty) \times S^1)\\
 W^{k-1,2}([0,\infty) \times S^1) \oplus W^{k-1/2,2}_-(S^1) \oplus W^{k-1,2}((-\infty,\infty) \times S^1)_1 \oplus.. \arrow[r,"Q_0 \oplus Q_1..."] & W^{k,2}([0,\infty)\times S^1))_0 \oplus W^{k,2}_1 \ldots \arrow[u, "g_R"]
\end{tikzcd}
\end{equation*}
where we clarify 
\begin{equation*}
    s_R|_{W^{k-1/2,2}_-(S^1)} = \op{Id}.
\end{equation*}
The subscripts under $W^{k,2}((-\infty,\infty)\times S^1))_i$ denote the copies of Sobolev spaces in the direct sum. 
And the splitting map $s_R$ and the gluing map $g_R$ are defined exactly the same way we did in section \ref{gluing}. We observe 
as before this $Q_R'$ is uniformly bounded as $\dt \rightarrow 0$. Let's verify that this constructs an approximate inverse to $\# D_i'$. We first observe away from the gluing region
\begin{equation*}
    \# D_i' Q_R' \eta = \eta
\end{equation*}
and near the gluing region as before we have
\[
\|\#_N D_i' Q_R' \eta -\eta\| \leq C/N \|\eta\|.
\]
Hence we can construct a right inverse of $\#D_i$ with uniformly bounded norm. Next since $D_\dt'$ is a uniformly bounded small perturbation of $\#_N D_i'$, it also has a uniformly bounded right inverse. 

To see that this operator is injective,  since we don't have index calculations (versions of index theorems probably exist but we cannot find an easy reference) we take a more direct approach, in part inspired by the appendix of \cite{colin2021embedded}. Suppose $\zeta_\dt \in Ker (\Pi_-,D_\dt')$ is of norm 1, consider $s=R$, for definiteness we first assume for all $\dt$ the norm of $\zeta_\dt$ restricted to $0<s<R$ is $\geq 1/2$. Let  $\beta_R := \beta[-\infty,2R;R]$, and consider $\beta_R \zeta_\dt$. Then we can consider it to lie in the domain of $(\Pi_-,\p_s-A_0):W^{k,2}([0,\infty)\times S^1) \rightarrow W_-^{k-1/2,2}(S^1) \times W^{k-1,2}([0,\infty)\times S^1)$. To estimate its image under $(\Pi_-,\p_s-A_0)$, first consider
\[
\|D_\dt \beta_R \zeta_\dt\| = \|\beta_R'\zeta_\dt\|\leq C/R.
\]
Observing that over $s<2R$ we have $\|\p_s-A_0 - D_\dt'\|\leq C\dt$, we have 
\[
(\Pi_-,\p_s-A_0)(\beta_R\zeta_\dt) = (0,  (\p_s-A_0)\beta_R\zeta_\dt)
\]
where $\|(\p_s-A_0)\beta_R\zeta_\dt\| \leq C/R$, but then the element
\[
\beta_R \zeta_\dt - (\Pi_-,\p_s-A_0)^{-1}((\Pi_-,\p_s-A_0)(\beta_R\zeta_\dt)) \in W^{k,2}([0,\infty)\times S^1)
\]
has norm $>1/3$, but lies in the kernel of $(\Pi_-,\p_s+A_0)$, which is a contradiction.\\
Similarly, if the norm of $\zeta_\dt$ when restricted to $s>R$ is $\geq 1/2$ for all $\dt >0$, then we use a similar cut off function $\hat{\beta}_R : = \beta_{[R/2;R/2,\infty]}$ to view $\hat{\beta}_R \zeta_\dt$ as element of $(D_\dt', W^{k,2}(v^*TM))$ and use the same process to produce a nonzero kernel of $D_\dt'$, which cannot exist since $D_\dt'$ is an isomorphism.
\end{proof}

We now state the finite interval analogue of the above theorems for later use.
\begin{theorem}
Let $v$ be a gradient trajectory. Let $D_\dt'$ be the linearization of $\db_{J_\dt}$ over $v$ with exponential weight removed via conjugation as above. We consider its restriction to $(s,t) \in [0,CR] \times S^1$, and the Sobolev space $W^{k,2}( [0,CR] \times S^1,\bb{R}^4)$. Consider the two projections $\Pi_\pm$, where they project to the positive/ negative eigenvalues of $A_0: = -A(0,t)-d$. Then the map
\[
(\Pi_-,\Pi_+,D_\dt'): W^{k,2}([0,CR] \times S^1, \bb{R}^4)\longrightarrow W_-^{k-1/2,2}(S^1) \times W_+^{k-1/2,2}(S^1) \times W^{k-1,2}([0,CR] \times S^1, \bb{R}^4) 
\]
defined by
\[
f(s,t)\longrightarrow (\Pi_- f(0,t), \Pi_+f(CR,t), D_\dt' f)
\]
is an isomorphism whose inverse has uniformly bounded norm as $\dt \rightarrow 0$.
\end{theorem}
\begin{proof}
As before we first show the map $(\Pi_-,\Pi_+,\p_s - A_0): W^{k,2}([0,CR] \times S^1, \bb{R}^4) \rightarrow W^{k-1/2,2}(S^1) \times W^{k-1/2,2}(S^1) \times W^{k-1,2}([0,CR] \times S^1, \bb{R}^4)$ is an isomorphism with uniformly bounded inverse $Q_0$. This is essentially the same proof as before, i.e. if $(\Pi_-,\Pi_+,\p_s - A_0)f=(g_-,g_+,h)$ with $f=\sum a_n e_n$, $g_\pm = \sum c_{n\pm} e_n$ and $h=\sum b_n e_n$ then we still have the formulas
\[
a_n = e^{\lambda_n s}\int_0^s b_n(s') e^{-\lambda_n (s')} ds' + c_{n-} e^{\lambda_ns }
\]
for $n<0$ and 
\[
a_n = e^{\lambda_n (s-CR)}\int_{CR}^s b_n(s') e^{-\lambda_n (s')} ds' + c_{n+} e^{\lambda_n (s-CR) }
\]
for $n>0$. This already implies injectivity.
The same proof shows $Q_0$ exists and is uniformly bounded as $\dt \rightarrow 0$. To elaborate a bit further, we still need to estimate sizes of 3 kinds of terms. For definiteness we focus on the case $n<0$. The terms we need to consider are of forms
\begin{enumerate}
    \item $\int_0^{CR}|c_n|^2 |\lambda_n|^{2k} e^{2\lambda_n s}ds$
    \item$\int_0^{CR} |\lambda_n|^{2(k-j-1)} |b_n^{(j)}|^2ds, j=0,...,k-1$
    \item $\int_0^{CR} e^{2\lambda_ns} |\lambda_n|^{2k} (\int_0^s b_n(s') e^{-\lambda_n (s')}ds')^2 ds$.
\end{enumerate}
The first two terms work exactly the same way as before with $CR$ replacing $\infty$. The third term requires a bit more care in that when we tried to estimate the $L^2$ norm of $a_n$, the domain of integration is different giving us an extra term via integration by parts. So instead we have
\[
\lambda_n^2 \int_0^{CR} |a_n|^2 ds \leq \int_0^{CR} |b_n|^2 ds +\lambda_n \{ (a_n(CR))^2 - (a_n(0))^2\} .
\]
The additional term we need to estimate is $|\lambda_n| a_n^2(CR)$. This is upper bounded by
\[
|\lambda_n| \cdot |c_{n-}|^2 e^{2\lambda CR} + |\lambda_n| e^{2\lambda_n CR}\left(\int_0^{CR}b_n(s')  e^{-\lambda_n s'}ds'\right)^2.
\]
The first term above, after multiplying by $|\lambda_n|^{2k-2}$, is upper bounded by the norm of $g_-$ with the correct weight of $|\lambda_n|$. To examine the second term note it is bounded above by \[
|\lambda_n| e^{2\lambda_n CR}\int_0^{CR}b_n^2(s')  ds' \int_0^{CR} e^{-2\lambda_n s'}ds' \leq C \int_0^{CR}b_n^2(s')  ds' 
\]
by Cauchy-Schwartz, and this has the right weight of $|\lambda_n|$ so that when we multiply by $|\lambda_n|^{2k-2}$ it is upper bounded by the norm of $g$. This concludes the discussion of the third bullet point. Putting all of these together as in the semi-infinite case we see that the inverse is well defined, and its norm is uniformly bounded above as $\dt \rightarrow 0$. 

To conclude $ (\Pi_-,\Pi_+,D_\dt')$ has uniformly bounded inverse we need to be slightly careful, since as $\dt \rightarrow 0$ the domain changes. Since $\dt R \rightarrow 0$ the actual operator $(\Pi_-,\Pi_+,D_\dt')$ is a size $\leq R\dt$ perturbation of $(\Pi_-,\Pi_+,\p_s-A_0)$, then by the above we can construct a right inverse with uniform bound $Q$ for $(\Pi_-,\Pi_+,D_\dt')$ and this implies surjectivity. To show injectivity we proceed similarly as before, we assume $\zeta_\dt$ has norm $1$ and lives in the kernel of $(\Pi_-,\Pi_+,D_\dt')$, then $\|(\Pi_-,\Pi_+,\p_s -A_0)\zeta_\dt \| \leq CR\dt \|\zeta\| $, then the element
\[
\zeta_\dt-Q_0(\Pi_-,\Pi_+,\p_s -A_0)\zeta_\dt
\]
is an element of norm $>1/2$ in the kernel of $(\Pi_-,\Pi_+,\p_s -A_0)$, contradiction.
\end{proof}

\subsection{Surjectivity of gluing}
In this subsection we finally prove surjectivity of gluing in our simplified setting. The idea is that we shall extend our vector fields $\zeta_{*,\dt, K, (r,a,p)_i}, i=1,2,v $ so that they satisfy the set of equations $\Theta_i=0, \Theta_v=0$, subject to our choice of right inverses, which we constructed in the pregluing section. Then this shows our holomorphic curve $u_\dt$ can be realized as a solution of $\Theta_i=0, \Theta_v=0$. Since we proved such solution is unique, this shows gluing is surjective. We will first focus on extending the vector fields $\zeta_{*,\dt, K, (r,a,p)_i} $ over the intermediate finite gradient trajectory. The extension to semi-infinte trajectories is similar but independent of this process so will be treated separately. 

We remark additionally since there are exponential weights in place, we clarify our notation: when we write a vector field $\zeta_*$ without $'$, we think of it as living in some exponentially weighted Sobolev space, when we write $\zeta_*'$ we think of it as living in an unweighted space where the weight has been removed by multiplication with the exponential weight. When we write $W^{k,2,d}$ we will always mean the exponential weight $e^{ds}$; we will write $W^{k,2,w}$ if a more complicated weight is used.

Finally we remark that we will work with Sobolev exponent $p=2$, then extend our result for $p>2$, since all of our linear theory was only worked out for $p=2$.
We first observe by virtue of $u_\dt$ being $J_\dt$-holomorphic, the vector fields $\zeta_{*,\dt, K, (r,a,p)_i} $ already satisfy $\Theta_*=0$ at most places. We focus on what happens around $u^2$ and where $u^2$ is glued to the finite gradient cylinder simply for ease of notation. Entirely analogous statements hold for $u^1$.
\begin{proposition}
For $(s_v,t_v) \in [3R-K,N_{\dt K}-3R+K] \times S^1 $, the vector field $\zeta_{v,\dt, K, (r,a,p)_i} $ satisfies $\Theta_v=0$.\\
For $(s_v,t_v) \in[0, R-K] \times S^1 \cup \Sigma_{2R}$, the vector field $\zeta_{2,\dt, K, (r,a,p)_2} $ satisfies $\Theta_2=0$. An entirely analogous statement is true near $u_1$.
\end{proposition}
Because of our choice cut off functions, the global vector field $\zeta_{\dt,K,(r,a,p)_i}$ agrees with $\zeta_{2,\dt, K, (r,a,p)_2} $ at $s_v=R-K$ and $\zeta_{v,\dt, K, (r,a,p)_i}$ at $s_v=3R-K$. Here we use $(s_v,t_v)$ coordinates, and see next proposition for using $(s_2,t_2)$ coordinates.
\begin{proposition}\label{proposition:uniquedeform}
There exists a unique vector field $\xi$ of norm less than $\tilde{\ep}$ over $W^{k,2,w}([R-K,3R-K] \times S^1, \bb{R}^4)$ where $w=d(s+K)$ that satisfies \[
\Pi_-\xi(R-K,t_v) = \Pi_- \zeta_{\dt,K,(r,a,p)_i}(R-K,t_v)\]
\[
\Pi_+\xi (3R-K) = \Pi_+ \zeta_{\dt,K,(r,a,p)_i}(3R-K,t_v)\]
and $v_K + \xi$ is $J_\dt$-holomorphic. An entirely analogous statement holds near the ends of $u^1$.
\end{proposition}
\begin{proof}
The $J_\dt$-holomorphicity condition amounts to $\xi$ solving a equation of the form
\[
D_\dt \xi + \mcal{F}(\xi)=0
\]
where $\mcal{F}$ is an expression bounded above in $C^k$ by $C|\xi|^2 + |\xi\|\p_t \xi|$. We next remove the exponential weights to get an equation
\[
D_\dt' \xi' + \mcal{F'}(\xi')=0
\]
where we also have $\mcal{F}' \leq C|\xi'|^2 + |\xi'| |\p_t \xi'|$. 
Then finding a solution to this equation with prescribed boundary conditions amounts to finding a fixed point of the map
\[
I: W^{k,2}([R-K,3R-K] \times S^1, \bb{R}^4) \longrightarrow W^{k,2}([R-K,3R-K] \times S^1, \bb{R}^4)
\]
defined by
\[
I(\xi') = Q(\Pi_- \zeta'_{\dt,K,(r,a,p)_i}(R-K,t_v),\Pi_+ \zeta'_{\dt,K,(r,a,p)_i}(3R-K,t_v), -\mcal{F}'(\xi'))
\]
where $Q$ is inverse of the operator $(\Pi_-,\Pi_+,D_\dt')$, and $\zeta'_{\dt,K,(r,a,p)_i}(R-K,t_v)$ is $\zeta_{\dt,K,(r,a,p)_i}(R-K,t_v)$ multiplied with the inverse of the exponential weight. That $Q$ exists, is an isomorphism with uniformly bounded norm follows from previous section on linear analysis. That $I$ is a contraction mapping principle follows the fact $\mcal{F}$ is quadratic, the images of projection maps $\Pi_\pm$ are independent of the input $\xi'$, as well as the fact that the norm of $Q$ is uniformly bounded as $\dt \rightarrow 0$. The fact that $I$ sends $\tilde{\ep}$ ball to itself is inherited in the fact $\mcal{F}'$ is quadratic. We also need to recall from previous estimates that the $W^{2,p,w}$ norms of (hence its $C^0$ norm) $\zeta_{\dt,K,(r,a,p)_i}$ can be made arbitrarily small as we take $\dt \rightarrow 0$ and the norms of $\Pi_\pm$ and $Q$ are uniformly bounded, which ensure the image of the contraction map $I$ land easily in the $\tilde{\ep}$ ball in the codomain (in our previous propositions we used $\tilde{\ep}^2$ to bound the norms, and this is where it comes in).  The theorem now follows from contraction mapping principle.
\end{proof}
We next extend $\zeta_{2,\dt,K,(r,a,p)_2}$ and $\zeta_{v,\dt,K,(r,a,p)_i}$ to solutions of $\Theta_2=0$ and $\Theta_v=0$ for $s_v<R$ and $s_2>R$. We recall there is a slight subtlety in that near the pregluing at $u_2$ there is a twist in the domain, i.e. an identification $t_v = t_2 +(r_1-r_2)$, and the vector fields over $u_2$ have coordinates $t_2$. We will be careful to make this identification, though we remark it doesn't cause any difficulties.
\begin{proposition}
There are vector fields $\hat{\zeta}_{2,\dt,K,(r,a,p)_2}$, $\hat{\zeta}_{v,\dt,K,(r,a,p)_i}$ defined over $W^{k,2,d}([R,\infty)\times S^1, \bb{R}^4)$ and $W^{k,2,d}([-\infty,3R-K)\times S^1, \bb{R}^4)$ respectively, both of norm $<\tilde{\ep}$ so that
\[
\Theta_v(\hat{\zeta}_{2,\dt,K,(r,a,p)_2},\hat{\zeta}_{v,\dt,K,(r,a,p)_i})=0
\]
\[
\Theta_2(\hat{\zeta}_{2,\dt,K,(r,a,p)_2},\hat{\zeta}_{v,\dt,K,(r,a,p)_i})=0
\]
where the exponential weight looks like $e^{ds}$ over $W^{k,2,d}([R,\infty)\times S^1, \bb{R}^4)$ and $e^{d(s+K)}$ over $W^{k,2,d}([-\infty,3R-K)\times S^1, \bb{R}^4)$. Further, we have the boundary conditions that
\[
\Pi_-(\hat{\zeta}_{2,\dt,K,(r,a,p)_2}(s_2=R,t_2)) = \Pi_-(\zeta_{\dt,K,(r,a,p)_i} (s_v=R-K, t_v -(r_1-r_2))
\]
\[
\Pi_+(\hat{\zeta}_{v,\dt,K,(r,a,p)_i})(s_v=3R-K,t_v)= \Pi_+(\zeta_{\dt,K,(r,a,p)_i} (s_v=3R-K, t_v).
\]
\end{proposition}
\begin{proof}
We immediately switch to primed coordinates by removing the weight. In these primed coordinates the equations look like
\[
\Theta_v' = D_\dt'  +\mcal{F}'_v
\]
where $\mcal{F}'_v$ can be upper bounded by quadratic expressions of $\hat{\zeta}'_{2,\dt,K,(r,a,p)_2}$ , $\hat{\zeta}'_{v,\dt,K,(r,a,p)_i}$, and their $t$ derivatives, as in Remark \ref{quadratic_term}. Likewise for
\[
\Theta_2' = D_\dt' +\mcal{F}'_2 + \mcal{E}'.
\]
We remark for $\Theta_2'$, the operator $D_\dt'$ is the linearization of the $\db_{J_\dt}$ along $u_2$, with exponential weight removed via conjugation. The dependence on $(r_2,a_2,p_2)$ of the linearization appears in the quadratic term $\mcal{F}_2'$. The term $\mcal{E}'$ is the corresponding error term which takes the form in pregluing section (it is slightly different since we are using $D_\dt'$ as the linear operator, but this is of no consequence).  For $\Theta_v'$, $D_\dt'$ is the linearization of $\db_{J_\dt}$ along $v$ with exponential weights removed.

Then finding an solution to the equation is tantamount to finding a fixed point of the operator
\[
I:W^{k,2}([R,\infty)\times S^1, \bb{R}^4) \oplus W^{k,2}([-\infty,3R-K)\times S^1, \bb{R}^4) \longrightarrow W^{k,2}([R,\infty)\times S^1, \bb{R}^4) \oplus W^{k,2}([-\infty,3R-K)\times S^1, \bb{R}^4)
\]
defined by:
\begin{align*}
I(\hat{\zeta}'_{2,\dt,K,(r,a,p)_2},\hat{\zeta}'_{v,\dt,K,(r,a,p)_i}) =& \{Q_2(\Pi_-(\zeta'_{\dt,K,(r,a,p)_i} (s_v=R, t_v -(r_1-r_2)), -\mcal{F}'_2 -\mcal{E}'), \\
&Q_v(\Pi_+(\zeta'_{\dt,K,(r,a,p)_i} (s_v=3R-K, t_v)), -\mcal{F}_v')\}.
\end{align*}
Where $Q_v$ is the inverse to the pair $(D_\dt', \Pi_+):W^{k,2}([-\infty,3R-K)\times S^1, \bb{R}^4) \rightarrow W^{k-1,2}([-\infty,3R-K)\times S^1, \bb{R}^4) \oplus W_+^{k-1/2,2}(S^1)$ where $\Pi_+$ take place at $s_v=3R-K$. $Q_2$ is the inverse to $(D_\dt', \Pi_-):W^{k,2}([R,\infty)\times S^1, \bb{R}^4) \rightarrow W^{k-1,2}([R,\infty)\times S^1, \bb{R}^4) \oplus W_-^{k-1/2,2}(S^1)$ where $\Pi_-$ takes place at $s_2=R$. It follows as in the previous proposition that $I$ is a contraction, from the $\tilde{\ep}$ ball to itself, and translating back to the weighted Sobolev spaces proves our theorem.
\end{proof}
It follows from the above proposition and uniqueness that the extensions extend smoothly past $s_2=R$ and $s_v=3R-K$, and they recover $u_\dt$:
\begin{proposition}
The concatenation of $\hat{\zeta}_{2,\dt,K,(r,a,p)_2}$ at $s_2=R$ with $\zeta_{2,\dt,K,(r,a,p)_2}$ at $s_2=R$ is of class $C^k$, we denote the resulting vector field by $\zeta_{2,\dt,K,(r,a,p)_2}$, with slight abuse in notation. A similar story holds for $\zeta_{v,\dt,K,(r,a,p)_i}$. The resulting vector fields $\zeta_{2,\dt,K,(r,a,p)_2}$ and $\zeta_{i,\dt,K,(r,a,p)_i}$ are $\leq \ep$ in $W^{2,p,d}(u_2^*TM)$ and $W^{2,p,d}(v_k^*TM)$ respectively, and satisfy the pair of equations $\Theta_2=0,\Theta_v=0$.
\end{proposition}
\begin{proof}
By Proposition \ref{proposition:uniquedeform} there exists a unique vector field over $s_v\in [R-K,3R-K]$ satisfying the boundary conditions imposed by $\zeta_{\dt,K,(r,a,p)_i}$, and by whose deformation of $v_K$ makes the resulting surface $J_\dt$-holomorphic. But observe $\beta_2 \hat{\zeta}_{2,\dt,K,(r,a,p)_2} + \beta_v \hat{\zeta}_{v,\dt,K,(r,a,p)_i}$ satisfy these conditions as well by virtue of the defining conditions for the pair $\hat{\zeta}_{2,\dt,K,(r,a,p)_2}$, and $\hat{\zeta}_{v,\dt,K,(r,a,p)_i}$: that they are solutions of the pair of equations $\Theta_2=0,\Theta_v=0$. Hence we conclude $\beta_2 \hat{\zeta}_{2,\dt,K,(r,a,p)_2} + \beta_v \hat{\zeta}_{v,\dt,K,(r,a,p)_i}$ agrees with $\zeta_{\dt,K,(r,a,p)_i} $ over $s_v \in [R-K,3R-K]$, and by our choice of cut off functions this implies the concatenation of $\hat{\zeta}_{2,\dt,K,(r,a,p)_2}$ with $\zeta_{2,\dt,K,(r,a,p)_2}$ is smooth, and likewise for $\hat{\zeta}_{v,\dt,K,(r,a,p)_2}$ with $\zeta_{v,\dt,K,(r,a,p)_2}$. That we can take $p>2$ when we only constructed $\hat{\zeta}_{2,\dt,K,(r,a,p)_2}$ for $p=2$ follows from the fact that the vector fields and their first order derivatives have $C^0$ norm $<1$, and in which case we have their $W^{2,p}$ norm bounded above by powers of their $W^{2,2}$ norm. Hence in this case the extended parts $\hat{\zeta}_{2,\dt,K,(r,a,p)_2}$ and $\hat{\zeta}_{v,\dt,K,(r,a,p)_2}$ have their $W^{2,p,d}$ norm (over $u^{2*}TM$ and $v_K^*TM$ respectively) bounded above by $\tilde{\ep}^{2/p}$, and for small enough $\tilde{\ep}$ this lands in the $\ep$ ball in $W^{2,p,d}(u^{2*}TM)$ and $W^{2,p,d}(v_K^*TM)$ respectively.
\end{proof}

We make one additional remark that the equations $\Theta_2=0$ and $\Theta_v=0$ also depends on the asymptotic vectors $(r,a,p)_i$, but from our constructions these vectors have norm $<\ep'$. 
\subsection{Extension of solutions near semi-infinite gradient trajectories}
In this subsection we briefly outline how to carry out the above in the case where $u^1$ is glued to a semi-infinite gradient trajectory. This is simpler than the finite gradient case because we don't need our vector fields to lie in $H_0$. 

Recall our conventions, we assume $u_\dt$ degenerates into the cascade $\{u^1,u^2\}$. We focus on what happens near a positive puncture of $u^1$, which has coordinate $(s_1',t_1') \in [0,\infty) \times S^1$. For large $K>0$ we can recall the decomposition of the domain of $u^1$
\[
\{(s_1',t_1') \in [K,\infty)\times S^1\} \cup \Sigma_{1K} \cup \textup{other punctures of} \, u^1.
\]
We will not worry about the other punctures of $u^1$ and only talk about $\Sigma_{1K} \cup [K,\infty)\times S^1$. Similarly we can break down the domain of $u_\dt$ into
\[
\Sigma_{\dt K} \cup [0,\infty)\times S^1 \cup \textup{other parts of } \, u_\dt.
\]
As above we will only care about $\Sigma_{\dt K} \cup [K,\infty)\times S^1$ and neglect other parts of $u_\dt$. 

Following our previous conventions, we take $\ep>0$ to be the $\ep$ that controls the size of $\ep$ balls we use in the contraction mapping principle and it is fixed for any choice of $\dt>0$. Let the parameter $\ep'$ depend on $K,\dt$ and go to zero as $\dt>0$. We further introduce $\tilde{\ep}$ which we consider to be of the form $0<\ep'<<\tilde{\ep} << \ep$ to help bound various norms of vector fields as in the previous discussion.

The convergence to cascade implies for given $K$, we can choose small enough $\dt>0$ so that there exists a vector field $\zeta_{1\dt}$ and variation of complex structure of $u_1$ so that 
\[
u_\dt |_{\Sigma_{\dt K}}= \exp_{u^1,\dt j_1}(\zeta_{1\dt})
\]
and for given $K$, as $\dt \rightarrow 0$, the $C^k$ norm of $\zeta_{1\dt}\rightarrow 0$ (we will take $\dt$ small enough so that it is bounded by $\ep'$). We now turn our attention to the cylindrical end of $u_\dt$, which is of the form $[K,\infty)\times S^1$.
\begin{proposition}
For $K$ large, (which would take $\dt \rightarrow 0$ with it in order to satisfy our previous assumptions) we have $u_\dt|_{[K,\infty)\times S^1}$ converges in $C^\infty_{loc}$ to trivial cylinders. This is also true uniformly, i.e. for given $\ep''>0$, there is a $K$ large enough so that for every small enough values of $\dt>0$, $u_\dt |_{[k,k+1]\times S^1}$ is within $\ep''$ (in the $C^k$ norm) of a trivial cylinder of the form $\gamma \times \bb{R}$ for all values of $k$ so that $[k,k+1]\times S^1 \subset [K,\infty)\times S^1$. 
\end{proposition}
\begin{proof}
The same proof as Proposition \ref{glob}.
\end{proof}
Therefore we can choose a large enough $K$ so that when $u_\dt$ is restricted to $[K,\infty)\times S^1$ the conditions for asymptotic estimates are met, namely we have the following:
\begin{proposition}
We take $\ep''>0$ small enough so that previous convergence estimate near Morse-Bott torus applies. Then there is a large enough $K$, so that for small enough $\ep'$ (which depends on $K$), and for small enough $\dt>0$ (which depends on $\ep'$), there is a gradient trajectory $v$ defined over the cylinder $(s_v,t_v)\in [K,\infty)\times S^1$ so that there is a vector field $\zeta_v$ over $v_K$ so that
\[
u_\dt |_{[K,\infty)\times S^1} = \exp_{v_K}(\zeta_v)
\]
and the norm of $\zeta_v$ measured in $C^k$ satisfies the bound
\[
\|\zeta_v(s_1')\| \leq \|\zeta_v(K)\|^{2/p}_{L^2(S^1)} e^{-\lambda s}.
\]
\end{proposition}
Retracing our footsteps we now construct a preglued curve $u_{r,a,p}$. There is a trivial cylinder $\gamma \times \bb{R}$ so that over the interval $[0,R)$ the difference between $v_K$ and $\gamma \times \bb{R}$ is bounded above by $R\dt \rightarrow 0$, then choose $(r,a,p)$ so that for $u^1+(r,a,p)$, the difference between
\[
|(u^1+(r,a,p)) (R,t_1') - (\gamma \times \bb{R}) (R,t_1')| \leq C(e^{-DR} + R\dt) 
\]
then using this choice of $(r,a,p)$ we construct a pregluing, by gluing together $u^1+(r,a,p) (R,t_1')$ to $v_K(R-K,t_1')$ as we did in the section for gluing. This constructs for us a preglued map 
\[
u_{(r,a,p)}: (\Sigma_{K\dt},\dt j_1) \cup [K,\infty) \times S^1 \longrightarrow M
\]
so that there exists a vector field $\zeta$ so that over $(\Sigma_{K\dt},\dt j_1) \cup [K,\infty) \times S^1$
\[
u_\dt = \exp_{u_{r,a,p}}(\zeta).
\]
We also have estimates of the size of $\zeta$:
\begin{proposition}
Over the semi-infinite interval $[0,\infty)\times S^1 \subset (\Sigma_{K\dt},\dt j_1) \cup [K,\infty) \times S^1$ we impose the exponential weight $e^{ds}$, Then with respect to this exponential weight the $W^{k,p,d}$ norm of $\zeta$ is bounded above by:
\[
\|\zeta\| \leq C\ep' (C+e^{dK})+CRe^{dR}\dt + Ce^{-D'K}
\]
which can be made arbitrarily small by taking $K$ large and $\ep'\rightarrow 0$ as $\dt \rightarrow 0$. In particular for given $\tilde{\ep}^2$ we can upper bound its norm by $\tilde{\ep}^2$ by taking $\dt \rightarrow 0$.
\end{proposition}
\begin{proof}
The same proof as Proposition \ref{norm_estimate}.
\end{proof}
Then we truncate $\zeta$ into $\beta_v\zeta_v + \beta_1 \zeta_1$ so that the pair $(\zeta_v,\zeta_1)$ solves the equations $\Theta_u=0, \Theta_v=0$ near the cylindrical end. Here $\Theta_v$ is the equation living over the gradient cylinder $v_K$, and $\Theta_u$ lives over $u^1$. This process is entirely analogous to the previous section, to wit, we apply the contraction mapping principle over domains of the form $[R,\infty)\times S^1$ to show $(\zeta_v,\zeta_1)$ can be extended to solutions of $\Theta_u=0, \Theta_v=0$. We
note that we no longer need to worry whether $\zeta_v$ lands in $H_0$ because there is no such requirement over $\Theta_v$. Further $\zeta_1$ already lands inside image of $Q_1$ because we arranged this when we preglued the finite gradient trajectories, and that conclusion is unaffected by extensions of $\zeta_1$ near the cylindrical neck. Hence we apply the above to each of the ends of $u^i$ and in conjunction with the extension of vector fields along the finite gradient trajectories, we conclude :
\begin{proposition}
The gluing construction is surjective in the case of 2-level cascades with one finite gradient cylinder segment in the intermediate cascade level. To be more precise, suppose $u_\dt$ degenerates into a transverse and rigid 2-level cascade $\{u^1,u^2\}$, with only one finite gradient trajectory in the intermediate cascade level, then $u_\dt$ corresponds (up to translation) to the unique solution of the system of equations $\mathbf{\Theta_v}=0,\mathbf{\Theta}_u=0$ with our given choice of right inverses.
\end{proposition}

\subsection{Multiple level cascades}
In this subsection we generalize our result to multiple level cascades. The main subtlety is when two consecutive levels meet along multiple ends on an intermediate cascade level. Hence we take that up first in what follows. The main difficulty will be setting up notation.
\subsubsection{2-level cascade meeting along multiple ends}
Let $u_\dt$ converge to a 2 level cascade $\{u^1,u^2\}$. Each $u^i$ is not necessarily connected. As before we first consider the finite gradient trajectories. The maps $u^1$ and $u^2$ meet along $i\in \{1,...,N\}$ free ends in the middle. Consider the tuple $(i,j)$ where $i\in \{1,..,N\}$ label the specific end, and $j\in \{1,2\}$ denotes whether the end belongs to $u^1$ or $u^2$. We fix cylindical ends around each puncture of the form $[0,\pm \infty) \times S^1$ (we won't bother labelling these with $(i,j)$ to avoid further clutter of notation). Recall the vector spaces with asymptotic vectors we associate to each end that meets the intermediate cascade level of $u^i$, which we denote by $V_{(i,j)}$. Each of these vector spaces are spanned by asymptotic vectors $(\p_a,\p_z,\p_x)$, we denote an element of these vector spaces by triples $(r,a,p)_{(i,j)}$. Recall there is a submanifold
\[
\Delta \subset \oplus_{(i,j)} V_{(i,j)}
\]
within an $\ep$ ball of the origin of $\oplus_{(i,j)} V_{(i,j)}$ so that if we used elements in $\Delta$ we would be able to construct a pregluing from the domains of $u^1$ and $u^2$. Recall the reason we have to do this is that, as we recall, moving each $p_{(i,j)}$ affects the $a$ distance between $u^1$ and $u^2$, and we need to make sure that the ends $(i,j)$ can be matched together.

Now given the degeneration of a $J_\dt$-holomorphic curve $u_\dt$ to the cascade $\cas{u} =\{ u^1,u^2\}$, let $K>0$ be large enough, for each end $i$ there is a gradient flow trajectory $v_i$ so that when restricted to the segment $[-s_i,s_i]\times S^1$, we have that
\[
|v_i(s_i,t)-u^1(-K,t)|, |v_i(-s_i,t)-u^2(K,t)| \leq \ep'
\]
and $u_\dt$ is very close to the gradient flow $v_i$.
Then as before we can constructed a preglued curve $u_{(r,a,p)_{(i,j)}}$ so that over the domain of the preglued curve, we have
\[
u_\dt = \exp_{u_{(r,a,p)_{(i,j)}}} (\zeta)
\]
for $\zeta$ a global vector field whose norm can be taken to be arbitrarily small by picking $K$ large enough and (consequently) $\ep'$ and $\dt$ small enough. Again here we are only worrying about the finite gradient trajectories, we will worry about the semi-infinite trajectories later.

Then we can split $\zeta$ into a sum of several other vector fields as before, namely we can write 
\[
\zeta = \zeta_1+\zeta_2 + \sum_i \zeta^i
\]
where $\zeta_i \in W^{2,p,d}(u^{i*}TM)$ for $i=1,2$, and $\zeta^i \in W^{2,p,w}(v_i^*TM)$ for $i=1,...,N$. Using global $a$ translation of entire cascade we can ensure $\zeta_1 \in Im Q_1$, and using a global increase in $p_{(i,1)}-p_{(i,2)}$ inside $\Delta$ we can ensure also $\zeta_2\in ImQ_2$. Here the definition of $Q_i$ is as before: we take compact neighborhoods of $u^i$ and require the integral of $\la\zeta, \p_a\ra$ over these neighborhoods is zero. This defines a codimension one subspace which we take to be the image of $Q_i$.

Finally to ensure $\zeta^i \in H_{0i}$. As before by exponential decay estimates the actual size of vector fields to make $\zeta^i \in H_{0i}$ are negligible compared to $\ep'$. The difference from the previous case is that now there are multiple ends to worry about. To do this we need some understanding of $\Delta$ as a manifold.

Recall for near any point $x\in \Delta$, its tangent space is spanned by 
\begin{align*}
&\{r_{(i,1)}, r_{(i,2)}\}, \quad \{ a_{(i,j)} \},\\
&\{ p_{(1,1)}-p_{(1,2)} =T , p_{(i,1)}-p_{(i,2)} = T + \dt f_i(a_{(1,1)},a_{(2,1)},p_{(i,2)},a_{(i,1)},a_{(i,2)} )
\}
\end{align*}
the functions $f_i$ have uniformly bounded $C^1$ norm. The reason they appear is because ends meeting at different values of $f$ travel different amounts of $a$ distance for the same change of $p$, so a correction term is needed so the preglued curve can be constructed.\\
Recall that for $\zeta_i \in H_{0i}$ we must have the functionals 
\[
L_{i,*}(\zeta_i)=0,\quad  *=r,a,p.
\]
For $*=r$, this can be adjusted for each $i$ by a change in $\{r_{(i,1)} = r_{(i,2)}\}$. For $*=p,a$, we first repeat the previous construction for $i=1$ verbatim to get vector fields $\zeta^1 \in H_{01}$ while keeping $\zeta_i \in \Im Q_i $. I.e. we take $a_{(1,j)}, p_{(1,j)}$ so that it does not induce global translations in $a$ direction of the thick parts of $u^1,u^2$ as they enter the pregluing to ensure $\zeta^1\in H_{01}$. For any other $i>1$, the only constraint is $p_{(i,1)}-p_{(i,2)} = T + f_i(a_{(1,1)},a_{(2,1)},p_{i,2},a_{(i,1)},a_{(i,2)})$, hence as before we first change $p_{(i,j)}$ simultaneously by $\Delta p_i$ to make $L_p(\zeta^i)=0$ and in this process we adjust $a_{(i,j)}$ to make the pregluing condition still hold. Finally we change $a_{(i,j)}$ by the same amount $\Delta a_{i,j}$ to make $L_a(\zeta^i)=0$ while preserving the previous equalities.

Using the same kind of machinery to extend the vector field $\zeta$ to solutions of $\Theta_1=0, \Theta_2=0, \Theta_{v_i}=0$, and using the exactly the same set up for semi-infinite gradient trajectories,
we arrive at the follow proposition:
\begin{proposition}
If a sequence of $J_\dt$-holomorphic curves $u_\dt$ degenerates into a transverse and rigid 2-level cascade $\{u^1,u^2\}$, then $u_\dt$ comes from the unique solution to our gluing construction, namely, $\Theta_1=0, \Theta_2=0, \Theta_{v_i}=0$, subject to our choice of right inverses.
\end{proposition}
\subsubsection{General case}
The general case proceeds largely analogously to the 2-level case. We shall be very brief in sketching it out. Assuming $u_\dt$ degenerates into an $n$-level transverse and rigid cascade, $\cas{u}=\{u^1,..,u^n\}$, then we use the notation $v_{(i,j)}$ to denote a finite gradient trajectory connection between $u^i$ and $u^{i+1}$, connecting between the $j$th end in that intermediate cascade level. As before we can find a pregluing $u_{pre} : \Sigma \rightarrow M$ depending on the data $(r,a,p)_{(i,j)} \in \oplus V_{(i,j)}$ so that there is a global vector field $\zeta$ so that
\[
u_\dt = exp_{u_{pre}}(\zeta)
\]
where $\zeta$ has very small norm.
and as before we split 
\[
\zeta = \sum_i \zeta_i + \sum_{(i,j)}\zeta^{i,j}
\]
for the intermediate cascade levels. by adjusting the asymptotic vector fields $p_{i,j}$ we can ensure $\zeta_i \in Im Q_i$, and using the same kind of adjustments as above we make sure $\zeta^{(i,j)}\in H_{0ij}$. Finally using the same analysis we extend them to solutions of $\mathbf{\Theta}_* =0$ - here we just mean the system of equations we used in the gluing construction, using the same kind of analysis to take case of semi-infinite gradient ends. Hence we have proved:
\begin{theorem}
The gluing construction is surjective in the following sense: if $u_\dt$ converges to a $n$-level transverse and rigid cascade $\cas{u}$. Then for each such cascade $\cas{u}$ after our choice of right inverses we constructed a unique glued curve for $\dt>0$ small enough, and $u_\dt$ agrees with this glued curve up to translation in the symplectization direction.
\end{theorem}

\begin{remark}
We note our theorem about correspondence between transverse rigid cascades and rigid $J_\dt$-holomorphic curves studies the correspondence of a \emph{single} cascade and a \emph{single} curve. Usually in Floer theory one needs to show the collection of all transverse and rigid cascades is in bijection with the collection of all rigid holomorphic curves. To apply our results in these circumstances one usually needs some finiteness assumptions on the cascades and the holomorphic curves. For more details see \cite{Tips}.
\end{remark}

\appendix
\section{Appendix: SFT compactness for cascades}
In this appendix we outline the SFT compactness result required for the degeneration of holomorphic curves to cascades. 

We borrow heavily the results and notation from the original SFT compactness paper \cite{SFT}. In fact our compactness theorem will follow from their setup in combination with our estimates of how $J_\dt$-holomorphic curves behave near Morse-Bott tori. The behaviour of holomorphic curves near a Morse-Bot torus is already discussed in Chapter 4 of \cite{BourPhd}, and is implicit in \cite{oancea}, for example  their Section 4.2 and Appendix. Hence this appendix is more of an expository nature for the sake of completeness, and we will point out the differences and similarities between our results and theirs in the course of proving our version of SFT compactness theorem.

\subsection{Deligne-Mumford moduli space of Riemann surfaces}
We begin with a review of the Deligne- Mumford compactification of stable Riemann surfaces. Most of the material in this section is taken directly from Section 4 of \cite{SFT}, but is repeated for the convenience of the reader.

Let $\mathbf{S}= (S,j,M)$ denote a closed Riemann surface $S$ with complex structure $j$ with marked points set $M$. The surface is called \emph{stable} if $2g +\mu \geq 3$, where $g$ is the genus and $\mu :=|M|$ is the number of marked points. Stability implies the automorphism group of the surface $\mathbf{S}$ is finite.

The uniformization theorem equips $\mathbf{\dot{S}} : = (S\setminus M,j)$ with a unique complete hyperbolic metric of constant curvature and finite volume, which we denote by $h^{\mathbf{S}}$. Each puncture in $\mathbf{\dot{S}}$ corresponds to a cusp in the metric. We let $\mcal{M}_{g,\mu}$ denote the moduli space of Riemann surfaces of signature $(g,\mu)$.

\subsubsection{Thick-Thin decomposition}
Fix $\ep>0$, given a stable Riemann surface $\mf{S}$, for $x\in \mf{\dot{S}}$ let $\rho(x)$ denote the injectivity radius of $h^{\mf{S}}$ at $x$. As in Section 4 of \cite{SFT}, we denote by $\textup{Thin}_\ep(\mf{S})$ and $\textup{Thick}_\ep (\mf{S})$ its $\ep$-thin and thick parts where
\begin{equation*}
    \textup{Thin}_\ep(\mf{S}) :=\overline{ \{x\in \dot{\mf{S}} | \rho(x) < \ep \}}
\end{equation*}
\begin{equation*}
    \textup{Thick}_\ep(\mf{S}) := \{x\in \dot{\mf{S}} | \rho(x) \geq \ep \}.
\end{equation*}
It is a fact of hyperbolic geometry that there is a constant $\ep_0 = sinh^{-1}(1)$ so that for all $\ep <\ep_0$ we have each component of $\textup{Thin}_\ep(\mf{S})$ is conformally equivalent to either a finite cylinder of the form $[-L,L]\times S^1$ or semi-infintie cylinder $[0,\infty)\times S^1$. Each compact component of the form $C= [-L,L]\times S^1$ contains a unique closed geodesic of length equal to $2\rho(C)$, which we denote by $\Gamma_C$. Here we set $\rho(C):=\inf_{x\in C} \rho(x)$.

\subsubsection{Oriented blow up of punctured Riemann surface}
Given $\mf{S} = (S,j,M)$, let $z\in M$, then as in \cite{SFT} we can define the oriented blow up $S^z$ as the circle compactification of $S\setminus z $ with boundary $\Gamma_z = T_zS/ \bb{R}_+^*$. The complex structure $j$ defines an $S^1$ action on $\Gamma_z$. The surface $S^z$ comes equipped with a map $\pi:S^z \rightarrow S\setminus \{z\}$ which collapses the blown up circle. Given a finite set 
$M =\{z_1,...,z_k\}$ we can similarly define the blown up space $S^M$ with boundary circles $\Gamma_1,..,\Gamma_k$, with projection $\pi: S^M\rightarrow S\setminus M$ that collapses the boundary circles.

\subsection{Stable nodal Riemann surface}
See Section 4.4 in \cite{SFT}.
Let $\mf{S}=(S,j,M,D)$ be a possibly disconnected Riemann surface, where $M,D$ are both marked points, and the cardinality of $D$ is even. We write $D =\{\overline{d}_1,\underline{d_1},...,\overline{d_k},\underline{d_k}\}$. The \emph{nodal Riemann surface} is the tuple $\mf{S}=(S,j,M,D)$ under the additional equivalence relations so that each pair $(\overline{d_i},\underline{d_i})$ and the set of all such special pairs are unordered.

From a given nodal Riemann surface $\mf{S}=(S,j,M,D)$ we can construct the following singular surface
\begin{equation}
    \hat{S}_D := S/\{\overline{d_i}\sim \underline{d_i}, i=1,..,k\}.
\end{equation}
The arithmetic genus of a nodal Riemann surface is defined to be $g=\frac{1}{2}\# D-b_0 +\sum_{i=1}^{b_0}g_i+1$, where $b_0$
 is the number of connected components of $S$. 
The \emph{signature} of a nodal Riemann surface is given by the pair $(g,\mu)$, where $g$ is the arithmetic genus and $\mu$ is the number of marked points in $M$.

A stable Riemann surface $\mf{S}=(S,j,M,D)$ is called \emph{decorated} if for each pair $(\overline{d_i},\underline{d_i})$ we include the information of orientation reversing orthogonal map
\begin{equation}
    r_i: \overline{\Gamma_i}:= (T_{\overline{d_i}}S\setminus 0)/\bb{R}_{>0} \longrightarrow \underline{\Gamma_i}:= (T_{\underline{d_i}}S\setminus 0)/\bb{R}_{>0}.
\end{equation}
We also consider \emph{partially decorated Riemann surfaces} where such $r_i$ maps are only given for a subset $D'\subset D$.

We consider the moduli space of nodal Riemann surfaces $\overline{\mcal{M}}_{g,\mu}$ and decorated nodal Riemann surface $\overline{\mcal{M}}_{g,\mu}^\$$ of signature $(g,\mu)$. The moduli space of smooth Riemann surfaces of signature $(g,\mu)$, which we write as $\mcal{M}_{g,\mu}$, includes naturally in the above spaces. We refer the reader to Section 4.5 in \cite{SFT} for detailed topologies of these spaces. For us we only need the notion of convergence, which we summarize below.

Given a decorated stable nodal Riemann surface ($r$ denotes the decoration), which we write as $(\mf{S},r) = (S,j,M,D,r)$, we first take its oriented blow up along points of $D$, to obtain boundary circles $\overline{\Gamma_i}$ and $\underline{\Gamma_i}$ associated to the pair $\{\overline{d_i},\underline{d_i}\}$, then using the orthogonal maps $r_i$, we glue the resulting pieces together along $\overline{\Gamma_i},\underline{\Gamma_i}$ and call the resulting surface $S^{D,r}$. The glued copy of $\overline{\Gamma_i}$ and $\underline{\Gamma_i}$ is called $\Gamma_i$. The surface $S^{D,r}$ has the same genus as the arithmetic genus of $(\mf{S},r)$, and inherits a uniformizing metric from $h^{j,M\cup D}$, which we write as $h^{\mf{S}}$. The metric $h^{\mf{S}}$ is defined away from the $\Gamma_i$ and points of $M$. We can talk about the thick/thin components of $\mf{S}$ and view them as subsets of $\dot{S}^{D,r}$. Every compact component $C$ of $\overline{Thin_\epsilon(S)} \subset S^{D,r}$ is a compact annulus, it has either a closed geodesic which we denote by $\Gamma_C$, or one of the special circles $\Gamma_i$ constructed above, which we will also denote by $\Gamma_C$.

Let $(\mf{S}_n,r_n) = \{S_n,j_n,M_n,D_n,r_n\}$ be a sequence of decorated stable nodal Riemann surfaces. We say $(\mf{S}_n,r_n)$ converges to a nodal stable Riemann surface $(\mf{S},r) = (S,j,M,D,r)$ if for large enough $n$ there are diffeomorphisms $\phi_n: S^{D,r} \rightarrow S_n^{D_n,r_n}$ with $\phi_n(M_n) =M$, and the following conditions hold (Section 4.5 in \cite{SFT}):
\begin{itemize} 
    \item $\mf{CRS1}$ For all $n\geq 1$, the images $\phi_n(\Gamma_i)$ of the special circles $\Gamma_i \subset S^{D,r}$ for $i=1,..,k$ are special circles or closed geodesics of the metrics $h^{j_n,M_n\cup D_n}$ on $\dot{S}^{D_n,r_n}$. Moreover, all special circles on $S^{D_n,r_n}$ are among these images.
    \item $\mf{CRS2}$ $h_n\rightarrow h$ in $C^\infty_{loc}(S^{D,r}\setminus (M \cup \bigcup _1^k  \Gamma_i))$ where $h_n:= \phi_n^*h^{j_n,M_n\cup D_n}$.
    \item $\mf{CRS3}$ Given a component $C$ of $Thin_\ep(\mf{S})\subset \dot{S}^{D,r}$, which contains a special circle $\Gamma_i$, and given a point $c_i \in \Gamma_i$, we consider for every $n\geq 1$ the geodesic arc $\dt_i^n$ for the induced metric $h^n= \phi_n^*h^{j_n,M_n\cup D_n}$, which intersects $\Gamma_i$ orthogonally at $c_i$ (even though the distance is infinite it still makes sense to talk about geodesics intersecting orthogonally at infinity), and whose ends are contained in the $\ep$-thick parts of $h^n$. Then $C\cap \dt^n_i$ converges as $n\rightarrow \infty$ in $C^0$ as a continuous geodesic for $h^\mf{S}$ which passes through the point $c_i$.
\end{itemize}
We note that $\mf{CRS2}$ is equivalent to $\phi^*_nj_n \rightarrow j$ in $C^\infty_{loc}(S^{D,r}\setminus (M\cup \bigcup_1^k \Gamma_i))$.
The topology on 
$\overline{\mcal{M}}_{g,\mu}$ is defined to be the weakest topology for which the forgetful map $\overline{\mcal{M}}_{g,\mu}^\$ \rightarrow \overline{\mcal{M}}_{g,\mu}$ defined by forgetting the $r_i$ is continuous.
Finally the compactness theorem.
\begin{theorem}[Theorem 4.2 in \cite{SFT}]
The spaces $\overline{\mcal{M}}_{g,\mu}$ and $\overline{\mcal{M}}_{g,\mu}^\$$ are compact metric spaces that contain $\mcal{M}_{g,\mu}$, and are equal to the closure of the inclusion of $\mcal{M}_{g,\mu}$ (i.e. they are compactifications of $\mcal{M}_{g,\mu}$). As we are in a metric space, sequential compactness suffices.
\end{theorem}
We now state a proposition which we will later need to find all components of a holomorphic building/cascade.

\begin{proposition}[Proposition 4.3 in \cite{SFT}]
Let $\mf{S}_n= (S_n,j_n,M_n,D_n)$ be a sequence of smooth marked nodal Riemann surfaces of signature $(g,\mu)$ which converges to a nodal curve $\mf{S}=(S,j,M,D)$ of signature $(g,\mu)$. Suppose for each $n\geq 1$ we are given a pair of points $Y_n= \{y_n^1,y_n^2\}\subset S_n \setminus (M_n \cup D_n)$ so that
\begin{equation}
    dist_n(y_n^1,y_n^2) \longrightarrow 0
\end{equation}
where $dist_n$ is with respect to the hyperbolic metric $h^{j_n,M_n\cup D_n}$. Suppose in addition there is a sequence $R_n\rightarrow +\infty$ such that there exists injective holomorphic maps $\phi_n:D_{R_n}\rightarrow S_n\setminus (M_n\cup D_n)$ where $D_{R_n}$ is the disk in $\bb{C}$ with radius $R_n$,satisfying $\phi_n(0)=y_n^1, \phi_n(1)=y_n^2$. Then there exists a subsequence of the new sequence $\mf{S}_n'= (S_n,j_n,M_n\cup Y_n,D_n)$ which converges to a nodal curve $\mf{S}'=(S',j',M',D')$ of signature $(g,\mu+2)$, which has one or two additional spherical components. One of these components contains the marked points $y^1,y^2$, which corresponds to the sequence $y_n^1,y_n^2$. The possible cases are illustrated in Fig 5 of \cite{SFT}.
\end{proposition}
Stated in words (and also explained in Section 4 of \cite{SFT}), the scenarios are as follows. Let $r_n$ and $r$ be decorations on stable nodal Riemann surfaces $\mf{S}_n$ and $\mf{S}$ respectively, and we have $\mf{S}_n \rightarrow \mf{S}$ in the sense specified above, and $\phi_n: S^{D,r} \rightarrow S^{D_n,r_n}_n$ be the corresponding diffeomorphism. Let $\hat{S}_D$ be the singular nodal Riemann surface obtained from $\mf{S}$ by gluing together the nodal points $D$, and $\pi: S^{D,r} \rightarrow \hat{S}^D$ the associated projection. Let $Z_n=\pi(\phi^{-1}(Y_n)) \subset \hat{S}_D$. Then the following can happen:
\begin{itemize}
    \item The points $z_n^1,z_n^2 \in Z_n$ converge to a point $z_0$, which does not belong to $M$ or $D$. Then the limit $\mf{S}'$ of $\mf{S}_n'$ has an extra sphere attached at $z_0$ on which lie two extra points $y^1,y^2$.
    \item The points $z_n^1,z_n^2 \in Z_n$ converge to a marked point $m\in M$. In this case the limit $\mf{S}'$ is $\mf{S}$ with two extra sphere $T_1$ and $T_2$ attached. The sphere $T_1$ is attached at $\infty$ to the original $m$, and has $m$ at its zero. The sphere $T_2$ has its $\infty$ point attached to $1\in T_1$ and $y^1,y^2$ lie on $T_2$.
    \item The points $z_n^1,z_n^2\in Z_n$ converge to a double point $d$ corresponding to pair of points $\{x,x'\}\in D$. Then we insert a sphere $T_1$ between nodes $x,x'$, with $x$ attached to $\infty$ on $T_1$ and $x'$ attached to $0$. We insert a second sphere $T_2$ whose $\infty$ point is attached to $1$ in $T_1$, and the two points $y_1$ and $y_2$ lie on $T_2$.
\end{itemize}

\subsection{SFT compactness theorem for Morse-Bott degenerations}
We are now ready to state the SFT compactness theorem for degeneration of holomorphic curves to cascades.
We first state a more careful definition of holomorphic cascades of height 1, taking into account of decorations. This is also taken directly out of \cite{SFT}.

Recall $\lambda$ is a Morse-Bott contact form, and $\lambda_\dt$ is its perturbation defined by $\lambda_\dt =e^{\dt f}\lambda$. Fix $L>>0$, then for all $\dt>0$ small enough all Reeb orbits with action $<L$ come from critical points of $f$ on each Morse-Bott torus.
\begin{definition}[Section 11.2 in \cite{SFT}]
Suppose we are given:
\begin{itemize}
    \item $n$ nodal stable $J$-holomorphic curves
    \begin{equation}
        u^i:=(a^i,\hat{u}^i;S_i,D_i,\overline{Z}^i\cup \underline{Z_i}), i=1,...,n
    \end{equation}
    where $u^i$ is a $J$-holomorphic map from $S^i$ to $\bb{R} \times Y$. The map $a^i$ goes from $S_i$ to $\bb{R}$, the symplectization direction; and $\hat{u}^i$ is the map to $Y$. The sets $\overline{Z}^i, \underline{Z_i}$ correspond to punctures that are asymptotic to Reeb orbits hit by $u^i$ at $s=+\infty$ and $s= -\infty$ respectively. Let $\Gamma_i^\pm$ denote the corresponding boundary circle after blowing up the marked points $\overline{Z}^i$, $\underline{Z}_i$ respectively.
    \item $n+1$ collections of cylinders that are lifts of gradient trajectories of $f$ along the Morse-Bott tori, which we write as
    \begin{equation}
        \{G_{j,i, T_i}, j=1,..,p_i\}, \, \, i=0,..,n.
    \end{equation}
    In the above, $i$ indexes which collection the cylinder is in, and $j$ indexes specific element in that collection. Said another way, $i$ indexes the specific level in the cascade, and $j$ refers to which gradient flow segment in the level. The numbers $T_i$ denotes the flow time along gradient flow of $f$,
    with $T_0= -\infty$, $T_n=\infty$, $0\leq T_i < \infty$ for $i=1,..,n$. Denote the domain of the cylinders  by $\tilde{S}_i$, and $\tilde{\Gamma_i}^\pm$ their boundary circles corresponding at positive/negative ends. Even though the gradient flows may be finite, we think of these domain cylinders as infinitely long, and will think of them as living in the thin part of the glued domain Riemann surface. We do this even if the flow time $T_i$ is zero.
    \item Each positive puncture of $u^i$ (with $i=1,..,n$) is matched with a negative puncture of $u^{i-1}$, where they cover Reeb orbits on the same Morse-Bott torus of the same multiplicity. Between these two matched pair of punctures there is a unique gradient trajectory $G_{j,i,T_i}$ that connects between them after gradient flow of time $T_i$. Then there are orientation reversing diffeomorphisms $\Phi_i: \Gamma_i^+ \rightarrow \tilde{\Gamma}_i^-$ and $\Psi_{i-1}: \tilde{\Gamma}_{i-1}^+ \rightarrow \Gamma_i^- $ which are orthogonal on each boundary component.
    \item We glue the domains $S^{Z_i}_i$ and $\tilde{S}_i$ via the maps $\Phi_i$ and $\Psi_i$, to obtain a surface
    \begin{equation}
        \overline{S}:= \tilde{S}_0 \cup _{\Psi_0} S_z^{Z_1}\cup_{\Phi_1} \cup ...\cup_{\Phi_n}\tilde{S}_{n}.
    \end{equation}
    
    The maps $u^i$ and $G_{j,i, T_i}$ fit together to define a continous map from $\overline{u}:S\rightarrow \bb{R}\times Y$. Here for defining $\overline{u}$, on the gradient segment parts we use the literal gradient flow of $f$ without re-scaling by $\dt$.
    \item For the surface $\overline{S}$, we describe its complex structure. The idea is to keep the thin parts corresponding to $ \underline {Z}_i ,\overline{Z}^{i+1}$, and insert between them an infinite cylinder corresponding to the connecting gradient trajectory (with one marked point added to make it stable), with now $\underline {Z}_i ,\overline{Z}^{i+1}$ viewed as nodal points which comes with their own special circles. In our case, two points among $\underline {Z}_i ,\overline{Z}^{i+1}$ are viewed as nodal points for each gradient segment we are gluing in. Then the new decorated Riemann surface underlying the cascade can then be written as
    \begin{align}
        &(\overline{S},M =\bigcup M_i \cup \{\textup{one for each gradient flow segment}\},\\
        &D = \bigcup_i \overline {Z}^i\cup \underline{Z_i} \cup \{\textup{punctures corresponding to gradient flow cylinders}\})
    \end{align}
    We note this does not necessarily guarantee the stability of the underlying domain $\overline{S}$, since the definitions of stability of Riemann surface and $J$ holomorphic curves are distinct (see remark below). However we can always add several marked points $M'$ to make the underlying nodal Riemann surface stable.
\end{itemize}
Then we say we have defined a $n$ level $J$ holomorphic cascade of curves of height 1.
\end{definition}

\begin{remark}
In the above definition by stable we mean stable in the sense of $J$- holomorphic curves, i.e. no level consists purely of trivial cylinders, and if a component of $J$ holomorphic curve is constant, then the underlying domain for that component is stable in the sense of Riemann surfaces. We will treat the issue of stability of domain separately.
\end{remark}
The definition of height $k$  holomorphic cascade is very similar, we stack $k$ height 1 cascades on top of one another, and identify the edge punctures with maps like $\Psi$ and $\Phi$. 
The definition of when two cascades are equivalent to one another is identical to the definition in Section 7.2 of \cite{SFT} of when two SFT buildings are equivalent to one another, with the addition that we  think of gradient flow trajectories in the cascade as extra levels with marked points.

Then we are ready to state the SFT compactness result.
\begin{definition}[Section 11.2 of \cite{SFT}]
Let $(u_{\dt_n},S_n,j_n,M_n,D_n,r_n)$ be a sequence of $J_{\dt_n}$-holomorphic curves.  And let $\cas{u}=\{u^1,..,u^m\}$ be a height $k$ holomorphic cascade (we allow $k$ infinite flow times), and let $(S,j,M, D,r)$ be the underlying decorated nodal Riemann surface. We say $(u_{\dt_n},S_n,j_n,M_n,D_n,r_n)$ converges to $\cas{u}$ if we can find an extra set of marked points $M'$ on $(S,j,M, D,r)$, and an extra sequence of marked points $M_n'$ on $(u_{\dt_n},S_n,j_n,M_n,D_n,r_n)$ to make the underlying nodal Riemann surfaces stable, with diffeomorphisms $\phi_n: S^{D,r} \rightarrow S^{D_n,r_n}$ with $\phi_n(M)=M_n$ and $\phi_n (M') = M_n'$ satisfying the convergence definition of stable decorated Riemann surfaces in $\mf{CRS1-3}$, and suppose in addition the following conditions hold:
\begin{itemize}
    \item $\mf{CGHC1}$ For every component $C$ of $S^{D,r}\setminus \bigcup\Gamma_i$ which is not a cylinder coming from a gradient flow, identify the corresponding component $C_n$ in $(u_{\dt_n},S_n,j_n,M_n\cup M_n',D_n,r_n)$, and if we write $u_{\dt_n} = (a_{\dt_n}, \hat{u}_{\dt_n})$ and similarly for $\cas{u}$. Then $\hat{u}_{\dt_n}|_{C_n}$ converges to $\hat{\cas{u}}|_C$ in $C^\infty_{loc}(Y)$
    \item $\mf{CGHC2}$ If $C_{ij}$ is the union of components of $S^{D,r}\setminus \bigcup \Gamma_i$ which correspond to the same level $j$ of height $i$ of $\cas{u}$, (recall specifying height $i$ specifies a height 1 cascade, and $j$ labels the level within that height 1 cascade), then there exists sequences $c^{ij}_n$ so that $a_n\circ \phi_n-a-c^{ij}_n|_{C_{ij}}\rightarrow 0$ in $C^\infty_{loc}$ 
\end{itemize}
Then we say the sequence of $J_{\dt,n}$-holomorphic curves are convergent to the $J$-holomorphic cascade $\cas{u}$. 
\end{definition}

\begin{theorem}\label{SFTtheorem}[Theorem 11.4 in \cite{SFT}]
Let $\dt_n >0$ and $\dt_n \rightarrow 0$, let $(u_{\dt_n},S_n,j_n,M_n\cup M_n',D_n,r_n)$ denote a sequence of $J_{\dt_n}$-holomorphic curves of fixed signature and asymptotic to the same Reeb orbits (recall as long as $\dt>0$ and all orbits have energy $<L$, the orbit themselves do not depend on $\dt$), then there exists a subsequence that converges to a $J$-holomorphic cascade of height $k$.
\end{theorem}

The rest of this appendix is dedicated to the proof of this theorem. First a theorem on gradient bounds:

\begin{theorem}\label{gradbound}
(Gradient Bounds, Lemma 10.7 in \cite{SFT}) Let $\dt_n\rightarrow 0$ and $(u_{\dt_n},S_n,j_n,M_n\cup M_n',D_n,r_n)$ be a sequence of $J_\dt$-holomorphic curves with fixed signature, and the curves $u_{\dt_n}$ have a uniform energy bound $E$. Then by Deligne-Mumford compactness the domain $(S_n,j_n,M_n\cup M_n',D_n,r_n)$ converges in the sense of $\mf{CRS1-3}$ to a decorated Riemann surface
\[
(\mf{S},j,M,D,\overline{Z}\cup \underline{Z} ,r)
\]
Then there exists a constant $K$ which only depends on the upper energy bound $E$ so that if we add to each $M_n$ an additional collection of marked points
\[
Y_n = \{y^1_n, w^1_n,...,y_n^K,w_n^K \} \subset \dot{S_n} = S_n \setminus (M_n\cup\underline{Z_n}\cup \overline{Z_n})
\]
we have the following uniform gradient bound
\begin{equation}
    \|\nabla u_{\dt_n}(x)\| \leq \frac{C}{\rho(x)}
\end{equation}
Here the gradient $\nabla u_{\dt_n}(x)$ is measured with respect the fixed $\bb{R}$ invariant metric in $\bb{R}\times Y$ in conjunction with the hyperbolic metric on $\dot{S_n}\setminus Y_n$, and $\rho(x)$ is the injectivity radius of the hyperbolic metric at $x$.
\end{theorem}
\begin{proof}
The same proof as in for Lemma 10.7 in \cite{SFT} goes through. The only two observations needed are: first the analogue of lemma 5.11 continues to hold even as we take $J_{\dt_n}\rightarrow J$. The second observation is that due to Morse-Bott assumption each plane or sphere that bubbles off also has a lower nonzero bound on energy, so the set of points that bubbles off is finite.
\end{proof}

\begin{proof}[Proof of Theorem \ref{SFTtheorem}]
\textbf{Step 1}. We first discuss convergence in the thick parts. The discussion largely mirrors the discussion of \cite{SFT} Section 10.2.2. Following the setup in Theorem \ref{gradbound}, we assume we have added enough marked points to the converging Riemann surfaces $(S_n,j_n,M_n\cup M_n', D_n,r_n) $ so that the gradient bound holds everywhere away from the marked points. We call the limit of the sequence $(\mf{S},j,M,D,\overline{Z}\cup \underline{Z} ,r)$. We let $\Gamma_i$ denote the special circles on $S$, then we may assume
\begin{equation*}
    \|\nabla u_{\dt_n} \circ \phi_n (x)\| \leq \frac{C}{\rho(x)}, x\in S\setminus \bigcup \Gamma_i
\end{equation*}
where $\phi_n$ is the diffeomorphism from $S\rightarrow S_n$ (defined away from the nodes) given by the definition of convergence. Then by Azerla-Ascoli and Gromov-Schwarz we can extract a subsequence that over the thick parts of $S_n$ converges in $C^\infty_{loc}(Y\times \bb{R})$ to a $J$ holomorphic map defined on thick parts of $\mf{S}$.

\textbf{Step 2}. Next we consider what happens on the thin parts near a node, following \cite{SFT} Section 10.2.3. Let $C_1,..,C_N$ denote the connected components of $S\setminus \cup \Gamma_i$, we have from the above discussion that $u_{\dt_n} \circ \phi$ converges to $J$-holomorphic maps in $C^\infty_{loc}$ over each of $C_i$. Call these maps $u_i$. The point is in this description there may be levels missing near the nodes that connect between $C_i$, and by examining closely what happens near the nodes we recover the entire cascade.

The first case is if $u_i$ is bounded in $\bb{R}\times Y$ near one of the nodes, then by the removal of singularities theorem then $u_i$ extends continuously to the node. If $u_i$ is unbounded near a node then it must converge to a Reeb orbit, and extend continuously to the circle at infinity which compactifies the puncture.

Given a pair of components of $S\setminus \cup \Gamma_i$, call them $C_i$ and $C_j$, that are adjacent to each other. The behaviour of $u_i$ and $u_j$ could be quite different. The maps $u_i$ and $u_j$ may be asymptotic to either a point or a Reeb orbit at their connecting node, and even if they are both asymptotic to Reeb orbits they might not even be asymptotic to the same one (not even Reeb orbits that land on the same Morse-Bott torus). The reason for this, as explained above, is that there may be further degenerations of the curve $u_{\dt_n}$ near this node. To capture this idea, let $\gamma^\pm$  denote the the asymptotic limit of $u_i$ and $u_j$ (which could be either a point or a Reeb orbit), then there is a component $T_n^\ep$ of the $\ep$-thin region of the hyperbolic metric $h^n = \phi_n^*h^{j_n,M_n}$ on $S=S^{D,r}$, with conformal parametrization
\[
g_n^\ep: A_n^\ep:= [-N_n^\ep, N_n^\ep] \times S^1 \longrightarrow T^\ep_n
\]
such that in $C^\infty(S^1)$
\[
\lim_{\ep \longrightarrow 0} \lim_{n\longrightarrow \infty} \hat{u}_{\dt_n}\circ \phi_n \circ g_n^\ep|_{\pm N_n^\ep \times S^1} = \gamma ^\pm
\]
Note that $g_n^\ep$ can be chosen to satisfy
\[
\|\nabla g_n^\ep(x)\| \leq C \rho (g_n^\ep(x))
\]
where the norm on the left hand side is measured with respect to the flat metric on the source and the hyperbolic metric on the target. Then under this parametrization we have
\[
\|\nabla u_{\dt_n}\circ \phi_n \circ g_n^\ep \| \leq C
\]
and if we take a subsequence of $\ep_k \rightarrow 0$ (also denoted by $\ep_k$), then we have
\[
\lim_{k\rightarrow +\infty} \hat{u}_{\dt_n}\circ \phi_k \circ g^{\ep_k}_k (\pm N_k \times S^1) = \gamma^\pm
\]
and hence obtain a homotopically unique map $\Phi:[0,1]\times S^1 \rightarrow Y$ satisfying $\Phi(0\times S^1) = \gamma^-$ and $\Phi(1\times S^1) = \gamma_+$.

Sinice we have a uniform bound on $\|\nabla u_{\dt_n}\circ \phi_n \circ g_n^\ep \|$, by Azerla-Ascoli it converges in $C^\infty_{loc}$ to holomorphic curves (which specfic curve it converges to might depend on which shift we are considering on the domain, this is akin to the degeneration of a gradient flow line to a broken gradient flow line in the Morse case). We break it down in to cases:
\begin{itemize}
\item $\int_{\gamma^+} \lambda - \int_{\gamma^-} \lambda =0$
\item $\int_{\gamma^+} \lambda - \int_{\gamma^-} \lambda>0$.
\end{itemize}

\textbf{Case 1}: We first consider when $\int_{\gamma^+} \lambda - \int_{\gamma^-} \lambda =0$. If both $\gamma^\pm$ are points, then they are connected by a sequence of $J$-holomorphic spheres touching each other at nodes, however in symplectizations all $J$ holomorphic sphere are points, so in this case $\gamma^\pm$ are the same point.

If one of the ends (say $\gamma^+$) is a Reeb orbit, and $\Gamma^-$ is a point. The fact that $u_{\dt_n}\circ \phi_n \circ g_n^\ep$ converges in $C^\infty_{loc}$ implies we can find a $J$ holomorphic plane with $\gamma^+$ as its positive puncture. But then this $J$ holomorphic plane must have zero energy, which contradicts the Morse-Bott assumption.

The last case is if both $\gamma^\pm$ are Reeb orbits. Then they must lie on the same Morse-Bott Torus, because the energy of the segment $u_{\dt_n}\circ \phi_n \circ g_n^\ep |_{[-N_n,N_n]\times S^1}$ approaches zero as $n\rightarrow \infty$, and there is not enough energy to support a cylinder connecting Reeb orbits from one Morse-Bott torus to another, hence the Reeb orbits must lie on the same Morse-Bott torus.

Then by Lemma \ref{glob} for large enough $n$ the derivatives of $u_{\dt_n}\circ \phi_n \circ g_n^\ep$ are pointwise bounded by $\ep>0$, then by Propositions \ref{finite conv}, \ref{inf conv}, there is a number $T\in [0,\infty]$, a segment of  gradient trajectory of  $f$ of time $T_n$, lifted to be a $J_{\dt_n}$-holomorphic curve, which we denote by $v_{\dt_n}$, such that after taking a subsequence, over $[-N_n,N_n]\times S^1$, $u_{\dt_n}\circ \phi_n \circ g_n^\ep$ is $C^k([-N_n,N_n] \times S^1)$ close to $v_\dt$. 

To elaborate a bit more, we note Propositions \ref{finite conv}, \ref{inf conv} only apply when we can establish the segment of $J_\dt$-holomorphic cylinder is uniformly bounded away from all except at most one critical point of $f$. If this is not the case, then necessarily $T_n \rightarrow +\infty$. We assume $u_{\dt_n}\circ \phi_n \circ g_n^\ep(-N_n \times S^1)$ approaches the minimum of $f$ and $u_{\dt_n}\circ \phi_n \circ g_n^\ep(N_n\times S^1)$ approaches the maximum of $f$. Then we can choose $L_n \in [-N_n,N_n]$ so that $u_{\dt_n}\circ \phi_n \circ g_n^\ep$ restricted to $[-N_n,L_n]\times S^1$ is uniformly bounded away from the maximum of $f$, and its restriction to $[L_n,N_n]\times S^1$ is uniformly bounded away from the minimum of $f$. Then we apply Proposition \ref{inf conv} to find two semi-infinite gradient cylinders $v_{\dt_n -}$ and $v_{\dt_n+}$ to which the restriction of $u_{\dt_n}\circ \phi_n \circ g_n^\ep$ to $[-N_n,L_n]\times S^1$ (resp. $[L_n,N_n]\times S^1)$ converges in $C^k$ norm. By local convergence the restriction of $v_{\dt_n-}$ and $v_{\dt_n+}$ to $L_n\times S^1$ are $C^k$ close to each other, so for our purposes\footnote{Establishing exponential decay estimates for gradient flow lines that go from critical point of $f$ to another critical point requires more careful analysis, and is outside the scope of this work. Incidentally this is related to the problem of gluing cascades of height greater than 1 - we need to think more carefully about how we choose our Sobolev spaces and place our exponential weights.} \footnote{We mention here our work is simplified because our critical manifold (the manifold that parametrizes the space of Reeb orbits) is $S^1$, hence there are no broken gradient trajectories. In the case where the critical manifold is higher dimensional the analysis near broken trajectories is more delicate, and is outside the scope of the current work. However it is probably within the convex span of current technology.} we can take $v_{\dt_n}$ to be either $v_{\dt_n +}$ or $v_{\dt_n-}$. 

The estimate we proved for its local behaviour also tells us how to define the relevant gluing maps $\Phi_i$ and $\Psi_i$. We should also attach a marked point to this cylindrical segment to make the domain stable.

\textbf{Case 2}: We consider the second case $\int_{\gamma^+}\lambda -\int_{\gamma^-}\lambda >0$. We first observe that there is a lower bound on $\int_{\gamma^+}\lambda -\int_{\gamma^-}\lambda $ by the Morse-Bott assumption. We shall see that over $[-N_n,N_n]\times S^1$ the map $u_{\dt_n}\circ \phi_n \circ g_n^\ep$ converges to a sequence of $J$ holomorphic cylinders (and in the case where $\gamma^-$ is a point, a collection of cylinders followed by a $J$-holomorphic plane) connected by gradient cylinders along Morse-Bott tori similar to the previous case. We first observe by the gradient bounds that there is no bubbling off of holomorphic planes, and that over any compact domain of $[-N_n,N_n]\times S^1$ the sequence $u_{\dt_n}\circ \phi_n \circ g_n^\ep$ converges uniformly to a $J$ holomorphic curve. We note this is very similar to the case in Morse theory where a gradient trajectory converges to a broken gradient trajectory.

Let $h$ denote the minimal energy of a nontrivial $J$ holomorphic curve in the Morse-Bott setting, after successively taking subsequences, we pick out numbers $a_n^i,b_n^i \in [-N_n,N_n]$ which partition the interval $[-N_n,N_n]$ so that the following holds:
\begin{itemize}
    \item $b_n^i-a_n^i \rightarrow \infty$, $a^{i+1}_n - b^i_n \rightarrow \infty$.
    \item $u_{\dt_n}\circ \phi_n \circ g_n^\ep$ converges uniformly to a nontrivial $J$ holomorphic curve $u^i$ over $[a_n^i,b_n^i]$.
    \item $u_{\dt_n}\circ \phi_n \circ g_n^\ep$ restricted to $[b_n^i, a_n^{i+1}]$ has energy $< h/20$. We shall show that in fact the energy goes to zero as $n\rightarrow \infty$.
\end{itemize}
We first observe by our assumptions there must be an interval of the form $[a_n^i,b_n^i]$, because otherwise the entire interval $[-N_n,N_n]$ the curve $u_{\dt_n}\circ \phi_n \circ g_n^\ep$ has energy less than $h/20$, hence over each compact subset the curve converges to trivial cylinders, then this implies that $\gamma^+$ and $\gamma^-$ are on the same Morse-Bott torus, which is the situation in case 1. 

We note in the second bullet point we required \emph{uniform} convergence over the interval $[a_n^i,b_n^i]$, as opposed to the usual convergence over compact set. The reason is that if we had $C^\infty_{loc}$ convergence over an interval of the form  $[a_n^i,b_n^i]$, and by looking at different compact subsets in the domain we got convergence in $C^\infty_{loc}$ into two different curves, we could have inserted more partitions into the interval $[a_n^i,b_n^i]$ until the three bullet points above are achieved.

We also observe that the evaluation maps $ev^+(u^i)$ and $ev^-(u^{i+1})$ land in the same Morse-Bott torus, since over $[b_n^i,a^{i+1}_n]$ the energy is too small to cross from one Morse-Bott torus to the next, hence in fact the energy of $u_{\dt_n}\circ \phi_n \circ g_n^\ep$ over $[b_n^i,a^{i+1}_n]$ converges to zero. As in the proof of the previous case, over the interval $[b_n^i,a^{i+1}_n]$, the map $u_{\dt_n}\circ \phi_n \circ g_n^\ep$ converges uniformly to a gradient flow trajectory $v^i$, as was shown in the previous case. As a technical point, once we have found $u^i$ and $u^{i+1}$, we should identify the region when they first enter the Morse-Bott torus, and perform the analysis as we did in propositions \ref{finite conv}, \ref{inf conv} to identify the correct length of the gradient trajectory (this may result in us moving the partition points $b_n^i,a_n^i$ to lengthen the segment that we think of as being the gradient trajectory). We add marked points to both domains of $v^i$ and $u^i$ to make the domain stable, and the gluing map $\Phi$ and $\Psi$ are naturally supplied by considerations of convergence.

\textbf{Step 3}: We remark that the behaviour of $u_{\dt_n}$ near a puncture (either symptotic to a Reeb orbit or to a point), around which the hyperbolic metric produces another thin region, is entirely analogous to the analysis we performed above: we can choose a cylindrical neighborhood of the form $[0,\infty)\times S^1$ or $(-\infty,0]\times S^1$, and along this neighborhood the curve $u_{\dt_n}$ reparametrized as above degenerates into a cascade of cylinders connected by Morse flow lines. The only additional piece of information which follows readily is that if in the original $u_{\dt_n}$ is asymptotic to $\gamma$ near this puncture, then the end of the chain of holomorphic cylinders and gradient trajectories also is also asymptotic to $\gamma$.

\textbf{Step 4}: Finally we discuss the level structure. 

Recall that after the previous modifications the domain of $u_{\dt_n}$ converges to a stable Riemann surface $(S,j,M,D, \overline{Z}\cup \underline{Z},r)$ so that each connected component of $S\setminus D_i$ is assigned either a $J$-holomorphic curve $u$, or a gradient cylinder $v$. We label the components of the domain associated with $J$-holomorphic curves $C_i$ and those labeled with gradient flow cylinders $\tilde{C}_i$. Now for each $C_i$ we pick a point $x_i \in C_i$ and define an ordering that
\[
C_i\leq C_j
\]
if 
\[
a_n(\phi_n(x_i))-a_n(\phi_n(x_j))< \infty
\]
and if $C_i \leq C_j$ and $C_j \leq C_i$, we say $C_i \sim C_j$. This ordering defines a level structure as in the SFT picture, then we add in the gradient flow $v_j$ by hand at each of the levels. We note that if a gradient flow flows across multiple levels of holomorphic curves, then it will appear at these levels as a trivial cylinder. With this convention we see that then the flow time at each intermediate cascade level is the same for all Morse-Bott tori on that level (if a gradient flow needed to flow longer it would simply appear as a trivial cylinder). Then we have the SFT compactness result as desired.
\end{proof}

\newpage
\printbibliography
\end{document}